\documentclass[english,onecolumn]{article}
\usepackage{lmodern}
\usepackage[T1]{fontenc}
\usepackage[latin9]{inputenc}
\usepackage{geometry}
\usepackage{array}
\usepackage{float}
\usepackage{rotfloat}
\usepackage{enumitem}
\usepackage{multirow}
\usepackage{amsmath}
\usepackage{amsthm}
\usepackage{amssymb}
\usepackage{cancel}
\usepackage{graphicx}
\usepackage{esint}

\makeatletter

\providecommand{\tabularnewline}{\\}
\floatstyle{ruled}
\newfloat{algorithm}{tbp}{loa}
\providecommand{\algorithmname}{Algorithm}
\floatname{algorithm}{\protect\algorithmname}

\numberwithin{equation}{section}
\numberwithin{figure}{section}
\theoremstyle{plain}
\newtheorem{thm}{\protect\theoremname}
\theoremstyle{definition}
\newtheorem{defn}[thm]{\protect\definitionname}
\theoremstyle{plain}
\newtheorem{prop}[thm]{\protect\propositionname}
\theoremstyle{plain}
\newtheorem{cor}[thm]{\protect\corollaryname}
\theoremstyle{remark}
\newtheorem{rem}[thm]{\protect\remarkname}
\theoremstyle{plain}
\newtheorem{lem}[thm]{\protect\lemmaname}
\newlist{casenv}{enumerate}{4}
\setlist[casenv]{leftmargin=*,align=left,widest={iiii}}
\setlist[casenv,1]{label={{\itshape\ \casename} \arabic*.},ref=\arabic*}
\setlist[casenv,2]{label={{\itshape\ \casename} \roman*.},ref=\roman*}
\setlist[casenv,3]{label={{\itshape\ \casename\ \alph*.}},ref=\alph*}
\setlist[casenv,4]{label={{\itshape\ \casename} \arabic*.},ref=\arabic*}

\usepackage{mathrsfs}
\usepackage{algorithm,algpseudocode}

\usepackage{colortbl}
\definecolor{lightgray}{rgb}{0.9,0.9,0.9}
\definecolor{lightred}{rgb}{1,0.8,0.8}
\definecolor{lightgreen}{rgb}{0.6,1,0.6}
\definecolor{lightyellow}{rgb}{1,1,0.5}
\definecolor{lightgrey}{rgb}{0.8,0.8,0.8}

\allowdisplaybreaks[1]

\makeatother

\usepackage{babel}
\providecommand{\casename}{Case}
\providecommand{\corollaryname}{Corollary}
\providecommand{\definitionname}{Definition}
\providecommand{\lemmaname}{Lemma}
\providecommand{\propositionname}{Proposition}
\providecommand{\remarkname}{Remark}
\providecommand{\theoremname}{Theorem}

\begin{document}
\title{Pairwise Multi-marginal Optimal Transport and Embedding for Earth
Mover's Distance}
\author{Cheuk Ting Li and Venkat Anantharam\\
EECS, UC Berkeley, Berkeley, CA, USA\\
Email: ctli@berkeley.edu, ananth@eecs.berkeley.edu}
\maketitle
\begin{abstract}
We investigate the problem of pairwise multi-marginal optimal transport,
that is, given a collection of probability distributions $\{P_{\alpha}\}$
on a Polish space $\mathcal{X}$, to find a coupling $\{X_{\alpha}\}$,
$X_{\alpha}\sim P_{\alpha}$, such that $\mathbf{E}[c(X_{\alpha},X_{\beta})]\le r\inf_{X\sim P_{\alpha},Y\sim P_{\beta}}\mathbf{E}[c(X,Y)]$
for all $\alpha,\beta$, where $c$ is a cost function and $r\ge1$.
In other words, every pair $(X_{\alpha},X_{\beta})$ has an expected
cost at most a factor of $r$ from its lowest possible value. This
can be regarded as a locality sensitive hash function for probability
distributions, and has applications such as robust and distributed
computation of transport plans. It can also be considered as a bi-Lipschitz
embedding of the collection of probability distributions into the
space of random variables taking values on $\mathcal{X}$. For $c(x,y)=\Vert x-y\Vert_{2}^{q}$
on $\mathbb{R}^{n}$, where $q>0$, we show that a finite $r$ is
attainable if and only if either $n=1$ or $0<q<1$. As $n\to\infty$,
the growth rate of the smallest possible $r$ is exactly $\Theta(n^{q/2})$
if $0<q<1$. Hence, the metric space of probability distributions
on $\mathbb{R}^{n}$ with finite $q$-th absolute moments, $0<q<1$,
with the earth mover's distance (or 1-Wasserstein distance) with respect
to the snowflake metric $c(x,y)=\Vert x-y\Vert_{2}^{q}$, is bi-Lipschitz
embeddable into $L_{1}$ with distortion $O(n^{q/2})$. If we consider
$c(x,y)=\Vert x-y\Vert_{2}$ (i.e., $q=1$) on the grid $[0..s]^{n}$
instead of $\mathbb{R}^{n}$, then $r=O(\sqrt{n}\log s)$ is attainable,
which implies the embeddability of the space of probability distributions
on $[0..s]^{n}$ into $L_{1}$ with distortion $O(\sqrt{n}\log s)$,
and improves upon the $O(n\log s)$ result by Indyk and Thaper. The
case of the discrete metric cost $c(x,y)=\mathbf{1}\{x\neq y\}$ and
more general metric and ultrametric costs are also investigated.

\end{abstract}

\section{Introduction}

The Monge-Kantorovich optimal transport problem \cite{monge1781memoire,kantorovich1942translocation}
 is to find a coupling between two probability distributions $P_{1},P_{2}$
over the Polish space $\mathcal{X}$ such that the expected cost $\mathbf{E}[c(X_{1},X_{2})]$
is minimized, where the marginal distributions satisfy $X_{i}\sim P_{i}$,
and $c:\mathcal{X}^{2}\to\mathbb{R}_{\ge0}$ is a cost function. Its
generalization to more than two marginal distributions has been studied
by Kellerer \cite{kellerer1984duality}, Gangbo and {\'S}wi{\k{e}}ch
\cite{gangbo1998optimal}, Heinich \cite{heinich2002probleme}, Carlier
\cite{carlier2003class}, and Pass \cite{pass2011uniqueness,pass2012local}.
Given a collection of probability distributions $P_{1},\ldots,P_{m}$
over $\mathcal{X}$, and the cost function $c_{\mathrm{m}}:\mathcal{X}^{m}\to\mathbb{R}_{\ge0}$,
the multi-marginal optimal transport problem is to minimize $\mathbf{E}[c_{\mathrm{m}}(X_{1},\ldots,X_{m})]$
over all couplings of $P_{1},\ldots,P_{m}$ (i.e., $X_{i}\sim P_{i}$).
Multi-marginal optimal transport has various applications, for example
in economics, condensed matter physics and image processing (see the
references in \cite{pass2012multi}).

In this paper, we consider a different generalization of the optimal
transport problem to more than two marginal distributions $P_{1},\ldots,P_{m}$,
which we call the \emph{pairwise multi-marginal optimal transport
problem}. Let $c:\mathcal{X}^{2}\to\mathbb{R}_{\ge0}$ be a symmetric
cost function. We study the set of achievable tuples of pairwise
costs $\{\mathbf{E}[c(X_{\alpha},X_{\beta})]\}_{\alpha,\beta}$ over
couplings of $\{P_{\alpha}\}_{\alpha}$.\footnote{This setting is related to the original multi-marginal optimal transport
problem in the sense that finding the set of achievable $m(m-1)/2$
tuples $\{\mathbf{E}[c(X_{\alpha},X_{\beta})]\}_{\alpha<\beta}$ is
equivalent to finding the optimal cost for the cost function $c_{\mathrm{m}}(x_{1},\ldots,x_{m})=\sum_{i<j}\nu_{i,j}c(x_{i},x_{j})$
for all values of $\{\nu_{i,j}\}$, i.e. the Lagrangian formulation.
The case where $c(x,y)=\Vert x-y\Vert_{2}^{2}$, $c_{\mathrm{m}}(x_{1},\ldots,x_{m})=\sum_{i<j}c(x_{i},x_{j})$
was studied in \cite{gangbo1998optimal}.} Specifically, we are interested in finding $r\ge1$ such that there
exists a coupling $\{X_{\alpha}\}_{\alpha}$ satisfying
\begin{equation}
C_{c}^{*}(P_{\alpha},P_{\beta})\le\mathbf{E}[c(X_{\alpha},X_{\beta})]\le rC_{c}^{*}(P_{\alpha},P_{\beta})\label{eq:ecr_bd}
\end{equation}
for all $\alpha,\beta$, where $C_{c}^{*}(P_{\alpha},P_{\beta}):=\inf_{X\sim P_{\alpha},\,Y\sim P_{\beta}}\mathbf{E}[c(X,Y)]$
is the optimal value of the original 2-marginal optimal transport
problem, which is the earth mover's distance (or 1-Wasserstein distance)
when $c$ is a metric. In other words, the expected cost between
each pair $(X_{\alpha},X_{\beta})$ is within a factor of $r$ from
the lowest possible expected cost when we only consider couplings
of $P_{\alpha},P_{\beta}$ and ignore the rest. This setting can be
generalized to an arbitrary collection of probability distributions
$\{P_{\alpha}\}_{\alpha}$, or even the collection of all distributions
over $\mathcal{X}$ in some cases.

This setting has appeared in other forms in the literature. For instance,
a coupling achieving \eqref{eq:ecr_bd} can be considered as a sketch
or a locality sensitive hash function for the estimation of $C_{c}^{*}$
in the settings in \cite{charikar2002similarity,indyk2003fast,lv2004image,andoni2009efficient}.
It is also utilized in the labeling problem for classification studied
in \cite{kleinberg2002approximation,archer2004approximate}. Furthermore,
this setting is connected to the embedding of metric spaces \cite{indyk2003fast,khot2006nonembeddability,naor2007planar},
in the sense that a coupling achieving \eqref{eq:ecr_bd} can be considered
as a bi-Lipschitz embedding of the collection of probability distributions
into the space of random variables taking values on $\mathcal{X}$
(see Section \ref{sec:geom}). Nevertheless, these previous works
do not regard their constructions as couplings, and are often only
applicable for finite spaces $\mathcal{X}$. The coupling interpretation
allows us to discover more applications in optimal transport, namely
robust, distributed and online computation of transport plans, and
a multi-agent matching problem with a fairness requirement. Refer
to Section \ref{sec:applications} for a discussion.

The case where $\mathcal{X}$ is finite, $\{P_{\alpha}\}_{\alpha}$
is the collection of all probability distributions over $\mathcal{X}$,
and $c$ is a metric was studied in Kleinberg and Tardos' work on
metric labeling \cite{kleinberg2002approximation}, and Charikar's
work on locality sensitive hashing \cite{charikar2002similarity}\footnote{We remark that \cite{kleinberg2002approximation} and \cite{charikar2002similarity}
do not regard their constructions as a coupling of all probability
distributions over $\mathcal{X}$. While they provide a way to compute
$X_{\alpha}\sim P_{\alpha}$ given any probability distribution $P_{\alpha}$,
the distribution of $X_{\alpha}$ is used as a means to certain goals
(metric labeling in \cite{kleinberg2002approximation}, and  the
approximation of earth mover's distance in \cite{charikar2002similarity}),
and is not considered as a goal itself. In this paper, we consider
$\{X_{\alpha}\}_{\alpha}$ as a coupling for the pairwise multi-marginal
optimal transport problem, and provide the first systematic study
on this problem.}, which show that $r=O(\log|\mathcal{X}|\log\log|\mathcal{X}|)$ is
achievable. Their results can be improved to $O(\log|\mathcal{X}|)$
using a tighter bound on the approximation of the metric by a tree
metric given by Fakcharoenphol, Rao and Talwar \cite{fakcharoenphol2004tight}.
In this paper, we show that
\[
r=55.7(1+\log|\mathcal{X}|)=O(\log|\mathcal{X}|)
\]
 is achievable, using a completely different (and arguably simpler)
construction compared to the tree metric construction.

The case where $c(x,y)=\mathbf{1}\{x\neq y\}$ is the discrete metric
was studied by Kleinberg and Tardos \cite{kleinberg2002approximation}
(only for finite $\mathcal{X}$), and also recently by Angel and Spinka
\cite{angel2019pairwise}.\footnote{Angel and Spinka also raised the question for general cost functions
\cite[Question 19]{angel2019pairwise}, though they did not give any
result for cost functions other than the discrete metric.} Their result implies that $r=2$ is attainable for the discrete metric,
that is, for any countable collection of probability distributions
$\{P_{\alpha}\}_{\alpha}$, there exists a coupling such that, for
all $\alpha,\beta$,
\[
\mathbf{P}(X_{\alpha}\neq X_{\beta})\le2d_{\mathrm{TV}}(P_{\alpha},P_{\beta}),
\]
where $d_{\mathrm{TV}}$ is the total variation distance. One of the
constructions in \cite{angel2019pairwise} (Coupling II) coincides
with the Poisson functional representation previously studied by Li
and El Gamal \cite{sfrl_trans}, and Li and Anantharam \cite{li2018unified}
(note that the construction in \cite{sfrl_trans,li2018unified} applies
to general distributions, while Coupling II in \cite{angel2019pairwise}
is only for discrete distributions). In this paper, we show the result
in \cite{angel2019pairwise} for any Polish space $\mathcal{X}$,
as a corollary of the exact formula for the Poisson matching lemma
in \cite{li2018unified}.  To the best of our knowledge, the pairwise
multi-marginal optimal transport setting (or any equivalent setting)
has not been studied for any other infinite spaces $\mathcal{X}$
and cost functions $c$.

In this paper, we present a general theorem which gives an upper bound
on $r$ when the cost function $c(x,y)=(d(x,y))^{q}$ is a snowflake
metric \cite{assouad1983plongements}, i.e., a power of a metric $d$,
where $0<q<1$ (see Theorem \ref{thm:metric_pow}). As a consequence,
when $c(x,y)=\Vert x-y\Vert_{2}^{q}$ over $\mathcal{X}=\mathbb{R}^{n}$
where $q>0$, and $\{P_{\alpha}\}_{\alpha}$ is the collection of
all probability distributions over $\mathbb{R}^{n}$, we can show
that 
\[
r=\frac{10.55}{1-q}n^{q/2}
\]
is attainable in \eqref{eq:ecr_bd} when $0<q<1$. We prove that such
$r$ does not exist when $n\ge2$ and $q\ge1$.  Moreover, we show
that, as $n\to\infty$, the growth rate of the smallest possible $r$
is exactly $\Theta(n^{q/2})$ when $0<q<1$. As a consequence, we
can show that the metric space of probability distributions on $\mathbb{R}^{n}$
with finite $q$-th absolute moments (i.e., probability distributions
$P$ with $\mathbf{E}_{X\sim P}[\Vert X\Vert_{2}^{q}]<\infty$), $0<q<1$,
with the earth mover's distance (or 1-Wasserstein distance) $C_{c}^{*}$
where $c(x,y)=\Vert x-y\Vert_{2}^{q}$, is embeddable into $L_{1}$
(the space of Lebesgue measurable functions $f:[0,1]\to\mathbb{R}$
with $\Vert f\Vert_{1}<\infty$) with a bi-Lipschitz embedding function
with distortion $10.55n^{q/2}/(1-q)$ (see Section \ref{sec:geom}).
In contrast, when $q=1$, $n\ge2$, the non-existence of such an $L_{1}$
embedding is proved in \cite{naor2007planar}.\footnote{We remark that the embedding for earth mover's distance with snowflake
metrics is also studied in \cite{bavckurs2014better}, though they
focus on the case where the probability distributions are discrete
and supported on small sets.}

If we consider only the grid points $\mathcal{X}=[0..s]^{n}$, $s\in\mathbb{N}$,
$c(x,y)=\Vert x-y\Vert_{2}$, and $\{P_{\alpha}\}_{\alpha}$ is the
collection of all probability distributions over $[0..s]^{n}$, then
we show that
\begin{align*}
r & =28.66\sqrt{n}\log(s+1)
\end{align*}
is attainable in \ref{eq:ecr_bd}. This implies the embeddability
of the space of probability distributions over $[0..s]^{n}$ into
$L_{1}$ with distortion $O(\sqrt{n}\log s)$, which improves upon
the $O(n\log s)$ result by Indyk and Thaper \cite{indyk2003fast}
(also see \cite{kleinberg2002approximation,charikar2002similarity,fakcharoenphol2004tight}).
This improvement is due to the fact that the construction used in
this paper, called sequential Poisson functional representation, does
not rely on a hierarchical partition of the space (e.g. the quadtree
or hyperoctree in \cite{indyk2003fast}, or tree metrics in \cite{kleinberg2002approximation,charikar2002similarity,fakcharoenphol2004tight}).
Partitioning the $n$-dimensional space into hypercubes is ill-fitted
for the $\ell_{2}$ metric, since a hypercube is more ``pointy''
compared to a ball. A hypercube has larger surface area and larger
diameter than a ball of the same volume; both contribute to a larger
distortion in \cite{indyk2003fast}. The sequential Poisson functional
representation utilizes balls instead of hypercubes, and thus is more
suitable for the $\ell_{2}$ metric. We give an algorithm with time
complexity $O(2^{n}|\mathcal{X}|\log^{2}|\mathcal{X}|)$ for computing
this coupling.

Furthermore, we show that if $c$ is an ultrametric over $\mathcal{X}$
(where $(\mathcal{X},c)$ is a complete separable metric space), $r=7.56$
is attainable. If $(\mathcal{X},c)$ is any complete separable metric
space, and $\{P_{\alpha}\}_{\alpha}$ is a finite collection with
size $m\ge2$, then $r=23.1\log m$ is attainable. The case where
$\mathcal{X}$ is a Riemannian manifold is also investigated.

This paper is organized as follows. In Section \ref{sec:main}, we
present the main results of this paper. In Section \ref{sec:applications},
we describe some applications of the pairwise multi-marginal optimal
transport problem. In Section \ref{sec:upc}, we define the universal
Poisson coupling, which is the main ingredient of the proofs of the
achievability results in this paper. In Section \ref{sec:spfr}, we
generalize the universal Poisson coupling to a random process. In
Section \ref{sec:lbs}, we present several impossibility results.
In Section \ref{sec:props}, we give some miscellaneous properties
that can be proved about the problem. In Section \ref{sec:trunc},
we study a modification of the definition of the pairwise multi-marginal
optimal transport problem which may be suitable for convex costs.
In Section \ref{sec:geom}, we discuss the bi-Lipschitz embedding
of the collection of probability distributions into the space of random
variables taking values on $\mathcal{X}$. In Section \ref{sec:unresolved},
we state several conjectures.

\subsection*{Notation}

Throughout the paper, we assume that the space $\mathcal{X}$ is Polish
with its Borel $\sigma$-algebra. Logarithms are to the natural base.
We use ``$:=$'' to denote equality by definition.  Write $\mathbb{R}_{\ge a}:=[a,\infty)$,
$\mathbb{Z}_{\ge a}:=\mathbb{R}_{\ge a}\cap\mathbb{Z}$, $\mathbb{N}:=\mathbb{Z}_{\ge1}$.
 Write $[a..b]:=[a,b]\cap\mathbb{Z}$, $[a..b):=[a,b)\cap\mathbb{Z}$,
$(a..b]:=(a,b]\cap\mathbb{Z}$, $(a..b):=(a,b)\cap\mathbb{Z}$.

For $f,g:\mathbb{N}\to\mathbb{R}_{\ge0}$, we write $f(n)=O(g(n))$
if $\lim\sup_{n\to\infty}f(n)/g(n)<\infty$. We write $f(n)=\Omega(g(n))$
if $\lim\inf_{n\to\infty}f(n)/g(n)>0$. We write $f(n)=\Theta(g(n))$
if $f(n)=O(g(n))$ and $f(n)=\Omega(g(n))$.

For a collection $\mathcal{E}$ of subsets of $\mathcal{X}$, the
$\sigma$-algebra generated by $\mathcal{E}$ is denoted as $\sigma(\mathcal{E})$.

The Lebesgue measure over $\mathbb{R}^{n}$ is denoted as $\lambda$.
The uniform distribution over $S$ (a finite set or a subset of $\mathbb{R}^{n}$
with finite positive Lebesgue measure) is denoted as $\mathrm{Unif}(S)$.
The degenerate distribution $\mathbf{P}\{X=a\}=1$ is denoted as $\delta_{a}$.
We use the terms ``probability measure'' and ``probability distribution''
interchangeably.

The standard basis for $\mathbb{R}^{n}$ is denoted as $\{\mathrm{e}_{1},\ldots,\mathrm{e}_{n}\}$,
where $(\mathrm{e}_{i})_{j}=\mathbf{1}\{i=j\}$. The Hamming distance
is defined as $d_{\mathrm{H}}(x,y):=|\{i:x_{i}\neq y_{i}\}|$ for
$x,y\in\mathcal{X}^{n}$.

We write $L_{p}$ for the space of Lebesgue measurable functions $f:[0,1]\to\mathbb{R}$
with $\Vert f\Vert_{p}:=(\int_{0}^{1}|f(t)|^{p}\mathrm{d}t)^{1/p}<\infty$.
We write $\ell_{p}:=\{x\in\mathbb{R}^{\mathbb{N}}:\Vert x\Vert_{p}<\infty\}$,
where $\Vert x\Vert_{p}:=(\sum_{i=1}^{\infty}|x_{i}|^{p})^{1/p}$.

Let $d$ be a metric over the space $\mathcal{X}$. For $S\subseteq\mathcal{X}$,
write $\mathrm{diam}(S):=\sup_{x,y\in S}d(x,y)$. Metric balls are
denoted as $\mathcal{B}_{d,w}(x):=\{y\in\mathcal{X}:\,d(x,y)\le w\}$,
e.g., $\mathcal{B}_{\Vert\cdot\Vert_{p},w}(x)$ is the $\ell_{p}$
ball of radius $w$ centered at $x$ (where $d(x,y)=\Vert x-y\Vert_{p}$).
We omit $d$ and write $\mathcal{B}_{w}(x)$ if $d$ is clear from
the context. The volume of the unit $\ell_{p}$ ball over $\mathbb{R}^{n}$
is \cite{dirichlet1839volume,wang2005volumes}
\begin{align}
\mathrm{V}_{n,p} & :=\lambda\big(\mathcal{B}_{\Vert\cdot\Vert_{p},1}(0)\big)\nonumber \\
 & =\frac{2^{n}(\mathit{\Gamma}(1+1/p))^{n}}{\mathit{\Gamma}(1+n/p)},\label{eq:lp_ball_vol}
\end{align}
where $\mathit{\Gamma}$ denotes the gamma function.

A function $f$ from the metric space $(\mathcal{X},d_{\mathcal{X}})$
to the metric space $(\mathcal{Y},d_{\mathcal{Y}})$ (where the metrics
can take the value $\infty$) is $r$-Lipschitz, $r>0$, if $d_{\mathcal{Y}}(f(x_{1}),f(x_{2}))\le rd_{\mathcal{X}}(x_{1},x_{2})$
for any $x_{1},x_{2}\in\mathcal{X}$. The function is $(r_{1},r_{2})$-bi-Lipschitz,
$r_{1},r_{2}>0$, if, for any $x_{1},x_{2}\in\mathcal{X}$,
\[
r_{2}^{-1}d_{\mathcal{X}}(x_{1},x_{2})\le d_{\mathcal{Y}}(f(x_{1}),f(x_{2}))\le r_{1}d_{\mathcal{X}}(x_{1},x_{2}),
\]
i.e., $f$ is $r_{1}$-Lipschitz, and has an inverse (restricted to
the range of $f$) that is $r_{2}$-Lipschitz. The distortion of a
function $f$ is the infimum of $r_{1}r_{2}$ such that $f$ is $(r_{1},r_{2})$-bi-Lipschitz.

For two $\sigma$-finite measures $\mu,\nu$ over $\mathcal{X}$ (a
Polish space with its Borel $\sigma$-algebra) such that $\nu$ is
absolutely continuous with respect to $\mu$ (denoted as $\nu\ll\mu$),
the Radon-Nikodym derivative is written as 
\[
\frac{\mathrm{d}\nu}{\mathrm{d}\mu}:\,\mathcal{X}\to[0,\infty).
\]
If $\nu_{1},\nu_{2}\ll\mu$ (but $\nu_{1}\ll\nu_{2}$ may not hold),
we write
\begin{equation}
\frac{\mathrm{d}\nu_{1}}{\mathrm{d}\nu_{2}}(x)=\frac{\mathrm{d}\nu_{1}}{\mathrm{d}\mu}(x)\left(\frac{\mathrm{d}\nu_{2}}{\mathrm{d}\mu}(x)\right)^{-1}\in[0,\infty],\label{eq:rnderiv}
\end{equation}
which is $0$ if $(\mathrm{d}\nu_{1}/\mathrm{d}\mu)(x)=0$, and is
$\infty$ if $(\mathrm{d}\nu_{1}/\mathrm{d}\mu)(x)>0$ and $(\mathrm{d}\nu_{2}/\mathrm{d}\mu)(x)=0$.

Let $\mathcal{X}_{i}$ be a measurable space with $\sigma$-algebra
$\mathcal{F}_{i}$ for $i\in I$, where $I$ is an index set. The
product $\sigma$-algebra (over the space $\prod_{i\in I}\mathcal{X}_{i}$)
is defined as
\[
\bigotimes_{i\in I}\mathcal{F}_{i}:=\sigma\bigg(\bigg\{\prod_{i\in I}E_{i}:\,E_{i}\in\mathcal{F}_{i},\,|\{i:\,E_{i}\neq\mathcal{X}_{i}\}|<\infty\bigg\}\bigg).
\]
If $(\mathcal{X}_{i},\mathcal{F}_{i})=(\mathcal{X},\mathcal{F})$
for all $i\in I$, we write $\mathcal{F}^{\otimes I}:=\bigotimes_{i\in I}\mathcal{F}_{i}$.

Let $\mathcal{X},\mathcal{Y},\mathcal{Z}$ be measurable spaces with
$\sigma$-algebras $\mathcal{F},\mathcal{G},\mathcal{H}$ respectively.
For a measure $\mu$ over $\mathcal{X}$, and a measurable function
$g:\mathcal{X}\to\mathcal{Y}$, the pushforward measure (which is
a measure over $\mathcal{Y}$) is denoted as $g_{*}\mu(E):=\mu(g^{-1}(E))$
for $E\in\mathcal{G}$. For sub-$\sigma$-algebra $\tilde{\mathcal{F}}\subseteq\mathcal{F}$,
denote the restriction of $\mu$ to $\tilde{\mathcal{F}}$ as $\mu\!\!\upharpoonright_{\tilde{\mathcal{F}}}$
(i.e., $\mu\!\!\upharpoonright_{\tilde{\mathcal{F}}}(E)=\mu(E)$ for
$E\in\tilde{\mathcal{F}}$). For $E\in\mathcal{F}$, denote the $E$-restriction
of $\mu$ as $\mu_{E}$ (i.e., $\mu_{E}(\tilde{E})=\mu(\tilde{E}\cap E)$
for $\tilde{E}\in\mathcal{F}$). For a probability measure $P$ over
$\mathcal{X}$ and $E\in\mathcal{F}$ with $P(E)>0$, denote the conditional
distribution as $P(\cdot|E)=(1/P(E))P_{E}$. If $\kappa:\mathcal{X}\times\mathcal{G}\to[0,1]$
is a probability kernel from $\mathcal{X}$ to $\mathcal{Y}$, and
$P$ is a probability measure over $\mathcal{X}$, then the semidirect
product (which is a probability measure over $\mathcal{X}\times\mathcal{Y}$)
is denoted as $P\kappa$, and the $\mathcal{Y}$-marginal of the semidirect
product is denoted as $\kappa\circ P$. We sometimes write $\kappa(E|x)=\kappa(x,E)$.
If $\kappa_{1}:\mathcal{X}\times\mathcal{G}\to[0,1]$, $\kappa_{2}:(\mathcal{X}\times\mathcal{Y})\times\mathcal{H}\to[0,1]$
are probability kernels from $\mathcal{X}$ to $\mathcal{Y}$, and
from $\mathcal{X}\times\mathcal{Y}$ to $\mathcal{Z}$, respectively,
then the semidirect product is denoted as $\kappa_{1}\kappa_{2}:\mathcal{X}\times(\mathcal{G}\otimes\mathcal{H})\to[0,1]$
(which satisfies $\kappa_{1}\kappa_{2}(x,E_{1}\times E_{2}):=\int_{E_{1}}\kappa_{2}((x,y),E_{2})\kappa_{1}(x,\mathrm{d}y)$
for any $E_{1}\in\mathcal{G}$, $E_{2}\in\mathcal{H}$).

For a measure $\mu$ over a topological space $\mathcal{X}$, its
support is defined as
\[
\mathrm{supp}(\mu):=\left\{ x\in\mathcal{X}:\,\mu(S)>0\;\forall\,\mathrm{open}\,\mathrm{set}\;S\ni x\right\} .
\]
Note that if $\mu$ is a $\sigma$-finite measure over a Polish space
$\mathcal{X}$, and $k\in\mathbb{N}$, then $|\mathrm{supp}(\mu)|\le k$
if and only if $\mu=\sum_{i=1}^{k}a_{i}\delta_{x_{i}}$ for some $x_{i}\in\mathcal{X}$
and $a_{i}\ge0$. \footnote{Since $\mathcal{X}$ is Polish, it is second-countable, and hence
strongly Lindel\"of. We have $(\mathrm{supp}(\mu))^{\mathrm{c}}=\{x\in\mathcal{X}:\,\exists\,\mathrm{open}\;S\ni x\;\mathrm{s.t.}\;\mu(S)=0\}=\bigcup_{\mathrm{open}\;S:\,\mu(S)=0}S$,
and hence there is a countable subcover $\{S_{i}\}_{i\in\mathbb{N}}$
satisfying $\mu(S_{i})=0$ such that $(\mathrm{supp}(\mu))^{\mathrm{c}}=\bigcup_{i}S_{i}$,
and hence $\mu((\mathrm{supp}(\mu))^{\mathrm{c}})=0$. If $|\mathrm{supp}(\mu)|\le k$,
we have $\mu=\sum_{x\in\mathrm{supp}(\mu)}\mu(\{x\})\delta_{x}$ (note
that $\mu(\{x\})<\infty$ by $\sigma$-finiteness). The other direction
follows directly from the fact that $\mathcal{X}$ is Hausdorff.}

The total variation distance between two probability distributions
$P,Q$ over $\mathcal{X}$ is denoted as $d_{\mathrm{TV}}(P,Q)=\sup_{A\subseteq\mathcal{X}\,\mathrm{measurable}}|P(A)-Q(A)|$.

For a set $\mathcal{X}$, denote the set of finite or countably infinite
subsets of $\mathcal{X}$ as $[\mathcal{X}]^{\le\aleph_{0}}$. For
a $\sigma$-finite measure $\mu$ over the measurable space $\mathcal{X}$,
denote the probability distribution of the set of points of a Poisson
process with intensity measure $\mu$ as $\mathrm{PP}(\mu)$, which
is a probability distribution over the space of integer-valued measures
over $\mathcal{X}$. Refer to \cite{last2017lectures} for the definition
of this space. 

The collection of all probability distributions over the measurable
space $(\mathcal{X},\mathcal{F})$ is denoted as $\mathcal{P}(\mathcal{X})$.
For a measure $\mu$ over $\mathcal{X}$, define $\mathcal{P}_{\ll\mu}(\mathcal{X}):=\{P\in\mathcal{P}(\mathcal{X}):\,P\ll\mu\}$.
For a collection of probability distributions $\{P_{\alpha}\}_{\alpha\in\mathcal{A}}$
over the space $\mathcal{X}$, the set of all couplings of $\{P_{\alpha}\}_{\alpha\in\mathcal{A}}$
on the standard probability space, denoted as $\Gamma_{\lambda}(\{P_{\alpha}\}_{\alpha\in\mathcal{A}})$,
is defined as the set of collections of random variables $\{X_{\alpha}\}_{\alpha\in\mathcal{A}}$,
$X_{\alpha}:[0,1]\to\mathcal{X}$ on the standard probability space
$([0,1],\mathcal{L}([0,1]),\lambda_{[0,1]})$ (where $\mathcal{L}([0,1])$
is the set of Lebesgue measurable sets) such that $X_{\alpha*}\lambda_{[0,1]}=P_{\alpha}$
(i.e., $X_{\alpha}\sim P_{\alpha}$) for all $\alpha\in\mathcal{A}$.
We use this definition  so that $\Gamma_{\lambda}(\{P_{\alpha}\}_{\alpha\in\mathcal{A}})$
can be defined on the standard probability space regardless of whether
$\mathcal{A}$ is countable. On the other hand, we can define the
set of all coupling distributions of $\{P_{\alpha}\}_{\alpha\in\mathcal{A}}$
as
\begin{equation}
\Gamma(\{P_{\alpha}\}_{\alpha}):=\left\{ Q\in\mathcal{P}(\mathcal{X}^{\mathcal{A}}):\,Q_{\alpha}=P_{\alpha},\,\forall\,\alpha\right\} ,\label{eq:coupling_def}
\end{equation}
where $\mathcal{X}^{\mathcal{A}}$ denotes the measurable space $(\mathcal{X}^{\mathcal{A}},\mathcal{F}^{\otimes\mathcal{A}})$,
and $Q_{\alpha}$ is the $\alpha$-th marginal of $Q$. Note that
$\Gamma(\{P_{\alpha}\}_{\alpha})$ is comprised of distributions over
$\mathcal{X}^{\mathcal{A}}$ (rather than random variables), which
is the more conventional definition of the set of couplings. When
$\mathcal{A}$ is finite or countably infinite, and $\mathcal{X}$
is Polish, then the two definitions are equivalent in the sense that
\[
\Gamma(\{P_{\alpha}\}_{\alpha})=\left\{ (u\mapsto\{X_{\alpha}(u)\}_{\alpha})_{*}\lambda_{[0,1]}:\,\{X_{\alpha}\}_{\alpha}\in\Gamma_{\lambda}(\{P_{\alpha}\}_{\alpha})\right\} .
\]
Nevertheless, when $\mathcal{A}$ is uncountable, the two definitions
are different since we may not be able to construct random variables
on the standard probability space with distribution $Q$ for some
$Q\in\mathcal{P}(\mathcal{X}^{\mathcal{A}})$.

\medskip{}

\[
\]

\section{Main Results\label{sec:main}}

In this paper, we only consider cost functions satisfying the following
conditions.
\begin{defn}
\label{def:costfcn}A \emph{symmetric cost function} over the Polish
space $\mathcal{X}$ is a function $c:\mathcal{X}^{2}\to\mathbb{R}_{\ge0}$
that is measurable (over the product $\sigma$-algebra of $\mathcal{X}^{2}$),
symmetric ($c(x,y)=c(y,x)$ for all $x,y$), not the constant zero
function, and satisfies $c(x,x)=0$ for all $x\in\mathcal{X}$. For
probability distributions $P,Q$ over $\mathcal{X}$ and a symmetric
cost function $c$, the optimal value of Kantorovich's optimal transport
problem (which is the 1-Wasserstein distance if $c$ is also a metric)
is denoted as
\begin{equation}
C_{c}^{*}(P,Q):=\inf_{(X,Y)\in\Gamma_{\lambda}(P,Q)}\mathbf{E}[c(X,Y)].\label{eq:otp2}
\end{equation}
\end{defn}

Let $\{P_{\alpha}\}_{\alpha\in\mathcal{A}}$ be a collection of probability
distributions over $\mathcal{X}$, a Polish space with its Borel $\sigma$-algebra,
where $\mathcal{A}$ is an arbitrary  index set. For any symmetric
cost function $c$ and coupling $\{X_{\alpha}\}_{\alpha\in\mathcal{A}}\in\Gamma_{\lambda}(\{P_{\alpha}\}_{\alpha\in\mathcal{A}})$
(i.e., $X_{\alpha}\sim P_{\alpha}$), we have $\mathbf{E}[c(X_{\alpha},X_{\beta})]\ge C_{c}^{*}(P_{\alpha},P_{\beta})$
for $\alpha,\beta\in\mathcal{A}$, where $C_{c}^{*}(P_{\alpha},P_{\beta})$
is the optimal value of the optimal transport problem \eqref{eq:otp2}
with only two marginals $P_{\alpha},P_{\beta}$. This lower bound
can be attained (or approached) when $|\mathcal{A}|=2$. However,
for $|\mathcal{A}|\ge3$, it may not be possible to attain this lower
bound for all pairs $\alpha,\beta$ simultaneously. We define the
pairwise coupling ratio to measure the gap from this lower bound.
\begin{defn}
\label{def:pcr}For a symmetric cost function $c$ over the Polish
space $\mathcal{X}$, a collection of probability distributions $\{P_{\alpha}\}_{\alpha\in\mathcal{A}}$
and a coupling $\{X_{\alpha}\}_{\alpha\in\mathcal{A}}\in\Gamma_{\lambda}(\{P_{\alpha}\}_{\alpha\in\mathcal{A}})$,
the \emph{pairwise coupling ratio} is defined as
\[
r_{c}(\{X_{\alpha}\}_{\alpha\in\mathcal{A}}):=\inf\left\{ r\ge1:\,\mathbf{E}[c(X_{\alpha},X_{\beta})]\le rC_{c}^{*}(P_{\alpha},P_{\beta}),\,\forall\,\alpha,\beta\in\mathcal{A}\right\} .
\]
The infimum is regarded as $\infty$ if no such $r$ exists. The
\emph{optimal pairwise coupling ratio} of the collection $\{P_{\alpha}\}_{\alpha\in\mathcal{A}}$
is defined as 
\begin{equation}
r_{c}^{*}(\{P_{\alpha}\}_{\alpha\in\mathcal{A}}):=\inf_{\{X_{\alpha}\}_{\alpha\in\mathcal{A}}\in\Gamma_{\lambda}(\{P_{\alpha}\}_{\alpha\in\mathcal{A}})}r_{c}(\{X_{\alpha}\}_{\alpha\in\mathcal{A}}).\label{eq:rc_def}
\end{equation}
We are particularly interested in the optimal pairwise coupling ratio
of the collection of all probability distributions over $\mathcal{X}$,
i.e., $r_{c}^{*}(\mathcal{P}(\mathcal{X}))$, and that of the collection
of all probability distributions dominated by some measure $\mu$,
i.e., $r_{c}^{*}(\mathcal{P}_{\ll\mu}(\mathcal{X}))$ . \footnote{While we usually denote a collection $\{P_{\alpha}\}_{\alpha\in\mathcal{A}}$
using the index set $\mathcal{A}$, any collection of probability
distributions $\mathcal{P}$ can be indexed using itself as the index
set, i.e., $\mathcal{P}=\{P_{\alpha}\}_{\alpha\in\mathcal{P}}$, $P_{\alpha}=\alpha$.
}
\end{defn}

The coupling $\{X_{\alpha}\}_{\alpha}$ can also be regarded as an
embedding of the collection of probability distributions $\{P_{\alpha}\}_{\alpha}$
into the space of random variables taking values on $\mathcal{X}$
(i.e., measurable functions from the sample space $[0,1]$ of the
standard probability space to $\mathcal{X}$), under the constraints
that each $P_{\alpha}$ is mapped to an $X_{\alpha}$ with distribution
$P_{\alpha}$, and the embedding roughly preserves the distance in
the sense that $\mathbf{E}[c(X_{\alpha},X_{\beta})]$ is within a
constant factor from $C_{c}^{*}(P_{\alpha},P_{\beta})$. This is formally
discussed in Section \ref{sec:geom}.

\medskip{}

\subsection{Discrete Metric}

In this paper, we find the optimal pairwise coupling ratio for the
discrete metric $\mathbf{1}_{\neq}(x,y):=\mathbf{1}\{x\neq y\}$,
where $C_{\mathbf{1}_{\neq}}^{*}(P,Q)=d_{\mathrm{TV}}(P,Q)$ is the
total variation distance. The proof is based on the Poisson functional
representation \cite{sfrl_trans,li2018unified}, and is given in Section
\ref{sec:upc}. We remark that a similar result is also given in \cite[Theorem 2, Proposition 6]{angel2019pairwise},
though the earlier work in \cite{sfrl_trans,li2018unified} provides
a more general and unified construction of the underlying coupling
(which is more general than Coupling II in \cite{angel2019pairwise}).
\footnote{Other constructions based on rejection sampling were given in \cite{kleinberg2002approximation}
and Coupling I in \cite{angel2019pairwise}. We use the Poisson functional
representation instead due to its simplicity.}
\begin{thm}
\label{thm:rd_bd}For any $\sigma$-finite measure $\mu$ over the
Polish space $\mathcal{X}$ with $|\mathrm{supp}(\mu)|\ge2$ (where
$|\mathrm{supp}(\mu)|$ is the cardinality of the support of $\mu$),
we have
\[
2\left(1-\frac{1}{|\mathrm{supp}(\mu)|}\right)\le r_{\mathbf{1}_{\neq}}^{*}(\mathcal{P}_{\ll\mu}(\mathcal{X}))\le\min\left\{ 2,\,\frac{|\mathrm{supp}(\mu)|+1}{3}\right\} .
\]
Moreover, there exists a coupling $\{X_{\alpha}\}_{\alpha}$ such
that $r_{\mathbf{1}_{\neq}}(\{X_{\alpha}\}_{\alpha})\le\min\{2,\,(|\mathrm{supp}(\mu)|+1)/3\}$.
In particular, when $|\mathrm{supp}(\mu)|=\infty$, we have 
\[
r_{\mathbf{1}_{\neq}}^{*}(\mathcal{P}_{\ll\mu}(\mathcal{X}))=2.
\]
\end{thm}

Hence, for any collection of probability distributions $\{P_{\alpha}\}_{\alpha\in\mathcal{A}}$
(where there exists a $\sigma$-finite measure $\mu$ such that $P_{\alpha}\ll\mu$
for all $\alpha$), there exists a coupling $\{X_{\alpha}\}_{\alpha\in\mathcal{A}}$
such that $\mathbf{P}(X_{\alpha}\neq X_{\beta})\le2d_{\mathrm{TV}}(P_{\alpha},P_{\beta})$
for all $\alpha,\beta$. As consequences of Theorem \ref{thm:rd_bd},
$r_{\mathbf{1}_{\neq}}^{*}(\{P_{\alpha}\}_{\alpha\in\mathcal{A}})\le2$
for any countable collection $\{P_{\alpha}\}_{\alpha\in\mathcal{A}}$
(assuming $\mathcal{A}=\mathbb{N}$, we can take $\mu=\sum_{i=1}^{\infty}2^{-i}P_{i}$),
$r_{\mathbf{1}_{\neq}}^{*}(\mathcal{P}(\mathbb{Z}))=2$, and $r_{\mathbf{1}_{\neq}}^{*}(\mathcal{P}_{\ll\lambda}(\mathbb{R}^{n}))=2$,
where $\mathcal{P}_{\ll\lambda}(\mathbb{R}^{n})$ is the collection
of all continuous probability distributions over $\mathbb{R}^{n}$.
\footnote{It is possible to remove the condition about $\mu$ and prove Theorem
\ref{thm:rd_bd} on $\mathcal{P}(\mathcal{X})$ instead of $\mathcal{P}_{\ll\mu}(\mathcal{X})$,
if we lift the restriction that the coupling has to be defined on
the standard probability space. The proof is given in Appendix \ref{subsec:pf_rd_bd_nonst}.

It is unknown whether $r_{\mathbf{1}_{\neq}}^{*}(\mathcal{P}(\mathcal{X}))\le2$
(if we require the coupling to be defined on the standard probability
space) for an uncountable Polish space $\mathcal{X}$. Nevertheless,
we can show that $r_{\mathbf{1}_{\neq}}^{*}(\{P\in\mathcal{P}(\mathcal{X}):\,\mathrm{supp}(P)\le k\})\le k$
for any Polish space $\mathcal{X}$ and $k\in\mathbb{N}$. Therefore
the existence of a $\sigma$-finite measure $\mu$ such that $P_{\alpha}\ll\mu$
for all non-degenerate $P_{\alpha}\in\{P_{\alpha}\}_{\alpha}$ (non-degenerate
means $P_{\alpha}\neq\delta_{x}$ for all $x\in\mathcal{X}$) is not
a necessary condition for $r_{\mathbf{1}_{\neq}}^{*}(\{P_{\alpha}\}_{\alpha})\le2$
to hold. The proof is given in Appendix \ref{subsec:pf_rd_bd_supp2}.}

\medskip{}

\subsection{Snowflake Metric Cost}

We present the main result in this paper, which is a general upper
bound on $r_{c}^{*}(\mathcal{P}(\mathcal{X}))$ for the case where
the symmetric cost function $c(x,y)=(d(x,y))^{q}$ is a snowflake
metric, i.e., power of a metric $d$, where $0<q<1$. The proof is
given in Section \ref{subsec:metric_pow}.
\begin{thm}
\label{thm:metric_pow}Let $(\mathcal{X},d)$ be a complete separable
metric space. Consider the symmetric cost function $c(x,y)=(d(x,y))^{q}$,
$0<q<1$. Let $\mathcal{B}_{w}(x):=\{y\in\mathcal{X}:\,d(x,y)\le w\}$.
Let $\mu$ be a $\sigma$-finite measure over $\mathcal{X}$, and
$\Psi>0$ satisfying:
\begin{itemize}
\item $0<\mu(\mathcal{B}_{w}(x))<\infty$ for any $x\in\mathcal{X}$, $w>0$;
\item For any $x,y\in\mathcal{X}$ and $w>0$,
\begin{equation}
\frac{\mu(\mathcal{B}_{w}(x)\backslash\mathcal{B}_{w}(y))}{\mu(\mathcal{B}_{w}(x))}\le\frac{\Psi d(x,y)}{w};\label{eq:metric_pow_delta_dist}
\end{equation}
\item $\mu$ satisfies that \footnote{Condition \eqref{eq:metric_pow_limsup} is automatically satisfied,
for example, if $\mu(\mathcal{B}_{w}(x))$ only depends on $w$, or
$\mathrm{diam}(\mathcal{X}):=\sup\{d(x,y):x,y\in\mathcal{X}\}<\infty$.
It is also implied by \eqref{eq:metric_pow_delta_dist} if $d$ is
an intrinsic metric \cite{burago2001course}. We can show this as
follows: for any $x,y$ and $w\ge d(x,y)$, let $k:=\lceil4\Psi\rceil$,
and $z_{0},\ldots,z_{k}$ such that $z_{0}=x$, $z_{k}=y$, and $d(x_{i-1},x_{i})\le2k^{-1}d(x,y)$
(see \cite[Corollary 2.4.17]{burago2001course}). By \eqref{eq:metric_pow_delta_dist},
we have $\mu(\mathcal{B}_{w}(x_{i-1}))\le\mu(\mathcal{B}_{w}(x_{i}))+\mu(\mathcal{B}_{w}(x_{i-1})\backslash\mathcal{B}_{w}(x_{i}))\le\mu(\mathcal{B}_{w}(x_{i}))+2k^{-1}\Psi\mu(\mathcal{B}_{w}(x_{i-1}))$,
and thus $(1-2k^{-1}\Psi)\mu(\mathcal{B}_{w}(x_{i-1}))\le\mu(\mathcal{B}_{w}(x_{i}))$.
Hence $(1-2k^{-1}\Psi)^{k}\mu(\mathcal{B}_{w}(x))\le\mu(\mathcal{B}_{w}(y))$.}
\begin{equation}
\underset{w\to\infty}{\lim\sup}\sup_{x,y\in\mathcal{X}:\,d(x,y)\le w}\frac{\mu(\mathcal{B}_{w}(x))}{\mu(\mathcal{B}_{w}(y))}<\infty.\label{eq:metric_pow_limsup}
\end{equation}
\end{itemize}
Then we have
\[
r_{c}^{*}(\mathcal{P}(\mathcal{X}))<7.56\cdot\frac{(2.47\Psi)^{q}}{1-q}.
\]
\end{thm}

In comparison, it is shown in \cite[Theorem 1]{leeb2018approximating}
that it is possible to approximate the snowflake metric $c(x,y)=(d(x,y))^{q}$,
$0<q<1$ by (the $q$-th power of) a random tree metric with expected
distortion $O(\dim\mathcal{X}/(1-q))$, where $\dim\mathcal{X}$ is
the doubling dimension of $\mathcal{X}$ (refer to \cite{leeb2018approximating}
for the definition).\footnote{It appears that \cite[Theorem 1]{leeb2018approximating} concerns
only the case where $\mathcal{X}$ is finite, though it should be
straightforward to generalize it to any bounded metric space (under
certain regularity conditions) by considering an infinite tree. Nevertheless,
there is no obvious way to generalize tree metrics to general unbounded
spaces.} Then a coupling for that tree metric can be constructed as in \cite{kleinberg2002approximation},
resulting in a pairwise coupling ratio at least $O(\dim\mathcal{X}/(1-q))$.
The bound $O(\Psi^{q}/(1-q))$ in Theorem \ref{thm:metric_pow} can
be significantly better. For example, when $\mathcal{X}=\mathbb{R}^{n}$,
$c(x,y)=\Vert x-y\Vert_{2}^{q}$, $0<q<1$, we have $\Psi=O(\sqrt{n})$,
which is much smaller than $\dim\mathcal{X}=O(n)$. More discussion
is given in Section \ref{subsec:lp_rn_main}. The construction in
Theorem \ref{thm:metric_pow} is based on the sequential Poisson functional
representation in Section \ref{sec:spfr}, not on tree metrics. The
improvement of Theorem \ref{thm:metric_pow} over \cite{leeb2018approximating}
shows an advantage of the sequential Poisson functional representation
over tree metrics.

\medskip{}

\subsection{$\ell_{p}$ Metric over $\mathbb{R}^{n}$\label{subsec:lp_rn_main}}

The one-dimensional case with convex cost can be addressed directly
via the quantile coupling $X_{\alpha}=F_{P_{\alpha}}^{-1}(U)$ (where
$F_{P_{\alpha}}^{-1}(u):=\inf\{x:\,F_{P_{\alpha}}(x)\ge u\}$ is the
inverse of the cumulative distribution function of $P_{\alpha}$),
$U\sim\mathrm{Unif}[0,1]$, which is optimal for each pair of probability
distributions, i.e., $\mathbf{E}[c(X_{\alpha},X_{\beta})]=C_{c}^{*}(P_{\alpha},P_{\beta})$
for any $\alpha,\beta$ (see \cite{rachev1998mass}), implying that
$r_{c}^{*}(\mathcal{P}(\mathbb{R}))=1$ in this case. This is stated
in the following proposition, of which the proof is omitted.
\begin{prop}
\label{prop:r_rc}When $\mathcal{X}=\mathbb{R}$, for any symmetric
cost function $c$ having the form $c(x,y)=\tilde{c}(|x-y|)$, where
$\tilde{c}$ is convex, we have
\[
r_{c}^{*}(\mathcal{P}(\mathbb{R}))=1.
\]
\smallskip{}
\end{prop}

For higher dimensional Euclidean spaces, we prove an upper bound for
$c(x,y)=\Vert x-y\Vert_{p}^{q}$, $p\in\mathbb{R}_{\ge1}\cup\{\infty\}$,
$0<q<1$ using Theorem \ref{thm:metric_pow}. The proof is given in
Section \ref{subsec:rn_lp}.
\begin{thm}
\label{thm:rn_rc_ub}When $\mathcal{X}=\mathbb{R}^{n}$, $n\ge1$,
$c(x,y)=\Vert x-y\Vert_{p}^{q}$, $p\in\mathbb{R}_{\ge1}\cup\{\infty\}$,
$0<q<1$, we have
\begin{align}
r_{c}^{*}(\mathcal{P}(\mathbb{R}^{n})) & <\frac{7.56}{1-q}\left(2.47n^{\mathbf{1}\{p>2\}(1-1/p)}\mathrm{V}_{n-1,p}/\mathrm{V}_{n,p}\right)^{q},\label{eq:rn_rc_ub_Bnp}
\end{align}
where $\mathrm{V}_{n,p}$ is the volume of the unit $\ell_{p}$ ball
given in \eqref{eq:lp_ball_vol}. As a result, we have
\[
r_{c}^{*}(\mathcal{P}(\mathbb{R}^{n}))<\frac{10.55}{1-q}n^{q\max\{1/p,\,1-1/p\}}.
\]
\smallskip{}
\end{thm}

We have seen in Proposition \ref{prop:r_rc} that there exists a coupling
that is optimal for each pair of probability distributions over $\mathbb{R}$
for a convex symmetric cost function. Two natural questions are whether
this continues to hold for the non-convex cost $c(x,y)=|x-y|^{q}$,
$q<1$ (note that $X_{\alpha}=F_{P_{\alpha}}^{-1}(U)$ can be arbitrarily
far from optimal by considering $P_{1}=\mathrm{Unif}[0,1]$, $P_{2}=\mathrm{Unif}[\epsilon,1+\epsilon]$
for small $\epsilon>0$), and for higher dimensional spaces $\mathbb{R}^{n}$,
$n\ge2$, with $c(x,y)=\Vert x-y\Vert_{p}^{q}$. The answers are both
negative, and $r_{c}^{*}>1$ in both cases. Moreover, in $\mathbb{R}^{n}$
when $n\ge2$, with $c(x,y)=\Vert x-y\Vert_{p}^{q}$ where $q\ge1$,
we have $r_{c}^{*}(\mathcal{P}(\mathbb{R}^{n}))=\infty$, i.e., a
uniform bound in the form of \eqref{eq:ecr_bd} does not exist. This
shows that the behavior of $r_{c}^{*}(\mathcal{P}(\mathbb{R}))$ is
very different from that of $r_{c}^{*}(\mathcal{P}(\mathbb{R}^{n}))$,
$n\ge2$. The proof of the following proposition is given in Section
\ref{sec:lbs} (for the bounds $r_{c}^{*}(\mathcal{P}(\mathbb{R}^{n}))\ge2$)
and Section \ref{sec:geom} (for the other bounds).
\begin{prop}
\label{prop:rn_rc_lb}When $\mathcal{X}=\mathbb{R}^{n}$, $c(x,y)=\Vert x-y\Vert_{p}^{q}$,
$p\in\mathbb{R}_{\ge1}\cup\{\infty\}$, $q>0$, we have
\[
\begin{array}{ll}
r_{c}^{*}(\mathcal{P}(\mathbb{R}^{n}))\ge2 & \text{for}\;n=1,\,q<1,\\
r_{c}^{*}(\mathcal{P}(\mathbb{R}^{n}))\ge\max\big\{2,\,\frac{1}{1000\sqrt{1-q}}\big\} & \text{for}\;n\ge2,\,q<1,\\
r_{c}^{*}(\mathcal{P}(\mathbb{R}^{n}))=\infty & \text{for}\;n\ge2,\,q\ge1.
\end{array}
\]
The above statements are also true for the cases given in Remark \ref{rem:rc_lb_cases}
below.

\smallskip{}
\end{prop}

While $r_{c}^{*}(\mathcal{P}(\mathbb{R}^{n}))=\infty$ for $n\ge2$,
$q\ge1$, we can have a finite $r_{c}^{*}$ if we restrict the space
to a finite set, as shown in the following proposition. The proof
is given in Section \ref{subsec:rn_lp}.
\begin{prop}
\label{prop:s_rc_ub}Let $\mathcal{X}\subseteq\mathbb{R}^{n}$ be
a finite set with $|\mathcal{X}|\ge2$, $n\ge1$, $c(x,y)=\Vert x-y\Vert_{p}^{q}$,
$p\in\mathbb{R}_{\ge1}\cup\{\infty\}$, $q\ge1$. We have
\[
r_{c}^{*}(\mathcal{P}(\mathcal{X}))<28.66n^{\max\{1/p,\,1-1/p\}}\gamma^{q-1}\log\left(\gamma/n^{\max\{1/p,\,1-1/p\}}+1\right),
\]
where $\gamma:=\max\{\Vert x-y\Vert_{p}:\,x,y\in\mathcal{X}\}/\min\{\Vert x-y\Vert_{p}:\,x,y\in\mathcal{X},\,x\neq y\}$.\smallskip{}
\end{prop}

Note that when $\mathcal{X}=[0..s]^{n}$ for $s\ge1$, we have $\gamma=n^{1/p}s$.
Consider the case $\mathcal{X}=[0..s]^{n}$, $s\ge1$, $p=2$, $q=1$.
We have $r_{c}^{*}(\mathcal{P}([0..s]^{n}))=O(\sqrt{n}\log s)$, which
is stronger than $O(n\log s)$ if we apply Theorem \ref{thm:metric_log},
or the previous results in \cite{indyk2003fast,kleinberg2002approximation,charikar2002similarity,fakcharoenphol2004tight}.
This implies that there exists a bi-Lipschitz embedding of $\mathcal{P}([0..s]^{n})$
into the space of random variables over $\mathcal{X}$ (and also into
$L_{1}$) with distortion $O(\sqrt{n}\log s)$ (see Proposition \ref{prop:bilip_L1}
and \ref{prop:ball_bilip}). An algorithm for computing $X_{\alpha}$
with time complexity $O(2^{n}|\mathcal{X}|\log^{2}|\mathcal{X}|)$
is given in Section \ref{subsec:metric_finite} and \ref{subsec:rn_lp}.
We remark that \cite{naor2007planar} also gives an embedding of $\mathcal{P}([0..s]^{2})$
($n=2$) into $L_{1}$ with distortion $O(\log s)$.\footnote{Note that the existence of an embedding into $L_{1}$ does not imply
the existence of an embedding into the space of random variables over
$\mathcal{X}$.} For a lower bound on the order of growth of $r_{c}^{*}(\mathcal{P}([0..s]^{n}))$,
using the result in \cite{naor2007planar} (see \eqref{eq:rn_rc_lb_order}
in Section \ref{sec:geom}), for any fixed $n\ge2$, we have $r_{c}^{*}(\mathcal{P}([0..s]^{n}))=\Omega(\sqrt{\log s})$
as $s\to\infty$.

\smallskip{}

It may also be of interest to find the rate of growth of $r_{c}^{*}(\mathcal{P}(\mathbb{R}^{n}))$
for $c(x,y)=\Vert x-y\Vert_{p}^{q}$ as $n$ increases. Theorem \ref{thm:rn_rc_ub}
gives $r_{c}^{*}(\mathcal{P}(\mathbb{R}^{n}))=O(n^{q\max\{1/p,\,1-1/p\}})$
if $q<1$. The following theorem shows a lower bound on $r_{c}^{*}(\mathcal{P}(\mathbb{R}^{n}))$
that increases with $n$. The proof is given in Appendix \ref{subsec:pf_rn_rc_lb_ball}.
\begin{thm}
\label{thm:rn_rc_lb_ball}When $\mathcal{X}=\mathbb{R}^{n}$, $n\ge2$,
$c(x,y)=\Vert x-y\Vert_{p}^{q}$, $p\in\mathbb{R}_{\ge1}\cup\{\infty\}$,
$0<q<1$, we have
\begin{align*}
 & r_{c}^{*}(\mathcal{P}(\mathbb{R}^{n}))\\
 & \ge\left(1-\frac{1}{n}\right)\left(\frac{n+q}{n+q-nq}\left(\frac{n/q+1}{\mathrm{V}_{n,p}}\right)^{1/(n/q+1)}+\frac{1}{n-1}\right.\\
 & \;\;\;\;\;\left.-\min\left\{ \frac{2^{-q}n^{q/p+1}}{n+q},\,\max\left\{ \left(\frac{n/q+1}{\mathrm{V}_{n,p}}\right)^{1/(n/q+1)}-1,0\right\} +\frac{qn^{1/p+1}}{2(n+1)}\right\} \right),
\end{align*}
where $\mathrm{V}_{n,p}$ is the volume of the unit $\ell_{p}$ ball
given in \eqref{eq:lp_ball_vol}. The above statement is also true
for the cases given in Remark \ref{rem:rc_lb_cases} below. As a result,
as $n\to\infty$, we have
\[
r_{c}^{*}(\mathcal{P}(\mathbb{R}^{n}))=\Omega(n^{q/p}),
\]
i.e., $r_{c}^{*}(\mathcal{P}(\mathbb{R}^{n}))$ grows at least as
fast as $n^{q/p}$.\smallskip{}
\end{thm}

Combining this with Theorem \ref{thm:rn_rc_ub}, for $0<q<1$, $1\le p\le2$,
we have
\[
r_{c}^{*}(\mathcal{P}(\mathbb{R}^{n}))=\Theta(n^{q/p}).
\]
The bounds in Proposition \ref{prop:r_rc}, Theorem \ref{thm:rn_rc_ub},
Proposition \ref{prop:rn_rc_lb} and Theorem \ref{thm:rn_rc_lb_ball}
are plotted in Figures \ref{fig:rn_plot_n12}, \ref{fig:rn_plot_n}
and \ref{fig:rn_plot_3d}.

A consequence of Theorem \ref{thm:rn_rc_lb_ball} is that $r_{c}^{*}(\mathcal{P}(\mathcal{X}))=\infty$
in the following infinite dimensional spaces. The proof is given in
Appendix \ref{subsec:pf_rn_rc_lb_infdim}.
\begin{cor}
\label{cor:rn_rc_lb_infdim}We have $r_{c}^{*}(\mathcal{P}(\mathcal{X}))=\infty$
for the following $\mathcal{X}$ and $c$:
\begin{itemize}
\item $\mathcal{X}=\ell_{p}$, $c(x,y)=\Vert x-y\Vert_{p}^{q}$, where $p\in\mathbb{R}_{\ge1}$,
$q>0$.\footnote{We have $r_{c}^{*}(\mathcal{P}(\mathcal{X}'))=\infty$ even if we
take $\mathcal{X}'\subseteq\mathcal{X}=\ell_{p}$ to be the set of
non-negative, non-increasing sequences that sum to 1. As a result,
if $\mathcal{X}''\subseteq\mathcal{P}(\mathbb{N})$ is a space of
probability distributions over $\mathbb{N}$ (and is also a Polish
space) large enough that every $x\in\mathcal{X}'$, when regarded
as a probability mass function, is in $\mathcal{X}''$, then we have
$r_{d_{\mathrm{TV}}}^{*}(\mathcal{P}(\mathcal{X}''))=\infty$ since
$d_{\mathrm{TV}}(f,g)=(1/2)\Vert x-y\Vert_{1}$. The proof is given
in Appendix \ref{subsec:pf_rn_rc_lb_infdim}.}
\item $\mathcal{X}=\mathrm{C}([0,1],\mathbb{R})$ is the space of all continuous
functions $f:[0,1]\to\mathbb{R}$ (with the topology and $\sigma$-algebra
generated by the $L_{\infty}$ metric), $c(f,g)=\Vert f-g\Vert_{p}^{q}$
is the $L_{p}$ metric to the power $q$, where $p\in\mathbb{R}_{\ge1}\cup\{\infty\}$,
$q>0$ satisfy $p<\infty$ or $q\ge1$.\footnote{We have $r_{c}^{*}(\mathcal{P}(\mathcal{X}'))=\infty$ (if $p<\infty$
or $q\ge1$) even if we take $\mathcal{X}'\subseteq\mathcal{X}=\mathrm{C}([0,1],\mathbb{R})$
to be the set of non-negative, infinitely differentiable, $1$-Lipschitz
functions $f$ with $\int_{0}^{1}f=1$. As a result, if $\mathcal{X}''\subseteq\mathcal{P}([0,1])$
is a space of probability distributions over $[0,1]$ (and is also
a Polish space) large enough that every $f\in\mathcal{X}'$, when
regarded as a probability density function, is in $\mathcal{X}''$,
then we have $r_{d_{\mathrm{TV}}}^{*}(\mathcal{P}(\mathcal{X}''))=\infty$
since $d_{\mathrm{TV}}(f,g)=(1/2)\Vert f-g\Vert_{1}$. The proof is
given in Appendix \ref{subsec:pf_rn_rc_lb_infdim}.}
\item $\mathcal{X}=\{x\in\{0,1\}^{\mathbb{N}}:\,\sum_{i}x_{i}<\infty\}$
is the space of infinite binary sequences with finitely many 1's,
$c(x,y)=\Vert x-y\Vert_{1}^{q}$ is the Hamming distance to the power
$q$, where $q>0$.\smallskip{}
\end{itemize}
\end{cor}

\medskip{}

\begin{rem}
\label{rem:rc_lb_cases}Proposition \ref{prop:rn_rc_lb} and Theorem
\ref{thm:rn_rc_lb_ball} are also true if we replace $\mathcal{P}(\mathbb{R}^{n})$
with any one of the following (for proofs, refer to the proofs of
the respective proposition or theorem):
\begin{itemize}
\item $\mathcal{P}(\mathbb{Z}^{n})$;
\item $\mathcal{P}_{\ll\lambda_{S}}(\mathbb{R}^{n})$ (i.e., continuous
probability distributions over $S$) for any $S\subseteq\mathbb{R}^{n}$
with $\lambda(S)>0$;
\item $\mathcal{P}(\mathcal{M})$ where $\mathcal{M}$ is a connected smooth
complete $n$-dimensional Riemannian manifold (without boundary),
and $c(x,y)=(d_{\mathcal{M}}(x,y))^{q}$, where $d_{\mathcal{M}}$
denotes the intrinsic distance on the manifold $\mathcal{M}$. In
this case, we let $p=2$ in Proposition \ref{prop:rn_rc_lb} and Theorem
\ref{thm:rn_rc_lb_ball}.\footnote{Note that if $\mathcal{M}$ satisfies these requirements then $(\mathcal{M},d_{\mathcal{M}})$
is a complete separable metric space. This is due to the fact that
every metric space is paracompact, and every Hausdorff (implied by
the metric), connected topological space locally homeomorphic to $\mathbb{R}^{n}$
is second-countable.}
\end{itemize}
\end{rem}

\medskip{}

\subsection{Riemannian Manifolds}

In this subsection, we consider the case $\mathcal{X}=\mathcal{M}$,
where $\mathcal{M}$ is a connected smooth complete real $n$-dimensional
Riemannian manifold (without boundary). Let $d_{\mathcal{M}}$ be
the intrinsic distance on the manifold $\mathcal{M}$. We consider
the symmetric cost function of the form $c(x,y)=(d_{\mathcal{M}}(x,y))^{q}$,
$q>0$.

We first consider the simple case where $\mathcal{M}=\{x\in\mathbb{R}^{2}:x_{1}^{2}+x_{2}^{2}=1\}$
is the circle. The proof of the following proposition is given in
Appendix \ref{subsec:pf_circle}.
\begin{prop}
\label{prop:circle}Let $\mathcal{M}=\{x\in\mathbb{R}^{2}:x_{1}^{2}+x_{2}^{2}=1\}$
be the circle, and $c(x,y)=(d_{\mathcal{M}}(x,y))^{q}$, $q>0$. We
have
\[
\begin{array}{ll}
r_{c}^{*}(\mathcal{P}(\mathcal{M}))\in[2,\,20.27/(1-q))\, & \text{for}\;0<q<1,\\
r_{c}^{*}(\mathcal{P}(\mathcal{M}))=2 & \text{for}\;q=1,\\
r_{c}^{*}(\mathcal{P}(\mathcal{M}))=\infty & \text{for}\;q>1.
\end{array}
\]
\end{prop}

\medskip{}

It is perhaps noteworthy that, compared to the case where $\mathcal{X}=\mathbb{R}$
is the real line, the bounds when $0<q<1$ are similar. However, for
the circle, we have $r_{c}^{*}=2$ when $q=1$, and $r_{c}^{*}=\infty$
when $q>1$ (compared to $r_{c}^{*}=1$ for the real line when $q\ge1$).
This shows that the behavior of $r_{c}^{*}$ for the circle is very
different from that for the real line.

For higher dimensional manifolds, the lower bounds in Proposition
\ref{prop:rn_rc_lb} and Theorem \ref{thm:rn_rc_lb_ball} also hold
(see Remark \ref{rem:rc_lb_cases}). We have the following upper bound
for the $n$-dimensional sphere and torus, proved using Theorem \ref{thm:metric_pow}.
The proof is given in Corollary \ref{cor:dyadic_manifold}.
\begin{prop}
Assume either $\mathcal{M}=\{x\in\mathbb{R}^{n+1}:\,\Vert x\Vert_{2}=1\}$
is the $n$-sphere, or $\mathcal{M}=\{x\in\mathbb{R}^{2n}:\,x_{2i-1}^{2}+x_{2i}^{2}=1\,\forall i\in[1..n]\}$
is the $n$-torus. Let $c(x,y)=(d_{\mathcal{M}}(x,y))^{q}$, $0<q<1$.
We have
\[
r_{c}^{*}(\mathcal{P}(\mathcal{M}))<\frac{20.27}{1-q}n^{q/2}.
\]
\end{prop}

\medskip{}

The following result applies to manifolds with non-negative Ricci
curvature, and compact manifolds with Ricci curvature bounded below.
The proof is given in Section \ref{subsec:manifold}.
\begin{thm}
\label{thm:ricci}Let $c(x,y)=(d_{\mathcal{M}}(x,y))^{q}$, $0<q<1$.
If the Ricci curvature of $\mathcal{M}$ is bounded below as $\mathrm{Ric}_{\mathcal{M}}\ge(n-1)K$
(i.e., $\mathrm{Ric}_{\mathcal{M}}(\xi,\xi)\ge(n-1)K\left\langle \xi,\xi\right\rangle $
for any $x\in\mathcal{M}$, $\xi\in T_{x}\mathcal{M}$) for some $K\in\mathbb{R}$,
we have:
\begin{enumerate}
\item If $K\ge0$, then
\[
r_{c}^{*}(\mathcal{P}(\mathcal{M}))<\frac{7.56}{1-q}(2.47n)^{q};
\]
\item If $K<0$ and $D:=\mathrm{diam}(\mathcal{M}):=\sup_{x,y\in\mathcal{M}}d_{\mathcal{M}}(x,y)<\infty$,
then
\[
r_{c}^{*}(\mathcal{P}(\mathcal{M}))<\frac{7.56}{1-q}\left(6.72n(D\sqrt{-K}+1)\right)^{q}.
\]
\end{enumerate}
\end{thm}

\medskip{}

For non-compact manifolds with negative Ricci curvature (e.g. the
hyperbolic space) , it is unknown whether $r_{c}^{*}(\mathcal{P}(\mathcal{M}))<\infty$
for $c(x,y)=(d_{\mathcal{M}}(x,y))^{q}$, $0<q<1$.\medskip{}

\subsection{Finite Metric Space\label{subsec:metric_finite_main}}

We consider the case where $\mathcal{X}$ is finite, and the symmetric
cost function $c$ is a metric over $\mathcal{X}$. The problem of
finding the exact value of $r_{c}^{*}(\{P_{\alpha}\}_{\alpha\in\mathcal{A}})$
(or even deciding whether $r_{c}^{*}(\{P_{\alpha}\}_{\alpha\in\mathcal{A}})=1$)
for a finite collection of probability distributions $\{P_{\alpha}\}_{\alpha\in\mathcal{A}}$
is NP-hard, which is shown in Appendix \ref{subsec:hardness}. Nevertheless,
it is possible to give bounds and efficient approximation algorithms,
as described in the following theorem. The proof is given in Section
\ref{subsec:metric_finite}.
\begin{thm}
\label{thm:metric_log}Let $(\mathcal{X},d)$ be a finite metric space,
and $c(x,y)=d(x,y)$. We have
\[
r_{c}^{*}(\mathcal{P}(\mathcal{X}))<55.7\left(1+\log|\mathcal{X}|\right).
\]
\end{thm}

\medskip{}

Theorem \ref{thm:metric_log} shows that $r_{c}^{*}(\mathcal{P}(\mathcal{X}))=O(\log|\mathcal{X}|)$,
which achieves the same order as in \cite{fakcharoenphol2004tight},
and improves upon $r_{c}^{*}(\mathcal{P}(\mathcal{X}))=O(\log|\mathcal{X}|\log\log|\mathcal{X}|)$
in \cite{kleinberg2002approximation,charikar2002similarity}. The
algorithm for computing $X_{\alpha}\sim P_{\alpha}$ in this coupling,
with time complexity $O(|\mathcal{X}|^{3}\log|\mathcal{X}|)$, is
given in Section \ref{subsec:metric_finite}.

Regarding the tightness of Theorem \ref{thm:metric_log}, it is shown
in Proposition \ref{prop:hamming_embed} (using a result in \cite{khot2006nonembeddability})
that when $\mathcal{X}=\{0,1\}^{n}$, $n\ge2$, and $c(x,y)=\Vert x-y\Vert_{1}$,
we have $r_{c}^{*}(\mathcal{P}(\{0,1\}^{n}))=\Omega(n)=\Omega(\log|\mathcal{X}|)$.
Therefore the bound $r_{c}^{*}(\mathcal{P}(\mathcal{X}))=O(\log|\mathcal{X}|)$
is tight.

For the case where the symmetric cost function is a power of a metric
(i.e., $c(x,y)=(d(x,y))^{q}$, $q>0$, where $d$ is a metric over
$\mathcal{X}$), the following bound shows that $r_{c}^{*}(\mathcal{P}(\mathcal{X}))=O(|\mathcal{X}|^{q})$.
The proof is given in Section \ref{subsec:ultra}.
\begin{thm}
\label{thm:metric}Let $(\mathcal{X},d)$ be a finite metric space,
and $c(x,y)=(d(x,y))^{q}$, $q>0$. We have
\[
r_{c}^{*}(\mathcal{P}(\mathcal{X}))<7.56(|\mathcal{X}|-1)^{q}.
\]
\end{thm}

This bound is weaker compared to Theorem \ref{thm:metric_log} when
$0<q\le1$ (where $c(x,y)=(d(x,y))^{q}$ is also a metric). Nevertheless,
when $q>1$, it is shown in \eqref{eq:circle_order} that there exists
a sequence of finite metric spaces where $r_{c}^{*}(\mathcal{P}(\mathcal{X}))=\Omega(|\mathcal{X}|^{q-1})$.

One may wonder whether it is possible to bound $r_{c}^{*}(\mathcal{P}(\mathcal{X}))$
in terms of $|\mathcal{X}|$ if no conditions are imposed on $c$.
This is impossible if $|\mathcal{X}|\ge4$, since it is shown in Proposition
\ref{prop:rc_prop_misc} that $r_{c}^{*}(\mathcal{P}(\mathcal{X}))$
can be arbitrarily large (or even infinite) when $|\mathcal{X}|=4$.\medskip{}

\subsection{Ultrametric Space}

\medskip{}
We have the following bound for the case where the symmetric cost
function $c$ is an ultrametric. Note that a symmetric cost function
$c$ is an ultrametric if $c(x,y)>0$ for $x\neq y$, and $c(x,z)\le\max\{c(x,y),c(y,z)\}$
for any $x,y,z$. The proof is given in Section \ref{subsec:ultra}.
\begin{thm}
\label{thm:ultrametric}Let $(\mathcal{X},c)$ be a complete separable
metric space. If $c$ is an ultrametric, then
\[
r_{c}^{*}(\mathcal{P}(\mathcal{X}))<7.56.
\]
\end{thm}

\medskip{}
We remark that it is shown in \cite{kloeckner2015geometric} that
if $(\mathcal{X},c)$ is a compact ultrametric space, then $(\mathcal{P}(\mathcal{X}),C_{c}^{*})$
is affinely isometric to a convex subset of $\ell_{1}$. In contrast,
Theorem \ref{thm:ultrametric} implies that there is a bi-Lipschitz
embedding of $(\mathcal{P}(\mathcal{X}),C_{c}^{*})$ into the space
of $\mathcal{X}$-valued random variables with distortion $7.56$,
where each $P\in\mathcal{P}(\mathcal{X})$ is mapped to a random variable
with distribution $P$ (see Section \ref{sec:geom}). The requirement
that $P$ is mapped to a random variable with distribution $P$ incurs
a penalty that the embedding is only bi-Lipschitz (instead of isometric).
This penalty is necessary, since it is shown in Proposition \ref{prop:rc_prop_misc}
that $r_{c}^{*}(\mathcal{P}(\mathcal{X}))>1$ for metric spaces $(\mathcal{X},c)$
unless $(\mathcal{X},c)$ can be isometrically embedded into $(\mathbb{R},\,(x,y)\mapsto|x-y|)$,
meaning that the only ultrametric space $(\mathcal{X},c)$ where $r_{c}^{*}(\mathcal{P}(\mathcal{X}))=1$
is the trivial example where $|\mathcal{X}|=2$.

\medskip{}

\subsection{Finite Collection of Probability Distributions}

We give a bound on $r_{c}^{*}$ when the symmetric cost function $c$
is a metric and the collection of distributions $\{P_{\alpha}\}_{\alpha\in\mathcal{A}}$
is finite. The proof uses the result on tree metrics in \cite{fakcharoenphol2004tight},
and the strategy in \cite{archer2004approximate}. It is given in
Appendix \ref{subsec:pf_finite_col}.
\begin{prop}
\label{prop:finite_col}Let $(\mathcal{X},c)$ be a complete separable
metric space. Let $\{P_{\alpha}\}_{\alpha\in\mathcal{A}}$ be a finite
collection of probability distributions over $\mathcal{X}$ with $|\mathcal{A}|\ge2$.
We have
\[
r_{c}^{*}(\{P_{\alpha}\}_{\alpha\in\mathcal{A}})<23.1\cdot\log|\mathcal{A}|.
\]
\end{prop}

This result (for finite $\{P_{\alpha}\}_{\alpha\in\mathcal{A}}$ and
arbitrary $\mathcal{X}$) appears to be similar to Theorem \ref{thm:metric_log}
(for finite $\mathcal{X}$ and arbitrary $\{P_{\alpha}\}_{\alpha\in\mathcal{A}}$),
with the roles of $\mathcal{X}$ and $\{P_{\alpha}\}_{\alpha\in\mathcal{A}}$
switched. Understanding the connection between the two results is
left for future research.\medskip{}

\begin{table}[H]
\noindent \hspace{-68pt}%
\begin{tabular}{|c|c|cl|}
\hline 
Collection $\{P_{\alpha}\}_{\alpha}$ & Cost $c(x,y)$ & \multicolumn{2}{c|}{Bounds on $r_{c}^{*}(\{P_{\alpha}\}_{\alpha})$}\tabularnewline
\hline 
\hline 
$\mathcal{P}([1..2])$ & \multirow{5}{*}{$\mathbf{1}\{x\neq y\}$} & $r_{c}^{*}=1$ & {\footnotesize{}(Folklore, Thm \ref{thm:rd_bd})}\tabularnewline
\cline{1-1} \cline{3-4} \cline{4-4} 
$\mathcal{P}([1..3])$ &  & $r_{c}^{*}=4/3$ & {\footnotesize{}(\cite{kleinberg2002approximation,iwasa2009approximation},
Thm \ref{thm:rd_bd})}\tabularnewline
\cline{1-1} \cline{3-4} \cline{4-4} 
$\mathcal{P}([1..4])$ &  & $3/2\le r_{c}^{*}\le5/3$ & {\footnotesize{}(Thm \ref{thm:rd_bd})}\tabularnewline
\cline{1-1} \cline{3-4} \cline{4-4} 
$\mathcal{P}([1..s]),\,s\ge5$ &  & $2(1-s^{-1})\le r_{c}^{*}\le2$ & {\footnotesize{}(\cite{kleinberg2002approximation,angel2019pairwise},
Thm \ref{thm:rd_bd})}\tabularnewline
\cline{1-1} \cline{3-4} \cline{4-4} 
$\mathcal{P}(\mathbb{Z})$ or $\mathcal{P}_{\ll\lambda}(\mathbb{R}^{n})$ &  & \multirow{1}{*}{$r_{c}^{*}=2$} & \multirow{1}{*}{{\footnotesize{}(\cite{angel2019pairwise}, Thm \ref{thm:rd_bd})}}\tabularnewline
\hline 
\multirow{2}{*}{$\mathcal{P}(\mathbb{R})$} & $|x-y|^{q}$, $q<1$ & $2\le r_{c}^{*}<9.34/(1-q)$ & {\footnotesize{}(Thm$\!$ \ref{thm:rn_rc_ub}, \eqref{eq:rn_rc_ub_Bnp},
$\!$Prop$\!$ \ref{prop:rn_rc_lb_1})}\tabularnewline
\cline{2-4} \cline{3-4} \cline{4-4} 
 & $|x-y|^{q}$, $q\ge1$ & $r_{c}^{*}=1$ & {\footnotesize{}(\cite{rachev1998mass}, Prop \ref{prop:r_rc})}\tabularnewline
\hline 
\multirow{3}{*}{$\mathcal{P}(\mathbb{R}^{n})$, $n\ge2$} & \multirow{2}{*}{$\Vert x-y\Vert_{p}^{q}$, $q<1$} & $\!\!\max\!\big\{2,\frac{1}{1000\sqrt{1-q}}\!\big\}\!\le\!r_{c}^{*}\!<\!\frac{10.55}{1-q}n^{q\max\{1/p,\,1-1/p\}}\!\!\!\!\!\!$ & {\footnotesize{}(Thm \ref{thm:rn_rc_ub}, Prop \ref{prop:rn_rc_lb_1})}\tabularnewline
\cline{3-4} \cline{4-4} 
 &  & $r_{c}^{*}=\Omega(n^{q/p})$ & {\footnotesize{}(Thm \ref{thm:rn_rc_lb_ball})}\tabularnewline
\cline{2-4} \cline{3-4} \cline{4-4} 
 & $\Vert x-y\Vert_{p}^{q}$, $q\ge1$ & $r_{c}^{*}=\infty$ & {\footnotesize{}(Prop \ref{prop:rn_rc_lb_2}, \cite{naor2007planar})}\tabularnewline
\hline 
$\mathcal{P}([0..s]^{n})$, $s\ge1$ & $\Vert x-y\Vert_{p}^{q}$, $p\le2$, $q\ge1$ & $r_{c}^{*}<28.66n^{q/p}s^{q-1}\log(s+1)$ & {\footnotesize{}(Prop \ref{prop:s_rc_ub})}\tabularnewline
\hline 
$\mathcal{P}(\ell_{p})$, $p<\infty$ & $\Vert x-y\Vert_{p}^{q}$ & $r_{c}^{*}=\infty$ & {\footnotesize{}(Cor \ref{cor:rn_rc_lb_infdim})}\tabularnewline
\hline 
$\mathcal{P}(\mathcal{P}(\mathbb{N}))$ & $d_{\mathrm{TV}}(x,y)$ & $r_{c}^{*}=\infty$ & {\footnotesize{}(Cor \ref{cor:rn_rc_lb_infdim})}\tabularnewline
\hline 
$\mathcal{P}(\mathrm{C}([0,1],\mathbb{R}))$ & $\Vert x-y\Vert_{p}^{q}$, $p\!<\!\infty$ or $q\!\ge\!1$ & $r_{c}^{*}=\infty$ & {\footnotesize{}(Cor \ref{cor:rn_rc_lb_infdim})}\tabularnewline
\hline 
$\mathcal{P}(\{0,1\}^{n})$, $n\ge2$ & \multirow{2}{*}{$\Vert x-y\Vert_{1}^{q}$} & $\!\!\!\max\!\big\{\!\frac{(n-1)^{q}-1}{n}\!+\!1,\,2\!-\!2^{1-n}\!\big\}\!\le\!r_{c}^{*}\!\le\!2n^{q},r_{c}^{*}\!=\!\Omega(n^{1-|1-q|}\!)\!\!\!\!$ & {\footnotesize{}(Prop \ref{prop:rn_rc_lb_1}, \ref{prop:rn_rc_lb_2},
\ref{prop:rc_prop_misc}, \ref{prop:hamming_embed})$\!$$\!$}\tabularnewline
\cline{1-1} \cline{3-4} \cline{4-4} 
$\!\!\mathcal{P}(\{x\!\in\!\{0,1\}^{\mathbb{N}}:\sum_{i}x_{i}\!<\!\infty\})\!\!$ &  & $r_{c}^{*}=\infty$ & {\footnotesize{}(Cor \ref{cor:rn_rc_lb_infdim}, \cite{khot2006nonembeddability})}\tabularnewline
\hline 
\multirow{3}{*}{$\begin{array}{c}
\mathcal{P}(\mathcal{M})\text{ for }\mathcal{M}\\
\text{being the circle}
\end{array}$} & $(d_{\mathcal{M}}(x,y))^{q}$, $q<1$ & $2\le r_{c}^{*}<20.27/(1-q)$ & {\footnotesize{}(Prop \ref{prop:circle})}\tabularnewline
\cline{2-4} \cline{3-4} \cline{4-4} 
 & $d_{\mathcal{M}}(x,y)$ & $r_{c}^{*}=2$ & {\footnotesize{}(Prop \ref{prop:circle})}\tabularnewline
\cline{2-4} \cline{3-4} \cline{4-4} 
 & $(d_{\mathcal{M}}(x,y))^{q}$, $q>1$ & $r_{c}^{*}=\infty$ & {\footnotesize{}(Prop \ref{prop:circle})}\tabularnewline
\hline 
\multirow{2}{*}{$\!\!\begin{array}{c}
\mathcal{P}(\mathcal{M})\text{ for }\mathcal{M}\text{ being the}\\
n\text{-sphere/torus},\,n\ge2
\end{array}\!\!$} & $(d_{\mathcal{M}}(x,y))^{q}$, $q<1$ & $\!\!\max\!\big\{2,\frac{1}{1000\sqrt{1-q}}\!\big\}\!\le\!r_{c}^{*}\!<\!\frac{20.27}{1-q}n^{q/2}$,
$\!r_{c}^{*}\!=\!\Theta(n^{q/2})\!\!\!\!$ & {\footnotesize{}(Cor$\!$ \ref{cor:dyadic_manifold},Prop$\!$ \ref{prop:rn_rc_lb_1},Thm$\!$
\ref{thm:rn_rc_lb_ball})$\!$$\!$$\!$}\tabularnewline
\cline{2-4} \cline{3-4} \cline{4-4} 
 & $(d_{\mathcal{M}}(x,y))^{q}$, $q\ge1$ & $r_{c}^{*}=\infty$ & {\footnotesize{}(Thm \ref{prop:rn_rc_lb_2})}\tabularnewline
\hline 
\multirow{2}{*}{$\mathcal{P}(\mathcal{X}),\,|\mathcal{X}|<\infty$} & $d(x,y)$, $d$ is a metric & $r_{c}^{*}<55.7(1+\log|\mathcal{X}|)$ & {\footnotesize{}(Thm \ref{thm:metric_log}, \cite{fakcharoenphol2004tight})}\tabularnewline
\cline{2-4} \cline{3-4} \cline{4-4} 
 & $(d(x,y))^{q}$, $d$ is a metric & $r_{c}^{*}<7.56(|\mathcal{X}|-1)^{q}$ & {\footnotesize{}(Thm \ref{thm:metric})}\tabularnewline
\hline 
$\mathcal{P}(\mathcal{X})$ & $\!\!$Ultrametric$\!$ that$\!$ metrizes$\!$ $\mathcal{X}$$\!\!$ & $r_{c}^{*}<7.56$ & {\footnotesize{}(Thm \ref{thm:ultrametric})}\tabularnewline
\hline 
$\{P_{\alpha}\}_{\alpha\in\mathcal{A}}$, $|\mathcal{A}|<\infty$ & $d(x,y)$, $d$ is a metric & $r_{c}^{*}<23.1\cdot\log|\mathcal{A}|$ & {\footnotesize{}(Prop \ref{prop:finite_col}, \cite{archer2004approximate})}\tabularnewline
\hline 
$\mathcal{P}([1..4])$ & $\mathbf{1}\{|x-y|=2\}$ & $r_{c}^{*}=\infty$ & {\footnotesize{}(Prop \ref{prop:rc_prop_misc})}\tabularnewline
\hline 
\end{tabular}\vspace{4pt}

\caption{\label{tab:listbd1}List of bounds on $r_{c}^{*}(\{P_{\alpha}\}_{\alpha})$,
where $s\in\mathbb{Z}_{\ge2}$, $p\in\mathbb{R}_{\ge1}\cup\{\infty\}$,
$q>0$.  The example $c(x,y)=\mathbf{1}\{|x-y|=2\}$ over $\mathcal{P}([1..4])$
is the smallest example where $r_{c}^{*}=\infty$.}
\end{table}

\begin{sidewaystable}[H]
\begin{tabular}{|c|c|c|c|c|c|c|l|}
\hline 
$n$ & $\mathcal{X}$ & $q$ & $p$ & Bounds on $r_{c}^{*}(\mathcal{P}(\mathcal{X}))$ & $\!\begin{array}{c}
r_{c}^{*}(\mathcal{P}(\mathcal{X}))\\
=\Omega(\cdots)
\end{array}\!$ & $\begin{array}{c}
r_{c}^{*}(\mathcal{P}(\mathcal{X}))\\
=O(\cdots)
\end{array}$ & $\begin{array}{c}
\text{Theorem /}\\
\text{Proposition}
\end{array}$\tabularnewline
\hline 
\hline 
\multirow{4}{*}{$1$} & \multirow{2}{*}{$[0..s]$} & $\!<\!1\!$ & $1$ & $\!\!\begin{array}{l}
2/(1+s^{q-1})\le r_{c}^{*}\\
\le\min\{\frac{9.34}{1-q},\,55.7(\log(s+1)+1),\,2s^{q},\,s^{1-q}\}
\end{array}\!\!$ & \multicolumn{2}{c|}{$1$} & {\footnotesize{}$\begin{array}{c}
\mbox{(Thm \ref{thm:metric_log}, \ref{thm:rn_rc_ub}, \eqref{eq:rn_rc_ub_Bnp},}\\
\mbox{Prop \ref{prop:rn_rc_lb_1}, \ref{prop:rc_prop_misc})}
\end{array}$$\!\!$}\tabularnewline
\cline{3-8} \cline{4-8} \cline{5-8} \cline{6-8} \cline{7-8} \cline{8-8} 
 &  & $\!\ge\!1\!$ & $1$ & $r_{c}^{*}=1$ & \multicolumn{2}{c|}{$1$} & {\footnotesize{}(\cite{rachev1998mass}, Prop \ref{prop:r_rc})}\tabularnewline
\cline{2-8} \cline{3-8} \cline{4-8} \cline{5-8} \cline{6-8} \cline{7-8} \cline{8-8} 
 & \multirow{2}{*}{$\mathbb{Z}$ or $\mathbb{R}$} & $\!<\!1\!$ & $1$ & $2\le r_{c}^{*}<9.34/(1-q)$ & \multicolumn{2}{c|}{$1$} & {\footnotesize{}(Thm$\!$ \ref{thm:rn_rc_ub}, \eqref{eq:rn_rc_ub_Bnp},
$\!$Prop$\!$ \ref{prop:rn_rc_lb_1})}\tabularnewline
\cline{3-8} \cline{4-8} \cline{5-8} \cline{6-8} \cline{7-8} \cline{8-8} 
 &  & $\!\ge\!1\!$ & $1$ & $r_{c}^{*}=1$ & \multicolumn{2}{c|}{$1$} & {\footnotesize{}(\cite{rachev1998mass}, Prop \ref{prop:r_rc})}\tabularnewline
\hline 
\multirow{9}{*}{$\!\ge\!2\!$} & \multirow{6}{*}{$[0..s]^{n}$} & \multirow{2}{*}{$\!<\!1\!$} & $\!\le\!2\!$ & $2(1-s^{-n})\le r_{c}^{*}<\frac{10.55}{1-q}n^{q/p}$ & \multicolumn{2}{c|}{$n^{q/p}$} & {\footnotesize{}(Thm \ref{thm:rn_rc_ub}, Prop \ref{prop:rn_rc_lb_1},
\ref{prop:hamming_embed})}\tabularnewline
\cline{4-8} \cline{5-8} \cline{6-8} \cline{7-8} \cline{8-8} 
 &  &  & $\!>\!2\!$ & $2(1-s^{-n})\le r_{c}^{*}\le\min\{\frac{10.55}{1-q}n^{q-q/p},\,2n^{q/p}s^{q}\}$ & $n^{q/p}$ & $\min\{n^{q-q/p},n^{q/p}s^{q}\}$ & {\footnotesize{}(Thm \ref{thm:rn_rc_ub}, Prop \ref{prop:rn_rc_lb_1},\ref{prop:rc_prop_misc},\ref{prop:hamming_embed})$\!\!$}\tabularnewline
\cline{3-8} \cline{4-8} \cline{5-8} \cline{6-8} \cline{7-8} \cline{8-8} 
 &  & \multirow{2}{*}{$1$} & $\!\le\!2\!$ & $\frac{1}{285}\sqrt{\log s}\le r_{c}^{*}<28.66n^{1/p}\log(s+1)$ & $n^{1/p}+\sqrt{\log s}$ & $n^{1/p}\log(s+1)$ & {\footnotesize{}(Prop \ref{prop:s_rc_ub}, \ref{prop:hamming_embed},
\eqref{eq:rn_rc_lb_order}, \cite{naor2007planar})$\!\!$}\tabularnewline
\cline{4-8} \cline{5-8} \cline{6-8} \cline{7-8} \cline{8-8} 
 &  &  & $\!>\!2\!$ & $\!\!\frac{1}{285}\sqrt{\log s}\le r_{c}^{*}<28.66n^{1-1/p}\log(s/n^{1-2/p}+1)\!\!$ & $n^{1/p}+\sqrt{\log s}$ & $n^{1-\frac{1}{p}}\log(\frac{s}{n^{1-2/p}}+1)$ & {\footnotesize{}(Prop \ref{prop:s_rc_ub}, \ref{prop:hamming_embed},
\eqref{eq:rn_rc_lb_order}, \cite{naor2007planar})$\!\!$}\tabularnewline
\cline{3-8} \cline{4-8} \cline{5-8} \cline{6-8} \cline{7-8} \cline{8-8} 
 &  & \multirow{2}{*}{$\!>\!1\!$} & $\!\le\!2\!$ & $\!\!\begin{array}{l}
\max\{(n-1)^{\frac{q}{p}}n^{-1},\frac{1}{2}\}s^{q-1}\le r_{c}^{*}\\
\;\;<28.66n^{\frac{q}{p}}s^{q-1}\log(s+1)
\end{array}\!\!$ & $\!\!\!n^{1-|1-\frac{q}{p}|}+n^{\frac{q}{p}-1}s^{q-1}+s^{q-1}\!\!\!$ & $n^{\frac{q}{p}}s^{q-1}\log(s+1)$ & {\footnotesize{}(Prop \ref{prop:s_rc_ub}, Prop \ref{prop:rn_rc_lb_2},
\ref{prop:hamming_embed})}\tabularnewline
\cline{4-8} \cline{5-8} \cline{6-8} \cline{7-8} \cline{8-8} 
 &  &  & $\!>\!2\!$ & $\!\!\begin{array}{l}
\max\{(n-1)^{\frac{q}{p}}n^{-1},\frac{1}{2}\}s^{q-1}\le r_{c}^{*}\\
\;\;<28.66n^{1+q/p-2/p}s^{q-1}\log(s/n^{1-2/p}+1)
\end{array}\!\!$ & $\!\!\!n^{1-|1-\frac{q}{p}|}+n^{\frac{q}{p}-1}s^{q-1}+s^{q-1}\!\!\!$ & $\!\!\!n^{1+\frac{q}{p}-\frac{2}{p}}s^{q-1}\log(\frac{s}{n^{1-2/p}}+1)\!\!\!$ & {\footnotesize{}(Prop \ref{prop:s_rc_ub}, \ref{prop:rn_rc_lb_2},
\ref{prop:hamming_embed})}\tabularnewline
\cline{2-8} \cline{3-8} \cline{4-8} \cline{5-8} \cline{6-8} \cline{7-8} \cline{8-8} 
 & \multirow{3}{*}{$\!\mathbb{Z}^{n}\!$ or $\mathbb{R}^{n}\!$} & \multirow{2}{*}{$\!<\!1\!$} & $\!\le\!2\!$ & $\!\max\big\{2,\frac{1}{1000\sqrt{1-q}}\big\}\le r_{c}^{*}<\frac{10.55}{1-q}n^{q/p}$ & \multicolumn{2}{c|}{$n^{q/p}$} & {\footnotesize{}(Thm$\!$ \ref{thm:rn_rc_ub},\ref{thm:rn_rc_lb_ball}$\!$
Prop$\!$ \ref{prop:rn_rc_lb_1})$\!\!$}\tabularnewline
\cline{4-8} \cline{5-8} \cline{6-8} \cline{7-8} \cline{8-8} 
 &  &  & $\!>\!2\!$ & $\!\max\big\{2,\frac{1}{1000\sqrt{1-q}}\big\}\le r_{c}^{*}<\frac{10.55}{1-q}n^{q-q/p}\!$ & $n^{q/p}$ & $n^{q-q/p}$ & {\footnotesize{}(Thm$\!$ \ref{thm:rn_rc_ub},\ref{thm:rn_rc_lb_ball}$\!$
Prop$\!$ \ref{prop:rn_rc_lb_1})$\!\!$}\tabularnewline
\cline{3-8} \cline{4-8} \cline{5-8} \cline{6-8} \cline{7-8} \cline{8-8} 
 &  & $\!\ge\!1\!$ & $\!$any$\!$ & $r_{c}^{*}=\infty$ & \multicolumn{2}{c|}{$\infty$} & {\footnotesize{}(Prop \ref{prop:rn_rc_lb_2}, \cite{naor2007planar})}\tabularnewline
\hline 
\end{tabular}

\caption{List of bounds on $r_{c}^{*}(\mathcal{P}(\mathcal{X}))$, where $c(x,y)=\Vert x-y\Vert_{p}^{q}$,
$p\in\mathbb{R}_{\ge1}\cup\{\infty\}$, $q>0$, $s\in\mathbb{N}$.
Here we write $r_{c}^{*}(\mathcal{P}(\mathcal{X}))=\Omega(g(n,s))$
(resp. $O(g(n,s))$) if $r_{c}^{*}(\mathcal{P}(\mathcal{X}))\ge\gamma_{p,q}g(n,s)$
(resp. $\le\gamma_{p,q}g(n,s)$) for all $n,s,p,q$, where the coefficient
$\gamma_{p,q}$ only depends on $p,q$. Therefore we write $O(n^{1/p}\log(s+1))$
instead of $O(n^{1/p}\log s)$ since the expression should also be
valid when $s=1$, $n\to\infty$. Note that the last two upper bounds
for the case $\mathcal{X}=[0..s]$ are due to the ratio bound in Proposition
\ref{prop:rc_prop_misc} with respect to the discrete metric and the
$\ell_{1}$ metric, respectively. }

\end{sidewaystable}

\begin{figure}[H]
\begin{centering}
\includegraphics[scale=0.43]{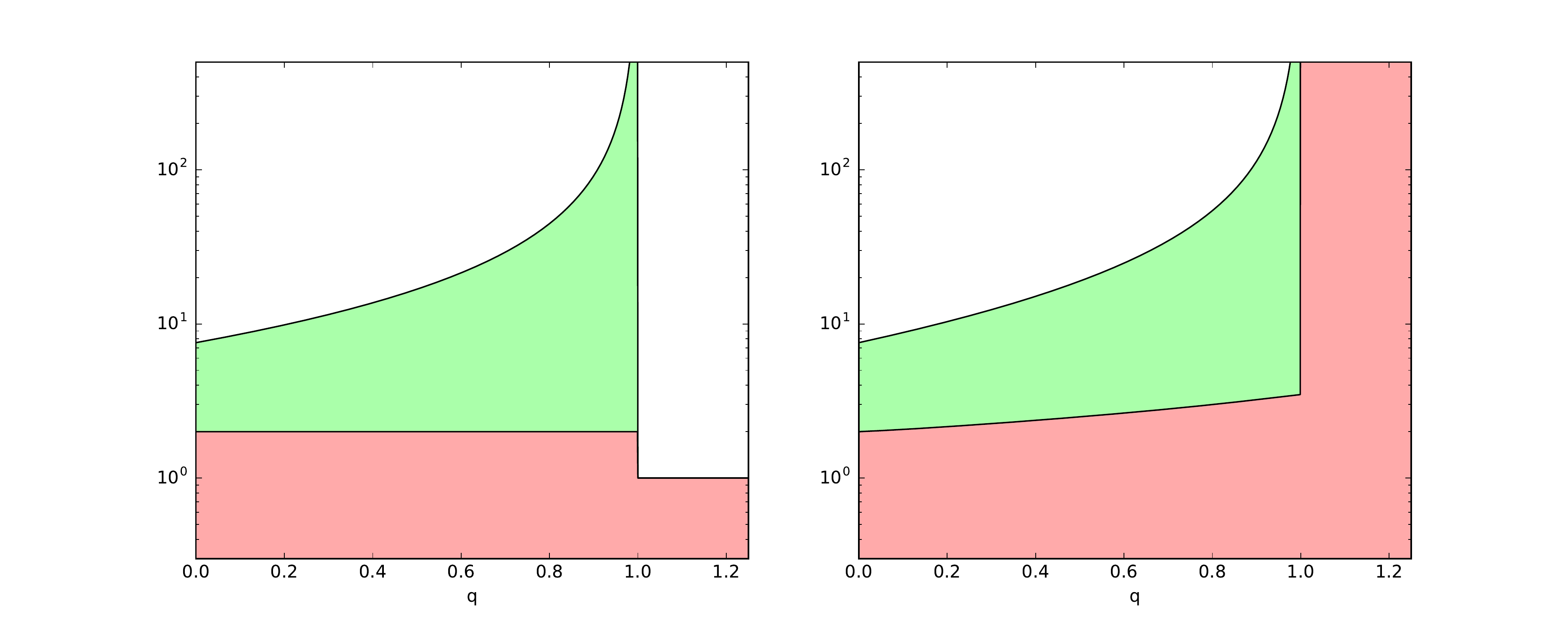}
\par\end{centering}
\caption{\label{fig:rn_plot_n12}Log-scale plot of the upper bound (Theorem
\ref{thm:rn_rc_ub}) and lower bound (Proposition \ref{prop:rn_rc_lb_1}
and Theorem \ref{thm:rn_rc_lb_ball}) on $r_{c}^{*}$ with $c(x,y)=\Vert x-y\Vert_{2}^{q}$
for $\mathbb{R}$ (left) and $\mathbb{R}^{2}$ (right) against $q$.
The green region is the region of possible values of $r_{c}^{*}$.
 The behavior of $r_{c}^{*}(\mathcal{P}(\mathbb{R}))$ is very different
from that of $r_{c}^{*}(\mathcal{P}(\mathbb{R}^{2}))$, as $r_{c}^{*}(\mathcal{P}(\mathbb{R}))$
is smaller when $q\ge1$ compared to $0<q<1$, whereas $r_{c}^{*}(\mathcal{P}(\mathbb{R}^{2}))$
is infinite for $q\ge1$ but finite for $0<q<1$. Note that the lower
bound for $\mathbb{R}^{2}$ tends to $\infty$ continuously as $q\to1^{-}$,
which may not be apparent in the figure.}

\end{figure}

\begin{figure}[H]
\begin{centering}
\includegraphics[scale=0.43]{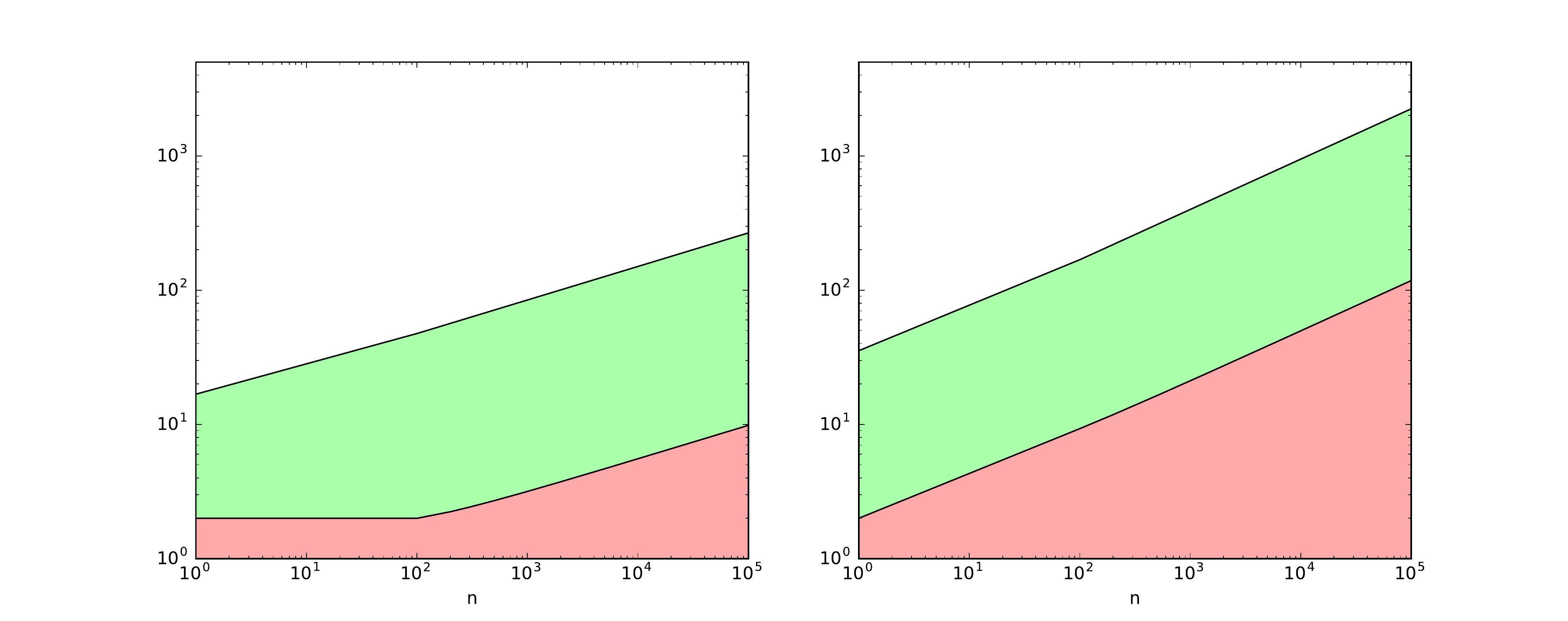}
\par\end{centering}
\caption{\label{fig:rn_plot_n}Log-log plot of the upper bound (Theorem \ref{thm:rn_rc_ub})
and lower bound (Proposition \ref{prop:rn_rc_lb_1} and Theorem \ref{thm:rn_rc_lb_ball})
on $r_{c}^{*}$ with $c(x,y)=\Vert x-y\Vert_{2}^{q}$ on $\mathbb{R}^{n}$
for $q=1/2$ (left) and $q=3/4$ (right) against $n$. The green region
is the region of possible values of $r_{c}^{*}$.}
\end{figure}
\begin{figure}[H]
\vspace{-5pt}

\begin{centering}
\includegraphics[scale=0.38]{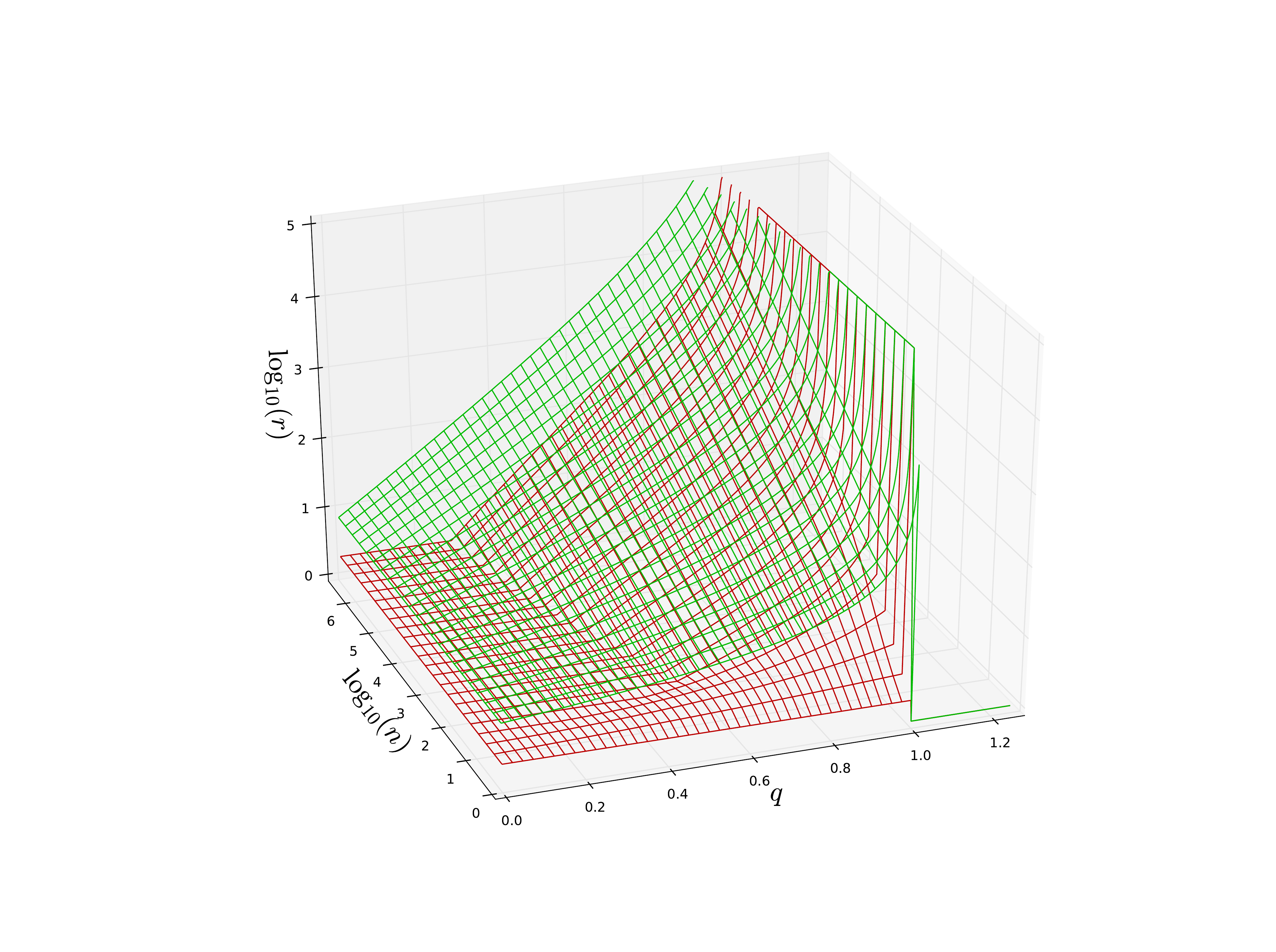}\vspace{-40pt}
\par\end{centering}
\caption{\label{fig:rn_plot_3d}Log-scale plot of the upper bound (Theorem
\ref{thm:rn_rc_ub}, green wireframe) and lower bound (Proposition
\ref{prop:rn_rc_lb_1} and Theorem \ref{thm:rn_rc_lb_ball}, red wireframe)
on $r_{c}^{*}$ with $c(x,y)=\Vert x-y\Vert_{2}^{q}$ on $\mathbb{R}^{n}$
for different $n$'s (in log scale) and $q$'s. Note that the numbers
on the vertical axis denote $\log_{10}(r)$ where $r$ is either the
upper or lower bound on $r_{c}^{*}$.}
\end{figure}

\section{Applications\label{sec:applications}}

In this section, we present several applications of the pairwise multi-marginal
optimal transport problem.

\subsection{Sketching and Locality Sensitive Hashing\label{subsec:lsh_app}}

Locality sensitive hashing for probability distributions has been
studied extensively, e.g., in \cite{charikar2002similarity,indyk2003fast,lv2004image,andoni2009efficient}.
A hash function $h_{Z}$ (indexed by a random variable $Z\sim Q$,
i.e., $h_{Z}$ is a random function) maps a probability distribution
$P_{\alpha}$ (in a collection $\{P_{\alpha}\}_{\alpha}$) into a
low-dimensional data structure $h_{Z}(P_{\alpha})$, such that $h_{Z}(P_{\alpha})$
is close to $h_{Z}(P_{\beta})$ whenever $P_{\alpha}$ is close to
$P_{\beta}$ (e.g. in earth mover's distance $C_{c}^{*}(P_{\alpha},P_{\beta})$
when $c$ is a metric). Such hash functions have various applications
including duplicate detection, nearest neighbor search, and estimation
of earth mover's distance.

A coupling $\{X_{\alpha}\}_{\alpha}$ of $\{P_{\alpha}\}_{\alpha}$
which attains a small $r:=r_{c}(\{X_{\alpha}\}_{\alpha})$ can be
regarded as a locality sensitive hash function for the earth mover's
distance. Assume there is a random variable $Z\sim\mathrm{Unif}[0,1]$
such that $X_{\alpha}=X_{\alpha}(Z)$ is a function of $Z$ for all
$\alpha$ (as in the definition of $\Gamma_{\lambda}(\{P_{\alpha}\}_{\alpha})$).
Let $h_{Z}:\{P_{\alpha}:\alpha\in\mathcal{A}\}\to\mathcal{X}$ be
defined by $h_{Z}(P_{\alpha}):=X_{\alpha}(Z)$, then we have
\[
C_{c}^{*}(P_{\alpha},P_{\beta})\le\mathbf{E}[c(h_{Z}(P_{\alpha}),h_{Z}(P_{\beta}))]\le rC_{c}^{*}(P_{\alpha},P_{\beta}).
\]
If $c$ is a metric, then $\mathbf{E}[c(h_{Z}(P_{\alpha}),h_{Z}(P_{\beta}))]$
is within a multiplicative factor from the earth mover's distance
$C_{c}^{*}(P_{\alpha},P_{\beta})$. Therefore we can estimate $C_{c}^{*}(P_{\alpha},P_{\beta})$
by drawing i.i.d. samples $Z_{1},\ldots,Z_{k}$, and taking the sample
mean of $c(h_{Z_{i}}(P_{\alpha}),h_{Z_{i}}(P_{\beta}))$.

The vector $\{h_{Z_{i}}(P_{\alpha})\}_{i\in[1..k]}\in\mathcal{X}^{k}$
also serves as a sketch of the distribution $P_{\alpha}$. For example,
let $\mathcal{X}=[0..s]^{n}$, $s\ge1$, $c(x,y)=\Vert x-y\Vert_{1}$.
Then $\{h_{Z_{i}}(P_{\alpha})\}_{i}\in\mathbb{R}^{nk}$, and $C_{c}^{*}(P_{\alpha},P_{\beta})$
can be estimated by the sample mean
\[
\frac{1}{k}\sum_{i=1}^{k}c(h_{Z_{i}}(P_{\alpha}),h_{Z_{i}}(P_{\beta}))=\frac{1}{k}\Vert\{h_{Z_{i}}(P_{\alpha})\}_{i}-\{h_{Z_{i}}(P_{\beta})\}_{i}\Vert_{1},
\]
which is the $\ell_{1}$ distance on $\mathbb{R}^{nk}$.  Refer to
Section \ref{subsec:rn_lp} for an algorithm for computing $h_{Z}$
for this example.

\medskip{}

\subsection{The Labeling Problem for Classification\label{subsec:label_prob}}

We describe the labeling problem studied in \cite{kleinberg2002approximation,archer2004approximate,huhnerbein2018image}
(also see the energy minimization problem studied in \cite{boykov1999fast}).
Consider a finite set of objects $\mathcal{A}$ to be classified,
and a finite set of classes or labels $\mathcal{X}$. A labeling is
a function $f:\mathcal{A}\to\mathcal{X}$ mapping each object to its
assigned label. The cost of assigning object $\alpha$ to label $x$
is given by $g(\alpha,x)$, where $g:\mathcal{A}\times\mathcal{X}\to\mathbb{R}_{\ge0}$.
For example, such a cost may come from an estimate of the likelihood
that object $\alpha$ belongs to the class $x$. We also have a set
of pairs of related objects $E\subseteq\mathcal{A}$. We want a related
pair of objects $(\alpha,\beta)\in E$ to be assigned similar labels,
that is, $c(f(\alpha),f(\beta))$ is small, where $c:\mathcal{X}\times\mathcal{X}\to\mathbb{R}_{\ge0}$
measures the difference between two labels (and is assumed to be a
symmetric cost function). The weight of a pair $(\alpha,\beta)\in E$
is given by $w(\alpha,\beta)\ge0$. The total cost of the labeling
$f$ is given by
\[
Q(f):=\sum_{\alpha\in\mathcal{A}}g(\alpha,f(\alpha))+\sum_{(\alpha,\beta)\in E}w(\alpha,\beta)c(f(\alpha),f(\beta)).
\]
The problem of minimizing $Q(f)$ is NP-hard, and can be expressed
as a mixed 0-1 integer program (see \cite{kleinberg2002approximation,boykov2001fast}).
Nevertheless, it is possible to give an approximate algorithm by pairwise
multi-marginal optimal transport, using the earthmover linear program
in \cite{kleinberg2002approximation,archer2004approximate}. We relax
the problem of minimizing $Q(f)$ to allow fractional labeling. A
fractional labeling is a collection of distributions $\{P_{\alpha}\}_{\alpha\in\mathcal{A}}$
over $\mathcal{X}$, where each object $\alpha$ is assigned a distribution
$P_{\alpha}$ over labels rather than only one label. For $(\alpha,\beta)\in E$,
the cost between $\alpha$ and $\beta$ is approximated by $C_{c}^{*}(P_{\alpha},P_{\beta})$.
Let the total approximate cost be
\[
\tilde{Q}(\{P_{\alpha}\}_{\alpha}):=\sum_{\alpha\in\mathcal{A}}\mathbf{E}_{X\sim P_{\alpha}}[g(\alpha,X)]+\sum_{(\alpha,\beta)\in E}w(\alpha,\beta)C_{c}^{*}(P_{\alpha},P_{\beta}).
\]
Since a labeling $f$ is a special case of a fractional labeling (by
letting $P_{\alpha}=\delta_{f(\alpha)}$), we have $\inf_{\{P_{\alpha}\}_{\alpha}}\tilde{Q}(\{P_{\alpha}\}_{\alpha})\le\inf_{f}Q(f)$.
The minimum of $\tilde{Q}(\{P_{\alpha}\}_{\alpha})$ can be found
 using linear programming (since $C_{c}^{*}(P_{\alpha},P_{\beta})$
can be expressed as a linear program).

Let $\{P_{\alpha}\}_{\alpha}$ be the minimizer of $\tilde{Q}(\{P_{\alpha}\}_{\alpha})$,
and $\{X_{\alpha}\}_{\alpha}$ be a coupling achieving $r_{c}(\{X_{\alpha}\}_{\alpha})\le r_{c}^{*}(\{P_{\alpha}\}_{\alpha})+\epsilon$.
We construct $f$ randomly as $f_{\{X_{\alpha}\}_{\alpha}}(\alpha):=X_{\alpha}$
(i.e., $f_{\{X_{\alpha}\}_{\alpha}}$ is a random function that depends
on the values of $\{X_{\alpha}\}_{\alpha}$). This construction gives
an expected cost
\begin{align*}
\mathbf{E}[Q(f_{\{X_{\alpha}\}_{\alpha}})] & =\mathbf{E}\left[\sum_{\alpha\in\mathcal{A}}g(\alpha,X_{\alpha})+\sum_{(\alpha,\beta)\in E}w(\alpha,\beta)c(X_{\alpha},X_{\beta})\right]\\
 & \le\sum_{\alpha\in\mathcal{A}}\mathbf{E}_{X\sim P_{\alpha}}[g(\alpha,X)]+\left(r_{c}^{*}(\{P_{\alpha}\}_{\alpha})+\epsilon\right)\sum_{(\alpha,\beta)\in E}w(\alpha,\beta)C_{c}^{*}(P_{\alpha},P_{\beta}),
\end{align*}
which is at most a factor $r_{c}^{*}(\{P_{\alpha}\}_{\alpha})+\epsilon$
away from $\tilde{Q}(\{P_{\alpha}\}_{\alpha})$, and hence at most
a factor $r_{c}^{*}(\{P_{\alpha}\}_{\alpha})+\epsilon$ away from
the optimum $\inf_{f}Q(f)$. This performance guarantee is universal
in the sense that the approximation factor is bounded by $r_{c}^{*}(\mathcal{P}(\mathcal{X}))+\epsilon$
regardless of $g$, $E$ and $w$. Efficient algorithms for computing
$\{X_{\alpha}\}_{\alpha}$ exist, for example, when $(\mathcal{X},c)$
is a finite metric space (see Section \ref{subsec:metric_finite}).

We remark that this is an example of a dependent randomized rounding
algorithm for converting a solution of a fractional relaxation of
a mixed 0-1 integer program to a solution of the original problem.
See \cite{bertsimas1999dependent} for discussions.

\medskip{}

\subsection{Robust Computation of Transport Plans}

Let $\{P_{\alpha}\}_{\alpha\in\mathcal{A}}$ be a collection of probability
distributions over $\mathcal{X}$. A \emph{transport plan computation
function} is a function $G:\{P_{\alpha}:\alpha\in\mathcal{A}\}\times\{P_{\alpha}:\alpha\in\mathcal{A}\}\to\mathcal{P}(\mathcal{X}^{2})$
which, given any pair of probability distributions $P_{\alpha},P_{\beta}$
in the collection, outputs $G(P_{\alpha},P_{\beta})\in\Gamma(P_{\alpha},P_{\beta})$
($G(P_{\alpha},P_{\beta})$ is a probability distribution over $\mathcal{X}^{2}$;
see \eqref{eq:coupling_def} for the definition of $\Gamma(P_{\alpha},P_{\beta})$)
such that $\mathbf{E}_{(X,Y)\sim G(P_{\alpha},P_{\beta})}[c(X,Y)]$
is close to the optimal value $C_{c}^{*}(P_{\alpha},P_{\beta})$.
The inputs $P_{\alpha},P_{\beta}$ may be perturbed into $P_{\alpha'},P_{\beta'}$
respectively (e.g. due to an imperfect estimation of $P_{\alpha},P_{\beta}$,
or due to an adversary), $\alpha',\beta'\in\mathcal{A}$, where $P_{\alpha'}$
is close to $P_{\alpha}$, and $P_{\beta'}$ is close to $P_{\beta}$,
in the sense that $C_{c}^{*}(P_{\alpha},P_{\alpha'})+C_{c}^{*}(P_{\beta},P_{\beta'})\le\epsilon$
for some $\epsilon>0$. A robust transport plan computation function
should be robust against such perturbations in the sense that $G(P_{\alpha'},P_{\beta'})\in\Gamma(P_{\alpha'},P_{\beta'})$
(the coupling computed using the perturbed distributions) is close
to $G(P_{\alpha},P_{\beta})$. More specifically, we say $G$ is $r$\emph{-robust},
$r>0$, if 
\begin{align*}
 & C_{c\times c}^{*}\left(G(P_{\alpha},P_{\beta}),\,G(P_{\alpha'},P_{\beta'})\right)\\
 & =\inf_{(X,Y)\sim G(P_{\alpha},P_{\beta}),\,(X',Y')\sim G(P_{\alpha'},P_{\beta'})}\mathbf{E}\left[c(X,X')+c(Y,Y')\right]\\
 & \le r\epsilon
\end{align*}
for any $\epsilon>0$ and $P_{\alpha},P_{\alpha'},P_{\beta},P_{\beta'}$
satisfying $C_{c}^{*}(P_{\alpha},P_{\alpha'})+C_{c}^{*}(P_{\beta},P_{\beta'})\le\epsilon$,
where $(c\times c):\mathcal{X}^{2}\times\mathcal{X}^{2}\to\mathbb{R}$
is defined by $(c\times c)((x,y),(x',y')):=c(x,x')+c(y,y')$, which
we call the 1-product cost function.

Note that if $c$ is a metric, then $C_{c}^{*}$ is the 1-Wasserstein
distance over $\mathcal{P}(\mathcal{X})$, and $C_{c\times c}^{*}$
is the 1-Wasserstein distance over $\mathcal{P}(\mathcal{X}^{2})$
with respect to the 1-product metric $c\times c$. Therefore, $G$
is $r$-robust if and only if $G$ is an $r$-Lipschitz function with
respect to $C_{c}^{*}\times C_{c}^{*}$ (the $1$-product of 1-Wasserstein
distances, which is a metric over the space of inputs pairs of probability
distributions) and $C_{c\times c}^{*}$.

Let $G^{*}(P_{\alpha},P_{\beta})$ be the transport plan computation
function that gives the optimal coupling attaining $\mathbf{E}_{(X,Y)\sim G^{*}(P_{\alpha},P_{\beta})}[c(X,Y)]=C_{c}^{*}(P_{\alpha},P_{\beta})$
for any $(P_{\alpha},P_{\beta})$. Then $G^{*}$ may not be robust
against perturbations. For example, consider $\mathcal{X}=\mathbb{R}^{2}$,
$c(x,y)=\sqrt{\Vert x-y\Vert_{2}}$ (which is a metric), and consider
the collection $\mathcal{P}(\mathcal{X})$ of all probability distributions.
Let 
\begin{equation}
P_{\alpha}=(\delta_{(1,0)}+\delta_{(-1,0)})/2,\;P_{\beta}=(\delta_{(2\epsilon^{2},1)}+\delta_{(0,-1)})/2,\;P_{\beta'}=(\delta_{(-2\epsilon^{2},1)}+\delta_{(0,-1)})/2.\label{eq:robust_eg}
\end{equation}
We have $C_{c}^{*}(P_{\beta},P_{\beta'})\le\epsilon$. The optimal
coupling between $(P_{\alpha},P_{\beta})$ is $G^{*}(P_{\alpha},P_{\beta})=(\delta_{((1,0),(2\epsilon^{2},1))}+\delta_{((-1,0),(0,-1))})/2$,
and the optimal coupling between $(P_{\alpha},P_{\beta'})$ is $G^{*}(P_{\alpha},P_{\beta'})=(\delta_{((1,0),(0,-1))}+\delta_{((-1,0),(-2\epsilon^{2},1))})/2$.
It can be checked that $C_{c\times c}^{*}(G^{*}(P_{\alpha},P_{\beta}),G^{*}(P_{\alpha},P_{\beta'}))\ge\sqrt{2}$,
which can be much larger than the perturbation $\epsilon$.

We can design a robust transport plan computation function using pairwise
multi-marginal optimal transport. Let $\{X_{\alpha}\}_{\alpha}\in\Gamma_{\lambda}(\{P_{\alpha}\}_{\alpha})$
and $G_{\mathrm{r}}(P_{\alpha},P_{\beta})$ be the joint distribution
of $(X_{\alpha},X_{\beta})$. We can guarantee that $\mathbf{E}_{(X,Y)\sim G_{\mathrm{r}}(P_{\alpha},P_{\beta})}[c(X,Y)]\le rC_{c}^{*}(P_{\alpha},P_{\beta})$.
If $P_{\alpha},P_{\beta}$ are perturbed into $P_{\alpha'},P_{\beta'}$
respectively, where $C_{c}^{*}(P_{\alpha},P_{\alpha'})+C_{c}^{*}(P_{\beta},P_{\beta'})\le\epsilon$,
then 
\begin{align*}
 & C_{c\times c}^{*}\left(G_{\mathrm{r}}(P_{\alpha},P_{\beta}),\,G_{\mathrm{r}}(P_{\alpha'},P_{\beta'})\right)\\
 & \le\mathbf{E}\left[c(X_{\alpha},X_{\alpha'})+c(X_{\beta},X_{\beta'})\right]\\
 & \le r_{c}(\{X_{\alpha}\}_{\alpha})C_{c}^{*}(P_{\alpha},P_{\alpha'})+r_{c}(\{X_{\alpha}\}_{\alpha})C_{c}^{*}(P_{\beta},P_{\beta'})\\
 & \le r_{c}(\{X_{\alpha}\}_{\alpha})\epsilon.
\end{align*}
Hence $G_{\mathrm{r}}$ is $r$-robust for $r=r_{c}(\{X_{\alpha}\}_{\alpha})$,
which can be made arbitrarily close to $r_{c}^{*}(\{P_{\alpha}\}_{\alpha})$.
In other words, at the expense of having a suboptimal coupling (within
a factor $r$ from optimal), $G_{\mathrm{r}}$ is robust in the sense
that perturbing the input probability distributions by $\epsilon$
(with respect to $C_{c}^{*}$) will only change the output coupling
by $r\epsilon$ (with respect to $C_{c\times c}^{*}$). For the case
$\mathcal{X}=\mathbb{R}^{2}$, $c(x,y)=\sqrt{\Vert x-y\Vert_{2}}$,
we can achieve $r=26$ by Theorem \ref{thm:rn_rc_ub}.

Besides the theoretical elegance of having a Lipschitz continuous
transport plan computation function, the notion of $r$-robustness
also has practical relevance. For instance, if two parties compute
transport plans based on two different perturbed versions of $(P_{\alpha},P_{\beta})$
(e.g. due to imperfect estimations of $P_{\alpha},P_{\beta}$) separately
using the same robust transport plan computation function, then the
two output couplings will be similar. Hence, there will be little
discrepancy when they perform actions based on those couplings.

The usage of a robust transport plan computation function also disincentivizes
the manipulation of the input probability distributions $P_{\alpha},P_{\beta}$
by another party. Assume $c$ is a metric. Suppose that party is capable
of perturbing $(P_{\alpha},P_{\beta})$ into $(P_{\alpha'},P_{\beta'})$
by $\epsilon$ (i.e., $C_{c}^{*}(P_{\alpha},P_{\alpha'})+C_{c}^{*}(P_{\beta},P_{\beta'})\le\epsilon$)
in order to manipulate the output coupling $G(P_{\alpha'},P_{\beta'})$,
so that the gain of that party measured by $\mathbf{E}_{(X,Y)\sim G(P_{\alpha'},P_{\beta'})}[h(X,Y)]$
is increased, where $h:\mathcal{X}^{2}\to\mathbb{R}$ is the gain
function. For example, in the setting where we design a transport
plan from mines to factories, the owner of a large factory that spans
multiple streets may choose to report a street address closer to a
mine he/she secretly prefers, so that the preferred mine would be
assigned to his/her factory. Such a manipulation can be effective
against the optimal transport plan computation function $G^{*}$ (refer
to the example in \eqref{eq:robust_eg}), but not against a robust
function. If $h$ is $\eta$-Lipschitz with respect to $c\times c$
for some $\eta>0$, then by using an $r$-robust transport plan computation
function $G$, we can guarantee that
\[
\left|\mathbf{E}_{(X,Y)\sim G(P_{\alpha'},P_{\beta'})}[h(X,Y)]-\mathbf{E}_{(X,Y)\sim G(P_{\alpha},P_{\beta})}[h(X,Y)]\right|\le\eta r\epsilon,
\]
i.e., that party cannot increase the gain by more than $\eta r\epsilon$
by manipulating $(P_{\alpha},P_{\beta})$.

\medskip{}

\subsection{Distributed Computation of Transport Plans}

Consider the setting where there are $k$ trucks to be assigned to
$k$ mines and $k$ factories, where each truck is assigned to one
mine and one factory, and is responsible for transportation between
them. Let the locations of mines and factories be $u_{1},\ldots,u_{k}\in\mathcal{X}$
and $v_{1},\ldots,v_{k}\in\mathcal{X}$ respectively. Let $P=k^{-1}\sum_{i=1}^{k}\delta_{u_{i}}$
be the distribution of mines (similarly define $Q$ for the distribution
of factories). The transportation cost between locations $x$ and
$y$ is $c(x,y)$. If the truck assignment is performed centrally,
then the optimal assignment can be obtained from solving the original
optimal transport problem.

We consider a variant of this setting where the assignment is performed
in a distributed manner. Assume $P$ belongs to a collection of probability
distributions $\{P_{\alpha}\}_{\alpha\in\mathcal{A}}$, and $Q$ belongs
to a collection of probability distributions $\{Q_{\beta}\}_{\beta\in\mathcal{B}}$.
Suppose the mining administration knows the distribution of mines
$P$, and is responsible for assigning trucks to mines, i.e., designing
a mapping $x:[1..k]\to\mathcal{X}$, where $x(i)$ is the location
of the mine to which the $i$-th truck is assigned. The mining administration
does not know $Q$, so its only information about $Q$ is that $Q$
belongs to $\{Q_{\beta}\}_{\beta\in\mathcal{B}}$. Similarly, the
industrial administration knows the distribution of factories $Q$
and that $P\in\{P_{\alpha}\}_{\alpha\in\mathcal{A}}$ (but not the
exact choice of $P$), and is responsible for assigning trucks to
factories (designing a mapping $y:[1..k]\to\mathcal{X}$). Let $U\sim\mathrm{Unif}[1..k]$,
$X:=x(U)$, $Y:=y(U)$. Then $(X,Y)$ is a coupling of $P,Q$. The
goal is to guarantee that such a distributed computation of the coupling
would give a total cost $k\mathbf{E}[c(X,Y)]$ at most a factor of
$r$ from the lowest possible total cost $kC_{c}^{*}(P,Q)$ for any
$P\in\{P_{\alpha}\}_{\alpha\in\mathcal{A}}$, $Q\in\{Q_{\beta}\}_{\beta\in\mathcal{B}}$,
where $r\ge1$ is as small as possible.

Since the mining administration produces the mapping $x$ using only
$P$, we can let $x_{\alpha}$ be the mapping produced when $P=P_{\alpha}$,
and $X_{\alpha}:=x_{\alpha}(U)$. Define $Y_{\beta}$ similarly. Then
$\{\{X_{\alpha}\}_{\alpha},\{Y_{\beta}\}_{\beta}\}$ is a coupling
of $\{\{P_{\alpha}\}_{\alpha},\{Q_{\beta}\}_{\beta}\}$. The problem
becomes that of minimizing 
\begin{equation}
\sup_{\alpha\in\mathcal{A},\,\beta\in\mathcal{B}}\frac{\mathbf{E}[c(X_{\alpha},Y_{\beta})]}{C_{c}^{*}(P_{\alpha},Q_{\beta})}\label{eq:dist_otc}
\end{equation}
subject to $X_{\alpha}\sim P_{\alpha}$, $Y_{\beta}\sim Q_{\beta}$
(assume $k$ is large enough that we can lift the restriction that
the probabilities in the probability mass function of $X_{\alpha}$
have to be multiples of $k^{-1}$).

This is a more general problem than that in Definition \eqref{def:pcr}
studied in this paper. If we assume that $\mathcal{A}=\mathcal{B}$,
$P_{\alpha}=Q_{\alpha}$ for any $\alpha\in\mathcal{A}$, and $c$
is a symmetric cost function satisfying $c(x,y)>0$ for any $x\neq y$,
then we can assume $X_{\alpha}=Y_{\alpha}$ almost surely for any
$\alpha\in\mathcal{A}$ (or else $\mathbf{E}[c(X_{\alpha},Y_{\alpha})]>0=C_{c}^{*}(P_{\alpha},Q_{\alpha})$,
making \eqref{eq:dist_otc} infinite). In this case, the infimum of
\eqref{eq:dist_otc} is the same as $r_{c}^{*}(\{P_{\alpha}\}_{\alpha})$.

\medskip{}

\subsection{Online Computation of Transport Plans}

Let $\{P_{\alpha}\}_{\alpha\in\mathcal{A}}$ be a collection of probability
distributions over $\mathcal{X}$. At the beginning, there are $k$
agents at locations $z_{0}(1),\ldots,z_{0}(k)$, where $z_{0}:[1..k]\to\mathcal{X}$.
The initial distribution of the agents is $P_{\alpha_{0}}=k^{-1}\sum_{i=1}^{k}\delta_{z_{0}(i)}$.
Let $b_{0}=0$. At time $t\in\mathbb{N}$, we observe a number $b_{t}\in\mathbb{Z}_{\ge-1}$
(the task index) and a probability distribution $P_{\alpha_{t}}$
(the task distribution) in the collection of probability distributions,
which falls in one of the following three cases:
\begin{enumerate}
\item (New tasks). A new batch of tasks is indicated by $b_{t}=t$. In this
case, $P_{\alpha_{t}}$ gives the distribution of the locations of
$k$ tasks in this batch, where each task requires an agent to complete.
We have to assign these new tasks to the agents. That is, we choose
$z_{t}:[1..k]\to\mathcal{X}$ such that $k^{-1}\sum_{i=1}^{k}\delta_{z_{t}(i)}=P_{\alpha_{t}}$.
\item (Existing tasks). If $0\le b_{t}<t$, then it signals each agent to
return to the task assigned to him/her at time $b_{t}$ for some follow-up
work (if $b_{t}=0$, this means that the agents should return to their
initial locations). The agents must return to their assigned tasks,
and cannot exchange tasks among themselves. In this case, $P_{\alpha_{t}}=P_{\alpha_{b_{t}}}$,
and $z_{t}=z_{b_{t}}$. We do not have to make any choice in this
case.
\item (Stop). If $b_{t}=-1$, then this marks the end of the procedure.
There are no further tasks and no need for further traveling.
\end{enumerate}
Note that upon receiving a new batch of tasks, we have to design $z_{t}$
online using only the information $b_{1},\ldots,b_{t}$, $P_{\alpha_{0}},\ldots,P_{\alpha_{t}}$,
and $z_{0},\ldots,z_{t-1}$. We do not know the stopping time until
it comes.

The cost of traveling from location $x$ to $y$ is given by $c(x,y)$.
Our goal is to assign the tasks to the agents such that the average
total traveling cost $k^{-1}\sum_{t=1}^{T}\sum_{i=1}^{k}c(z_{t-1}(i),z_{t}(i))$
is small, where $T$ is the time just before stopping (i.e., $b_{T+1}=-1$).
Let $U\sim\mathrm{Unif}[1..k]$ and $Z_{t}:=z_{t}(U)$. Then designing
$z_{t}$ upon receiving a new batch of tasks is equivalent to designing
a random variable $Z_{t}$ (dependent of $Z_{0},\ldots,Z_{t-1}$)
with marginal distribution $P_{\alpha_{t}}$ (assume $k$ is large
enough that we can lift the restriction that the probabilities in
the probability mass function of $(Z_{0},\ldots,Z_{t})$ have to be
multiples of $k^{-1}$). The average total traveling cost can be given
by $k^{-1}\sum_{t=1}^{T}\sum_{i=1}^{k}c(z_{t-1}(i),z_{t}(i))=\sum_{t=1}^{T}\mathbf{E}[c(Z_{t-1},Z_{t})]$.
We want to guarantee that
\begin{equation}
\sum_{t=1}^{T}\mathbf{E}[c(Z_{t-1},Z_{t})]\le r\inf_{\{\tilde{Z}_{t}\}_{t}:\,\tilde{Z}_{t}\sim P_{\alpha_{t}},\,\tilde{Z}_{t}=\tilde{Z}_{b_{t}}\,\forall t\in[0..T]}\sum_{t=1}^{T}\mathbf{E}[c(\tilde{Z}_{t-1},\tilde{Z}_{t})]\label{eq:online_plan}
\end{equation}
for any $\{b_{t}\}_{t}$ and $\{\alpha_{t}\}_{t}$, where $r\ge1$
is as small as possible. The infimum in the right hand side is the
optimal average total cost for an offline scheme where we decide on
all $\{Z_{t}\}_{t}$ together using the knowledge of all the $\{b_{t}\}_{t}$
and $\{P_{\alpha_{t}}\}_{t}$. In other words, the online scheme gives
an average total cost within a factor of $r$ from that given by the
optimal offline scheme. The requirement that the scheme is online
only results in a multiplicative penalty of $r$.

We consider the following three online schemes:
\begin{enumerate}
\item (Greedy scheme). Upon receiving a new batch of tasks at time $t$,
design $Z_{t}$ such that $\mathbf{E}[c(Z_{t-1},Z_{t})]=C_{c}^{*}(P_{\alpha_{t-1}},P_{\alpha_{t}})$,
i.e., minimize the cost of travelling from time $t-1$ to time $t$.
\item (Reactive minimax scheme). Upon receiving a new batch of tasks, design
$Z_{t}$ such that \linebreak$\max_{0\le i<t}\mathbf{E}[c(Z_{i},Z_{t})]/C_{c}^{*}(P_{\alpha_{i}},P_{\alpha_{t}})$
is minimized. 
\item (Preemptive minimax scheme). Let $\{X_{\alpha}\}_{\alpha}\in\Gamma_{\lambda}(\{P_{\alpha}\}_{\alpha})$,
where $r_{c}(\{X_{\alpha}\}_{\alpha})\le r_{c}^{*}(\{P_{\alpha}\}_{\alpha})+\epsilon$
for some small $\epsilon>0$. Upon receiving a new batch of tasks,
let $Z_{t}=X_{\alpha_{t}}$.
\end{enumerate}
For the preemptive minimax scheme, we have $\mathbf{E}[c(Z_{t-1},Z_{t})]\le r_{c}(\{X_{\alpha}\}_{\alpha})C_{c}^{*}(P_{\alpha_{t-1}},P_{\alpha_{t}})$,
and hence
\begin{align*}
 & \sum_{t=1}^{T}\mathbf{E}[c(Z_{t-1},Z_{t})]\\
 & \le r_{c}(\{X_{\alpha}\}_{\alpha})\sum_{t=1}^{T}C_{c}^{*}(P_{\alpha_{t-1}},P_{\alpha_{t}})\\
 & \le r_{c}(\{X_{\alpha}\}_{\alpha})\inf_{\{\tilde{Z}_{t}\}_{t}:\,\tilde{Z}_{t}\sim P_{\alpha_{t}},\,\tilde{Z}_{t}=\tilde{Z}_{b_{t}}\,\forall t\in[0..T]}\sum_{t=1}^{T}\mathbf{E}[c(\tilde{Z}_{t-1},\tilde{Z}_{t})],
\end{align*}
and we can achieve $r=r_{c}(\{X_{\alpha}\}_{\alpha})\le r_{c}^{*}(\{P_{\alpha}\}_{\alpha})+\epsilon$
in \eqref{eq:online_plan} if $r_{c}^{*}(\{P_{\alpha}\}_{\alpha})<\infty$.

On the other hand, the greedy scheme  fails to satisfy \eqref{eq:online_plan}.
Consider $\mathcal{X}=\mathbb{R}^{2}$, $c(x,y)=\sqrt{\Vert x-y\Vert_{2}}$
(which is a metric), and consider the collection $\mathcal{P}(\mathcal{X})$
of all probability distributions. The preemptive minimax scheme achieves
$r=26$ by Theorem \ref{thm:rn_rc_ub}. We now give an example where
the greedy scheme fails. Let $l\in\mathbb{N}$, $4\le l<T$,
\[
(b_{t},P_{\alpha_{t}})=\begin{cases}
\left(t,\,\frac{1}{2}(\delta_{(\cos(\pi t/l),\sin(\pi t/l))}+\delta_{(-\cos(\pi t/l),-\sin(\pi t/l))})\right) & \mathrm{if}\;t\le l-2\\
(0,\,P_{\alpha_{0}}) & \mathrm{if}\;t\ge l-1\;\mathrm{and}\;t-l\;\mathrm{is}\,\mathrm{odd}\\
(l-2,\,P_{\alpha_{l-2}}) & \mathrm{if}\;t\ge l-1\;\mathrm{and}\;t-l\;\mathrm{is}\,\mathrm{even}.
\end{cases}
\]
It can be checked that the greedy scheme  selects $Z_{t}=(U\cos(\pi t/l),U\sin(\pi t/l))$
for $t=0,\ldots,l-2$, where $U\sim\mathrm{Unif}\{1,-1\}$. The average
total cost is 
\[
(l-2)\sqrt{2\sin\left(\frac{\pi}{2l}\right)}+(T-l+2)\sqrt{2\sin\left(\frac{\pi(l-2)}{2l}\right)}\approx\sqrt{\pi l}+(T-l)\sqrt{2}
\]
for large $l,T$, which can be much larger than the average total
cost given by the preemptive minimax scheme, which is upper-bounded
by
\begin{align*}
26\sum_{t=1}^{T}C_{c}^{*}(P_{\alpha_{t-1}},P_{\alpha_{t}}) & =26\left((l-2)\sqrt{2\sin\left(\frac{\pi}{2l}\right)}+(T-l+2)\sqrt{2\sin\left(\frac{2\pi}{2l}\right)}\right)\\
 & \approx26\left(\sqrt{\pi l}+(T-l)\sqrt{2\pi/l}\right).
\end{align*}
If $T=l^{2}$, then the average total cost of the greedy scheme grows
like $l^{2}$, whereas that of the preemptive minimax scheme grows
like $l^{3/2}$, which implies that the average total cost of the
optimal offline scheme grows at most like $l^{3/2}$.

It is uncertain whether the reactive minimax scheme achieves the same
$r$ (or within a multiplicative factor) as the preemptive minimax
scheme. Nevertheless, the preemptive minimax scheme does not involve
solving an optimization problem for each time step, and thus is simpler
to implement.

If $\mathcal{A}$ is finite and $\sup_{\alpha,\beta\in\mathcal{A}}\mathbf{E}_{(X,Y)\sim P_{\alpha}\times P_{\beta}}[c(X,Y)]$
is finite, we can actually show that finding the smallest $r$ in
\eqref{eq:online_plan} among online schemes is equivalent to the
pairwise multi-marginal optimal transport problem of finding $r_{c}^{*}(\{P_{\alpha}\}_{\alpha})$.
The achievability of $r=r_{c}^{*}(\{P_{\alpha}\}_{\alpha})+\epsilon$
in \eqref{eq:online_plan} for any $\epsilon>0$ follows from the
preemptive minimax scheme. For the converse, fix any online scheme
achieving $r$ in \eqref{eq:online_plan}, and let $T\ge|\mathcal{A}|-1$,
$b_{t}=t$, and $P_{\alpha_{t}}$ be distinct for $0\le t\le|\mathcal{A}|-1$.
Let $X_{\beta}:=Z_{i}$, where $i\in[0..|\mathcal{A}|-1]$ such that
$\alpha_{i}=\beta$. Since the scheme is online, we can choose $b_{t},P_{\alpha_{t}}$
for $t\ge|\mathcal{A}|$ without affecting $\{Z_{i}\}_{i\le|\mathcal{A}|-1}$.
For any two probability distributions $P_{\alpha_{i}},P_{\alpha_{j}}$
in $\{P_{\alpha}\}_{\alpha}$, $i,j\in[0..|\mathcal{A}|-1]$, take
$b_{t}=i$ for $|\mathcal{A}|\le t\le T$ where $t$ is odd, and $b_{t}=j$
for $|\mathcal{A}|\le t\le T$ where $t$ is even. We have
\begin{align*}
 & (T-|\mathcal{A}|)\mathbf{E}[c(X_{\alpha_{i}},X_{\alpha_{j}})]\\
 & =(T-|\mathcal{A}|)\mathbf{E}[c(Z_{i},Z_{j})]\\
 & \le\sum_{t=1}^{T}\mathbf{E}[c(Z_{t-1},Z_{t})]\\
 & \stackrel{(a)}{\le}r\inf_{\{\tilde{Z}_{t}\}_{t}:\,\tilde{Z}_{t}\sim P_{\alpha_{t}},\,\tilde{Z}_{t}=\tilde{Z}_{b_{t}}\,\forall t\in[0..T]}\sum_{t=1}^{T}\mathbf{E}[c(\tilde{Z}_{t-1},\tilde{Z}_{t})]\\
 & \stackrel{(b)}{\le}r\left(|\mathcal{A}|\sup_{\alpha,\beta\in\mathcal{A}}\mathbf{E}_{(X,Y)\sim P_{\alpha}\times P_{\beta}}[c(X,Y)]+(T-|\mathcal{A}|+2)C_{c}^{*}(P_{\alpha_{i}},P_{\alpha_{j}})\right),
\end{align*}
where (a) is by \eqref{eq:online_plan}, and (b) is because we can
take $(\tilde{Z}_{i},\tilde{Z}_{j})$ to approach the optimal value
of $C_{c}^{*}(P_{\alpha_{i}},P_{\alpha_{j}})$, and $\tilde{Z}_{i'}\sim P_{\alpha_{i'}}$
for $i'\in[0..|\mathcal{A}|-1]\backslash\{i,j\}$ to be independent
of $(\tilde{Z}_{i},\tilde{Z}_{j})$ and independent among themselves.
The number of times $(i,j)$ or $(j,i)$ appears consecutively in
$b_{0},b_{1},\ldots,b_{T}$ is at least $T-|\mathcal{A}|$, and at
most $T-|\mathcal{A}|+2$. Letting $T\to\infty$, we have $\mathbf{E}[c(X_{\alpha_{i}},X_{\alpha_{j}})]\le rC_{c}^{*}(P_{\alpha_{i}},P_{\alpha_{j}})$.
Therefore, $r_{c}^{*}(\{P_{\alpha}\}_{\alpha})\le r_{c}(\{X_{\alpha}\}_{\alpha})\le r$.

\medskip{}

\subsection{Image Registration}

Consider the following simplified setting of image registration. There
are $m$ images represented as probability measures $P_{1},\ldots,P_{m}$
on $\mathcal{X}=\mathbb{R}^{2}$, where $P_{i}$ has a probability
density function proportional to the grayscale data of the $i$-th
image. The goal is to transform these images onto a common coordinate
system, i.e., find invertible maps $f_{i}:\mathbb{R}^{2}\to\mathbb{R}^{2}$
such that $f_{i*}P_{i}$ are the same across $i=1,\ldots,m$. Letting
$U\sim f_{1*}P_{1}$, $X_{i}:=f_{i}^{-1}(U)$, we have $X_{i}\sim P_{i}$,
and hence $\{X_{i}\}_{i}$ is a coupling of $\{P_{i}\}_{i}$.

We also want these transformations to distort the images as little
as possible. The case $m=2$ has been studied extensively (e.g. in
\cite{peleg1989unified,haker2004optimal}), where we minimize $\mathbf{E}[c(X_{1},X_{2})]$
among couplings $(X_{1},X_{2})\in\Gamma_{\lambda}(P_{1},P_{2})$,
where $c$ is a cost function. The squared distance cost $c(x,y)=\Vert x-y\Vert_{2}^{2}$
is usually considered, where the minimum value of the optimization
problem is given by the squared 2-Wasserstein distance between $P_{1}$
and $P_{2}$. This can be readily generalized to $m>2$ if there is
a natural ordering of the images (e.g. in optical flow where the images
are frames of a video), where we minimize $\mathbf{E}[\sum_{i=1}^{m-1}c(X_{i},X_{i+1})]$
among $\{X_{i}\}_{i}\in\Gamma_{\lambda}(\{P_{i}\}_{i})$. This can
be solved by minimizing $\mathbf{E}[c(X_{i},X_{i+1})]$ separately,
and then putting the optimal couplings in $\Gamma_{\lambda}(P_{i},P_{i+1})$,
$i=1,\ldots,m-1$ together to form a coupling in $\Gamma_{\lambda}(\{P_{i}\}_{i})$.
Another generalization to $m>2$ is when there is a target image $P_{1}$,
and we are only concerned with the cost between $P_{1}$ and $P_{i}$,
$i\ge2$, i.e., we minimize $\mathbf{E}[\sum_{i=2}^{m}c(X_{1},X_{i})]$
among $\{X_{i}\}_{i}\in\Gamma_{\lambda}(\{P_{i}\}_{i})$. This optimization
problem can be reduced to the case $m=2$ similarly.

Nevertheless, there are settings where there is no natural ordering
of the images, and there is no special target image, e.g. when the
images are photos of an object taken at different angles or under
different lighting conditions. In this case, it is perhaps natural
to minimize $r\ge1$ such that $\mathbf{E}[c(X_{i},X_{j})]\le rC_{c}^{*}(P_{i},P_{j})$
for all $i,j\in[1..m]$. This guarantees that any pair $(X_{i},X_{j})$
gives an expected cost at most a factor of $r$ from the minimum possible
expected cost. The minimum $r$ is given by the optimal pairwise coupling
ratio $r_{c}^{*}(\{P_{i}\}_{i})$. \medskip{}
\footnote{An unfortunate fact is that $r_{c}^{*}(\mathcal{P}(\mathbb{R}^{2}))=\infty$
when $c(x,y)=\Vert x-y\Vert_{2}^{2}$ by Proposition \ref{prop:rn_rc_lb},
meaning that there is no uniform upper bound on $r_{c}^{*}(\{P_{i}\}_{i})$.
Nevertheless, $r_{c}^{*}(\{P_{i}\}_{i})$ can still be finite (since
we are not considering the set $\mathcal{P}(\mathbb{R}^{2})$ of all
distributions). A modification of $r_{c}^{*}$ which might allow us
to consider $\mathcal{P}(\mathbb{R}^{2})$ is given in Section \ref{sec:trunc}.}

\subsection{Multi-agent Matching with Fairness Requirement}

We consider a variant of the multi-agent matching problem \cite{carlier2010matching,chiappori2010hedonic,pass2012multi}.
Suppose there are $m$ categories of agents, each consisting of $k$
agents. Let $P_{i}$ be the distribution of agents in category $i$,
where the space $\mathcal{X}$ represents the locations of the agents
(it can also represent preferences, skill sets, etc). The goal is
to form teams, where each team consists of one agent from each category.
Let $Q_{X_{1},\ldots,X_{m}}$ be the distribution of locations of
team members of a team (i.e., the distribution of the location vector
$(X_{1},\ldots,X_{m})$ of a random team, where $X_{i}$ is the location
of the member of category $i$). Then $Q_{X_{1},\ldots,X_{m}}$ is
a coupling of $P_{1},\ldots,P_{m}$. The cost of placing two agents
at locations $x,y$ in the same team is given by the pairwise cost
function $c(x,y)$ (e.g. traveling cost).

While we can form teams by minimizing the total cost $k\mathbf{E}_{Q}[\sum_{i<j}c(X_{i},X_{j})]$
over couplings $Q_{X_{1},\ldots,X_{m}}$ (which falls in the scope
of the original multi-marginal optimal transport problem), such a
team assignment may not be satisfactory to some categories, since
each category (or the association of agents in a category) is only
interested in the costs relevant to their agents. In order to ensure
fairness, $k\mathbf{E}_{Q}[c(X_{i},X_{j})]$ should be small for all
pairs $i<j$. If the bound in \eqref{eq:ecr_bd} is attainable, then
we can guarantee that each $k\mathbf{E}_{Q}[c(X_{i},X_{j})]$ is within
a factor of $r$ from its lowest possible value $kC_{c}^{*}(P_{i},P_{j})$.\medskip{}

\section{Universal Poisson Coupling\label{sec:upc}}

In this section, we review the Poisson functional representation \cite{sfrl_trans,li2018unified},
with a different notation. We also remark that the same construction
(but only for discrete distributions) was given later in \cite{angel2019pairwise}.
\begin{defn}
[Poisson functional representation]\label{def:pfr}Let $\phi\in[\mathcal{X}\times\mathbb{R}_{\ge0}]^{\le\aleph_{0}}$,
where $\mathcal{X}$ is a Polish space with its Borel $\sigma$-algebra.
For $\mu$ a $\sigma$-finite measure over $\mathcal{X}$, and $P\ll\mu$
a probability measure on $\mathcal{X}$, define
\[
\varrho_{P\Vert\mu}(\phi):=\bigg(\underset{(x,t)\in\phi:\,\frac{\mathrm{d}P}{\mathrm{d}\mu}(x)>0}{\arg\min}t\Big(\frac{\mathrm{d}P}{\mathrm{d}\mu}(x)\Big)^{-1}\bigg)_{1}
\]
with arbitrary tie-breaking, where $(\cdot)_{1}$ denotes the $x$-component
of the pair. We omit $\mu$ and only write $\varrho_{P}(\phi)$ if
$\mu$ is clear from the context.

Let $\Phi\sim\mathrm{PP}(\mu\times\lambda_{\mathbb{R}_{\ge0}})$.
Since a Poisson process is a proper point process \cite[Corollary 6.5]{last2017lectures},
there exist random variables $(X_{1},T_{1}),(X_{2},T_{2}),\ldots\in\mathcal{X}\times\mathbb{R}_{\ge0}$
and a random variable $K\in\mathbb{Z}_{\ge0}\cup\{\infty\}$ such
that $\sum_{i=1}^{K}\delta_{(X_{i},T_{i})}=\Phi$ almost surely. We
define
\begin{equation}
\varrho_{P\Vert\mu}(\Phi):=\varrho_{P\Vert\mu}\bigg(\bigcup_{i=1}^{K}\{(X_{i},T_{i})\}\bigg).\label{eq:pfr_pois}
\end{equation}

\medskip{}
\end{defn}

By the mapping theorem \cite{kingman1992poisson,last2017lectures}
(also see Appendix A of \cite{sfrl_trans}), $\varrho_{P}(\Phi)$
is a random variable and $\varrho_{P}(\Phi)\sim P$. For a collection
of probability distributions $\{P_{\alpha}\}_{\alpha\in\mathcal{A}}$,
we can define a coupling $X_{\alpha}:=\varrho_{P_{\alpha}}(\Phi)\sim P_{\alpha}$.
We call this the \emph{universal Poisson coupling}. It is universal
in the sense that we construct each $X_{\alpha}\sim P_{\alpha}$ on
the probability space of the Poisson process individually, without
taking the other probability distributions in the collection into
account. Hence, we can couple the collection of all probability distributions
dominated by $\mu$.

The following result is proved in \cite[Appendix A]{li2018unified}
regarding the probability that $\varrho_{Q}(\Phi)=\varrho_{P}(\Phi)$
for probability measures $P,Q$ on $\mathcal{X}$.
\begin{lem}
[Poisson matching lemma \cite{li2018unified}]\label{lem:pml}\textup{}
Let $\Phi\sim\mathrm{PP}(\mu\times\lambda_{\mathbb{R}_{\ge0}})$,
and $P,Q$ be probability measures on $\mathcal{X}$ with $P,Q\ll\mu$.
Then we have the following almost surely:
\begin{align*}
 & \mathbf{P}\big(\varrho_{Q}(\Phi)=\varrho_{P}(\Phi)\,\big|\,\varrho_{P}(\Phi)\big)\\
 & =\left(f(\varrho_{P}(\Phi))\int\max\left\{ \frac{f(y)}{f(\varrho_{P}(\Phi))},\,\frac{g(y)}{g(\varrho_{P}(\Phi))}\right\} \mu(\mathrm{d}y)\right)^{-1},
\end{align*}
where $f:=\mathrm{d}P/\mathrm{d}\mu$, $g:=\mathrm{d}Q/\mathrm{d}\mu$.
\end{lem}

For two probability distributions $P$ and $Q$, the probability that
$\varrho_{P}(\Phi)\neq\varrho_{Q}(\Phi)$ can be regarded as a distance
between $P$ and $Q$. Using Lemma \ref{lem:pml}, this probability
can be evaluated to be the following expression. The formula for the
special case where the distributions are discrete is also given in
\cite{angel2019pairwise}.
\begin{defn}
For two probability distributions $P,Q\in\mathcal{P}(\mathcal{X})$,
the \emph{Poisson coupling distance} is defined as
\[
d_{\mathrm{PC}}(P,Q)=1-\int\left(\int\max\left\{ \frac{f(y)}{f(x)},\,\frac{g(y)}{g(x)}\right\} \mu(\mathrm{d}y)\right)^{-1}\mu(\mathrm{d}x),
\]
where $\mu$ is a $\sigma$-finite measure over $\mathcal{X}$ with
$P,Q\ll\mu$ (we can take $\mu=P+Q$), $f:=\mathrm{d}P/\mathrm{d}\mu$,
$g:=\mathrm{d}Q/\mathrm{d}\mu$. We consider $(\int\max\{\frac{f(y)}{f(x)},\,\frac{g(y)}{g(x)}\}\mu(\mathrm{d}y))^{-1}=0$
if $f(x)=0$ or $g(x)=0$.
\end{defn}

It can be checked that for $P,Q\ll\mu$,
\begin{equation}
d_{\mathrm{PC}}(P,Q)=\mathbf{P}(\varrho_{P}(\Phi)\neq\varrho_{Q}(\Phi)),\label{eq:dpc_pneq}
\end{equation}
where $\Phi\sim\mathrm{PP}(\mu\times\lambda_{\mathbb{R}_{\ge0}})$.
\begin{prop}
\label{prop:dpc_metric}$d_{\mathrm{PC}}$ is a metric over $\mathcal{P}(\mathcal{X})$.
\end{prop}

\begin{proof}
[Proof of Proposition \ref{prop:dpc_metric}] We have
\begin{align*}
 & d_{\mathrm{PC}}(P,Q)=0\\
\Leftrightarrow & \mathbf{P}(\varrho_{P}(\Phi)\neq\varrho_{Q}(\Phi))=0\\
\Leftrightarrow & P=Q
\end{align*}
since $\varrho_{P}(\Phi)\sim P$, $\varrho_{Q}(\Phi)\sim Q$. For
the triangle inequality, let $P_{1},P_{2},P_{3}\ll\mu$,
\begin{align*}
 & d_{\mathrm{PC}}(P_{1},P_{2})+d_{\mathrm{PC}}(P_{2},P_{3})\\
 & =\mathbf{P}(\varrho_{P_{1}}(\Phi)\neq\varrho_{P_{2}}(\Phi))+\mathbf{P}(\varrho_{P_{2}}(\Phi)\neq\varrho_{P_{3}}(\Phi))\\
 & \ge\mathbf{P}(\varrho_{P_{1}}(\Phi)\neq\varrho_{P_{2}}(\Phi)\;\mathrm{or}\;\varrho_{P_{2}}(\Phi)\neq\varrho_{P_{3}}(\Phi))\\
 & \ge\mathbf{P}(\varrho_{P_{1}}(\Phi)\neq\varrho_{P_{3}}(\Phi))\\
 & =d_{\mathrm{PC}}(P_{1},P_{3}).
\end{align*}
\end{proof}
Moreover, we can bound $d_{\mathrm{PC}}(P,Q)$ by $d_{\mathrm{TV}}(P,Q)$.
\begin{prop}
\label{prop:dpc_prop}For probability measures $P,Q$ on $\mathcal{X}$,
\[
d_{\mathrm{TV}}(P,Q)\le d_{\mathrm{PC}}(P,Q)\le\frac{2d_{\mathrm{TV}}(P,Q)}{1+d_{\mathrm{TV}}(P,Q)}.
\]
Moreover, if $|\mathrm{supp}(P)\cup\mathrm{supp}(Q)|\le2$, then $d_{\mathrm{PC}}(P,Q)=d_{\mathrm{TV}}(P,Q)$.
\end{prop}

\begin{proof}
[Proof of Proposition \ref{prop:dpc_prop}] For the lower bound,
$d_{\mathrm{PC}}(P,Q)=\mathbf{P}(\varrho_{P}(\Phi)\neq\varrho_{Q}(\Phi))\ge d_{\mathrm{TV}}(P,Q)$
since $\varrho_{P}(\Phi)\sim P$, $\varrho_{Q}(\Phi)\sim Q$. For
the upper bound,
\begin{align*}
d_{\mathrm{PC}}(P,Q) & =1-\int\left(\int\max\left\{ \frac{f(y)}{f(x)},\,\frac{g(y)}{g(x)}\right\} \mu(\mathrm{d}y)\right)^{-1}\mu(\mathrm{d}x)\\
 & \le1-\int\left(\int\frac{\max\{f(y),g(y)\}}{\min\{f(x),g(x)\}}\mu(\mathrm{d}y)\right)^{-1}\mu(\mathrm{d}x)\\
 & =1-\int\left(\frac{\int\max\{f(y),g(y)\}\mu(\mathrm{d}y)}{\min\{f(x),g(x)\}}\right)^{-1}\mu(\mathrm{d}x)\\
 & =1-\frac{\int\min\{f(x),g(x)\}\mu(\mathrm{d}x)}{\int\max\{f(y),g(y)\}\mu(\mathrm{d}y)}\\
 & =1-\frac{1-d_{\mathrm{TV}}(P,Q)}{1+d_{\mathrm{TV}}(P,Q)}\\
 & =\frac{2d_{\mathrm{TV}}(P,Q)}{1+d_{\mathrm{TV}}(P,Q)}.
\end{align*}
If $|\mathrm{supp}(P)\cup\mathrm{supp}(Q)|\le2$, let $\mathrm{supp}(P)\cup\mathrm{supp}(Q)\subseteq\{x,y\}$.
Assume $P(x)\le Q(x)$ without loss of generality (treat $a/0=\infty$
for $a>0$), then
\begin{align*}
d_{\mathrm{PC}}(P,Q) & =1-\left(\max\left\{ \frac{P(x)}{P(x)},\,\frac{Q(x)}{Q(x)}\right\} +\max\left\{ \frac{P(y)}{P(x)},\,\frac{Q(y)}{Q(x)}\right\} \right)^{-1}\\
 & \;\;\;\;\;\;-\left(\max\left\{ \frac{P(x)}{P(y)},\,\frac{Q(x)}{Q(y)}\right\} +\max\left\{ \frac{P(y)}{P(y)},\,\frac{Q(y)}{Q(y)}\right\} \right)^{-1}\\
 & =1-\left(1+\frac{P(y)}{P(x)}\right)^{-1}-\left(1+\frac{Q(x)}{Q(y)}\right)^{-1}\\
 & =1-P(x)-Q(y)\\
 & =Q(x)-P(x)\\
 & =d_{\mathrm{TV}}(P,Q).
\end{align*}
\end{proof}
We are ready to give the pairwise coupling ratio for the discrete
metric.
\begin{proof}
[Proof of Theorem \ref{thm:rd_bd}] We first show the upper bound.
By Proposition \ref{prop:dpc_prop}, for all $\alpha,\beta$
\begin{align*}
\mathbf{P}(\varrho_{P_{\alpha}}(\Phi)\neq\varrho_{P_{\beta}}(\Phi)) & \le\frac{2d_{\mathrm{TV}}(P_{\alpha},P_{\beta})}{1+d_{\mathrm{TV}}(P_{\alpha},P_{\beta})}\\
 & \le2d_{\mathrm{TV}}(P_{\alpha},P_{\beta}),
\end{align*}
where $\Phi\sim\mathrm{PP}(\mu\times\lambda_{\mathbb{R}_{\ge0}})$.
Hence $r_{\mathbf{1}_{\neq}}(\{\varrho_{P_{\alpha}}\}_{\alpha})\le2$.

We then prove the bound $r_{\mathbf{1}_{\neq}}^{*}(\mathcal{P}_{\ll\mu}(\mathcal{X}))\le(s+1)/3$
when $s:=|\mathrm{supp}(\mu)|<\infty$.\footnote{It is mentioned in \cite{kleinberg2002approximation} that it is shown
by Sanjeev Arora and Julia Chuzhoy that the approximation ratios $4/3$
and $11/6$ are achievable for $s=3$ and $s=4$ respectively for
the metric labelling problem (which, depending on the construction
they use, may mean that $r_{\mathbf{1}_{\neq}}^{*}(\mathcal{P}_{\ll\mu}(\mathcal{X}))$
is upper-bounded by those numbers, respectively). Note that the bound
$r_{\mathbf{1}_{\neq}}^{*}(\mathcal{P}_{\ll\mu}(\mathcal{X}))\le5/3$
for $s=4$ here is better than $11/6$. A construction very similar
to the one described here has appeared in \cite{iwasa2009approximation}
(which considers the best $\varsigma$ instead of a random one).} Without loss of generality, assume $\mathrm{supp}(\mu)=[1..s]$,
and thus $\mathcal{P}_{\ll\mu}(\mathcal{X})=\mathcal{P}([1..s])$.
We construct a coupling $\{X_{\alpha}\}_{\alpha}$ of $\{P_{\alpha}\}_{\alpha}=\mathcal{P}([1..s])$
by letting $\varsigma:[1..s]\to[1..s]$ be a uniform random permutation
of $[1..s]$, $U\sim\mathrm{Unif}[0,1]$ independent of $\varsigma$,
and $X_{\alpha}:=F_{\varsigma_{*}P_{\alpha}}^{-1}(U)$ (i.e., the
quantile coupling on $\{\varsigma_{*}P_{\alpha}\}_{\alpha}$). For
any $P_{\alpha},P_{\beta}\in\mathcal{P}([1..s])$, and $(\tilde{X}_{\alpha},\tilde{X}_{\beta})\in\Gamma_{\lambda}(P_{\alpha},P_{\beta})$
(assume $(\tilde{X}_{\alpha},\tilde{X}_{\beta})$ is independent of
$(\varsigma,U)$),
\begin{align*}
 & \mathbf{P}(X_{\alpha}\neq X_{\beta})\\
 & \le\mathbf{E}\left[\mathbf{E}\left[|\varsigma(X_{\alpha})-\varsigma(X_{\beta})|\,\big|\,\varsigma\right]\right]\\
 & \stackrel{(a)}{\le}\mathbf{E}\left[\mathbf{E}\left[|\varsigma(\tilde{X}_{\alpha})-\varsigma(\tilde{X}_{\beta})|\,\big|\,\varsigma\right]\right]\\
 & =\mathbf{E}\left[\mathbf{E}\left[|\varsigma(\tilde{X}_{\alpha})-\varsigma(\tilde{X}_{\beta})|\,\big|\,\tilde{X}_{\alpha},\tilde{X}_{\beta}\right]\right]\\
 & =\mathbf{E}\left[\mathbf{1}\{\tilde{X}_{\alpha}\neq\tilde{X}_{\beta}\}\sum_{i=1}^{s-1}i\mathbf{P}\left(|\varsigma(\tilde{X}_{\alpha})-\varsigma(\tilde{X}_{\beta})|=i\,\big|\,\tilde{X}_{\alpha}\neq\tilde{X}_{\beta}\right)\right]\\
 & =\mathbf{E}\left[\mathbf{1}\{\tilde{X}_{\alpha}\neq\tilde{X}_{\beta}\}\sum_{i=1}^{s-1}i\frac{s-i}{s(s-1)/2}\right]\\
 & =\mathbf{E}\left[\mathbf{1}\{\tilde{X}_{\alpha}\neq\tilde{X}_{\beta}\}\frac{s(s-1)(s+1)/6}{s(s-1)/2}\right]\\
 & =\frac{s+1}{3}\mathbf{P}(\tilde{X}_{\alpha}\neq\tilde{X}_{\beta}),
\end{align*}
where (a) is by the optimality of the quantile coupling for the cost
function $|x-y|$. Hence
\begin{align*}
\mathbf{P}(X_{\alpha}\neq X_{\beta}) & \le\inf_{(\tilde{X}_{\alpha},\tilde{X}_{\beta})\in\Gamma_{\lambda}(P_{\alpha},P_{\beta})}\frac{s+1}{3}\mathbf{P}(\tilde{X}_{\alpha}\neq\tilde{X}_{\beta})\\
 & =\frac{s+1}{3}d_{\mathrm{TV}}(P_{\alpha},P_{\beta}).
\end{align*}

We then show the lower bound. Let $m\le|\mathrm{supp}(\mu)|$. By
the definition of $\mathrm{supp}(\mu)$, we can find disjoint open
sets $E_{i}$, $i=1,\ldots,m$, where $\mu(E_{i})>0$. Let $Q_{i}\ll\mu_{E_{i}}$
be probability measures. Let $\mathcal{A}:=\{1,\ldots,m\}$, $P_{\alpha}:=\sum_{i\in\mathcal{A}\backslash\{\alpha\}}Q_{i}/(m-1)$.
Then $d_{\mathrm{TV}}(P_{\alpha},P_{\beta})=1/(m-1)$ for any $\alpha\neq\beta$.
For any coupling $\{X_{\alpha}\}_{\alpha\in\mathcal{A}}$, let $X_{m+1}:=X_{1}$,
then
\begin{align*}
 & \sum_{\alpha=1}^{m}\mathbf{P}(X_{\alpha}\neq X_{\alpha+1})\\
 & =\mathbf{E}\Big[\sum_{\alpha=1}^{m}\mathbf{1}\{X_{\alpha}\neq X_{\alpha+1}\}\Big]\\
 & \ge2,
\end{align*}
where the last inequality is because the $X_{\alpha}$'s cannot be
all equal at a time (for $x\in E_{i}$, there is a probability distribution
$P_{i}$ not supported at $x$), so at least 2 pairs among the cycle
$(X_{1},X_{2}),\ldots,(X_{m-1},X_{m}),(X_{m},X_{1})$ are not equal.
Hence there exists $\alpha$ such that $\mathbf{P}(X_{\alpha}\neq X_{\alpha+1})\ge2/m$.
Therefore $r_{\mathbf{1}_{\neq}}^{*}(\{P_{\alpha}\}_{\alpha})\ge2(m-1)/m$.
The case $|\mathrm{supp}(\mu)|=\infty$ follows from letting $m\to\infty$.

We remark that a similar technique is also used in \cite[Proposition 6]{angel2019pairwise}.
Nevertheless, by considering cycles (instead of all pairs in \cite{angel2019pairwise}),
we are able to give sharper bounds for general symmetric cost functions
in Proposition \ref{prop:rn_rc_lb_1}.
\end{proof}
\begin{rem}
It is possible to remove the condition about $\mu$ and prove Theorem
\ref{thm:rd_bd} on $\mathcal{P}(\mathcal{X})$ instead of $\mathcal{P}_{\ll\mu}(\mathcal{X})$,
if we remove the restriction that the coupling has to be defined on
the standard probability space. The proof is given in Appendix \ref{subsec:pf_rd_bd_nonst}.

\end{rem}

\section{Sequential Poisson Functional Representation\label{sec:spfr}}

The Poisson coupling $(X,Y)$, $X\sim P$, $Y\sim Q$ of probability
distributions $P,Q$ tends to have a high probability of matching
(i.e., $\mathbf{P}(X=Y)$) when $d_{\mathrm{TV}}(P,Q)$ is small,
as shown in Proposition \ref{prop:dpc_prop}. However, it is not concerned
with the distance between $X$ and $Y$, only whether they are equal.
In order to construct a coupling $(X,Y)$ where $X$ and $Y$ are
close (but not necessarily equal), instead of applying the Poisson
functional representation to $X$ directly, we can apply the Poisson
functional representation to a noisy version $Z$ of $X$ instead.

For example, let $\gamma>0$, and $\kappa$ be the probability kernel
from $\mathbb{R}^{n}$ to $\mathbb{R}^{n}$, where $\kappa(x,\cdot)=\mathrm{Unif}(\mathcal{B}_{\Vert\cdot\Vert_{2},\gamma/2}(x))$
(i.e., it is the conditional distribution of $X+W$ given $X$, where
$W\sim\mathrm{Unif}(\mathcal{B}_{\Vert\cdot\Vert_{2},\gamma/2}(0))$
independent of $X$). For a collection of continuous distributions
$\{P_{\alpha}\}_{\alpha\in\mathcal{A}}$ over $\mathbb{R}^{n}$, we
construct the coupling $\{X_{\alpha}\}_{\alpha\in\mathcal{A}}$ as
$(\Phi_{1},\Phi_{2})\stackrel{iid}{\sim}\mathrm{PP}(\lambda_{\mathbb{R}^{n}}\times\lambda_{\mathbb{R}_{\ge0}})$,
$Z_{\alpha}:=\varrho_{\kappa\circ P_{\alpha}}(\Phi_{1})$, $X_{\alpha}:=\varrho_{P_{\alpha}(\cdot|Z_{\alpha})}(\Phi_{2})$,
where $P(\cdot|Z)$ denotes the conditional distribution of $X$ given
$Z$ when $(X,Z)\sim P\kappa$, i.e., given $P_{\alpha}$, we first
use the the Poisson functional representation on $\Phi_{1}$ to find
$Z_{\alpha}$, and then identify the exact value of $X_{\alpha}$
conditioned on $Z_{\alpha}$ using $\Phi_{2}$. Since $Z_{\alpha}\sim\kappa\circ P_{\alpha}$
and $X_{\alpha}|Z_{\alpha}\sim P_{\alpha}(\cdot|Z_{\alpha})$, we
have $X_{\alpha}\sim P_{\alpha}$. Also,
\begin{align*}
 & \mathbf{P}\left(\Vert X_{\alpha}-X_{\beta}\Vert_{2}\le\gamma\right)\\
 & \ge\mathbf{P}(Z_{\alpha}=Z_{\beta})\\
 & =1-d_{\mathrm{PC}}(\kappa\circ P_{\alpha},\,\kappa\circ P_{\beta}),
\end{align*}
which may be larger than $1-d_{\mathrm{PC}}(P_{\alpha},\,P_{\beta})$
(the bound we would obtain if we apply the Poisson functional representation
directly to $P_{\alpha}$).

Nevertheless, this construction can only give a bound on $\mathbf{P}(\Vert X_{\alpha}-X_{\beta}\Vert_{2}\le\gamma)$
at a specific value of $\gamma$ which is fixed as a parameter of
the construction. It is not strong enough to bound $\mathbf{E}[\Vert X_{\alpha}-X_{\beta}\Vert_{2}^{q}]$
as required in Theorem \ref{thm:rn_rc_ub}. Therefore, instead of
using only one auxiliary random variable $Z_{\alpha}$, we use a sequence
of random variables $\{Z_{\alpha,i}\}_{i\in\mathbb{Z}}$.\medskip{}

\begin{defn}
[Sequential Poisson functional representation]\label{def:spfr} Let
$\mathcal{X}$ be a Polish space with its Borel $\sigma$-algebra
$\mathcal{F}$. Let $I\subseteq\mathbb{Z}$ be a nonempty interval
(i.e., $I=[k..l]$, $[k..\infty)$, $(-\infty..l]$ or $(-\infty..\infty)$
for some $k\le l\in\mathbb{Z}$). Write $I_{<i}:=\{j\in I:\,j<i\}$,
and define $I_{\le i}$, $I_{>i}$, $I_{\ge i}$ similarly. Let $\mu_{i}$
be a $\sigma$-finite measure over $\mathcal{X}$ for $i\in I$. Let
$\bar{P}:\mathcal{F}^{\otimes I}\to[0,1]$ be a probability distribution
on $\mathcal{X}^{I}$ (note that $\mathcal{F}^{\otimes I}$ is the
Borel $\sigma$-algebra of the Polish space $\mathcal{X}^{I}$, since
a Polish space is second-countable). Let $\phi_{i}\in[\mathcal{X}\times\mathbb{R}_{\ge0}]^{\le\aleph_{0}}$,
$i\in I$. For $J\subseteq I$, write $\phi_{J}:=\{\phi_{i}\}_{i\in J}$,
$\mu_{J}:=\{\mu_{i}\}_{i\in J}$, and $\bar{P}_{J}:\mathcal{F}^{\otimes J}\to[0,1]$,
defined by $\bar{P}_{J}(E):=\bar{P}(\{z\in\mathcal{X}^{I}:\,\{z_{j}\}_{j\in J}\in E\})$
(i.e., $\bar{P}_{J}$ is the marginal distribution of the random variables
with indices in $J$). For $i\in I$, let $\bar{P}_{i|I_{<i}}:\mathcal{X}^{I_{<i}}\times\mathcal{F}\to[0,1]$
be a regular conditional distribution of $Z_{i}$ given $\{Z_{j}\}_{j\in I_{<i}}$,
where $\{Z_{j}\}_{j\in I}\sim\bar{P}$ (write $\bar{P}_{i|I_{<i}}(E|\{z_{j}\}_{j\in I_{<i}}):=\bar{P}_{i|I_{<i}}(\{z_{j}\}_{j\in I_{<i}},E)$).
  If $\inf I=-\infty$, we require an additional parameter $\nu_{I}=\{\nu_{i}\}_{i\in I}$,
where $\nu_{i}$ is a probability distribution.

We require the following conditions on $(\bar{P}_{I},\{\bar{P}_{i|I_{<i}}\}_{i\in I},\mu_{I},\nu_{I})$
(we refer to these two conditions together as the \emph{SPFR condition}):
\begin{enumerate}
\item For all $\{z_{j}\}_{j\in I_{<i}}\in\mathcal{X}^{I_{<i}}$, we have
\begin{equation}
\bar{P}_{i|I_{<i}}(\cdot|\{z_{j}\}_{j\in I_{<i}})\ll\mu_{i}.\label{eq:spfr_abscont_kernel}
\end{equation}
\item If $\inf I=-\infty$, we require that the given $\nu_{I}$ satisfy
that $\nu_{i}\ll\mu_{i}$ for $i\in I$,
\begin{equation}
\sum_{i\in I_{\le0}}d_{\mathrm{TV}}\bigg(\bigg(\prod_{j\in I_{<i}}\nu_{j}\bigg)\bar{P}_{i|I_{<i}},\,\prod_{j\in I_{\le i}}\nu_{j}\bigg)<\infty,\label{eq:spfr_kinfty_lim}
\end{equation}
and
\begin{equation}
\lim_{i\to-\infty}d_{\mathrm{TV}}\bigg(\bar{P}_{I_{\le i}},\,\prod_{j\in I_{\le i}}\nu_{j}\bigg)=0.\label{eq:spfr_kinfty_lim2}
\end{equation}
\end{enumerate}
We define the \emph{sequential Poisson functional representation (SPFR)}
$\bar{\varrho}_{\{\bar{P}_{i|I_{<i}}\}_{i\in I}\Vert\mu_{I}}(\phi_{I})\in\mathcal{X}^{I}$
(write $\bar{\varrho}_{\{\bar{P}_{i|I_{<i}}\}_{i\in I}\Vert\mu_{I},j}(\phi_{I})$
for its $j$-th component) recursively:
\begin{casenv}
\item $I=[k..l]$: Define
\begin{align}
\bar{\varrho}_{\{\bar{P}_{i|I_{<i}}\}_{i\in I}\Vert\mu_{I}}(\phi_{I}) & :=\Big(\bar{\varrho}_{\{\bar{P}_{i|I_{<i}}\}_{i\in I_{<l}}\Vert\mu_{I_{<l}}}(\phi_{I_{<l}}),\,\varrho_{\bar{P}_{l|I_{<l}}(\cdot\,|\,\bar{\varrho}_{\bar{P}_{I_{<l}}\Vert\mu_{I_{<l}}}(\phi_{I_{<l}}))\,\Vert\,\mu_{l}}(\phi_{l})\Big),\label{eq:spfr}
\end{align}
where the outermost $(\cdot,\cdot)$ denotes concatenation, and $\varrho$
denotes the original Poisson functional representation in Definition
\ref{def:pfr}. If $k=l$, then $\bar{\varrho}_{\{\bar{P}_{i|I_{<i}}\}_{i\in I}\Vert\mu_{I}}(\phi_{I}):=\varrho_{\bar{P}_{l}\Vert\mu_{l}}(\phi_{l})$.
\item $I=(-\infty..l]$: Define 
\begin{equation}
\bar{\varrho}_{\{\bar{P}_{i|I_{<i}}\}_{i\in I}\Vert\mu_{I},l}(\phi_{I}):=\lim_{k\to-\infty}\bar{\varrho}_{\{\bar{P}_{i|I_{<i}}(\cdot|(\{\varrho_{\nu_{j}\Vert\mu_{j}}(\phi_{j})\}_{j<k},\cdot))\}_{i\in[k..l]}\Vert\mu_{[k..l]},l}(\phi_{[k..l]}),\label{eq:spfr_kinfty}
\end{equation}
where $\bar{P}_{i|I_{<i}}(\cdot|(\{\varrho_{\nu_{j}\Vert\mu_{j}}(\phi_{j})\}_{j<k},\cdot)):\,(\{z_{j}\}_{j\in[k..i-1]},E)\mapsto\bar{P}_{i|I_{<i}}(E|(\{\varrho_{\nu_{j}\Vert\mu_{j}}(\phi_{j})\}_{j<k},\{z_{j}\}_{j\in[k..i-1]}))$
is a probability kernel from $\mathcal{X}^{[k..i-1]}$ to $\mathcal{X}$,
the $\bar{\varrho}$ in the right hand side is defined using \eqref{eq:spfr},
and the limit is with respect to the discrete metric (not the topology
of $\mathcal{X}$), i.e., $\lim_{k\to-\infty}a_{k}=b$ $\Leftrightarrow$
$\inf\{k:\,a_{k}\neq b\}>-\infty$. We will show later that the limit
exists. We then define
\[
\bar{\varrho}_{\{\bar{P}_{i|I_{<i}}\}_{i\in I}\Vert\mu_{I}}(\phi_{I}):=\big\{\bar{\varrho}_{\{\bar{P}_{i|I_{<i}}\}_{i\in I_{\le j}}\Vert\mu_{I_{\le j}},j}(\phi_{I_{\le j}})\big\}_{j\in I},
\]
where each term is given by \eqref{eq:spfr_kinfty}. We will show
later that this also satisfies the equality in \eqref{eq:spfr}.
\item $I=[k..\infty)$ or $(-\infty..\infty)$: Define
\begin{equation}
\bar{\varrho}_{\{\bar{P}_{i|I_{<i}}\}_{i\in I}\Vert\mu_{I}}(\phi_{I}):=\big\{\bar{\varrho}_{\{\bar{P}_{i|I_{<i}}\}_{i\in I_{\le j}}\Vert\mu_{I_{\le j}},j}(\phi_{I_{\le j}})\big\}_{j\in I}.\label{eq:spfr_linfty}
\end{equation}
\end{casenv}
For notational simplicity, we write $\bar{\varrho}_{\bar{P}_{I}\Vert\mu_{I}}(\phi_{I}):=\bar{\varrho}_{\{\bar{P}_{i|I_{<i}}\}_{i\in I}\Vert\mu_{I}}(\phi_{I})$,
where $\{\bar{P}_{i|I_{<i}}\}_{i\in I}$ are regular conditional distributions
of $\bar{P}$ such that $(\bar{P}_{I},\{\bar{P}_{i|I_{<i}}\}_{i\in I},\mu_{I},\nu_{I})$
satisfies the SPFR condition. We say $(\bar{P}_{I},\mu_{I},\nu_{I})$
satisfies the SPFR condition if there exist regular conditional distributions
$\{\bar{P}_{i|I_{<i}}\}_{i\in I}$ such that $(\bar{P}_{I},\{\bar{P}_{i|I_{<i}}\}_{i\in I},\mu_{I},\nu_{I})$
satisfies the SPFR condition. This simplified notation is justified
by part 1 of Lemma \ref{lem:spfr_dist} below (the choice of $\{\bar{P}_{i|I_{<i}}\}_{i\in I}$
does not matter). We omit $\nu_{I}$ from the notation $\bar{\varrho}_{\bar{P}_{I}\Vert\mu_{I}}$
since the choice of $\nu_{I}$ is clear from the context. We sometimes
omit $\mu_{I}$ and write $\bar{\varrho}_{\bar{P}_{I}}:=\bar{\varrho}_{\bar{P}_{I}\Vert\mu_{I}}$,
$\bar{\varrho}_{\bar{P}_{I},i}:=\bar{\varrho}_{\bar{P}_{I}\Vert\mu_{I},i}$
if $\mu_{I}$ is clear from the context.
\end{defn}

\medskip{}

We then show that if the SPFR condition is satisfied, and we have
$\Phi_{i}\sim\mathrm{PP}(\mu_{i}\times\lambda_{\mathbb{R}_{\ge0}})$
independent across $i\in I$, then $\bar{\varrho}_{\bar{P}_{I}}(\Phi_{I})\sim\bar{P}$
(the definition of $\bar{\varrho}_{\bar{P}_{I}}$ is extended to Poisson
processes similarly as in \eqref{eq:pfr_pois}). The proof is given
in Appendix \ref{subsec:pf_hpfr_dist}.
\begin{lem}
\label{lem:spfr_dist}Let $I\subseteq\mathbb{Z}$ be a nonempty interval,
and $(\bar{P}_{I},\mu_{I},\nu_{I})$ satisfy the SPFR condition. Let
$\Phi_{i}\sim\mathrm{PP}(\mu_{i}\times\lambda_{\mathbb{R}_{\ge0}})$
be independent across $i\in I$. Then we have:
\begin{enumerate}
\item \label{enu:spfr_dist_rcd_welldef}If $\{\bar{P}_{i|I_{<i}}\}_{i\in I}$
and $\{\bar{P}'_{i|I_{<i}}\}_{i\in I}$ are both regular conditional
distributions of $\bar{P}$ such that $(\bar{P}_{I},\{\bar{P}_{i|I_{<i}}\}_{i\in I},\mu_{I},\nu_{I})$
and $(\bar{P}_{I},\{\bar{P}'_{i|I_{<i}}\}_{i\in I},\mu_{I},\nu_{I})$
satisfy the SPFR condition, then $\bar{\varrho}_{\{\bar{P}_{i|I_{<i}}\}_{i\in I}\Vert\mu_{I}}(\Phi_{I})=\bar{\varrho}_{\{\bar{P}'_{i|I_{<i}}\}_{i\in I}\Vert\mu_{I}}(\Phi_{I})$
almost surely. In other words, $\bar{\varrho}_{\bar{P}_{I}}(\Phi_{I})$
does not depend on the choice of $\{\bar{P}_{i|I_{<i}}\}_{i\in I}$.
\item \label{enu:spfr_dist_dist}$\bar{\varrho}_{\bar{P}_{I}}(\Phi_{I})\sim\bar{P}$.
\item \label{enu:spfr_dist_subint}For any $i\in I$, we have $\bar{\varrho}_{\bar{P}_{I_{\le i}}}(\Phi_{I_{\le i}})=\{\bar{\varrho}_{\bar{P}_{I},j}(\Phi_{I})\}_{j\in I_{\le i}}$
almost surely.
\item \label{enu:spfr_dist_recur}If $\sup I=l<\infty$, then \eqref{eq:spfr}
is satisfied almost surely (even if $\inf I=-\infty$).
\item \label{enu:spfr_dist_kinfty_match}If $\inf I=-\infty$, then with
probability 1,
\begin{equation}
\inf\left\{ i\in I:\,\bar{\varrho}_{\bar{P}_{I}\Vert\mu_{I},i}(\Phi_{I})\neq\varrho_{\nu_{i}\Vert\mu_{i}}(\Phi_{i})\right\} >-\infty.\label{eq:spfr_dist_agree}
\end{equation}
\end{enumerate}
\end{lem}

\medskip{}

We then give a bound on the probability that $\bar{\varrho}_{\bar{P}_{I}}(\Phi_{I})\neq\bar{\varrho}_{\bar{Q}_{I}}(\Phi_{I})$
for two probability distributions $\bar{P},\bar{Q}$.
\begin{lem}
\label{lem:spfr_dpcbd}Let $I\subseteq\mathbb{Z}$ be a nonempty interval
with $\sup I<\infty$. Let $\bar{P},\bar{Q}$ be probability distributions
over $\mathcal{X}^{I}$ such that $(\bar{P}_{I},\mu_{I},\nu_{I})$
and $(\bar{Q}_{I},\mu_{I},\nu_{I})$ satisfy the SPFR condition. Let
$\Phi_{i}\sim\mathrm{PP}(\mu_{i}\times\lambda_{\mathbb{R}_{\ge0}})$
independent across $i\in I$. Then we have
\begin{align*}
 & \mathbf{P}\Big(\bar{\varrho}_{\bar{P}_{I}}(\Phi_{I})\neq\bar{\varrho}_{\bar{Q}_{I}}(\Phi_{I})\Big)\\
 & \le2\sum_{i\in I}(1+\mathbf{1}\{i+1\in I\})d_{\mathrm{TV}}\big(\bar{P}_{I_{\le i}},\,\bar{Q}_{I_{\le i}}\big).
\end{align*}
\end{lem}

\begin{proof}
[Proof of Lemma \ref{lem:spfr_dpcbd}] Let $\sup I=l$. For brevity,
write $\bar{\varrho}_{\bar{P}}^{i}=\bar{\varrho}_{\bar{P}_{I_{\le i}}}(\Phi_{I_{\le i}})$,
and let $\bar{\varrho}_{\bar{P}}^{i}=\emptyset$ if $I_{\le i}=\emptyset$.
We have
\begin{align*}
 & \mathbf{P}\Big(\bar{\varrho}_{\bar{P}}^{l}\neq\bar{\varrho}_{\bar{Q}}^{l}\Big)\\
 & \stackrel{(a)}{=}\sum_{i\in I}\mathbf{P}\Big(\bar{\varrho}_{\bar{P}}^{i-1}=\bar{\varrho}_{\bar{Q}}^{i-1}\;\mathrm{and}\;\bar{\varrho}_{\bar{P}}^{i}\neq\bar{\varrho}_{\bar{Q}}^{i}\Big)\\
 & \stackrel{(b)}{=}\sum_{i\in I}\mathbf{P}\Big(\bar{\varrho}_{\bar{P}}^{i-1}=\bar{\varrho}_{\bar{Q}}^{i-1}\;\mathrm{and}\;\varrho_{\bar{P}_{i|I_{<i}}\big(\cdot\,|\,\bar{\varrho}_{\bar{P}}^{i-1}\big)\,\Vert\,\mu_{i}}(\Phi_{i})\neq\varrho_{\bar{Q}_{i|I_{<i}}\big(\cdot\,|\,\bar{\varrho}_{\bar{Q}}^{i-1}\big)\,\Vert\,\mu_{i}}(\Phi_{i})\Big)\\
 & \le\sum_{i\in I}\mathbf{P}\Big(\varrho_{\bar{P}_{i|I_{<i}}\big(\cdot\,|\,\bar{\varrho}_{\bar{P}}^{i-1}\big)\,\Vert\,\mu_{i}}(\Phi_{i})\neq\varrho_{\bar{Q}_{i|I_{<i}}\big(\cdot\,|\,\bar{\varrho}_{\bar{P}}^{i-1}\big)\,\Vert\,\mu_{i}}(\Phi_{i})\Big)\\
 & \stackrel{(c)}{=}\sum_{i\in I}\int d_{\mathrm{PC}}\left(\bar{P}_{i|I_{<i}}\big(\cdot\,|\,\{z_{j}\}_{j\in I_{<i}}\big),\,\bar{Q}_{i|I_{<i}}\big(\cdot\,|\,\{z_{j}\}_{j\in I_{<i}}\big)\right)\bar{P}_{I_{<i}}(\mathrm{d}\{z_{j}\}_{j\in I_{<i}})\\
 & \stackrel{(d)}{\le}2\sum_{i\in I}\int d_{\mathrm{TV}}\left(\bar{P}_{i|I_{<i}}\big(\cdot\,|\,\{z_{j}\}_{j\in I_{<i}}\big),\,\bar{Q}_{i|I_{<i}}\big(\cdot\,|\,\{z_{j}\}_{j\in I_{<i}}\big)\right)\bar{P}_{I_{<i}}(\mathrm{d}\{z_{j}\}_{j\in I_{<i}})\\
 & =2\sum_{i\in I}d_{\mathrm{TV}}\big(\bar{P}_{I_{<i}}\bar{P}_{i|I_{<i}},\,\bar{P}_{I_{<i}}\bar{Q}_{i|I_{<i}}\big)\\
 & \le2\sum_{i\in I}\left(d_{\mathrm{TV}}\big(\bar{P}_{I_{<i}}\bar{P}_{i|I_{<i}},\,\bar{Q}_{I_{<i}}\bar{Q}_{i|I_{<i}}\big)+d_{\mathrm{TV}}\big(\bar{P}_{I_{<i}}\bar{Q}_{i|I_{<i}},\,\bar{Q}_{I_{<i}}\bar{Q}_{i|I_{<i}}\big)\right)\\
 & =2\sum_{i\in I}\left(d_{\mathrm{TV}}\big(\bar{P}_{I_{\le i}},\,\bar{Q}_{I_{\le i}}\big)+d_{\mathrm{TV}}\big(\bar{P}_{I_{<i}},\,\bar{Q}_{I_{<i}}\big)\right)\\
 & =2\sum_{i\in I}(1+\mathbf{1}\{i+1\in I\})d_{\mathrm{TV}}\big(\bar{P}_{I_{\le i}},\,\bar{Q}_{I_{\le i}}\big),
\end{align*}
where (a) is because there almost surely exists $i\in I$ such that
$\bar{\varrho}_{\bar{P}}^{i-1}=\bar{\varrho}_{\bar{Q}}^{i-1}$ (either
$I=[k..l]$, where $\bar{\varrho}_{\bar{P}}^{k-1}=\bar{\varrho}_{\bar{Q}}^{k-1}=\emptyset$,
or $I=(-\infty..l]$, where such an $i$ exists almost surely by \eqref{eq:spfr_dist_agree}),
(b) is by \eqref{eq:spfr} (which holds for all cases by Lemma \ref{lem:spfr_dist}),
(c) is by \eqref{eq:dpc_pneq} and $\bar{\varrho}_{\bar{P}}^{i-1}\sim\bar{P}_{I_{<i}}$
by Lemma \ref{lem:spfr_dist}, and (d) is by Proposition \ref{prop:dpc_prop}.
\end{proof}
\medskip{}

\subsection{Finite Metric Space\label{subsec:metric_finite}}

We prove Theorem \ref{thm:metric_log} about finite metric spaces.
\begin{proof}
[Proof of Theorem \ref{thm:metric_log}] Let $\mu$ be the counting
measure over $\mathcal{X}$. Let $\eta:=1.56$. Let $\mathcal{B}_{w}(x):=\{y\in\mathcal{X}:\,d(x,y)\le w\}$.
Let $\mathrm{U}\mathcal{B}_{w}$ be a probability kernel from $\mathcal{X}$
to $\mathcal{X}$ defined by 
\begin{equation}
\mathrm{U}\mathcal{B}_{w}(x,E):=\frac{\mu(\mathcal{B}_{w}(x)\cap E)}{\mu(\mathcal{B}_{w}(x))}.\label{eq:metric_pow_pf_unifball-1}
\end{equation}
Let $I:=[i_{0}..i_{1}]$, $i_{0}:=\lfloor-\eta^{-1}\log\max\{d(x,y):\,x\neq y\}\rfloor-1$,
$i_{1}:=\lfloor-\eta^{-1}\log\min\{d(x,y):\,x\neq y\}\rfloor+1$.
Let $\theta\in[0,1]$. For a collection of probability distributions
$\{P_{\alpha}\}_{\alpha}$, let 
\begin{equation}
\bar{P}_{\alpha,\theta}:=\Big(\prod_{i\in I}\mathrm{U}\mathcal{B}_{e^{-\eta(i+\theta)}}\Big)\circ P_{\alpha},\label{eq:metric_pow_patheta-1}
\end{equation}
i.e., $\bar{P}_{\alpha,\theta}$ is the distribution of $\{Z_{i}\}_{i\in I}$,
where $X\sim P_{\alpha}$, and $Z_{i}|X\sim\mathrm{U}\mathcal{B}_{e^{-\eta(i+\theta)}}$
are conditionally independent across $i$ given $X$. We then apply
the sequential Poisson functional representation on $\bar{P}=\bar{P}_{\alpha,\theta}$,
$\mu_{i}=\mu$ to obtain $\bar{\varrho}_{\bar{P}_{\alpha,\theta}}$.
Define
\[
X_{\alpha,\theta}:=\bar{\varrho}_{\bar{P}_{\alpha,\theta},i_{1}}(\Phi_{I}),\;X_{\alpha}:=X_{\alpha,\varTheta},
\]
where $\varTheta\sim\mathrm{Unif}[0,1]$ independent of $\Phi_{i}\stackrel{iid}{\sim}\mathrm{PP}(\mu\times\lambda_{\mathbb{R}_{\ge0}})$.
It is clear that the SPFR condition is satisfied. Since $\mathrm{U}\mathcal{B}_{e^{-\eta(i_{1}+\theta)}}(x,\cdot)=\delta_{x}$
for any $x$, the $i_{1}$-th marginal of $\bar{P}_{\alpha,\theta}$
is $P_{\alpha}$, and hence $X_{\alpha,\theta}\sim P_{\alpha}$. Note
that if $\theta_{1},\theta_{2}$ satisfy ``$e^{-\eta(i+\theta_{1})}\ge d(x,y)$
$\Leftrightarrow$ $e^{-\eta(i+\theta_{2})}\ge d(x,y)$'' for $i\in I$,
$x,y\in\mathcal{X}$ (which induces a finite partition of $[0,1]$,
the range of $\theta$), then $\bar{P}_{\alpha,\theta_{1}}=\bar{P}_{\alpha,\theta_{2}}$
and $X_{\alpha,\theta_{1}}=X_{\alpha,\theta_{2}}$. Hence $X_{\alpha,\theta}$
depends on $\theta$ only through which set in the finite partion
$\theta$ lies in, and thus $X_{\alpha}=X_{\alpha,\Theta}$ is a random
variable. Therefore, $X_{\alpha}\sim P_{\alpha}$.

Let $F_{x}(w):=\mu(\mathcal{B}_{w}(x))/\mu(\mathcal{X})$ (let $F_{x}(w)=0$
if $w<0$). Consider two probability distributions $P_{\alpha},P_{\beta}$.
Fix any coupling $(\tilde{X}_{\alpha},\tilde{X}_{\beta})\in\Gamma_{\lambda}(P_{\alpha},P_{\beta})$.
Let $\tilde{D}:=d(\tilde{X}_{\alpha},\tilde{X}_{\beta})$. For any
$i\ge i_{0}$ and $\theta\in[0,1]$,
\begin{align}
 & \mathbf{P}\Big(d(X_{\alpha,\theta},X_{\beta,\theta})>2e^{-\eta(i+\theta)}\Big)\nonumber \\
 & \le\mathbf{P}\Big(\bar{\varrho}_{\bar{P}_{\alpha,\theta},\min\{i,i_{1}\}}(\Phi_{I})\neq\bar{\varrho}_{\bar{P}_{\beta,\theta},\min\{i,i_{1}\}}(\Phi_{I})\Big)\nonumber \\
 & \stackrel{(a)}{\le}2\sum_{j=i_{0}}^{\min\{i,i_{1}\}}(1+\mathbf{1}\{j<i\})d_{\mathrm{TV}}\big(\bar{P}_{\alpha,\theta,I_{\le j}},\,\bar{P}_{\beta,\theta,I_{\le j}}\big)\nonumber \\
 & =2\sum_{j=i_{0}}^{\min\{i,i_{1}\}}(1+\mathbf{1}\{j<i\})d_{\mathrm{TV}}\left(\Big(\prod_{k=i_{0}}^{j}\mathrm{U}\mathcal{B}_{e^{-\eta(k+\theta)}}\Big)\circ P_{\alpha},\,\Big(\prod_{k=i_{0}}^{j}\mathrm{U}\mathcal{B}_{e^{-\eta(k+\theta)}}\Big)\circ P_{\beta}\right)\nonumber \\
 & \stackrel{(b)}{\le}2\sum_{j=i_{0}}^{\min\{i,i_{1}\}}(1+\mathbf{1}\{j<i\})\mathbf{E}\left[d_{\mathrm{TV}}\left(\Big(\prod_{k=i_{0}}^{j}\mathrm{U}\mathcal{B}_{e^{-\eta(k+\theta)}}\Big)(\tilde{X}_{\alpha},\cdot),\,\Big(\prod_{k=i_{0}}^{j}\mathrm{U}\mathcal{B}_{e^{-\eta(k+\theta)}}\Big)(\tilde{X}_{\beta},\cdot)\right)\right]\nonumber \\
 & \stackrel{(c)}{\le}2\sum_{j=i_{0}}^{\min\{i,i_{1}\}}(1+\mathbf{1}\{j<i\})\mathbf{E}\left[\sum_{k=i_{0}}^{j}d_{\mathrm{TV}}\left(\mathrm{U}\mathcal{B}_{e^{-\eta(k+\theta)}}(\tilde{X}_{\alpha},\cdot),\,\mathrm{U}\mathcal{B}_{e^{-\eta(k+\theta)}}(\tilde{X}_{\beta},\cdot)\right)\right]\nonumber \\
 & \le2\sum_{j=i_{0}}^{i}(1+\mathbf{1}\{j<i\})\mathbf{E}\left[\sum_{k=i_{0}}^{j}\left(\frac{\mu(\mathcal{B}_{e^{-\eta(k+\theta)}}(\tilde{X}_{\alpha})\backslash\mathcal{B}_{e^{-\eta(k+\theta)}}(\tilde{X}_{\beta}))}{\mu(\mathcal{B}_{e^{-\eta(k+\theta)}}(\tilde{X}_{\alpha}))}+\frac{\mu(\mathcal{B}_{e^{-\eta(k+\theta)}}(\tilde{X}_{\beta})\backslash\mathcal{B}_{e^{-\eta(k+\theta)}}(\tilde{X}_{\alpha}))}{\mu(\mathcal{B}_{e^{-\eta(k+\theta)}}(\tilde{X}_{\beta}))}\right)\right]\nonumber \\
 & \stackrel{(d)}{\le}2\sum_{j=i_{0}}^{i}(1+\mathbf{1}\{j<i\})\mathbf{E}\left[\sum_{k=i_{0}}^{j}\left(1-\frac{F_{\tilde{X}_{\alpha}}(e^{-\eta(k+\theta)}-\tilde{D})}{F_{\tilde{X}_{\alpha}}(e^{-\eta(k+\theta)})}+1-\frac{F_{\tilde{X}_{\beta}}(e^{-\eta(k+\theta)}-\tilde{D})}{F_{\tilde{X}_{\beta}}(e^{-\eta(k+\theta)})}\right)\right]\nonumber \\
 & =2\sum_{j=i_{0}}^{i}(2i-2j+1)\mathbf{E}\left[2-\frac{F_{\tilde{X}_{\alpha}}(e^{-\eta(j+\theta)}-\tilde{D})}{F_{\tilde{X}_{\alpha}}(e^{-\eta(j+\theta)})}-\frac{F_{\tilde{X}_{\beta}}(e^{-\eta(j+\theta)}-\tilde{D})}{F_{\tilde{X}_{\beta}}(e^{-\eta(j+\theta)})}\right],\label{eq:metrc_finite_dbd}
\end{align}
where (a) is by Lemma \ref{lem:spfr_dpcbd}, (b) is by $\tilde{X}_{\alpha}\sim P_{\alpha}$,
$\tilde{X}_{\beta}\sim P_{\beta}$, and the convexity of $d_{\mathrm{TV}}$,
(c) is due to $d_{\mathrm{TV}}(\prod_{k=i_{0}}^{j}Q_{k},\,\prod_{k=i_{0}}^{j}\tilde{Q}_{k})\le\sum_{k=i_{0}}^{j}d_{\mathrm{TV}}(Q_{k},\tilde{Q}_{k})$
for probability distributions $Q_{k}$, $\tilde{Q}_{k}$ over $\mathcal{X}$
(where $\prod_{k=i_{0}}^{j}Q_{k}$ denotes the product measure), and
(d) is because if $y\in\mathcal{B}_{e^{-\eta(k+\theta)}}(\tilde{X}_{\alpha})\backslash\mathcal{B}_{e^{-\eta(k+\theta)}}(\tilde{X}_{\beta})$,
then $d(y,\tilde{X}_{\alpha})>e^{-\eta(k+\theta)}-d(\tilde{X}_{\alpha},\tilde{X}_{\beta})$.
The above inequality is also true for $i<i_{0}$ since $\mathbf{P}(d(X_{\alpha,\theta},X_{\beta,\theta})>2e^{-\eta(i_{0}+\theta)})=0$.
Hence,
\begin{align*}
 & \mathbf{E}\left[d(X_{\alpha,\varTheta},X_{\beta,\varTheta})\right]\\
 & =\int_{0}^{1}\mathbf{E}\left[d(X_{\alpha,\theta},X_{\beta,\theta})\right]\mathrm{d}\theta\\
 & =\int_{0}^{1}\int_{0}^{\infty}\mathbf{P}\left(d(X_{\alpha,\theta},X_{\beta,\theta})>t\right)\mathrm{d}t\mathrm{d}\theta\\
 & \le\int_{0}^{1}\int_{0}^{\infty}\mathbf{P}\left(d(X_{\alpha,\theta},X_{\beta,\theta})>2e^{-\eta\left(\left\lceil -\eta^{-1}\log(t/2)-\theta\right\rceil +\theta\right)}\right)\mathrm{d}t\mathrm{d}\theta\\
 & \le\int_{0}^{1}\int_{0}^{\infty}\bigg(2\sum_{j=i_{0}}^{\left\lceil -\eta^{-1}\log(t/2)-\theta\right\rceil }\big(2\left\lceil -\eta^{-1}\log(t/2)-\theta\right\rceil -2j+1\big)\\
 & \;\;\;\;\;\;\;\cdot\mathbf{E}\left[2-\frac{F_{\tilde{X}_{\alpha}}(e^{-\eta(j+\theta)}-\tilde{D})}{F_{\tilde{X}_{\alpha}}(e^{-\eta(j+\theta)})}-\frac{F_{\tilde{X}_{\beta}}(e^{-\eta(j+\theta)}-\tilde{D})}{F_{\tilde{X}_{\beta}}(e^{-\eta(j+\theta)})}\right]\bigg)\mathrm{d}t\mathrm{d}\theta\\
 & \le\int_{0}^{1}\int_{0}^{\infty}\bigg(2\sum_{j=-\infty}^{\left\lceil -\eta^{-1}\log(t/2)-\theta\right\rceil }\big(-2\eta^{-1}\log(t/2)-2(j+\theta)+3\big)\\
 & \;\;\;\;\;\;\;\cdot\mathbf{E}\left[2-\frac{F_{\tilde{X}_{\alpha}}(e^{-\eta(j+\theta)}-\tilde{D})}{F_{\tilde{X}_{\alpha}}(e^{-\eta(j+\theta)})}-\frac{F_{\tilde{X}_{\beta}}(e^{-\eta(j+\theta)}-\tilde{D})}{F_{\tilde{X}_{\beta}}(e^{-\eta(j+\theta)})}\right]\bigg)\mathrm{d}t\mathrm{d}\theta\\
 & =2\int_{0}^{\infty}\int_{-\infty}^{-\eta^{-1}\log(t/2)+1}\big(-2\eta^{-1}\log(t/2)-2\gamma+3\big)\mathbf{E}\left[2-\frac{F_{\tilde{X}_{\alpha}}(e^{-\eta\gamma}-\tilde{D})}{F_{\tilde{X}_{\alpha}}(e^{-\eta\gamma})}-\frac{F_{\tilde{X}_{\beta}}(e^{-\eta\gamma}-\tilde{D})}{F_{\tilde{X}_{\beta}}(e^{-\eta\gamma})}\right]\mathrm{d}\gamma\mathrm{d}t\\
 & =2\int_{-\infty}^{\infty}\left(\int_{0}^{2e^{-\eta(\gamma-1)}}\big(-2\eta^{-1}\log(t/2)-2\gamma+3\big)\mathrm{d}t\right)\mathbf{E}\left[2-\frac{F_{\tilde{X}_{\alpha}}(e^{-\eta\gamma}-\tilde{D})}{F_{\tilde{X}_{\alpha}}(e^{-\eta\gamma})}-\frac{F_{\tilde{X}_{\beta}}(e^{-\eta\gamma}-\tilde{D})}{F_{\tilde{X}_{\beta}}(e^{-\eta\gamma})}\right]\mathrm{d}\gamma\\
 & =2\int_{-\infty}^{\infty}2e^{-\eta(\gamma-1)}\left(2\eta^{-1}\log2-2\gamma+3-2\eta^{-1}\left(\log\left(2e^{-\eta(\gamma-1)}\right)-1\right)\right)\\
 & \;\;\;\;\;\;\cdot\mathbf{E}\left[2-\frac{F_{\tilde{X}_{\alpha}}(e^{-\eta\gamma}-\tilde{D})}{F_{\tilde{X}_{\alpha}}(e^{-\eta\gamma})}-\frac{F_{\tilde{X}_{\beta}}(e^{-\eta\gamma}-\tilde{D})}{F_{\tilde{X}_{\beta}}(e^{-\eta\gamma})}\right]\mathrm{d}\gamma\\
 & =8\int_{-\infty}^{\infty}e^{-\eta(\gamma-1)}\left(\eta^{-1}\log2-\gamma+3/2-\eta^{-1}\left(\log2-\eta(\gamma-1)-1\right)\right)\\
 & \;\;\;\;\;\;\cdot\mathbf{E}\left[2-\frac{F_{\tilde{X}_{\alpha}}(e^{-\eta\gamma}-\tilde{D})}{F_{\tilde{X}_{\alpha}}(e^{-\eta\gamma})}-\frac{F_{\tilde{X}_{\beta}}(e^{-\eta\gamma}-\tilde{D})}{F_{\tilde{X}_{\beta}}(e^{-\eta\gamma})}\right]\mathrm{d}\gamma\\
 & =8\left(\frac{1}{\eta}+\frac{1}{2}\right)\int_{-\infty}^{\infty}e^{-\eta(\gamma-1)}\mathbf{E}\left[2-\frac{F_{\tilde{X}_{\alpha}}(e^{-\eta\gamma}-\tilde{D})}{F_{\tilde{X}_{\alpha}}(e^{-\eta\gamma})}-\frac{F_{\tilde{X}_{\beta}}(e^{-\eta\gamma}-\tilde{D})}{F_{\tilde{X}_{\beta}}(e^{-\eta\gamma})}\right]\mathrm{d}\gamma\\
 & =8\left(\frac{1}{\eta}+\frac{1}{2}\right)\int_{0}^{\infty}e^{\eta}t\mathbf{E}\left[2-\frac{F_{\tilde{X}_{\alpha}}(t-\tilde{D})}{F_{\tilde{X}_{\alpha}}(t)}-\frac{F_{\tilde{X}_{\beta}}(t-\tilde{D})}{F_{\tilde{X}_{\beta}}(t)}\right]\frac{1}{t\eta}\mathrm{d}t\\
 & =8\left(\frac{1}{\eta}+\frac{1}{2}\right)\frac{e^{\eta}}{\eta}\int_{0}^{\infty}\mathbf{E}\left[2-\frac{F_{\tilde{X}_{\alpha}}(t-\tilde{D})}{F_{\tilde{X}_{\alpha}}(t)}-\frac{F_{\tilde{X}_{\beta}}(t-\tilde{D})}{F_{\tilde{X}_{\beta}}(t)}\right]\mathrm{d}t,
\end{align*}
where
\begin{align*}
 & \int_{0}^{\infty}\left(1-\frac{F_{x}(t-\gamma)}{F_{x}(t)}\right)\mathrm{d}t\\
 & =\int_{0}^{\infty}\left(\mathbf{1}\{t-\gamma<0\}\frac{F_{x}(0)}{F_{x}(t)}+\int_{\max\{t-\gamma,\,0\}}^{t}\frac{1}{F_{x}(t)}\mathrm{d}F_{x}(\tau)\right)\mathrm{d}t\\
 & \le\int_{0}^{\infty}\left(\mathbf{1}\{t-\gamma<0\}+\int_{\max\{t-\gamma,\,0\}}^{t}\frac{1}{F_{x}(\tau)}\mathrm{d}F_{x}(\tau)\right)\mathrm{d}t\\
 & =\gamma+\gamma\int_{0}^{\infty}\frac{1}{F_{x}(\tau)}\mathrm{d}F_{x}(\tau)\\
 & =\gamma+\gamma\int_{0}^{\infty}\mathrm{d}\log(F_{x}(\tau))\\
 & =\gamma\left(1+\log|\mathcal{X}|\right),
\end{align*}
where the last equality is because $F_{x}(0)=1/|\mathcal{X}|$ and
$F_{x}(\tau)=1$ for $\tau\ge\max\{d(x,y):\,x\neq y\}$. Therefore,
\begin{align*}
 & \mathbf{E}\left[d(X_{\alpha,\varTheta},X_{\beta,\varTheta})\right]\\
 & \le8\left(\frac{1}{\eta}+\frac{1}{2}\right)\frac{e^{\eta}}{\eta}\mathbf{E}\left[2\tilde{D}\left(1+\log|\mathcal{X}|\right)\right]\\
 & =16\left(\frac{1}{\eta}+\frac{1}{2}\right)\frac{e^{\eta}}{\eta}\left(1+\log|\mathcal{X}|\right)\mathbf{E}[\tilde{D}]\\
 & \le55.692\left(1+\log|\mathcal{X}|\right)\mathbf{E}[\tilde{D}]
\end{align*}
by substituting $\eta=1.56$. Hence,
\begin{align*}
 & \mathbf{E}\left[d(X_{\alpha,\varTheta},X_{\beta,\varTheta})\right]\\
 & \le\inf_{(\tilde{X}_{\alpha},\tilde{X}_{\beta})\in\Gamma_{\lambda}(P_{\alpha},P_{\beta})}55.692\left(1+\log|\mathcal{X}|\right)\mathbf{E}\left[d(\tilde{X}_{\alpha},\tilde{X}_{\beta})\right]\\
 & =55.692\left(1+\log|\mathcal{X}|\right)C_{c}^{*}(P_{\alpha},P_{\beta}).
\end{align*}
\end{proof}
\medskip{}
We now describe the algorithm for computing $X_{\alpha}$ given $P_{\alpha}$.
Let $\eta:=1.56$, $\varTheta\sim\mathrm{Unif}[0,1]$, and $V_{i,x}\stackrel{iid}{\sim}\mathrm{Exp}(1)$
for $x\in\mathcal{X}$, $i\in[i_{0}..i_{1}]$, where $i_{0}:=\lfloor-\eta^{-1}\log\max\{d(x,y):\,x\neq y\}\rfloor-1$,
$i_{1}:=\lfloor-\eta^{-1}\log\min\{d(x,y):\,x\neq y\}\rfloor+1$.\footnote{Since $\mathcal{X}$ is finite, we do not require all points in $\Phi_{i}$
to compute the Poisson functional representation. We require only
the time of the first point with first coordinate $x$ for each $x\in\mathcal{X}$,
i.e., $Z_{i,x}=\inf\{t:\,(x,t)\in\Phi_{i}\}$. It can be checked that
$Z_{i,x}\stackrel{iid}{\sim}\mathrm{Exp}(1)$ when $\mu$ is the counting
measure.} Let $h_{\{V_{i,x}\}_{i,x},\varTheta}:\mathcal{P}(\mathcal{X})\to\mathcal{X}$
such that $X_{\alpha}=h_{\{V_{i,x}\}_{i,x},\varTheta}(P_{\alpha})$,
i.e., it is the function that computes the coupling. In practice,
the pseudorandom numbers $\varTheta$ and $\{V_{i,x}\}_{i,x}$ can
be generated on the fly using a random seed, and thus only the random
seed is needed to be passed into the algorithm. We give an algorithm
for computing $h_{\{V_{i,x}\}_{i,x},\varTheta}$ as follows:

\begin{algorithm}[H]
\textbf{$\;\;\;\;$Input:} $\varTheta\sim\mathrm{Unif}[0,1]$, $\{V_{i,x}\}_{i\in[i_{0}..i_{1}],x\in\mathcal{X}}\stackrel{iid}{\sim}\mathrm{Exp}(1)$,
$P\in\mathcal{P}(\mathcal{X})$

\textbf{$\;\;\;\;$Output:} $h_{\{V_{x}\}_{x},\varTheta}(P)\in\mathcal{X}$

\smallskip{}

\begin{algorithmic}

\State{$S\leftarrow\mathcal{X}$}

\State{$\tilde{p}_{x}\leftarrow P(x)$ for $x\in\mathcal{X}$}

\State{$i\leftarrow i_{0}$}

\While{$|S|>1$}

\State{$w\leftarrow e^{-\eta(i+\varTheta)}$}

\State{$s_{x}\leftarrow|\mathcal{B}_{w}(x)|$ for $x\in S$}

\State{$\hat{p}_{x}\leftarrow\sum_{y\in S:\,d(x,y)\le w}\tilde{p}_{y}/s_{y}$
for $x\in\mathcal{X}$}

\State{$z\leftarrow\arg\min_{x}V_{i,x}/\hat{p}_{x}$}

\State{$S\leftarrow\{x\in S:\,d(x,z)\le w\}$}

\State{$\tilde{p}_{x}\leftarrow\tilde{p}_{x}/s_{x}$ for $x\in S$}

\State{$\{\tilde{p}_{x}\}_{x\in S}\leftarrow\{\tilde{p}_{x}/\sum_{y\in S}\tilde{p}_{y}\}_{x\in S}$}

\State{$i\leftarrow i+1$}

\EndWhile

\Return{$x\in S$ (the only element)}

\end{algorithmic}

\caption{\label{alg:lsh}Locality sensitive hash function for finite metric
space}
\end{algorithm}

In the algorithm, $\tilde{p}_{x}$ represents the posterior distribution
$\mathbf{P}_{X|\{Z_{j}\}_{j<i}}$, where $X\sim P$, and $Z_{i}|X\sim\mathrm{U}\mathcal{B}_{e^{-\eta(i+\theta)}}$
are conditionally independent across $i$ given $X$, and $S$ is
the support of $\mathbf{P}_{X|\{Z_{j}\}_{j<i}}$. Also, $\hat{p}_{x}$
represents the distribution $\mathbf{P}_{Z_{i}|\{Z_{j}\}_{j<i}}=\mathrm{U}\mathcal{B}_{e^{-\eta(i+\theta)}}\circ\mathbf{P}_{X|\{Z_{j}\}_{j<i}}$.
After generating $z=Z_{i}$ using the Poisson functional representation
on the distribution $\mathbf{P}_{Z_{i}|\{Z_{j}\}_{j<i}}$, the posterior
distribution is updated as $\mathbf{P}_{X|\{Z_{j}\}_{j\le i}}(x)\propto\mathbf{P}_{X|\{Z_{j}\}_{j<i}}(x)\mathrm{U}\mathcal{B}_{e^{-\eta(i+\theta)}}(x,z)$.
While we normalize $\tilde{p}_{x}$ in the algorithm to make them
sum to $1$, this step is not necessary.

The time complexity of the algorithm is
\[
O\left(|\mathcal{X}|^{2}\log\frac{\max\{d(x,y):\,x\neq y\}}{\min\{d(x,y):\,x\neq y\}}\right).
\]
It is possible to make the time complexity independent of the values
of $d(x,y)$ by noting that the performance guarantee in Theorem \ref{thm:metric_log}
comes from \eqref{eq:metrc_finite_dbd}. Hence, instead of considering
all $i\in[i_{0}..i_{1}]$, we can only consider $i$'s in the set
$\{\lfloor-\eta^{-1}\log(e^{\eta\theta}d(x,y)/2)\rfloor+1:\,x\neq y\}$
(i.e., the $i$'s so that $2e^{-\eta(i+\theta)}$ is just smaller
than some $d(x,y)$), while retaining the same guarantee in Theorem
\ref{thm:metric_log}. This set has size at most $O(|\mathcal{X}|\log|\mathcal{X}|)$
(see Proposition \ref{prop:metric_dist_set} below). Therefore the
time complexity becomes $O(|\mathcal{X}|^{3}\log|\mathcal{X}|)$.

\medskip{}

\begin{prop}
\label{prop:metric_dist_set}Let $(\mathcal{X},d)$ be a finite metric
space, and $\eta>0$. We have
\[
\left|\left\{ \lfloor-\eta^{-1}\log d(x,y)\rfloor:\,x\neq y\right\} \right|\le\frac{|\mathcal{X}|\log|\mathcal{X}|}{\eta}+2|\mathcal{X}|-2.
\]
\end{prop}

\begin{proof}
[Proof of Proposition \ref{prop:metric_dist_set}] We prove the
claim by induction on $|\mathcal{X}|$. The claim is clearly true
when $|\mathcal{X}|=1$. Assume the claim is true for any metric space
$(\mathcal{X},d)$ with $|\mathcal{X}|<k$. Consider a metric space
$(\mathcal{X},d)$ with $|\mathcal{X}|=k$. Treat $(\mathcal{X},d)$
as a complete graph with edge weights given by $d$, and consider
its minimum spanning tree. Let $(x_{1},x_{2})$ be one of the longest
edges in the minimum spanning tree. Removing $(x_{1},x_{2})$ from
the minimum spanning tree breaks the tree into two connected components.
Let the vertex sets of the two components be $S_{1}$ and $S_{2}$.
Consider any $y_{1}\in S_{1}$, $y_{2}\in S_{2}$. We have $d(y_{1},y_{2})\ge d(x_{1},x_{2})$,
or else replacing the edge $(x_{1},x_{2})$ by $(y_{1},y_{2})$ decreases
the total weight of the minimum spanning tree. Also we have $d(y_{1},y_{2})\le(k-1)d(x_{1},x_{2})$,
since there is a path from $y_{1}$ to $y_{2}$ along the minimum
spanning tree where the weight of each edge is at most $d(x_{1},x_{2})$.
Hence,
\begin{align*}
 & \left|\left\{ \lfloor-\eta^{-1}\log d(x,y)\rfloor:\,x,y\in\mathcal{X},\,x\neq y\right\} \right|\\
 & \le\sum_{i=1}^{2}\left|\left\{ \lfloor-\eta^{-1}\log d(x,y)\rfloor:\,x,y\in S_{i},\,x\neq y\right\} \right|+\left|\left\{ \lfloor-\eta^{-1}\log d(y_{1},y_{2})\rfloor:\,y_{1}\in S_{1},y_{2}\in S_{2}\right\} \right|\\
 & \stackrel{(a)}{\le}\sum_{i=1}^{2}\left(\frac{|S_{i}|\log|S_{i}|}{\eta}+2|S_{i}|-2\right)+\left(\frac{-\log d(x_{1},x_{2})}{\eta}-\frac{-\log(k-1)d(x_{1},x_{2})}{\eta}+2\right)\\
 & \stackrel{(b)}{\le}\left(\frac{(k-1)\log(k-1)}{\eta}+2(k-1)-2\right)+\left(\frac{\log(k-1)}{\eta}+2\right)\\
 & \le\frac{k\log k}{\eta}+2k-2,
\end{align*}
where (a) is by applying the induction hypothesis on $S_{1},S_{2}$,
and $d(x_{1},x_{2})\le d(y_{1},y_{2})\le(k-1)d(x_{1},x_{2})$, and
(b) is by the convexity of $t\log t$.
\end{proof}
\medskip{}

For the case $\mathcal{X}=[0..s]^{n}$, $d(x,y)=\Vert x-y\Vert_{2}$,
a slight modification of this algorithm attains the $O(\sqrt{n}\log s)$
result in Proposition \ref{prop:s_rc_ub}, with a time complexity
$O(2^{n}|\mathcal{X}|\log^{2}|\mathcal{X}|)$. See Remark \ref{rem:grid_alg}
for a discussion.

\medskip{}

\subsection{Snowflake Metric Cost\label{subsec:metric_pow}}

We can apply the sequential Poisson functional representation to prove
Theorem \ref{thm:metric_pow} concerning the case where the symmetric
cost function $c(x,y)=(d(x,y))^{q}$ is a power of a metric $d$ over
a complete separable metric space $(\mathcal{X},d)$, and $0<q<1$.
\begin{proof}
[Proof of Theorem \ref{thm:metric_pow}] Let $\mu$ be a $\sigma$-finite
measure over $\mathcal{X}$, and $\Psi>0$, satisfying $0<\mu(\mathcal{B}_{w}(x))<\infty$
for any $x\in\mathcal{X}$, $w>0$, and
\[
\frac{\mu(\mathcal{B}_{w}(x)\backslash\mathcal{B}_{w}(y))}{\mu(\mathcal{B}_{w}(x))}\le\frac{\Psi d(x,y)}{w}
\]
for any $x,y\in\mathcal{X}$ and $w>0$. Let $\mathcal{F}$ be the
Borel $\sigma$-algebra of $\mathcal{X}$. Let $l\in\mathbb{N}$ and
$\eta>0$ be fixed numbers which will be chosen later. Fix any point
$x_{0}\in\mathcal{X}$. Let $\mathrm{U}\mathcal{B}_{w}$ be a probability
kernel from $\mathcal{X}$ to $\mathcal{X}$ defined by 
\begin{equation}
\mathrm{U}\mathcal{B}_{w}(x,E):=\frac{\mu(\mathcal{B}_{w}(x)\cap E)}{\mu(\mathcal{B}_{w}(x))}.\label{eq:metric_pow_pf_unifball}
\end{equation}
Let $\theta\in[0,1]$. For a collection of probability distributions
$\{P_{\alpha}\}_{\alpha}$, let 
\begin{equation}
\bar{P}_{\alpha,\theta}:=\Big(\prod_{i\in\mathbb{Z}}\mathrm{U}\mathcal{B}_{e^{-\eta(i+\theta)}}\Big)\circ P_{\alpha},\label{eq:metric_pow_patheta}
\end{equation}
i.e., $\bar{P}_{\alpha,\theta}$ is the distribution of $\{Z_{i}\}_{i\in\mathbb{Z}}$,
where $X\sim P_{\alpha}$, and $Z_{i}|X\sim\mathrm{U}\mathcal{B}_{e^{-\eta(i+\theta)}}$
are conditionally independent across $i$ given $X$. We then apply
the sequential Poisson functional representation on $I=\mathbb{Z}$,
$\bar{P}=\bar{P}_{\alpha,\theta}$, $\mu_{i}=\mu$, $\nu_{i}=\mathrm{U}\mathcal{B}_{e^{-\eta(i+\theta)}}(x_{0},\cdot)$
to obtain $\bar{\varrho}_{\bar{P}_{\alpha,\theta}}$. Define
\begin{equation}
X_{\alpha,\theta}:=\lim_{i\to\infty}\bar{\varrho}_{\bar{P}_{\alpha,\theta},i}(\Phi_{I}),\label{eq:metric_pow_limdef}
\end{equation}
\[
X_{\alpha}:=X_{\alpha,\varTheta},
\]
where $\varTheta\sim\mathrm{Unif}(\{j/l:\,j\in[0..l-1]\})$ independent
of $\Phi_{i}\stackrel{iid}{\sim}\mathrm{PP}(\mu\times\lambda_{\mathbb{R}_{\ge0}})$,
and the limit in \eqref{eq:metric_pow_limdef} is taken with respect
to the metric $d$.

Refer to Appendix \ref{subsec:pf_metric_pow_spfr} for the proof that
the SPFR condition is satisfied. We now show that $X_{\alpha,\theta}$
is defined (the limit exists) and $X_{\alpha,\theta}\sim P_{\alpha}$.
Consider the function $g:\mathcal{X}^{\mathbb{Z}}\to\mathcal{X}$,
$g(\{z_{i}\}_{i})=\lim_{i\to\infty}z_{i}$ if the limit exists, $g(\{z_{i}\}_{i})=x_{0}$
otherwise. Since the pointwise limit of a sequence of measurable functions
is measurable, $g$ is measurable over the $\sigma$-algebra $\mathcal{F}^{\otimes\mathbb{Z}}$.
Let $X\sim P_{\alpha}$, and $Z_{i}|X\sim\mathrm{U}\mathcal{B}_{e^{-\eta(i+\theta)}}$
are conditionally independent across $i$ given $X$. Since $d(Z_{i},Z_{j})\le e^{-\eta i}+e^{-\eta j}$
almost surely (in the probability space containing $X,\{Z_{i}\}_{i}$),
$\{Z_{i}\}_{i}$ is a Cauchy sequence almost surely, and $X=g(\{Z_{i}\}_{i})$
almost surely. Since $\{Z_{i}\}_{i}$ has the same distribution as
$\{\bar{\varrho}_{\bar{P}_{\alpha,\theta},i}(\Phi_{I})\}_{i}$ over
$\mathcal{F}^{\otimes\mathbb{Z}}$, $X=g(\{Z_{i}\}_{i})$ has the
same distribution as $g(\{\bar{\varrho}_{\bar{P}_{\alpha,\theta},i}(\Phi_{I})\}_{i})$.
Also $\{\bar{\varrho}_{\bar{P}_{\alpha,\theta},i}(\Phi_{I})\}_{i}$
is a Cauchy sequence almost surely (in the probability space containing
$\Phi_{I}$). Hence, $X_{\alpha,\theta}=g(\{\bar{\varrho}_{\bar{P}_{\alpha,\theta},i}(\Phi_{I})\}_{i})\sim P_{\alpha}$
for any $\theta\in[0,1]$, and hence $X_{\alpha}\sim P_{\alpha}$.

Consider two probability distributions $P_{\alpha},P_{\beta}$. Fix
any coupling $(\tilde{X}_{\alpha},\tilde{X}_{\beta})\in\Gamma_{\lambda}(P_{\alpha},P_{\beta})$.
For any $i\in\mathbb{Z}$ and $\theta\in[0,1]$,
\begin{align}
 & \mathbf{P}\Big(d(X_{\alpha,\theta},X_{\beta,\theta})>2e^{-\eta(i+\theta)}\Big)\nonumber \\
 & \le\mathbf{P}\Big(\bar{\varrho}_{\bar{P}_{\alpha,\theta},i}(\Phi_{I})\neq\bar{\varrho}_{\bar{P}_{\alpha,\beta},i}(\Phi_{I})\Big)\nonumber \\
 & \stackrel{(a)}{\le}2\sum_{j\le i}(1+\mathbf{1}\{j<i\})d_{\mathrm{TV}}\big(\bar{P}_{\alpha,\theta,I_{\le j}},\,\bar{P}_{\beta,\theta,I_{\le j}}\big)\nonumber \\
 & =2\sum_{j\le i}(1+\mathbf{1}\{j<i\})d_{\mathrm{TV}}\left(\Big(\prod_{k\le j}\mathrm{U}\mathcal{B}_{e^{-\eta(k+\theta)}}\Big)\circ P_{\alpha},\,\Big(\prod_{k\le j}\mathrm{U}\mathcal{B}_{e^{-\eta(k+\theta)}}\Big)\circ P_{\beta}\right)\nonumber \\
 & \stackrel{(b)}{\le}2\sum_{j\le i}(1+\mathbf{1}\{j<i\})\mathbf{E}\left[d_{\mathrm{TV}}\left(\Big(\prod_{k\le j}\mathrm{U}\mathcal{B}_{e^{-\eta(k+\theta)}}\Big)(\tilde{X}_{\alpha},\cdot),\,\Big(\prod_{k\le j}\mathrm{U}\mathcal{B}_{e^{-\eta(k+\theta)}}\Big)(\tilde{X}_{\beta},\cdot)\right)\right]\nonumber \\
 & \stackrel{(c)}{\le}2\sum_{j\le i}(1+\mathbf{1}\{j<i\})\mathbf{E}\left[\min\left\{ \sum_{k\le j}d_{\mathrm{TV}}\left(\mathrm{U}\mathcal{B}_{e^{-\eta(k+\theta)}}(\tilde{X}_{\alpha},\cdot),\,\mathrm{U}\mathcal{B}_{e^{-\eta(k+\theta)}}(\tilde{X}_{\beta},\cdot)\right),\,1\right\} \right]\nonumber \\
 & \stackrel{(d)}{\le}2\sum_{j\le i}(1+\mathbf{1}\{j<i\})\mathbf{E}\left[\min\left\{ \sum_{k\le j}\Psi e^{\eta(k+\theta)}d(\tilde{X}_{\alpha},\tilde{X}_{\beta}),\,1\right\} \right]\nonumber \\
 & =2\sum_{j\le i}(1+\mathbf{1}\{j<i\})\mathbf{E}\left[\min\left\{ \frac{\Psi e^{\eta(j+\theta)}}{1-e^{-\eta}}d(\tilde{X}_{\alpha},\tilde{X}_{\beta}),\,1\right\} \right],\label{eq:metric_pow_pdbd}
\end{align}
where (a) is by Lemma \ref{lem:spfr_dpcbd}, (b) is by $\tilde{X}_{\alpha}\sim P_{\alpha}$,
$\tilde{X}_{\beta}\sim P_{\beta}$, and the convexity of $d_{\mathrm{TV}}$,
(c) is due to $d_{\mathrm{TV}}(\prod_{k=1}^{\infty}Q_{k},\,\prod_{k=1}^{\infty}\tilde{Q}_{k})\le\sum_{k=1}^{\infty}d_{\mathrm{TV}}(Q_{k},\tilde{Q}_{k})$
for probability distributions $Q_{k}$, $\tilde{Q}_{k}$ over $\mathcal{X}$
(where $\prod_{k=1}^{\infty}Q_{k}$ denotes the product measure) \footnote{This can be shown by letting $(W_{k},\tilde{W}_{k})$ be independent
across $k$ such that $W_{k}\sim Q_{k}$, $\tilde{W}_{k}\sim\tilde{Q}_{k}$
and $\mathbf{P}(W_{k}\neq\tilde{W}_{k})\le d_{\mathrm{TV}}(Q_{k},\tilde{Q}_{k})+2^{-k}\epsilon$.
We have $d_{\mathrm{TV}}(\prod_{k=1}^{\infty}Q_{k},\,\prod_{k=1}^{\infty}\tilde{Q}_{k})\le\mathbf{P}(\{W_{k}\}_{k}\neq\{\tilde{W}_{k}\}_{k})\le\sum_{k=1}^{\infty}\mathbf{P}(W_{k}\neq\tilde{W}_{k})\le\sum_{k=1}^{\infty}d_{\mathrm{TV}}(Q_{k},\tilde{Q}_{k})+\epsilon$
for any $\epsilon>0$.}, and (d) is because
\begin{align}
 & d_{\mathrm{TV}}\left(\mathrm{U}\mathcal{B}_{w}(x,\cdot),\,\mathrm{U}\mathcal{B}_{w}(y,\cdot)\right)\nonumber \\
 & =\max\left\{ \mathrm{U}\mathcal{B}_{w}(x,\mathcal{B}_{w}(x)\backslash\mathcal{B}_{w}(y)),\,\mathrm{U}\mathcal{B}_{w}(y,\mathcal{B}_{w}(y)\backslash\mathcal{B}_{w}(x))\right\} \nonumber \\
 & =\max\left\{ \frac{\mu(\mathcal{B}_{w}(x)\backslash\mathcal{B}_{w}(y))}{\mu(\mathcal{B}_{w}(x))},\,\frac{\mu(\mathcal{B}_{w}(y)\backslash\mathcal{B}_{w}(x))}{\mu(\mathcal{B}_{w}(y))}\right\} \nonumber \\
 & \le\frac{\Psi d(x,y)}{w}\label{eq:metric_pow_pf_dtv}
\end{align}
by \eqref{eq:metric_pow_delta_dist}. Hence,
\begin{align*}
 & \mathbf{E}\left[(d(X_{\alpha,\varTheta},X_{\beta,\varTheta}))^{q}\right]\\
 & =l^{-1}\sum_{k=0}^{l-1}\mathbf{E}\left[(d(X_{\alpha,k/l},X_{\beta,k/l}))^{q}\right]\\
 & =l^{-1}\sum_{k=0}^{l-1}\int_{0}^{\infty}\mathbf{P}\left((d(X_{\alpha,k/l},X_{\beta,k/l}))^{q}>t\right)\mathrm{d}t\\
 & \le l^{-1}\sum_{k=0}^{l-1}\int_{0}^{\infty}\mathbf{P}\left(d(X_{\alpha,k/l},X_{\beta,k/l})>2e^{-\eta\left(\left\lceil -\eta^{-1}\log(t^{1/q}/2)-k/l\right\rceil +k/l\right)}\right)\mathrm{d}t\\
 & \le l^{-1}\sum_{k=0}^{l-1}\int_{0}^{\infty}\bigg(2\sum_{j\le\left\lceil -\eta^{-1}\log(t^{1/q}/2)-k/l\right\rceil }\big(1+\mathbf{1}\{j<\left\lceil -\eta^{-1}\log(t^{1/q}/2)-k/l\right\rceil \}\big)\\
 & \;\;\;\;\;\;\;\cdot\mathbf{E}\left[\min\left\{ \frac{\Psi e^{\eta(j+k/l)}}{1-e^{-\eta}}d(\tilde{X}_{\alpha},\tilde{X}_{\beta}),\,1\right\} \right]\bigg)\mathrm{d}t\\
 & \stackrel{(a)}{=}2l^{-1}\int_{0}^{\infty}\bigg(\sum_{k<\left\lceil l(-\eta^{-1}\log(t^{1/q}/2)+1)\right\rceil }\big(1+\mathbf{1}\{k<\left\lceil -l\eta^{-1}\log(t^{1/q}/2)\right\rceil \}\big)\mathbf{E}\left[\min\left\{ \frac{\Psi e^{\eta k/l}}{1-e^{-\eta}}d(\tilde{X}_{\alpha},\tilde{X}_{\beta}),\,1\right\} \right]\bigg)\mathrm{d}t\\
 & \le2\int_{0}^{\infty}\int_{-\infty}^{-\eta^{-1}\log(t^{1/q}/2)+1+2l^{-1}}(1+\mathbf{1}\{\gamma<-\eta^{-1}\log(t^{1/q}/2)+2l^{-1}\})\mathbf{E}\left[\min\left\{ \frac{\Psi e^{\eta\gamma}}{1-e^{-\eta}}d(\tilde{X}_{\alpha},\tilde{X}_{\beta}),\,1\right\} \right]\mathrm{d}\gamma\mathrm{d}t\\
 & =2\int_{-\infty}^{\infty}\left(\int_{0}^{(2e^{-\eta(\gamma-1-2l^{-1})})^{q}}(1+\mathbf{1}\{\gamma<-\eta^{-1}\log(t^{1/q}/2)+2l^{-1}\})\mathrm{d}t\right)\mathbf{E}\left[\min\left\{ \frac{\Psi e^{\eta\gamma}}{1-e^{-\eta}}d(\tilde{X}_{\alpha},\tilde{X}_{\beta}),\,1\right\} \right]\mathrm{d}\gamma\\
 & =2\int_{-\infty}^{\infty}\left((2e^{-\eta(\gamma-1-2l^{-1})})^{q}+(2e^{-\eta(\gamma-2l^{-1})})^{q}\right)\mathbf{E}\left[\min\left\{ \frac{\Psi e^{\eta\gamma}}{1-e^{-\eta}}d(\tilde{X}_{\alpha},\tilde{X}_{\beta}),\,1\right\} \right]\mathrm{d}\gamma\\
 & =2^{1+q}\int_{0}^{\infty}e^{2q\eta l^{-1}}(e^{q\eta}+1)\left(\frac{t\Psi}{1-e^{-\eta}}\right)^{q}\mathbf{E}\left[\min\left\{ t^{-1}d(\tilde{X}_{\alpha},\tilde{X}_{\beta}),\,1\right\} \right]\frac{1}{\eta t}\mathrm{d}t\\
 & =2^{1+q}\eta^{-1}e^{2q\eta l^{-1}}(e^{q\eta}+1)\left(\frac{\Psi}{1-e^{-\eta}}\right)^{q}\int_{0}^{\infty}t^{q-2}\mathbf{E}\left[\min\left\{ d(\tilde{X}_{\alpha},\tilde{X}_{\beta}),\,t\right\} \right]\mathrm{d}t\\
 & =2^{1+q}e^{2q\eta l^{-1}}\frac{e^{q\eta}+1}{q\eta(1-q)}\left(\frac{\Psi}{1-e^{-\eta}}\right)^{q}\int_{0}^{\infty}q(1-q)t^{q-2}\mathbf{E}\left[\min\left\{ d(\tilde{X}_{\alpha},\tilde{X}_{\beta}),\,t\right\} \right]\mathrm{d}t\\
 & \stackrel{(b)}{=}2^{1+q}e^{2q\eta l^{-1}}\frac{e^{q\eta}+1}{q\eta(1-q)}\left(\frac{\Psi}{1-e^{-\eta}}\right)^{q}\mathbf{E}\left[(d(\tilde{X}_{\alpha},\tilde{X}_{\beta}))^{q}\right],
\end{align*}
where (a) is by substituting $k\leftarrow jl+k$, and (b) is due to
$\mathbf{E}[Z^{q}]=\int_{0}^{\infty}\mathbf{E}[\min\{Z,t\}]q(1-q)t^{q-2}\mathrm{d}t$
for any random variable $Z\ge0$ (see \eqref{eq:ezpow}). Substituting
$\eta:=5/(3q)$, we have
\begin{align*}
 & \mathbf{E}\left[(d(X_{\alpha,\varTheta},X_{\beta,\varTheta}))^{q}\right]\\
 & \le2^{1+q}e^{2(5/3)l^{-1}}\frac{e^{5/3}+1}{(5/3)(1-q)}\left(\frac{\Psi}{1-e^{-5/(3q)}}\right)^{q}\mathbf{E}\left[(d(\tilde{X}_{\alpha},\tilde{X}_{\beta}))^{q}\right]\\
 & \le2^{1+q}e^{2(5/3)l^{-1}}\frac{e^{5/3}+1}{(5/3)(1-q)}\left(\frac{\Psi}{1-e^{-5/3}}\right)^{q}\mathbf{E}\left[(d(\tilde{X}_{\alpha},\tilde{X}_{\beta}))^{q}\right]\\
 & \le7.554e^{2(5/3)l^{-1}}\frac{(2.47\Psi)^{q}}{1-q}\mathbf{E}\left[(d(\tilde{X}_{\alpha},\tilde{X}_{\beta}))^{q}\right]\\
 & \le7.555\frac{(2.47\Psi)^{q}}{1-q}\mathbf{E}\left[(d(\tilde{X}_{\alpha},\tilde{X}_{\beta}))^{q}\right]
\end{align*}
by substituting $l=\lceil2(5/3)/\log(7.555/7.554)\rceil$. Hence,
\begin{align}
 & \mathbf{E}\left[(d(X_{\alpha,\varTheta},X_{\beta,\varTheta}))^{q}\right]\nonumber \\
 & \le\inf_{(\tilde{X}_{\alpha},\tilde{X}_{\beta})\in\Gamma_{\lambda}(P_{\alpha},P_{\beta})}7.555\cdot\frac{(2.47\Psi)^{q}}{1-q}\mathbf{E}\left[(d(\tilde{X}_{\alpha},\tilde{X}_{\beta}))^{q}\right]\nonumber \\
 & \le7.555\cdot\frac{(2.47\Psi)^{q}}{1-q}C_{c}^{*}(P_{\alpha},P_{\beta}).\label{eq:metric_pow_betterconst}
\end{align}
\end{proof}
\begin{rem}
The reason for letting $\varTheta\sim\mathrm{Unif}(\{j/l:\,j\in[0..l-1]\})$
instead of fixing $\varTheta=0$ (i.e., taking $l=1$) is to make
the construction less ``scale-dependent'', i.e., treating the balls
$\mathcal{B}_{e^{-\eta k}}$ for integer $k$ less differently compared
to balls of other radii. This can lead to a slightly better constant.
While we take a large value of $l$ so $\varTheta$ is close to being
distributed as $\mathrm{Unif}[0,1]$, we do not directly let $\varTheta\sim\mathrm{Unif}[0,1]$
to avoid having to prove that $X_{\alpha,\varTheta}$ is a random
variable (note that if $X_{\alpha,j/l}$ are random variables for
$j\in[0..l-1]$, and $\varTheta\sim\mathrm{Unif}(\{j/l:\,j\in[0..l-1]\})$
independent of $\{X_{\alpha,j/l}\}_{j}$, then $X_{\alpha,\varTheta}$
is a random variable).
\end{rem}

\medskip{}

\begin{rem}
\label{rem:lsh}The sequential Poisson functional representation can
be considered as a locality sensitive hash function for the estimation
of $C_{c}^{*}$ in the setting in \cite{charikar2002similarity}.
Let $h_{\phi_{I},\theta}:\mathcal{P}(\mathcal{X})\to\mathcal{X}$
be defined by $h_{\phi_{I},\theta}(P):=\lim_{i\to\infty}\bar{\varrho}_{\bar{P}_{\theta},i}(\phi_{I})$,
where $\bar{P}_{\theta}$ is defined in \eqref{eq:metric_pow_patheta}
with $P_{\alpha}\leftarrow P$. Let $\varTheta\sim\mathrm{Unif}(\{j/l:\,j\in[0..l-1]\})$
be independent of $\Phi_{i}\stackrel{iid}{\sim}\mathrm{PP}(\mu\times\lambda_{\mathbb{R}_{\ge0}})$.
Then by \eqref{eq:metric_pow_betterconst}, for any $P,Q\in\mathcal{P}(\mathcal{X})$,
\[
C_{c}^{*}(P,Q)\le\mathbf{E}\left[(d(h_{\Phi_{I},\varTheta}(P),h_{\Phi_{I},\varTheta}(Q)))^{q}\right]\le7.56\cdot\frac{(2.47\Psi)^{q}}{1-q}C_{c}^{*}(P,Q).
\]
Therefore, we can estimate $C_{c}^{*}(P,Q)$ by drawing i.i.d. samples
of $(\varTheta,\Phi_{I})$ and computing the sample mean of $(d(h_{\Phi_{I},\varTheta}(P),h_{\Phi_{I},\varTheta}(Q)))^{q}$,
which is asymptotically accurate up to a multiplicative factor.\medskip{}
\end{rem}

\medskip{}

The following weaker form of Theorem \ref{thm:metric_pow} is sometimes
easier to check.
\begin{cor}
\label{cor:metric_pow_c}Let $(\mathcal{X},d)$ be complete separable
metric space. Consider the symmetric cost function $c(x,y)=(d(x,y))^{q}$,
$0<q<1$. Let $\mathcal{B}_{w}(x):=\{y\in\mathcal{X}:\,d(x,y)\le w\}$.
Let $\mu$ be a $\sigma$-finite measure over $\mathcal{X}$, and
$\Psi\ge1$ satisfying that for any $x\in\mathcal{X}$,
\begin{equation}
w\mapsto w^{-\Psi}\mu(\mathcal{B}_{w}(x))\label{eq:metric_pow_c_noninc}
\end{equation}
is a non-increasing function of $w\in\mathbb{R}_{>0}$ that takes
values in $\mathbb{R}_{>0}$ (i.e., always positive and finite). Then
we have
\[
r_{c}^{*}(\mathcal{P}(\mathcal{X}))<7.56\cdot\frac{(2.47\Psi)^{q}}{1-q}.
\]
\end{cor}

\begin{proof}
[Proof of Corollary \ref{cor:metric_pow_c}] We check \eqref{eq:metric_pow_delta_dist}.
For any $x,y\in\mathcal{X}$, $w>0$, if $d(x,y)<w$, then
\begin{align*}
 & \frac{\mu(\mathcal{B}_{w}(x)\backslash\mathcal{B}_{w}(y))}{\mu(\mathcal{B}_{w}(x))}\\
 & \le\frac{\mu(\mathcal{B}_{w}(x)\backslash\mathcal{B}_{w-d(x,y)}(x))}{\mu(\mathcal{B}_{w}(x))}\\
 & =1-\frac{\mu(\mathcal{B}_{w-d(x,y)}(x))}{\mu(\mathcal{B}_{w}(x))}\\
 & =1-\frac{(w-d(x,y))^{-\Psi}\mu(\mathcal{B}_{w-d(x,y)}(x))}{w^{-\Psi}\mu(\mathcal{B}_{w}(x))}\left(\frac{w-d(x,y)}{w}\right)^{\Psi}\\
 & \stackrel{(a)}{\le}1-\left(\frac{w-d(x,y)}{w}\right)^{\Psi}\\
 & \stackrel{(b)}{\le}\frac{\Psi d(x,y)}{w},
\end{align*}
where (a) is by \eqref{eq:metric_pow_c_noninc}, and (b) is because
$\Psi\ge1$. The case $d(x,y)\ge w$ follows directly from $\Psi\ge1$.

We then check \eqref{eq:metric_pow_limsup}. For any $x,y\in\mathcal{X}$,
$w>0$ such that $d(x,y)\le w$, by \eqref{eq:metric_pow_c_noninc},
\begin{align*}
\frac{\mu(\mathcal{B}_{w}(x))}{\mu(\mathcal{B}_{w}(y))} & \le\frac{\mu(\mathcal{B}_{2w}(y))}{\mu(\mathcal{B}_{w}(y))}\le2^{\Psi}.
\end{align*}
 The result follows from Theorem \ref{thm:metric_pow}.
\end{proof}
\medskip{}

\subsection{$\ell_{p}$ Metric over $\mathbb{R}^{n}$\label{subsec:rn_lp}}

We now prove Theorem \ref{thm:rn_rc_ub} by applying Theorem \ref{thm:metric_pow}
to find $r_{c}^{*}(\mathcal{P}(\mathbb{R}^{n}))$ for $c(x,y)=\Vert x-y\Vert_{p}^{q}$,
$0<q<1$, $p\in\mathbb{R}_{\ge1}\cup\{\infty\}$. The proof of Proposition
\ref{prop:s_rc_ub} is also given.
\begin{proof}
[Proof of Theorem \ref{thm:rn_rc_ub} and Proposition \ref{prop:s_rc_ub}]Let
$d(x,y)=\Vert x-y\Vert_{p}$, $\mu=\lambda$.  The condition \eqref{eq:metric_pow_limsup}
is clearly satisfied. To find $\Psi$ for \eqref{eq:metric_pow_delta_dist},
we fix any $x,z\in\mathbb{R}^{n}$, $w>0$. For $p\le2$, 
\begin{align*}
 & \lambda(\mathcal{B}_{w}(x)\backslash\mathcal{B}_{w}(x+z))\\
 & \stackrel{(a)}{\le}\Vert z\Vert_{2}\lambda_{n-1}\big(\mathrm{P}_{z^{\perp}}(\mathcal{B}_{w}(x))\big)\\
 & \stackrel{(b)}{\le}\Vert z\Vert_{2}\lambda_{n-1}\big(\mathrm{P}_{\mathrm{e}_{1}^{\perp}}(\mathcal{B}_{w}(x))\big)\\
 & =\Vert z\Vert_{2}\mathrm{V}_{n-1,p}w^{n-1}\\
 & \le\Vert z\Vert_{p}\mathrm{V}_{n-1,p}w^{n-1},
\end{align*}
where (a) is because $\mathcal{B}_{w}(x)$ is convex, so $\mathcal{B}_{w}(x)\cap\{y+tz:\,t\in\mathbb{R}\}$
is a line segment for any $y\in\mathbb{R}^{n}$, and hence $(\mathcal{B}_{w}(x)\backslash\mathcal{B}_{w}(x+z))\cap\{y+tz:\,t\in\mathbb{R}\}$
is a line segment with length at most $\Vert z\Vert_{2}$, so by Fubini's
theorem, $\lambda(\mathcal{B}_{w}(x)\backslash\mathcal{B}_{w}(x+z))$
can be bounded by the product of $\Vert z\Vert_{2}$ and the area
of the projection of $\mathcal{B}_{w}(x)$ to the $(n-1)$-dimensional
hyperplane orthogonal to $z$ (such a projection is denoted as $\mathrm{P}_{z^{\perp}}(x):=x-\left\langle x,z\right\rangle z/\Vert z\Vert_{2}^{2}$,
and $\lambda_{n-1}$ denotes the Lebesgue measure over that hyperplane),
(b) is by \cite[Theorem 12]{barthe2002hyperplane}, and $\mathrm{V}_{n,p}$
is the volume of the unit $\ell_{p}$ ball given in \eqref{eq:lp_ball_vol}.

For $p>2$,
\begin{align*}
 & \lambda(\mathcal{B}_{w}(x)\backslash\mathcal{B}_{w}(x+z))\\
 & \le\sum_{i=1}^{n}\lambda\bigg(\mathcal{B}_{w}\bigg(x+\sum_{j=1}^{i-1}z_{j}\mathrm{e}_{j}\bigg)\,\backslash\,\mathcal{B}_{w}\bigg(x+\sum_{j=1}^{i}z_{j}\mathrm{e}_{j}\bigg)\bigg)\\
 & =\sum_{i=1}^{n}\lambda\left(\mathcal{B}_{w}(x)\backslash\mathcal{B}_{w}(x+z_{i}\mathrm{e}_{i})\right)\\
 & \le\sum_{i=1}^{n}|z_{i}|\lambda_{n-1}\big(\mathrm{P}_{\mathrm{e}_{i}^{\perp}}(\mathcal{B}_{w}(x))\big)\\
 & =\Vert z\Vert_{1}\mathrm{V}_{n-1,p}w^{n-1}\\
 & \le n^{1-1/p}\Vert z\Vert_{p}\mathrm{V}_{n-1,p}w^{n-1}.
\end{align*}
Hence,
\begin{align*}
 & \frac{\lambda(\mathcal{B}_{w}(x)\backslash\mathcal{B}_{w}(x+z))}{\lambda(\mathcal{B}_{w}(x))}\\
 & \le\frac{\Vert z\Vert_{p}n^{\mathbf{1}\{p>2\}(1-1/p)}\mathrm{V}_{n-1,p}}{w\mathrm{V}_{n,p}}.
\end{align*}
Let $\Psi=n^{\mathbf{1}\{p>2\}(1-1/p)}\mathrm{V}_{n-1,p}/\mathrm{V}_{n,p}$.
 By Theorem \ref{thm:metric_pow},
\begin{align*}
 & r_{c}^{*}(\mathcal{P}(\mathcal{X}))\\
 & <\frac{7.56}{1-q}\left(\frac{2.47n^{\mathbf{1}\{p>2\}(1-1/p)}\mathrm{V}_{n-1,p}}{\mathrm{V}_{n,p}}\right)^{q},
\end{align*}
where
\begin{align}
\frac{\mathrm{V}_{n-1,p}}{\mathrm{V}_{n,p}} & =\frac{\mathit{\Gamma}(1+n/p)}{2\mathit{\Gamma}(1+1/p)\mathit{\Gamma}(1+(n-1)/p)}\nonumber \\
 & \stackrel{(a)}{\le}\frac{(\mathit{\Gamma}(1+(n-1)/p))^{1-1/p}(\mathit{\Gamma}(2+(n-1)/p))^{1/p}}{2\mathit{\Gamma}(1+1/p)\mathit{\Gamma}(1+(n-1)/p)}\nonumber \\
 & =\frac{(1+(n-1)/p)^{1/p}}{2\mathit{\Gamma}(1+1/p)},\label{eq:Bnp_ratio}
\end{align}
where (a) is because $\log\mathit{\Gamma}(t)$ is convex. Therefore,
\begin{align*}
 & r_{c}^{*}(\mathcal{P}(\mathcal{X}))\\
 & <\frac{7.56}{1-q}\left(\frac{1.235n^{\mathbf{1}\{p>2\}(1-1/p)}(1+(n-1)/p)^{1/p}}{\mathit{\Gamma}(1+1/p)}\right)^{q}\\
 & \le\frac{7.56}{1-q}\left(1.3946n^{1/p+\mathbf{1}\{p>2\}(1-1/p)}\right)^{q}\\
 & \le\frac{10.543}{1-q}n^{q/p+\mathbf{1}\{p>2\}(q-q/p)}.
\end{align*}
To improve the bound for $p>2$, by the ratio bound in Proposition
\ref{prop:rc_prop_misc},
\begin{align*}
r_{c}^{*}(\mathcal{P}(\mathbb{R}^{n})) & \le n^{q(1/2-1/p)}r_{\Vert\cdot\Vert_{2}^{q}}^{*}(\mathcal{P}(\mathbb{R}^{n}))\\
 & <n^{q(1/2-1/p)}\frac{10.543}{1-q}n^{q/2}\\
 & =\frac{10.55}{1-q}n^{q(1-1/p)}.
\end{align*}
Hence, for any $p\in\mathbb{R}_{\ge1}\cup\{\infty\}$,
\[
r_{c}^{*}(\mathcal{P}(\mathbb{R}^{n}))<\frac{10.55}{1-q}n^{q\max\{1/p,\,1-1/p\}}.
\]

For Proposition \ref{prop:s_rc_ub} concerning the case $\mathcal{X}\subseteq\mathbb{R}^{n}$
is finite, $c(x,y)=\Vert x-y\Vert_{p}^{q}$, $q\ge1$, by the ratio
bound in Proposition \ref{prop:rc_prop_misc}, for any $0<\tilde{q}<1$,
\begin{align*}
r_{c}^{*}(\mathcal{P}(\mathcal{X})) & \le\gamma^{q-\tilde{q}}r_{\Vert x-y\Vert_{p}^{\tilde{q}}}^{*}(\mathcal{P}(\mathcal{X}))\\
 & <\gamma^{q-\tilde{q}}\frac{10.543}{1-\tilde{q}}n^{\tilde{q}\max\{1/p,\,1-1/p\}}\\
 & =10.543\gamma^{q}\frac{1}{1-\tilde{q}}(n^{\max\{1/p,\,1-1/p\}}/\gamma)^{\tilde{q}}.
\end{align*}
If $\gamma>e\cdot n^{\max\{1/p,\,1-1/p\}}$, substituting $\tilde{q}=1+1/\log(n^{\max\{1/p,\,1-1/p\}}/\gamma)$,
we have
\begin{align*}
r_{c}^{*}(\mathcal{P}(\mathcal{X})) & <10.543\gamma^{q}e(n^{\max\{1/p,\,1-1/p\}}/\gamma)\log\left(\gamma/n^{\max\{1/p,\,1-1/p\}}\right)\\
 & <28.659n^{\max\{1/p,\,1-1/p\}}\gamma^{q-1}\log\left(\gamma/n^{\max\{1/p,\,1-1/p\}}\right)\\
 & \le28.659n^{\max\{1/p,\,1-1/p\}}\gamma^{q-1}\log\left(\gamma/n^{\max\{1/p,\,1-1/p\}}+1\right).
\end{align*}
If $\gamma\le e\cdot n^{\max\{1/p,\,1-1/p\}}$, by the ratio bound
in Proposition \ref{prop:rc_prop_misc} (with respect to the discrete
metric), 
\begin{align*}
r_{c}^{*}(\mathcal{P}(\mathcal{X})) & \le2\gamma^{q}\\
 & \stackrel{(a)}{\le}2\gamma^{q}\frac{n^{\max\{1/p,\,1-1/p\}}\gamma^{-1}\log\left(\gamma/n^{\max\{1/p,\,1-1/p\}}+1\right)}{e^{-1}\log(e+1)}\\
 & <4.14n^{\max\{1/p,\,1-1/p\}}\gamma^{q-1}\log\left(\gamma/n^{\max\{1/p,\,1-1/p\}}+1\right),
\end{align*}
where (a) is because $t\log(1/t+1)$ is an increasing function.

\end{proof}
\[
\]

\begin{rem}
\label{rem:grid_alg}For the sake of applying Algorithm \ref{alg:lsh}
on the grid $[0..s]^{n}$, we first embed the grid into a discrete
torus $[0..2s]^{n}$. The $\ell_{p}$ distance between two points
$x,y\in[0..2s]^{n}$ on the discrete torus is given by $d_{\mathrm{T},p}(x,y):=\min_{z\in\mathbb{Z}^{n}}\Vert x-y+(2s+1)z\Vert_{p}$.
\footnote{Note that the embedding of the grid $[0..s]^{n}$ (with the normal
$\ell_{p}$ distance over $\mathbb{R}^{n}$) into the torus $[0..2s]^{n}$
(by the identity map) is isometric. The reason for using the discrete
torus is twofold: it allows the use of fast Fourier transform (which
performs convolution over the discrete torus); and a ball over the
discrete torus has size that only depends on its radius (not its center).} We want the probability distributions $\mathrm{U}\mathcal{B}_{w}(x,\cdot)$
in \ref{eq:metric_pow_pf_unifball} to be defined over $[0..2s]^{n}$,
not $\mathbb{R}^{n}$, so that the probability distributions are discrete
and can be handled by the algorithm. To achieve this, we take $\mu$
to be the counting measure over $[0..2s]^{n}$ (instead of $\lambda$).
Denote  $\mathcal{B}_{\mathbb{R}^{n},w}(x)\subseteq\mathbb{R}^{n}$
for the $\ell_{p}$ ball over $\mathbb{R}^{n}$. Instead of using
$\mathrm{U}\mathcal{B}_{w}(x,\cdot)$ in \eqref{eq:metric_pow_patheta}
(which is a probability kernel from $\mathbb{R}^{n}$ to $\mathbb{R}^{n}$),
we use probability kernel from $[0..2s]^{n}$ to $[0..2s]^{n}$ given
by
\begin{align*}
\mathrm{round}_{*}\mathrm{U}\mathcal{B}_{w}(x,E) & =\mathrm{U}\mathcal{B}_{w}(x,\mathrm{round}^{-1}(E))\\
 & =\frac{\lambda(\mathcal{B}_{\mathbb{R}^{n},w}(x)\cap\mathrm{round}^{-1}(E))}{\lambda(\mathcal{B}_{\mathbb{R}^{n},w}(x))}
\end{align*}
for $x\in[0..2s]^{n}$, $E\subseteq[0..2s]^{n}$, where $\mathrm{round}:\mathbb{R}^{n}\to[0..2s]^{n}$
is defined by $(\mathrm{round}(x))_{i}\equiv\lfloor x_{i}+1/2\rfloor$
($\mathrm{mod}\;2s+1$). The SPFR condition is clearly satisfied due
to the finiteness of $[0..2s]^{n}$ (and hence we only need to consider
a finite interval $I$). We still have \eqref{eq:metric_pow_pf_dtv}
since
\begin{align*}
 & d_{\mathrm{TV}}\left(\mathrm{round}_{*}\mathrm{U}\mathcal{B}_{w}(x,\cdot),\,\mathrm{round}_{*}\mathrm{U}\mathcal{B}_{w}(y,\cdot)\right)\\
 & =\min_{z\in\mathbb{Z}^{n}}d_{\mathrm{TV}}\left(\mathrm{round}_{*}\mathrm{U}\mathcal{B}_{w}(x+(2s+1)z,\cdot),\,\mathrm{round}_{*}\mathrm{U}\mathcal{B}_{w}(y,\cdot)\right)\\
 & \le\min_{z\in\mathbb{Z}^{n}}d_{\mathrm{TV}}\left(\mathrm{U}\mathcal{B}_{w}(x+(2s+1)z,\cdot),\,\mathrm{U}\mathcal{B}_{w}(y,\cdot)\right)\\
 & \le\frac{\Psi d_{\mathrm{T},p}(x,y)}{w}.
\end{align*}
The only difference is that instead of $\mathrm{supp}(\mathrm{U}\mathcal{B}_{w}(x,\cdot))\cap\mathrm{supp}(\mathrm{U}\mathcal{B}_{w}(y,\cdot))\neq\emptyset$
$\Rightarrow$ $\Vert x-y\Vert_{p}\le2w$ used in \eqref{eq:metric_pow_pdbd},
we have $\mathrm{supp}(\mathrm{round}_{*}\mathrm{U}\mathcal{B}_{w}(x,\cdot))\cap\mathrm{supp}(\mathrm{round}_{*}\mathrm{U}\mathcal{B}_{w}(y,\cdot))\neq\emptyset$
$\Rightarrow$ $d_{\mathrm{T},p}(x,y)\le4w$ for any $x,y\in[0..2s]^{n}$,
$w>0$. This is due to the fact that if $x,z\in[0..2s]^{n}$,
\begin{align*}
 & z\in\mathrm{supp}(\mathrm{round}_{*}\mathrm{U}\mathcal{B}_{w}(x,\cdot))\\
 & \Rightarrow\mathrm{round}^{-1}(\{z\})\cap\mathcal{B}_{\mathbb{R}^{n},w}(x)\neq\emptyset\\
 & \Rightarrow\exists\tilde{z}\in\mathbb{Z}^{n}:\,\tilde{z}_{i}\equiv z_{i}\;(\mathrm{mod}\;2s+1)\;\mathrm{and}\;\{v\in\mathbb{R}^{n}:\Vert v-\tilde{z}\Vert_{\infty}\le\frac{1}{2}\}\cap\mathcal{B}_{\mathbb{R}^{n},w}(x)\neq\emptyset\\
 & \stackrel{(a)}{\Rightarrow}\exists\tilde{z}\in\mathbb{Z}^{n}:\,\tilde{z}_{i}\equiv z_{i}\;(\mathrm{mod}\;2s+1)\;\mathrm{and}\;\mathcal{B}_{\mathbb{R}^{n},w}(\tilde{z})\cap\mathcal{B}_{\mathbb{R}^{n},w}(x)\neq\emptyset\\
 & \Rightarrow d_{\mathrm{T},p}(x,z)\le2w,
\end{align*}
where (a) is because if $\Vert v-\tilde{z}\Vert_{\infty}\le1/2$,
then $\Vert v-\tilde{z}\Vert_{p}\le\Vert v-\tilde{x}\Vert_{p}$ for
any $\tilde{x}\in\mathbb{Z}^{n}$. Therefore, using $\mathrm{round}_{*}\mathrm{U}\mathcal{B}_{w}(x,\cdot)$
instead of $\mathrm{U}\mathcal{B}_{w}(x,\cdot)$ results in at most
a multiplicative penalty of $2$ on $r_{c}$, i.e., when $\mathcal{X}=[0..2s]^{n}$,
$c(x,y)=(d_{\mathrm{T},p}(x,y))^{q}$, $q\ge1$, it achieves
\begin{align*}
r_{c}(\{X_{\alpha}\}_{\alpha}) & <57.32n^{\max\{1/p,\,1-1/p\}}\gamma^{q-1}\log\left(\gamma/n^{\max\{1/p,\,1-1/p\}}+1\right)\\
 & =57.32n^{(q-1)/p+\max\{1/p,\,1-1/p\}}s^{q-1}\log\left(s/n^{\max\{0,\,1-2/p\}}+1\right),
\end{align*}
where $\gamma=n^{1/p}s$. This coupling $\{X_{\alpha}\}_{\alpha}$
can be computed using a modified version of Algorithm \ref{alg:lsh}
given below, with $\mathcal{X}=[0..2s]^{n}$ (the discrete torus),
$\eta:=5/(3(1-1/\log s))$, $i_{0}:=\lfloor-\eta^{-1}\log(n^{1/p}s)\rfloor-1$,
$i_{1}:=1$.

\begin{algorithm}[H]
\textbf{$\;\;\;\;$Input:} $\varTheta\sim\mathrm{Unif}[0,1]$, $\{V_{i,x}\}_{i\in[i_{0}..i_{1}],x\in\mathcal{X}}\stackrel{iid}{\sim}\mathrm{Exp}(1)$,
$P\in\mathcal{P}(\mathcal{X})$

\textbf{$\;\;\;\;$Output:} $h_{\{V_{x}\}_{x},\varTheta}(P)\in\mathcal{X}$

\smallskip{}

\begin{algorithmic}

\State{$\tilde{p}_{x}\leftarrow P(x)$ for $x\in\mathcal{X}$}

\State{$i\leftarrow i_{0}$}

\While{$|\{x\in\mathcal{X}:\,\tilde{p}_{x}>0\}|>1$}

\State{$w\leftarrow e^{-\eta(i+\varTheta)}$}

\State{$\hat{p}_{x}\leftarrow\sum_{y\in\mathcal{X}}\mathrm{round}_{*}\mathrm{U}\mathcal{B}_{w}(y,\{x\})\tilde{p}_{y}$
for $x\in\mathcal{X}$}

\State{$z\leftarrow\arg\min_{x}V_{i,x}/\hat{p}_{x}$}

\State{$\tilde{p}_{x}\leftarrow\mathrm{round}_{*}\mathrm{U}\mathcal{B}_{w}(x,\{z\})\tilde{p}_{x}$
for $x\in\mathcal{X}$}

\State{$\{\tilde{p}_{x}\}_{x\in\mathcal{X}}\leftarrow\{\tilde{p}_{x}/\sum_{y\in\mathcal{X}}\tilde{p}_{y}\}_{x\in\mathcal{X}}$}

\State{$i\leftarrow i+1$}

\EndWhile

\Return{$x$ such that $\tilde{p}_{x}>0$}

\end{algorithmic}

\caption{\label{alg:lsh_grid}Locality sensitive hash function for the discrete
torus.}
\end{algorithm}

The computation of $\hat{p}_{x}$ in the algorithm is a convolution,
and thus can be performed using fast Fourier transform ($O(2^{n}ns^{n}\log s)$
time for each computation).\footnote{The convolution kernel $\{\mathrm{round}_{*}\mathrm{U}\mathcal{B}_{w}(0,\{x\})\}_{x\in\mathcal{X}}$
involves the volume $\lambda(\mathcal{B}_{\mathbb{R}^{n},w}(0)\cap\mathrm{round}^{-1}(\{x\}))$
of the intersection of a ball and a hypercube, which can be computed
or estimated using various methods, for example, using Fourier series
as in \cite{rousseau1997intersection}.} The number of iterations is $i_{1}-i_{0}+1=O(\log(n^{1/p}s))=O(\log n+\log s)$.
Hence, the time complexity of the algorithm is 
\[
O\left((2^{n}ns^{n}\log s)(\log n+\log s)\right)\le O(2^{n}|\mathcal{X}|\log^{2}|\mathcal{X}|).
\]
\end{rem}

\[
\]

\medskip{}

\subsection{Ultrametric Cost\label{subsec:ultra}}

We then prove Theorem \ref{thm:ultrametric} by applying Theorem \ref{thm:metric_pow}
to the case where $c$ is an ultrametric (i.e., $c(x,z)\le\max\{c(x,y),c(y,z)\}$
for any $x,y,z$).
\begin{proof}
[Proof of Theorem \ref{thm:ultrametric}]Since $(\mathcal{X},c)$
is a separable metric space, let $\{z_{i}\}_{i\in\mathbb{N}}$, $z_{i}\in\mathcal{X}$
be a countable dense subset of $\mathcal{X}$. Define $\mu:=\sum_{i=1}^{\infty}2^{-i}\delta_{z_{i}}$.
Let $q>0$, and $d(x,y)=(c(x,y))^{1/q}$, $\mu=\lambda$. It is straightforward
to check that $d$ is also an ultrametric that metrizes $\mathcal{X}$.
We have $\mathcal{B}_{d,w}(x)\cap\{z_{i}:i\in\mathbb{N}\}\neq\emptyset$,
and hence $0<\mu(\mathcal{B}_{d,w}(x))<\infty$, for any $x\in\mathcal{X}$,
$w>0$. Since $d$ is an ultrametric, $\mathcal{B}_{d,w}(x)=\mathcal{B}_{d,w}(y)$
if $d(x,y)\le w$, and hence \eqref{eq:metric_pow_delta_dist} is
satisfied for $\Psi=1$. The condition \eqref{eq:metric_pow_limsup}
is satisfied since $d(x,y)\le w$ implies $\mathcal{B}_{d,w}(x)=\mathcal{B}_{d,w}(y)$.
Applying Theorem \ref{thm:metric_pow} (using the constant $7.555$
in \eqref{eq:metric_pow_betterconst}) on $d$, $q$, $\mu$, we have
\[
r_{c}^{*}(\mathcal{P}(\mathcal{X}))\le7.555\cdot\frac{(2.47)^{q}}{1-q}.
\]
Letting $q\to0$, we have $r_{c}^{*}(\mathcal{P}(\mathcal{X}))<7.56$.
\end{proof}
\medskip{}
We then prove Theorem \ref{thm:metric}, which concerns the case where
$\mathcal{X}$ is finite, $c(x,y)=(d(x,y))^{q}$ is a power of a metric
$d$, and $q>0$.
\begin{proof}
[Proof of Theorem \ref{thm:metric}]Let $q>0$. Define an ultrametric
$\tilde{d}$ over $\mathcal{X}$ by
\begin{align*}
\tilde{d}(x,y) & :=\left(\inf\left\{ \max_{j\in[0..k-1]}d(z_{j},z_{j+1}):\,k\in\mathbb{N},\,z_{0},\ldots,z_{k}\in\mathcal{X},\,z_{0}=x,\,z_{k}=y\right\} \right)^{q}.
\end{align*}
It is straightforward to check that $\tilde{d}$ is an ultrametric,
and $\tilde{d}(x,y)\le(d(x,y))^{q}\le(|\mathcal{X}|-1)^{q}\tilde{d}(x,y)$.
Applying Theorem \ref{thm:ultrametric} on $\tilde{d}$, we have $r_{\tilde{d}}^{*}(\mathcal{P}(\mathcal{X}))<7.56$.
The result follows from the ratio bound in \eqref{prop:rc_prop_misc}.
\end{proof}
\medskip{}

\subsection{Riemannian Manifolds\label{subsec:manifold}}

In this subsection, we consider the case $\mathcal{X}=\mathcal{M}$,
where $\mathcal{M}$ is a connected smooth complete real $n$-dimensional
Riemannian manifold. Let $d_{\mathcal{M}}$ be the intrinsic distance
on the manifold $\mathcal{M}$, and write $\mathcal{B}_{w}(x)=\mathcal{B}_{d_{\mathcal{M}},w}(x)$.
Let $\mu$ be the measure over $\mathcal{M}$ induced by the Riemannian
volume form.

We now prove Theorem \ref{thm:ricci} for the case where the Ricci
curvature is bounded below.
\begin{proof}
[Proof of Theorem \ref{thm:ricci}] We first consider the case where
$\mathrm{Ric}\ge0$. By the Bishop-Gromov inequality \cite{petersen2016riemannian},
for any $x\in\mathcal{M}$,
\[
w\mapsto\frac{\mu(\mathcal{B}_{w}(x))}{w^{n}\mathrm{V}_{n,2}}
\]
is a non-increasing function of $w\in\mathbb{R}_{>0}$. The result
follows from Corollary \ref{cor:metric_pow_c}.

We then consider the case where $\mathrm{Ric}_{\mathcal{M}}\ge(n-1)K$,
$K<0$. Let $D:=\mathrm{diam}(\mathcal{M})$. By the Bishop-Gromov
inequality, for any $x\in\mathcal{M}$,
\begin{equation}
w\mapsto\frac{\mu(\mathcal{B}_{w}(x))}{(-K)^{-n/2}\breve{\mathrm{V}}_{n}(w\sqrt{-K})}\label{eq:ricci_bg}
\end{equation}
is a non-increasing function of $w\in\mathbb{R}_{>0}$, where
\[
\breve{\mathrm{V}}_{n}(z):=n\mathrm{V}_{n,2}\int_{0}^{z}(\sinh t)^{n-1}\mathrm{d}t
\]
is the volume of the geodesic ball of radius $z$ in the $n$-dimensional
hyperbolic space of Ricci curvature $-(n-1)$ (denote its derivative
as $\breve{\mathrm{V}}_{n}'(z)$). Note that the denominator in \eqref{eq:ricci_bg}
is the volume of the ball of radius $w$ in the hyperbolic space of
Ricci curvature $(n-1)K$. Let 
\[
\Psi:=D\sqrt{-K}\frac{\breve{\mathrm{V}}_{n}'(D\sqrt{-K})}{\breve{\mathrm{V}}_{n}(D\sqrt{-K})}.
\]
When $0<w\le\mathrm{diam}(\mathcal{M})$,
\begin{align*}
 & \frac{\mathrm{d}}{\mathrm{d}w}\log\left(w^{-\Psi}\breve{\mathrm{V}}_{n}(w\sqrt{-K})\right)\\
 & =\frac{\sqrt{-K}\breve{\mathrm{V}}_{n}'(w\sqrt{-K})}{\breve{\mathrm{V}}_{n}(w\sqrt{-K})}-\frac{\Psi}{w}\\
 & \stackrel{(a)}{\le}\frac{\sqrt{-K}\breve{\mathrm{V}}_{n}'(w\sqrt{-K})}{\breve{\mathrm{V}}_{n}(w\sqrt{-K})}-\frac{1}{w}\frac{w\sqrt{-K}\breve{\mathrm{V}}_{n}'(w\sqrt{-K})}{\breve{\mathrm{V}}_{n}(w\sqrt{-K})}\\
 & \le0,
\end{align*}
where (a) is because $t\breve{\mathrm{V}}_{n}'(t)/\breve{\mathrm{V}}_{n}(t)$
is non-decreasing (see Appendix \ref{sec:hyperbolic_inc} for the
proof). Hence $w^{-\Psi}\breve{\mathrm{V}}_{n}(w)$ is non-increasing
for $0<w\le D$. Multiplying this with \eqref{eq:ricci_bg}, $w^{-\Psi}\mu(\mathcal{B}_{w}(x))$
is non-increasing for $0<w\le D$. It is obviously non-increasing
when $w\ge D$ (since $\mu(\mathcal{B}_{w}(x))=\mu(\mathcal{M})$).
By Corollary \ref{cor:metric_pow_c},
\[
r_{c}^{*}(\mathcal{P}(\mathcal{M}))<\frac{7.56}{1-q}\left(2.47D\sqrt{-K}\frac{\breve{\mathrm{V}}_{n}'(D\sqrt{-K})}{\breve{\mathrm{V}}_{n}(D\sqrt{-K})}\right)^{q}.
\]
To further bound the right hand side, if $n\ge3$, for any $z>0$,
\begin{align*}
 & \frac{z\breve{\mathrm{V}}_{n}'(z)}{\breve{\mathrm{V}}_{n}(z)}\\
 & =\frac{z(\sinh z)^{n-1}}{\int_{0}^{z}(\sinh t)^{n-1}\mathrm{d}t}\\
 & \le\frac{z((e^{z}-e^{-z})/2)^{n-1}}{\int_{z(1-(n-1)^{-1})}^{z}((e^{t}-e^{t-2z(1-(n-1)^{-1})})/2)^{n-1}\mathrm{d}t}\\
 & =\frac{ze^{z(n-1)}(1-e^{-2z})^{n-1}}{(n-1)^{-1}(e^{z(n-1)}-e^{z(1-(n-1)^{-1})(n-1)})(1-e^{-2z(1-(n-1)^{-1})})^{n-1}}\\
 & =\frac{(n-1)z}{1-e^{-z}}\left(\frac{1-e^{-2z}}{1-e^{-2z(1-(n-1)^{-1})}}\right)^{n-1}\\
 & \stackrel{(a)}{\le}(n-1)(z+1)\left(\frac{1}{1-(n-1)^{-1}}\right)^{n-1}\\
 & \stackrel{(b)}{\le}en(z+1),
\end{align*}
where (a) is because $z/(1-e^{-z})\le z+1$, and $(1-y)/(1-y^{a})\le1/a$
for any $0\le y<1$, $a\le1$ (which can be shown by $1-y^{a}-a(1-y)\ge0$
by convexity), and (b) is because $(y/(y+1))(1-1/y)^{-y}\le e$ for
$y\ge2$. If $n=2$, then
\begin{align*}
\frac{z\breve{\mathrm{V}}_{2}'(z)}{\breve{\mathrm{V}}_{2}(z)} & =\frac{z\sinh z}{\int_{0}^{z}\sinh t\mathrm{d}t}\\
 & =\frac{z\sinh z}{\cosh z-1}\\
 & \le z+2\\
 & \le en(z+1).
\end{align*}
Therefore $z\breve{\mathrm{V}}_{n}'(z)/\breve{\mathrm{V}}_{n}(z)\le en(z+1)$
for any $n\ge1$, $z>0$ (the case $n=1$ is obvious). Hence,
\[
r_{c}^{*}(\mathcal{P}(\mathcal{M}))<\frac{7.56}{1-q}\left(6.72n(D\sqrt{-K}+1)\right)^{q}.
\]
\end{proof}
\medskip{}

Another approach is to embed $\mathcal{M}$ into a Euclidean space.
\begin{cor}
\label{cor:dyadic_manifold}Let $c(x,y)=(d_{\mathcal{M}}(x,y))^{q}$,
$0<q<1$. If $\mathcal{M}$ can be smoothly (not necessarily isometrically)
embedded in $\mathbb{R}^{m}$ by the embedding function $f:\mathcal{M}\to\mathbb{R}^{m}$,
then we have
\begin{equation}
r_{c}^{*}(\mathcal{P}(\mathcal{M}))<\frac{7.56}{1-q}\left(\frac{2.47\mathrm{V}_{m-1,2}}{\mathrm{V}_{m,2}}\cdot\frac{\sup_{x\neq y\in\mathcal{M}}(d_{\mathcal{M}}(x,y)/\Vert f(x)-f(y)\Vert_{2})}{\inf_{x\neq y\in\mathcal{M}}(d_{\mathcal{M}}(x,y)/\Vert f(x)-f(y)\Vert_{2})}\right)^{q}.\label{eq:dyadic_manifold_bd}
\end{equation}
This has the following consequences:
\begin{itemize}
\item For the $n$-sphere $\mathcal{M}=\{x\in\mathbb{R}^{n+1}:\,\Vert x\Vert_{2}=1\}$,
\[
r_{c}^{*}(\mathcal{P}(\mathcal{M}))<\frac{20.27}{1-q}n^{q/2}.
\]
\item For the $n$-torus $\mathcal{M}=\{x\in\mathbb{R}^{2n}:\,x_{2i-1}^{2}+x_{2i}^{2}=1\,\forall i\in[1..n]\}$,
\[
r_{c}^{*}(\mathcal{P}(\mathcal{M}))<\frac{20.27}{1-q}n^{q/2}.
\]
\end{itemize}
\end{cor}

\begin{proof}
[Proof of Corollary \ref{cor:dyadic_manifold}] Note that \eqref{eq:dyadic_manifold_bd}
is a direct consequence of Theorem \ref{thm:rn_rc_ub} and the ratio
bound in Proposition \ref{prop:rc_prop_misc}. When $\mathcal{M}=\{x\in\mathbb{R}^{n+1}:\,\Vert x\Vert_{2}=1\}$,
letting $f(x)=x$, we have $\sup_{x\neq y}(d_{\mathcal{M}}(x,y)/\Vert x-y\Vert_{2})=\pi/2$,
and $\inf_{x\neq y}(d_{\mathcal{M}}(x,y)/\Vert x-y\Vert_{2})=1$.
Hence,
\begin{align*}
r_{c}^{*}(\mathcal{P}(\mathcal{M})) & <\frac{7.56}{1-q}\left(\frac{2.47(\pi/2)\mathrm{V}_{n,2}}{\mathrm{V}_{n+1,2}}\right)^{q}\\
 & \stackrel{(a)}{\le}\frac{7.56}{1-q}\left(2.47(\pi/4)\frac{(1+n/2)^{1/2}}{\mathit{\Gamma}(1+1/2)}\right)^{q}\\
 & \le\frac{7.56}{1-q}\left(2.47(\pi/4)\frac{(3n/2)^{1/2}}{\mathit{\Gamma}(1+1/2)}\right)^{q}\\
 & \le\frac{20.27n^{q/2}}{1-q},
\end{align*}
where (a) is by \eqref{eq:Bnp_ratio}. 

When $\mathcal{M}=\{x\in\mathbb{R}^{2n}:\,x_{2i-1}^{2}+x_{2i}^{2}=1\,\forall i\in[1..n]\}$,
letting $f(x)=x$, we have $\sup_{x\neq y}(d_{\mathcal{M}}(x,y)/\Vert x-y\Vert_{2})=\pi/2$,
and $\inf_{x\neq y}(d_{\mathcal{M}}(x,y)/\Vert x-y\Vert_{2})=1$.
Hence,
\begin{align*}
r_{c}^{*}(\mathcal{P}(\mathcal{M})) & <\frac{7.56}{1-q}\left(\frac{2.47(\pi/2)\mathrm{V}_{2n-1,2}}{\mathrm{V}_{2n,2}}\right)^{q}\\
 & \stackrel{(a)}{\le}\frac{7.56}{1-q}\left(2.47(\pi/4)\frac{(1+(2n-1)/2)^{1/2}}{\mathit{\Gamma}(1+1/2)}\right)^{q}\\
 & \le\frac{7.56}{1-q}\left(2.47(\pi/4)\frac{(3n/2)^{1/2}}{\mathit{\Gamma}(1+1/2)}\right)^{q}\\
 & \le\frac{20.27n^{q/2}}{1-q},
\end{align*}
where (a) is by \eqref{eq:Bnp_ratio}.
\end{proof}
\medskip{}

\section{Lower Bounds on $r_{c}^{*}$\label{sec:lbs}}

In this section, we prove Proposition \ref{prop:rn_rc_lb}, which
is separated into Proposition \ref{prop:rn_rc_lb_1} and \ref{prop:rn_rc_lb_2}.
\begin{prop}
\label{prop:rn_rc_lb_1}For any Polish space $\mathcal{X}$ and symmetric
cost function $c$,
\[
r_{c}^{*}(\mathcal{P}(\mathcal{X}))\ge\sup_{x_{1},\ldots,x_{k}\in\mathcal{X}}\frac{2(k-1)\min_{1\le i<j\le k}c(x_{i},x_{j})}{\sum_{i=1}^{k}c(x_{i},x_{i+1})},
\]
where the supremum is over $k\in\mathbb{N}$ and sequences $x_{1},\ldots,x_{k}\in\mathcal{X}$,
and we let $x_{k+1}:=x_{1}$. As a result, for any $p\in\mathbb{R}_{\ge1}\cup\{\infty\}$,
$q>0$, $n\in\mathbb{N}$, $s,s_{1},\ldots,s_{n}\in\mathbb{N}$, $S\subseteq\mathbb{R}^{n}$
with $\lambda(S)>0$,

\[
\begin{array}{rll}
r_{c}^{*}(\mathcal{P}([0..s])) & \!\!\!\!\!\ge\frac{2}{1+s^{q-1}} & \text{for}\;c(x,y)=|x-y|^{q},\;q<1,\\
r_{c}^{*}(\mathcal{P}(\mathbb{Z})) & \!\!\!\!\!\ge2 & \text{for}\;c(x,y)=|x-y|^{q},\;q<1,\;\dagger\\
r_{c}^{*}(\mathcal{P}(\mathcal{X})) & \!\!\!\!\!\ge2(1-(2\lfloor|\mathcal{X}|/2\rfloor)^{-1}) & \text{for}\;\mathcal{X}=[0..s_{1}]\times\cdots\times[0..s_{n}],\,n\ge2,\,c(x,y)=\Vert x-y\Vert_{p}^{q},\\
r_{c}^{*}(\mathcal{P}(\mathbb{Z}^{n})) & \!\!\!\!\!\ge2 & \text{for}\;n\ge2,\,c(x,y)=\Vert x-y\Vert_{p}^{q},\;\dagger\\
r_{c}^{*}(\mathcal{P}(\mathcal{X})) & \!\!\!\!\!\ge2(1-|\mathcal{X}|^{-1}) & \text{for}\;\mathcal{X}=[0..s_{1}]\times\cdots\times[0..s_{n}],\,n\ge1,\,c(x,y)=(d_{\mathrm{H}}(x,y))^{q},\\
r_{c}^{*}(\mathcal{P}(\mathbb{Z}^{n})) & \!\!\!\!\!\ge2 & \text{for}\;n\ge1,\,c(x,y)=(d_{\mathrm{H}}(x,y))^{q}.
\end{array}
\]
The lines marked with $\dagger$ are also true for the cases given
in Remark \ref{rem:rc_lb_cases}.

\smallskip{}
\end{prop}

\begin{proof}
[Proof of Proposition \ref{prop:rn_rc_lb_1}] Fix a sequence $x_{1},\ldots,x_{k}\in\mathcal{X}$
and let $x_{k+1}:=x_{1}$. Let $P_{i}:=(k-1)^{-1}\sum_{j\in\{1,\ldots,k\}\backslash\{i\}}\delta_{x_{i}}$
for $i=1,\ldots,k$, and let $P_{k+1}:=P_{1}$. We have 
\[
C_{c}^{*}(P_{i+1},P_{i})\le(k-1)^{-1}c(x_{i},x_{i+1}).
\]
Fix any coupling $\{X_{i}\}_{i\in[1..k]}$ of $\{P_{i}\}_{i\in[1..k]}$
and let $r:=r_{c}(\{X_{i}\}_{i})$. Let $X_{k+1}:=X_{1}$, then
\begin{align*}
 & \sum_{i=1}^{k}\mathbf{P}(X_{i+1}\neq X_{i})\\
 & =\mathbf{E}\Big[\sum_{i=1}^{k}\mathbf{1}\{X_{i+1}\neq X_{i}\}\Big]\\
 & \ge2,
\end{align*}
where the last inequality is because, with probability 1, $X_{1},\ldots,X_{k}$
are not all equal (for any $x$, there is a distribution $P_{i}$
not supported at $x$), so there are at least two pairs among the
cycle of pairs $(X_{1},X_{2}),\ldots,(X_{k-1},X_{k}),(X_{k},X_{1})$
with $X_{i+1}\neq X_{i}$. Also,
\begin{align*}
 & \mathbf{P}(X_{i+1}\neq X_{i})\min_{1\le i<j\le k}c(x_{i},x_{j})\\
 & \le\mathbf{E}\left[c(X_{i+1},X_{i})\right]\\
 & \le rC_{c}^{*}(P_{i+1},P_{i})\\
 & \le r(k-1)^{-1}c(x_{i},x_{i+1}).
\end{align*}
Hence,
\[
r\ge\frac{2(k-1)\min_{1\le i<j\le k}c(x_{i},x_{j})}{\sum_{i=1}^{k}c(x_{i},x_{i+1})}.
\]

For $c(x,y)=|x-y|^{q}$, $q<1$ over $[0..s]$, let $k:=s+1$, $x_{i}:=i-1$
for $1\le i\le s+1$. Then
\begin{align*}
r_{c}^{*}(\mathcal{P}([0..s])) & \ge\frac{2(k-1)\min_{1\le i<j\le k}c(x_{i},x_{j})}{\sum_{i=1}^{k}c(x_{i},x_{i+1})}\\
 & =\frac{2s}{s+s^{q}}\\
 & =\frac{2}{1+s^{q-1}}.
\end{align*}
The bound for $r_{c}^{*}(\mathcal{P}(\mathbb{Z}))$ follows from letting
$s\to\infty$. The other cases in Remark \ref{rem:rc_lb_cases} follow
from the same arguments as in Appendix \ref{subsec:bd_ball_cont}.

For $c(x,y)=\Vert x-y\Vert_{p}^{q}$ over $\mathcal{X}:=[1..s_{1}]\times\cdots\times[1..s_{n}]$,
$n\ge2$, $s_{i}\ge2$ (we consider this instead of $\mathcal{X}=[0..s_{1}]\times\cdots\times[0..s_{n}]$
due to notational simplicity), we will construct $x_{1},\ldots,x_{k}$
as a vertex-disjoint cycle of the grid graph $G_{s_{1},\ldots,s_{n}}$
(the grid graph $G_{a_{1},\ldots,a_{n}}$ is the graph with vertex
set $[1..a_{1}]\times\cdots\times[1..a_{n}]$, and two vertices $x,y$
connected if $\Vert x-y\Vert_{1}=1$). Note that $G_{a_{1},\ldots,a_{n}}$
has a subgraph that is isomorphic to $G_{a_{1}\cdot a_{2},a_{3},a_{4},\ldots,a_{n}}$
(consider a Hamiltonian path of $G_{a_{1},a_{2}}$, and the Cartesian
graph product of that path with $G_{a_{3},a_{4},\ldots,a_{n}}$).
Applying this fact repeatedly, $G_{s_{1},\ldots,s_{n}}$ has a subgraph
isomorphic to $G_{s_{1},\prod_{i=2}^{n}s_{i}}$. If $\prod_{i=1}^{n}s_{i}$
is even, there exists a vertex-disjoint cycle of length $\prod_{i=1}^{n}s_{i}$
in $G_{s_{1},\prod_{i=2}^{n}s_{i}}$, and hence in $G_{s_{1},\ldots,s_{n}}$
(since $G_{a,b}$ is Hamiltonian if $a$ or $b$ is even \cite[p. 148]{skiena1991graph}).
If $\prod_{i=1}^{n}s_{i}$ is odd, since any longest vertex-disjoint
path between any two vertices of $G_{a,b}$ contains at least $ab-2$
vertices \cite{keshavarz2012linear}, there exists a vertex-disjoint
cycle in $G_{s_{1},\prod_{i=2}^{n}s_{i}}$ with length at least $\prod_{i=1}^{n}s_{i}-2$
(by considering the longest vertex-disjoint path between $(0,0)$
and $(0,1)$). Since $G_{s_{1},\prod_{i=2}^{n}s_{i}}$ is bipartite
\cite[p. 148]{skiena1991graph}, the cycle has even length, and hence
has length $\prod_{i=1}^{n}s_{i}-1$. Therefore we can find a vertex-disjoint
cycle $x_{1},\ldots,x_{k}$ in $G_{s_{1},\ldots,s_{n}}$ with $k=2\lfloor(1/2)\prod_{i=1}^{n}s_{i}\rfloor$,
and
\begin{align*}
r_{c}^{*}(\mathcal{P}(\mathcal{X})) & \ge\frac{2(k-1)\min_{1\le i<j\le k}c(x_{i},x_{j})}{\sum_{i=1}^{k}c(x_{i},x_{i+1})}\\
 & =2\left(1-\left(2\left\lfloor \frac{1}{2}\prod_{i=1}^{n}s_{i}\right\rfloor \right)^{-1}\right).
\end{align*}
The bound for $r_{c}^{*}(\mathcal{P}(\mathbb{Z}^{n}))$, $n\ge2$
follows from considering $[1..s]^{n}$ and letting $s\to\infty$.
The other cases in Remark \ref{rem:rc_lb_cases} follow from the same
arguments as in Appendix \ref{subsec:bd_ball_cont}.

For $c(x,y)=(d_{\mathrm{H}}(x,y))^{q}$ over $\mathcal{X}:=[1..s_{1}]\times\cdots\times[1..s_{n}]$,
$s_{i}\ge2$ (we consider this instead of $\mathcal{X}=[0..s_{1}]\times\cdots\times[0..s_{n}]$
due to notational simplicity), we will construct $x_{1},\ldots,x_{k}$
as a vertex-disjoint cycle of the rook graph $K_{s_{1},\ldots,s_{n}}$
(the rook graph $K_{a_{1},\ldots,a_{n}}$ is the graph with vertex
set $[1..a_{1}]\times\cdots\times[1..a_{n}]$, and two vertices $x,y$
connected if $d_{\mathrm{H}}(x,y)=1$). We will show $K_{s_{1},\ldots,s_{n}}$
is Hamiltonian. Assume $n\ge2$ (otherwise $K_{s_{1}}$ is the complete
graph, which is clearly Hamiltonian). Since $G_{a_{1},\ldots,a_{n}}$
is a subgraph of $K_{a_{1},\ldots,a_{n}}$, $K_{s_{1},\ldots,s_{n}}$
has a subgraph isomorphic to $K_{s_{1}}\square G_{s_{2},\ldots,s_{n}}$
(where $\square$ denotes the Cartesian graph product), and hence
$K_{s_{1},\ldots,s_{n}}$ has a subgraph isomorphic to $K_{s_{1}}\square G_{\prod_{i=2}^{n}s_{i}}$.
If $\prod_{i=2}^{n}s_{i}$ is even, then $K_{s_{1}}\square G_{\prod_{i=2}^{n}s_{i}}$
(and hence $K_{s_{1},\ldots,s_{n}}$) is Hamiltonian since $G_{s_{1},\prod_{i=2}^{n}s_{i}}$
is Hamiltonian. If $\prod_{i=2}^{n}s_{i}$ is odd, let $\tilde{x}_{1},\ldots,\tilde{x}_{\prod_{i=1}^{n}s_{i}-s_{1}}$
be a Hamiltonian cycle of $G_{s_{1},\prod_{i=2}^{n}s_{i}-1}$. The
cycle must contain the edge $((1,\prod_{i=2}^{n}s_{i}-1),\,(2,\prod_{i=2}^{n}s_{i}-1))$
(since the degree of the node $(1,\prod_{i=2}^{n}s_{i}-1)$ in $G_{s_{1},\prod_{i=2}^{n}s_{i}-1}$
is 2). Without loss of generality, assume $\tilde{x}_{\prod_{i=1}^{n}s_{i}-s_{1}}=(1,\prod_{i=2}^{n}s_{i}-1)$,
$\tilde{x}_{1}=(2,\prod_{i=2}^{n}s_{i}-1)$. Let $\tilde{x}_{\prod_{i=1}^{n}s_{i}-s_{1}+1}=(1,\prod_{i=2}^{n}s_{i})$,
$\tilde{x}_{\prod_{i=1}^{n}s_{i}-s_{1}+1+i}=(s+1-i,\prod_{i=2}^{n}s_{i})$
for $i=1,\ldots,s_{1}-1$, then $\tilde{x}_{1},\ldots,\tilde{x}_{\prod_{i=1}^{n}s_{i}}$
is a Hamiltonian cycle of $K_{s_{1}}\square G_{\prod_{i=2}^{n}s_{i}}$,
and hence $K_{s_{1},\ldots,s_{n}}$ is Hamiltonian. Let $x_{1},\ldots,x_{k}$
be a Hamiltonian cycle of $K_{s_{1},\ldots,s_{n}}$,
\begin{align*}
r_{c}^{*}(\mathcal{P}(\mathcal{X})) & \ge\frac{2(k-1)\min_{1\le i<j\le k}c(x_{i},x_{j})}{\sum_{i=1}^{k}c(x_{i},x_{i+1})}\\
 & =2\left(1-\left(\prod_{i=1}^{n}s_{i}\right)^{-1}\right).
\end{align*}
The bound for $r_{c}^{*}(\mathcal{P}(\mathbb{Z}^{n}))$ follows from
considering $[1..s]^{n}$ and letting $s\to\infty$.
\end{proof}
\smallskip{}

\begin{prop}
\label{prop:rn_rc_lb_2}For any Polish space $\mathcal{X}$ and symmetric
cost function $c$,
\[
r_{c}^{*}(\mathcal{P}(\mathcal{X}))\ge\sup\left\{ \left(\sum_{i=1}^{k}\frac{c(x_{i},x_{i+1})+c(x_{i+k},x_{i+k+1})}{c(x_{i},x_{i+k+1})+c(x_{i+k},x_{i+1})-c(x_{i},x_{i+1})-c(x_{i+k},x_{i+k+1})}\right)^{-1}\right\} +1,
\]
where the supremum is over $k\in\mathbb{N}$ and sequences $x_{1},\ldots,x_{2k}\in\mathcal{X}$
satisfying
\begin{equation}
c(x_{i},x_{i+k+1})+c(x_{i+k},x_{i+1})>c(x_{i},x_{i+1})+c(x_{i+k},x_{i+k+1})\label{eq:rn_rc_lb_2_cond}
\end{equation}
for $i=1,\ldots,k$, where we let $x_{2k+1}=x_{1}$. This has the
following consequences:
\begin{enumerate}
\item 
\begin{align}
r_{c}^{*}(\mathcal{P}(\mathcal{X})) & \ge\sup\bigg\{\Big(\frac{c(x_{1},x_{2})}{c(x_{2},x_{3})+c(x_{3},x_{1})-c(x_{1},x_{2})}+\frac{c(x_{2},x_{3})}{c(x_{3},x_{1})+c(x_{1},x_{2})-c(x_{2},x_{3})}\nonumber \\
 & \;\;+\frac{c(x_{3},x_{1})}{c(x_{1},x_{2})+c(x_{2},x_{3})-c(x_{3},x_{1})}\Big)^{-1}\bigg\}+1,\label{eq:rn_rc_lb_tri}
\end{align}
where the supremum is over $x_{1},x_{2},x_{3}\in\mathcal{X}$ where
the denominators in the three fractions above are positive.
\item For any $p\in\mathbb{R}_{\ge1}\cup\{\infty\}$, $q>0$, $n\in\mathbb{Z}_{\ge2}$,
$s\in\mathbb{N}$, for $c(x,y)=\Vert x-y\Vert_{p}^{q}$,
\begin{align*}
r_{c}^{*}(\mathcal{P}([0..s]^{n})) & \ge\frac{(n-1)^{q/p}s^{q}-1}{ns}+1\\
 & \ge(n-1)^{q/p}n^{-1}s^{q-1}.
\end{align*}
As a result, if $q>1$,
\[
r_{c}^{*}(\mathcal{P}(\mathbb{Z}^{n}))=\infty.
\]
This is also true for the other two cases given in Remark \ref{rem:rc_lb_cases}.
\item For any $q>0$, $n\in\mathbb{Z}_{\ge2}$, for $c(x,y)=(d_{\mathrm{H}}(x,y))^{q}$
over $[1..2]^{n}$,
\[
r_{c}^{*}(\mathcal{P}([1..2]^{n}))\ge\frac{(n-1)^{q}-1}{n}+1.
\]
\end{enumerate}
\smallskip{}
\end{prop}

\begin{proof}
[Proof of Proposition \ref{prop:rn_rc_lb_2}] Fix a sequence $x_{1},\ldots,x_{2k}\in\mathcal{X}$
satisfying \eqref{eq:rn_rc_lb_2_cond}, and let $x_{2k+1}:=x_{1}$.
Let $P_{i}:=\frac{1}{2}\delta_{x_{i}}+\frac{1}{2}\delta_{x_{i+k}}$
for $i=1,\ldots,k+1$ (note that $P_{k+1}=P_{1}$). We have
\[
C_{c}^{*}(P_{i},P_{i+1})\le\frac{1}{2}\left(c(x_{i},x_{i+1})+c(x_{i+k},x_{i+k+1})\right).
\]
Fix any coupling $\{X_{i}\}_{i\in[1..k]}$ of $\{P_{i}\}_{i\in[1..k]}$
and let $r:=r_{c}(\{X_{i}\}_{i})$. Since $X_{i}\sim P_{i}=\frac{1}{2}\delta_{x_{i}}+\frac{1}{2}\delta_{x_{i+k}}$,
we can find random variables $Z_{i}\in\{0,1\}$ with uniform marginals
$Z_{i}\sim\mathrm{Unif}\{0,1\}$ such that $X_{i}=x_{i+kZ_{i}}$.
Let $X_{k+1}:=X_{1}$, and $Z_{k+1}:=1-Z_{1}$ (so that $X_{i}=x_{i+kZ_{i}}$
is also satisfied for $i=k+1$). Then
\begin{align*}
 & rC_{c}^{*}(P_{i},P_{i+1})\\
 & \ge\mathbf{E}\left[c(X_{i},X_{i+1})\right]\\
 & =\mathbf{E}\left[c(x_{i+kZ_{i}},\,x_{i+1+kZ_{i+1}})\right]\\
 & =\left(1-\mathbf{P}(Z_{i+1}\neq Z_{i})\right)\frac{1}{2}\left(c(x_{i},x_{i+1})+c(x_{i+k},x_{i+k+1})\right)\\
 & \;\;\;\;+\mathbf{P}(Z_{i+1}\neq Z_{i})\frac{1}{2}\left(c(x_{i},x_{i+k+1})+c(x_{i+k},x_{i+1})\right).
\end{align*}
Hence,
\begin{align*}
\mathbf{P}(Z_{i}\neq Z_{i+1}) & \le\frac{rC_{c}^{*}(P_{i},P_{i+1})-\frac{1}{2}\left(c(x_{i},x_{i+1})+c(x_{i+k},x_{i+k+1})\right)}{\frac{1}{2}\left(c(x_{i},x_{i+k+1})+c(x_{i+k},x_{i+1})-c(x_{i},x_{i+1})-c(x_{i+k},x_{i+k+1})\right)}\\
 & \le\frac{(r-1)\left(c(x_{i},x_{i+1})+c(x_{i+k},x_{i+k+1})\right)}{c(x_{i},x_{i+k+1})+c(x_{i+k},x_{i+1})-c(x_{i},x_{i+1})-c(x_{i+k},x_{i+k+1})}.
\end{align*}
Therefore,
\begin{align*}
 & (r-1)\sum_{i=1}^{k}\frac{c(x_{i},x_{i+1})+c(x_{i+k},x_{i+k+1})}{c(x_{i},x_{i+k+1})+c(x_{i+k},x_{i+1})-c(x_{i},x_{i+1})-c(x_{i+k},x_{i+k+1})}\\
 & \ge\sum_{i=1}^{k}\mathbf{P}(Z_{i}\neq Z_{i+1})\\
 & \ge\mathbf{P}(Z_{k+1}\neq Z_{1})\\
 & =1.
\end{align*}
The result follows.

To show \eqref{eq:rn_rc_lb_tri}, apply the proposition on the sequence
$x_{1},x_{1},x_{2},x_{2},x_{3},x_{3}$.

For $c(x,y)=\Vert x-y\Vert_{p}^{q}$ over $\mathcal{X}:=[0..s_{1}]\times\cdots\times[0..s_{n}]$,
$n\ge2$, $s_{i}\ge1$, let $k:=\sum_{i=1}^{n}s_{i}$. Define $x_{1},\ldots,x_{2k}$
recursively as $x_{1}:=(0,\ldots,0)$, $x_{i}:=\max\{y:\,\Vert x_{i-1}-y\Vert_{1}=1\}$
for $2\le i\le k$ where the maximum is with respect to lexicographical
order, $x_{i}:=(s_{1},\ldots,s_{n})-x_{i-k}$ for $k+1\le i\le2k$.
\begin{align*}
 & r_{c}^{*}(\mathcal{P}(\mathcal{X}))\\
 & \ge\left(\sum_{i=1}^{k}\frac{c(x_{i},x_{i+1})+c(x_{i+k},x_{i+k+1})}{c(x_{i},x_{i+k+1})+c(x_{i+k},x_{i+1})-c(x_{i},x_{i+1})-c(x_{i+k},x_{i+k+1})}\right)^{-1}+1\\
 & =\left(\sum_{i=1}^{k}\frac{2}{c(x_{i},x_{i+k+1})+c(x_{i+k},x_{i+1})-2}\right)^{-1}+1\\
 & \stackrel{(a)}{\ge}\left(k\frac{2}{2(\sum_{i=1}^{n}s_{i}^{p}-\max_{i\in[1:n]}s_{i}^{p})^{q/p}-2}\right)^{-1}+1\\
 & =\Big(\sum_{i=1}^{n}s_{i}\Big)^{-1}\left(\Big(\sum_{i=1}^{n}s_{i}^{p}-\max_{i\in[1:n]}s_{i}^{p}\Big)^{q/p}-1\right)+1,
\end{align*}
where (a) is because by construction, $(x_{i})_{j}\in\{0,s_{j}\}$
for all $j\in[1:n]$ where $(x_{i})_{j}=(x_{i+1})_{j}$, and hence
$|(x_{i+k})_{j}-(x_{i})_{j}|=|s_{j}-2(x_{i})_{j}|=s_{j}$ and $|(x_{i+k+1})_{j}-(x_{i})_{j}|=|s_{j}-(x_{i+1})_{j}-(x_{i})_{j}|=s_{j}$
for all $j\in[1:n]$ where $(x_{i})_{j}=(x_{i+1})_{j}$ (there are
$n-1$ such $j$'s). In particular, when $\mathcal{X}=[0..s]^{n}$,
\[
r_{c}^{*}(\mathcal{P}(\mathcal{X}))\ge\frac{(n-1)^{q/p}s^{q}-1}{ns}+1.
\]
The case for $\mathcal{X}=\mathbb{Z}^{n}$, $q>1$ follows from letting
$s\to\infty$. The other cases in Remark \ref{rem:rc_lb_cases} follow
from the same arguments as in Appendix \ref{subsec:bd_ball_cont}.

For $c(x,y)=(d_{\mathrm{H}}(x,y))^{q}$ over $[1..2]^{n}$, since
$d_{\mathrm{H}}(x,y)=\Vert x-y\Vert_{1}$ over $[1..2]^{n}$, we have
\[
r_{c}^{*}([1..2]^{n})\ge\frac{(n-1)^{q}-1}{n}+1.
\]
\end{proof}
\medskip{}

\section{Miscellaneous Properties of $r_{c}^{*}$\label{sec:props}}

In this section, we list some properties of $r_{c}^{*}$.
\begin{prop}
\label{prop:rc_prop_misc}Let $\mathcal{X}$ be a Polish space, $\{P_{\alpha}\}_{\alpha\in\mathcal{A}}$
be a collection of probability distributions over $\mathcal{X}$,
and $c,c_{1},c_{2}$ be symmetric cost functions. For brevity, we
write
\[
\sup\frac{c_{1}}{c_{2}}:=\sup_{x,y\in\mathcal{X}:\,c_{1}(x,y)>0}\frac{c_{1}(x,y)}{c_{2}(x,y)},
\]
where we treat $a/0=\infty$ for $a>0$, then we have the following
properties:
\begin{enumerate}
\item (Ratio bound) 
\begin{equation}
r_{c_{1}}^{*}(\{P_{\alpha}\}_{\alpha})\le r_{c_{2}}^{*}(\{P_{\alpha}\}_{\alpha})\left(\sup\frac{c_{1}}{c_{2}}\right)\left(\sup\frac{c_{2}}{c_{1}}\right)\label{eq:rc_prop_misc_ratio}
\end{equation}
if $\sup c_{1}/c_{2},\sup c_{2}/c_{1}<\infty$. As a result, if $\mathcal{X}$
is finite or countably infinite, by considering $c_{2}(x,y)=\mathbf{1}\{x\neq y\}$,
\[
r_{c}^{*}(\{P_{\alpha}\}_{\alpha})\le\frac{2\sup_{x,y\in\mathcal{X}}c(x,y)}{\inf_{x,y\in\mathcal{X},\,x\neq y}c(x,y)},
\]
and if $|\mathcal{X}|<\infty$ and $c(x,y)>0$ for all $x\neq y$,
then $r_{c}^{*}(\{P_{\alpha}\}_{\alpha})<\infty$.\medskip{}
\item (Continuity) If $|\mathcal{X}|<\infty$, then $r_{c}^{*}(\{P_{\alpha}\}_{\alpha})$
is a continuous function of $c$ over the $|\mathcal{X}|(|\mathcal{X}|-1)/2$-dimensional
space of symmetric cost functions $c$ satisfying $c(x,y)>0$ for
all $x\neq y$. \footnote{We require $c(x,y)>0$ for all $x\neq y$ here since it is possible
that $r_{c}^{*}(\{P_{\alpha}\})=\infty$ if this condition is not
satisfied (e.g. for $\mathcal{X}=[1..4]$, $c(1,3)=c(2,4)=1$, $c(x,y)=0$
otherwise, we have $r_{c}^{*}(\mathcal{P}([1..4]))=\infty$).}\medskip{}
\item (Bound by cardinality) Consider $\sup_{c}r_{c}^{*}(\mathcal{P}([1..s]))$
for $s\in\mathbb{Z}_{\ge2}$, where the supremum is over symmetric
cost functions $c$. We have
\begin{align*}
\sup_{c}r_{c}^{*}(\mathcal{P}([1..2])) & =1,\\
\sup_{c}r_{c}^{*}(\mathcal{P}([1..3])) & =4/3,\\
\sup_{c}r_{c}^{*}(\mathcal{P}([1..4])) & =\infty.
\end{align*}
This shows that it is impossible to bound $r_{c}^{*}(\mathcal{P}(\mathcal{X}))$
by $|\mathcal{X}|$ alone (as long as $|\mathcal{X}|\ge4$) without
taking the properties of $c$ into account (e.g. as in Theorem \ref{thm:metric_log}
and Theorem \ref{thm:metric}). An example where $r_{c}^{*}(\mathcal{P}([1..4]))=\infty$
is for the symmetric cost function $c(1,3)=c(2,4)=1$, $c(x,y)=0$
otherwise.\medskip{}
\item (Range) If $|\mathcal{X}|\ge4$, then for any $r\in\mathbb{R}_{\ge1}\cup\{\infty\}$,
there exists a symmetric cost function $c$ where $r_{c}^{*}(\mathcal{P}(\mathcal{X}))=r$.\medskip{}
\item (Necessary condition for $r_{c}^{*}=1$) If $r_{c}^{*}(\mathcal{P}(\mathcal{X}))=1$,
then for any $x_{1},x_{2},x_{3}\in\mathcal{X}$, the following ``anti-triangle
inequality'' is satisfied:
\begin{align*}
c(x_{1},x_{2}) & \ge c(x_{2},x_{3})+c(x_{3},x_{1})\\
\mathrm{or}\;c(x_{2},x_{3}) & \ge c(x_{3},x_{1})+c(x_{1},x_{2})\\
\mathrm{or}\;c(x_{3},x_{1}) & \ge c(x_{1},x_{2})+c(x_{2},x_{3}).
\end{align*}
Also, for any $x_{0},x_{1},x_{2},x_{3}\in\mathcal{X}$,
\begin{align}
c(x_{1},x_{2}) & \ge c(x_{0},x_{1})+c(x_{0},x_{2})\nonumber \\
\mathrm{or}\;c(x_{2},x_{3}) & \ge c(x_{0},x_{2})+c(x_{0},x_{3})\nonumber \\
\mathrm{or}\;c(x_{0},x_{1}) & \ge c(x_{0},x_{3})+c(x_{1},x_{3})\nonumber \\
\mathrm{or}\;c(x_{0},x_{3}) & \ge c(x_{0},x_{1})+c(x_{1},x_{3}).\label{eq:ati4}
\end{align}
\item (Metric cost with $r_{c}^{*}=1$) If $(\mathcal{X},c)$ is a metric
space, then $r_{c}^{*}(\mathcal{P}(\mathcal{X}))=1$ if and only if
$(\mathcal{X},c)$ can be isometrically embedded into the metric space
$(\mathbb{R},\,(x,y)\mapsto|x-y|)$ (i.e., there exists $g:\mathcal{X}\to\mathbb{R}$
such that $c(x,y)=|g(x)-g(y)|$ for any $x,y\in\mathcal{X}$).\medskip{}
\end{enumerate}
\end{prop}

\begin{proof}
[Proof of Proposition \ref{prop:rc_prop_misc}] $\,$
\begin{enumerate}
\item Assume $\sup c_{1}/c_{2},\sup c_{2}/c_{1}<\infty$. Fix any $\{X_{\alpha}\}_{\alpha}\in\Gamma_{\lambda}(\{P_{\alpha}\}_{\alpha})$.
For any $\alpha,\beta\in\mathcal{A}$,
\[
\mathbf{E}[c_{1}(X_{\alpha},X_{\beta})]\le\mathbf{E}[c_{2}(X_{\alpha},X_{\beta})]\sup\frac{c_{1}}{c_{2}},
\]
\[
C_{c_{1}}^{*}(P_{\alpha},P_{\beta})\sup\frac{c_{2}}{c_{1}}\ge C_{c_{2}}^{*}(P_{\alpha},P_{\beta}),
\]
and hence
\begin{align*}
 & r_{c_{1}}(\{X_{\alpha}\}_{\alpha})\\
 & =\inf\left\{ r\ge1:\,\mathbf{E}[c_{1}(X_{\alpha},X_{\beta})]\le rC_{c_{1}}^{*}(P_{\alpha},P_{\beta})\,\forall\alpha,\beta\in\mathcal{A}\right\} \\
 & \le\inf\left\{ r\ge1:\,\mathbf{E}[c_{2}(X_{\alpha},X_{\beta})]\sup\frac{c_{1}}{c_{2}}\le rC_{c_{2}}^{*}(P_{\alpha},P_{\beta})/\sup\frac{c_{2}}{c_{1}}\,\forall\alpha,\beta\in\mathcal{A}\right\} \\
 & =r_{c_{2}}(\{X_{\alpha}\}_{\alpha})\left(\sup\frac{c_{1}}{c_{2}}\right)\left(\sup\frac{c_{2}}{c_{1}}\right).
\end{align*}
The result follows.\medskip{}
\item This is a consequence of the ratio bound, since if $c_{i}\to c$,
then $\sup\frac{c_{i}}{c}\to1$ and $\sup\frac{c}{c_{i}}\to1$.\medskip{}
\item By Proposition \eqref{prop:dpc_prop}, $\sup_{c}r_{c}^{*}(\mathcal{P}([1..2]))=1$
since $c$ is a constant multiple of $\mathbf{1}_{\neq}$. For $\mathcal{X}=[1..3]$,
let $\{P_{\alpha}\}_{\alpha}=\mathcal{P}(\mathcal{X})$. We have $\sup_{c}r_{c}^{*}(\mathcal{P}([1..3]))\ge r_{\mathbf{1}_{\neq}}^{*}(\mathcal{P}([1..3]))\ge4/3$
by Theorem \eqref{thm:rd_bd}. For the upper bound, consider any $c$
and assume $c(1,2)\le c(2,3)\le c(1,3)$ without loss of generality.
If $c(1,2)=0$, then let $g:[1..3]\to\mathbb{R}$, $g(1)=0$, $g(2)=c(1,3)-c(2,3)$,
$g(3)=2c(1,3)-c(2,3)$, and let $c':\mathbb{R}^{2}\to\mathbb{R}$,
$c'(x,y)=\max\{|x-y|-g(2),0\}$ (a convex function of $|x-y|$), then
$c(x,y)=c'(g(x),g(y))$, and hence $r_{c}^{*}(\mathcal{P}([1..3]))\le r_{c'}^{*}(\mathcal{P}(\mathbb{R}))=1$
by Proposition \eqref{prop:r_rc}. If $c(1,2)>0$ and $c(1,3)\ge c(1,2)+c(2,3)$,
then let $g:[1..3]\to\mathbb{R}$, $g(1)=-c(1,2)$, $g(2)=0$, $g(3)=c(2,3)$,
and let $c':\mathbb{R}^{2}\to\mathbb{R}$ be defined by 
\[
c'(x,y)=|x-y|+\max\{|x-y|-c(2,3),0\}\frac{c(1,3)-c(1,2)-c(2,3)}{c(1,2)},
\]
which is a convex function of $|x-y|$. Note that $c(x,y)=c'(g(x),g(y))$.
Hence $r_{c}^{*}(\mathcal{P}([1..3]))\le r_{c'}^{*}(\mathcal{P}(\mathbb{R}))=1$
by Proposition \eqref{prop:r_rc}. If $c(1,2)>0$ and $c(1,3)<c(1,2)+c(2,3)$
(and thus $c$ is a metric), then define a coupling $\{X_{\alpha}\}_{\alpha}$
of $\{P_{\alpha}\}_{\alpha}$ by letting $\varsigma:[1..3]\to[1..3]$
be a random permutation defined by
\[
(\varsigma(1),\varsigma(2),\varsigma(3))=\begin{cases}
(1,2,3) & \text{with prob.}\;\frac{\frac{c(1,3)}{\gamma-2c(1,3)}}{\frac{c(1,2)}{\gamma-2c(1,2)}+\frac{c(2,3)}{\gamma-2c(2,3)}+\frac{c(1,3)}{\gamma-2c(1,3)}}\\
(2,3,1) & \text{with prob.}\;\frac{\frac{c(2,3)}{\gamma-2c(2,3)}}{\frac{c(1,2)}{\gamma-2c(1,2)}+\frac{c(2,3)}{\gamma-2c(2,3)}+\frac{c(1,3)}{\gamma-2c(1,3)}}\\
(3,1,2) & \text{with prob.}\;\frac{\frac{c(1,2)}{\gamma-2c(1,2)}}{\frac{c(1,2)}{\gamma-2c(1,2)}+\frac{c(2,3)}{\gamma-2c(2,3)}+\frac{c(1,3)}{\gamma-2c(1,3)}},
\end{cases}
\]
where $\gamma:=c(1,2)+c(2,3)+c(1,3)$, and $X_{\alpha}:=F_{\varsigma_{*}P_{\alpha}}^{-1}(U)$
(i.e., the quantile coupling on $\{\varsigma_{*}P_{\alpha}\}_{\alpha}$).
Let $g_{\varsigma}:[1..3]\to\mathbb{R}$ be defined by $g_{\varsigma}(x):=-\mathbf{1}\{\varsigma(x)=1\}c(\varsigma^{-1}(1),\varsigma^{-1}(2))+\mathbf{1}\{\varsigma(x)=3\}c(\varsigma^{-1}(2),\varsigma^{-1}(3))$.
Note that $c(x,y)\le|g_{\varsigma}(x)-g_{\varsigma}(y)|$. For any
$P_{\alpha},P_{\beta}\in\mathcal{P}([1..3])$, and $(\tilde{X}_{\alpha},\tilde{X}_{\beta})\in\Gamma_{\lambda}(P_{\alpha},P_{\beta})$
(assume $(\tilde{X}_{\alpha},\tilde{X}_{\beta})$ is independent of
$(\varsigma,U)$), we have
\begin{align*}
 & \mathbf{E}[c(X_{\alpha},X_{\beta})]\\
 & \le\mathbf{E}\left[\mathbf{E}\left[|g_{\varsigma}(X_{\alpha})-g_{\varsigma}(X_{\beta})|\,\big|\,\varsigma\right]\right]\\
 & \stackrel{(a)}{\le}\mathbf{E}\left[\mathbf{E}\left[|g_{\varsigma}(\tilde{X}_{\alpha})-g_{\varsigma}(\tilde{X}_{\beta})|\,\big|\,\varsigma\right]\right]\\
 & =\mathbf{E}\left[\mathbf{E}\left[|g_{\varsigma}(\tilde{X}_{\alpha})-g_{\varsigma}(\tilde{X}_{\beta})|\,\big|\,\tilde{X}_{\alpha},\tilde{X}_{\beta}\right]\right]\\
 & \stackrel{(b)}{=}\mathbf{E}\Bigg[\mathbf{1}\{\tilde{X}_{\alpha}\neq\tilde{X}_{\beta}\}\Bigg(\Bigg(1-\frac{\frac{c(\tilde{X}_{\alpha},\tilde{X}_{\beta})}{\gamma-2c(\tilde{X}_{\alpha},\tilde{X}_{\beta})}}{\frac{c(1,2)}{\gamma-2c(1,2)}+\frac{c(2,3)}{\gamma-2c(2,3)}+\frac{c(1,3)}{\gamma-2c(1,3)}}\Bigg)c(\tilde{X}_{\alpha},\tilde{X}_{\beta})\\
 & \;\;\;\;\;\;\;\;\;+\frac{\frac{c(\tilde{X}_{\alpha},\tilde{X}_{\beta})}{\gamma-2c(\tilde{X}_{\alpha},\tilde{X}_{\beta})}}{\frac{c(1,2)}{\gamma-2c(1,2)}+\frac{c(2,3)}{\gamma-2c(2,3)}+\frac{c(1,3)}{\gamma-2c(1,3)}}(\gamma-c(\tilde{X}_{\alpha},\tilde{X}_{\beta}))\Bigg)\Bigg]\\
 & =\mathbf{E}\left[\mathbf{1}\{\tilde{X}_{\alpha}\neq\tilde{X}_{\beta}\}c(\tilde{X}_{\alpha},\tilde{X}_{\beta})\left(1+\frac{1}{\frac{c(1,2)}{\gamma-2c(1,2)}+\frac{c(2,3)}{\gamma-2c(2,3)}+\frac{c(1,3)}{\gamma-2c(1,3)}}\right)\right]\\
 & =\left(1+\frac{1}{\frac{c(1,2)}{\gamma-2c(1,2)}+\frac{c(2,3)}{\gamma-2c(2,3)}+\frac{c(1,3)}{\gamma-2c(1,3)}}\right)\mathbf{E}[c(\tilde{X}_{\alpha},\tilde{X}_{\beta})]\\
 & \stackrel{(c)}{\le}\left(1+\frac{1}{3\frac{\gamma/3}{\gamma-2\gamma/3}}\right)\mathbf{E}[c(\tilde{X}_{\alpha},\tilde{X}_{\beta})]\\
 & =\frac{4}{3}\mathbf{E}[c(\tilde{X}_{\alpha},\tilde{X}_{\beta})],
\end{align*}
where (a) is by the optimality of the quantile coupling for the cost
function $|x-y|$, (b) is due to the definition of $\varsigma$ and
that $|g_{\varsigma}(\tilde{X}_{\alpha})-g_{\varsigma}(\tilde{X}_{\beta})|=\gamma-c(\tilde{X}_{\alpha},\tilde{X}_{\beta})$
if $\{\varsigma(\tilde{X}_{\alpha}),\varsigma(\tilde{X}_{\beta})\}=\{1,3\}$,
and (c) is due to the convexity of $t\mapsto t/(\gamma-2t)$.\\
$ $\\
For $\mathcal{X}=[1..4]$, consider the symmetric cost function $c(1,3)=c(2,4)=1$,
$c(x,y)=\epsilon$ otherwise, $0\le\epsilon<1$. Applying Proposition
\ref{prop:rn_rc_lb_2} on the sequence $(1,1,1,2,2,3,4,4)$, 
\begin{align*}
 & r_{c}^{*}(\mathcal{P}([1..4]))\\
 & \ge\bigg(\frac{c(2,3)}{c(1,2)+c(1,3)-c(2,3)}+\frac{c(3,4)}{c(1,3)+c(1,4)-c(3,4)}\\
 & \;\;+\frac{c(1,2)}{c(1,4)+c(2,4)-c(1,2)}+\frac{c(1,4)}{c(1,2)+c(2,4)-c(1,4)}\bigg)^{-1}+1\\
 & =\frac{1}{4\epsilon}+1.
\end{align*}
Letting $\epsilon\to0$ (or directly setting $\epsilon=0$) gives
$\sup_{c}r_{c}^{*}(\mathcal{P}([1..4]))=\infty$.\medskip{}
\item Without loss of generality, assume $\mathcal{X}=[1..4]$ (we can partition
$\mathcal{X}$ into 4 sets, and let $c(x,y)$ only depend on which
set $x$ and $y$ lie in). In the previous example, $c(x,y)>0$ for
$x\neq y$ and $r_{c}^{*}(\mathcal{P}([1..4]))$ can be arbitrarily
large. For $c(x,y)=|x-y|$, $r_{c}^{*}(\mathcal{P}([1..4]))=1$. The
result follows from the continuity of $r_{c}^{*}(\mathcal{P}([1..4]))$,
and that $r_{c}^{*}(\mathcal{P}([1..4]))$ can be infinite as in the
previous example.
\item The anti-triangle inequality is a direct consequence of \eqref{eq:rn_rc_lb_tri}
in Proposition \ref{prop:rn_rc_lb_2}. We can deduce \eqref{eq:ati4}
by applying Proposition \ref{prop:rn_rc_lb_2} on the sequence $(x_{0},x_{0},x_{0},x_{1},x_{1},x_{2},x_{3},x_{3})$.
\item If $c$ is a metric and $r_{c}^{*}(\mathcal{P}(\mathcal{X}))=1$,
then by the triangle and anti-triangle inequality, for any $x,y,z\in\mathcal{X}$,
we have 
\begin{equation}
c(x,z)\in\{|c(x,y)-c(y,z)|,\,c(x,y)+c(y,z)\},\label{eq:eti3}
\end{equation}
which can be deduced by considering which of the 3 cases in the anti-triangle
inequality holds. Fix any $w\in\mathcal{X}$. Define a relation ``$\sim$''
over $\mathcal{X}\backslash\{w\}$ by $x\sim y$ if $c(x,y)=|c(x,w)-c(y,w)|$.
We first show that it is an equivalence relation. If $x,y,z\in\mathcal{X}\backslash\{w\}$,
$x\sim y$ and $y\sim z$, then by applying \eqref{eq:ati4} on $(w,x,y,z)$,
either $c(w,x)=c(w,z)+c(x,z)$ or $c(w,z)=c(w,x)+c(x,z)$ (the first
two cases in \eqref{eq:ati4} does not hold due to $x\sim y$, $y\sim z$),
and thus $c(x,z)=|c(w,x)-c(w,z)|$, and $x\sim z$. Hence ``$\sim$''
is an equivalence relation. \\
\\
We then show that if $x\nsim y$ and $y\nsim z$, then $x\sim z$.
Assume $x,y,z\in\mathcal{X}\backslash\{w\}$, $x\nsim y$ and $y\nsim z$.
Applying \eqref{eq:eti3} on $(x,y,z)$ and $(x,w,z)$,
\begin{align*}
c(x,z) & \in\left\{ |c(x,y)-c(z,y)|,\,c(x,y)+c(z,y)\right\} \cap\left\{ |c(x,w)-c(z,w)|,\,c(x,w)+c(z,w)\right\} \\
 & \stackrel{(a)}{=}\left\{ \left|c(x,w)-c(z,w)\right|,\,c(x,w)+c(z,w)+2c(y,w)\right\} \cap\left\{ |c(x,w)-c(z,w)|,\,c(x,w)+c(z,w)\right\} \\
 & \stackrel{(b)}{=}\left\{ |c(x,w)-c(z,w)|\right\} ,
\end{align*}
where (a) is because $c(x,y)=c(x,w)+c(y,w)$ (by \eqref{eq:eti3}
and $x\nsim y$) and $c(z,y)=c(z,w)+c(y,w)$, and (b) is because $c(y,w)>0$
since $y\neq w$. Hence $x\sim z$. This implies that there are at
most 2 equivalence classes in $\mathcal{X}\backslash\{w\}$.\\
\\
Let $S\subseteq\mathcal{X}\backslash\{w\}$ be a nonempty equivalence
class, and $S':=\mathcal{X}\backslash\{w\}\backslash S$. Since there
are at most 2 equivalence classes, $S'$ is an equivalence class (or
is empty). Define an embedding function $g:\mathcal{X}\to\mathbb{R}$
by $g(w):=0$, $g(x):=c(x,w)$ if $x\in S$, $g(x):=-c(x,w)$ if $x\in S'$.
We now check that $c(x,y)=|g(x)-g(y)|$ for any $x,y\in\mathcal{X}$
(which implies that $g$ is injective). This is obvious when $x=w$
or $y=w$. If $x,y\in S$, we have $c(x,y)=|g(x)-g(y)|$ by the definition
of ``$\sim$''. The same holds if $x,y\in S'$. If $x\in S$, $y\in S'$,
then $c(x,y)=c(x,w)+c(y,w)=|g(x)-g(y)|$.
\end{enumerate}
\end{proof}
\medskip{}

\section{Truncated Pairwise Coupling Ratio\label{sec:trunc}}

One shortcoming of the definition of the pairwise multi-marginal coupling
setting is that it is unsuitable for convex costs. For instance, when
$\mathcal{X}=\mathbb{R}^{n}$, $n\ge2$, $c(x,y)=\Vert x-y\Vert_{2}^{2}$,
where $C_{c}^{*}$ is the widely-used squared 2-Wasserstein distance,
we have $r_{c}^{*}=\infty$ by Proposition \ref{prop:rn_rc_lb}. This
motivates us to modify Definition \ref{def:pcr} to accomodate convex
costs.
\begin{defn}
\label{def:trunc}For a symmetric cost function $c$ over the Polish
space $\mathcal{X}$, a collection of probability distributions $\{P_{\alpha}\}_{\alpha\in\mathcal{A}}$,
a coupling $\{X_{\alpha}\}_{\alpha\in\mathcal{A}}\in\Gamma_{\lambda}(\{P_{\alpha}\}_{\alpha\in\mathcal{A}})$,
and $\eta\ge0$, the \emph{$\eta$-truncated pairwise coupling ratio}
is defined as
\begin{align*}
 & \tilde{r}_{c,\eta}(\{X_{\alpha}\}_{\alpha\in\mathcal{A}})\\
 & :=\sup_{\alpha,\beta\in\mathcal{A}:\,C_{c}^{*}(P_{\alpha},P_{\beta})<\infty}\mathbf{E}\left[\min\left\{ \frac{c(X_{\alpha},X_{\beta})}{C_{c}^{*}(P_{\alpha},P_{\beta})},\,\eta\right\} \right],
\end{align*}
where we treat $0/0=1$ and $t/0=\infty$ for $t>0$ here. Intuitively,
this definition allows us to ignore some large values of $c(X_{\alpha},X_{\beta})/C_{c}^{*}(P_{\alpha},P_{\beta})$
by paying a penalty $\eta$. The \emph{optimal $\eta$-truncated pairwise
coupling ratio} of the collection $\{P_{\alpha}\}_{\alpha\in\mathcal{A}}$
is defined as 
\[
\tilde{r}_{c,\eta}^{*}(\{P_{\alpha}\}_{\alpha\in\mathcal{A}}):=\inf_{\{X_{\alpha}\}_{\alpha\in\mathcal{A}}\in\Gamma_{\lambda}(\{P_{\alpha}\}_{\alpha\in\mathcal{A}})}\tilde{r}_{c,\eta}(\{X_{\alpha}\}_{\alpha\in\mathcal{A}}).
\]
\end{defn}

Note that $\tilde{r}_{c,\eta}^{*}(\{P_{\alpha}\}_{\alpha\in\mathcal{A}})\le\eta$
and is increasing with $\eta$, which follows directly from the definition.

We give a bound on $\tilde{r}_{c,\eta}^{*}$ for the case where $\mathcal{X}$
is the circle, which shows that $\tilde{r}_{c,\eta}^{*}$ grows sublinearly
with $\eta$ in this case. The proof is given in Appendix \ref{subsec:pf_circle}.
\begin{prop}
\label{prop:circle_trunc}Let $\mathcal{M}=\{x\in\mathbb{R}^{2}:x_{1}^{2}+x_{2}^{2}=1\}$
be the circle, and $c(x,y)=(d_{\mathcal{M}}(x,y))^{q}$, $q\ge1$.
We have, for $\eta\ge0$,
\[
\tilde{r}_{c,\eta}^{*}(\mathcal{P}(\mathcal{M}))\le q\eta^{1-1/q}+1.
\]
\end{prop}

\medskip{}

Nevertheless, it is unknown whether $\tilde{r}_{c,\eta}^{*}$ grows
sublinearly with $\eta$ for the case $\mathcal{X}=\mathbb{R}^{n}$,
$n\ge2$, $c(x,y)=\Vert x-y\Vert_{2}^{2}$.\medskip{}

\section{Embedding $\mathcal{P}(\mathcal{X})$ into the Space of Random Variables\label{sec:geom}}

The embedding of a metric space into another metric space has been
studied extensively (e.g. see \cite{bourgain1985lipschitz,linial1995geometry,alon1995graph,bartal1996probabilistic,indyk2003fast,khot2006nonembeddability,naor2007planar,andoni2018snowflake}).
The embedding function is required to approximately preserve distances
(e.g. it is bi-Lipschitz). Popular choices of the target metric space
are the sequence spaces $\ell_{1}$, $\ell_{2}$, and function spaces
$L_{1}$, $L_{2}$, due to their theoretic and computational simplicity.
In this section, we show that if the cost function is a metric, the
pairwise multi-marginal optimal transport problem can be stated as
a problem of finding an embedding of the space of probability distributions,
with the 1-Wasserstein distance, into the space of random variables
on the standard probability space (which is also a space of functions).

Let $(\mathcal{X},d)$ be a complete separable metric space where
$d$ takes values in $\mathbb{R}_{\ge0}$ (it cannot take the value
$\infty$), with Borel $\sigma$-algebra $\mathcal{F}$. Write $\mathcal{V}_{\lambda}(\mathcal{X})$
for the space of equivalence classes of random variables (i.e., measurable
functions) $X:[0,1]\to\mathcal{X}$ on the standard probability space
$([0,1],\mathcal{L}([0,1]),\lambda_{[0,1]})$, modulo the equivalence
relation of almost sure equality.\footnote{The subscript $\lambda$ in $\mathcal{V}_{\lambda}(\mathcal{X})$
represents the probability measure of the underlying space $([0,1],\mathcal{L}([0,1]),\lambda_{[0,1]})$.
While writing $\mathcal{V}_{\lambda_{[0,1]}}(\mathcal{X})$ is more
accurate, we write $\mathcal{V}_{\lambda}(\mathcal{X})$ for notational
simplicity.} Let $d_{\lambda}:\mathcal{V}_{\lambda}(\mathcal{X})\times\mathcal{V}_{\lambda}(\mathcal{X})\to\mathbb{R}_{\ge0}\cup\{\infty\}$
be defined by $d_{\lambda}(X,Y)=\mathbf{E}[d(X,Y)]=\int_{0}^{1}d(X(u),Y(u))\mathrm{d}u$.
Then $(\mathcal{V}_{\lambda}(\mathcal{X}),d_{\lambda})$ is a metric
space (where the metric may take the value $\infty$).\footnote{For the purpose of generating a topology, we can convert a metric
that may take the value $\infty$ into a metric that is always finite
by $\tilde{d}_{\lambda}(X,Y):=\min\{d_{\lambda}(X,Y),\,1\}$. Nevertheless,
we allow metrics to take the value $\infty$ in this section (unless
otherwise specified like the case for $d$) to facilitate discussions
on Lipschitz continuity, which also applies to metrics that may take
the value $\infty$.} Let $\varpi_{\lambda}:\mathcal{V}_{\lambda}(\mathcal{X})\to\mathcal{P}(\mathcal{X})$
be defined by $\varpi_{\lambda}(X):=X_{*}\lambda_{[0,1]}$ (i.e.,
the distribution of $X$). Clearly, $\varpi_{\lambda}$ is a 1-Lipschitz
function from the metric space $(\mathcal{V}_{\lambda}(\mathcal{X}),d_{\lambda})$
to the metric space $(\mathcal{P}(\mathcal{X}),C_{d}^{*})$, where
$C_{d}^{*}$ is the 1-Wasserstein distance (which may take the value
$\infty$), i.e.,
\[
C_{d}^{*}(\varpi_{\lambda}(X),\varpi_{\lambda}(Y))\le d_{\lambda}(X,Y)
\]
for any $X,Y\in\mathcal{V}_{\lambda}(\mathcal{X})$. The problem
of finding $r_{d}^{*}(\mathcal{P}(\mathcal{X}))$ is equivalent to
that of finding the best Lipschitz constant among right inverses of
$\varpi_{\lambda}$, as demonstrated in the following proposition.
The proof is omitted since it is a direct consequence of the definition
of $r_{d}^{*}$.
\begin{prop}
\label{prop:bilip}For any $\mathcal{P}'\subseteq\mathcal{P}(\mathcal{X})$,
we have
\begin{align*}
r_{d}^{*}(\mathcal{P}') & =\inf\left\{ r\ge1:\,\exists\,r\text{-Lipschitz}\;\xi:(\mathcal{P}',C_{d}^{*})\to(\mathcal{V}_{\lambda}(\mathcal{X}),d_{\lambda})\;\mathrm{s.t.}\;\varpi_{\lambda}\circ\xi=\mathrm{id}_{\mathcal{P}'}\right\} ,
\end{align*}
where $\varpi_{\lambda}\circ\xi=\mathrm{id}_{\mathcal{P}'}$ means
that $\varpi_{\lambda}(\xi(P))=P$ for all $P\in\mathcal{P}'$. Note
that any $\xi$ satisfying the above conditions is $(r,1)$-bi-Lipschitz
since its left inverse $\varpi_{\lambda}$ is $1$-Lipschitz. As a
result, if $r_{d}^{*}(\mathcal{P}')<r<\infty$, then there exists
an $(r,1)$-bi-Lipschitz embedding function of $(\mathcal{P}',C_{d}^{*})$
into $(\mathcal{V}_{\lambda}(\mathcal{X}),d_{\lambda})$ that is a
right inverse of $\varpi_{\lambda}$.
\end{prop}

\medskip{}

If we are interested in embedding into $L_{1}$ instead of $(\mathcal{V}_{\lambda}(\mathcal{X}),d_{\lambda})$,
as in the previous works on metric space embedding (e.g. \cite{indyk2003fast,khot2006nonembeddability,naor2007planar}),
this can be achieved by combining an embedding function of $(\mathcal{P}(\mathcal{X}),C_{d}^{*})$
into $(\mathcal{V}_{\lambda}(\mathcal{X}),d_{\lambda})$, and an embedding
function of $(\mathcal{X},d)$ into $L_{1}$.
\begin{prop}
\label{prop:bilip_L1}Let $\mathcal{P}_{d}(\mathcal{X}):=\{P\in\mathcal{P}(\mathcal{X}):\,\int\int d(x,y)P(\mathrm{d}x)P(\mathrm{d}y)<\infty\}$.
If $r_{d}^{*}(\mathcal{P}_{d}(\mathcal{X}))<r<\infty$, and $\psi$
is a bi-Lipschitz embedding function of $(\mathcal{X},d)$ into $L_{1}$
with distortion at most $\theta$, then there exists a bi-Lipschitz
embedding function of $(\mathcal{P}_{d}(\mathcal{X}),C_{d}^{*})$
into $L_{1}$ with distortion at most $r\theta$.
\end{prop}

\begin{proof}
[Proof of Proposition \ref{prop:bilip_L1}] Assume $\psi$ is $(\theta,1)$-bi-Lipschitz
without loss of generality. By Proposition \ref{prop:bilip}, there
exists an $(r,1)$-bi-Lipschitz embedding function $\xi$ of $(\mathcal{P}_{d}(\mathcal{X}),C_{d}^{*})$
into $(\mathcal{V}_{\lambda}(\mathcal{X}),d_{\lambda})$ that is a
right inverse of $\varpi_{\lambda}$. Let $\Upsilon:[0,1]\to[0,1]^{2}$
be a measure-preserving function\footnote{Measure-preserving means that $\Upsilon_{*}\lambda_{[0,1]}=\lambda_{[0,1]^{2}}$.
We can simply let $\Upsilon(t):=(\sum_{i=1}^{\infty}2^{-i}b_{2i-1},\,\sum_{i=1}^{\infty}2^{-i}b_{2i})$,
where $\sum_{i=1}^{\infty}2^{-i}b_{i}=t$ is the binary expansion
of $t$ (choose the terminating expansion in case of ambiguity).}, and write $\Upsilon(t)=(\Upsilon_{1}(t),\Upsilon_{2}(t))$. Let
$\zeta:\mathcal{P}_{d}(\mathcal{X})\to L_{1}$ be defined by
\[
\zeta(P)(t):=\psi\big(\xi(P)(\Upsilon_{1}(t))\big)\big(\Upsilon_{2}(t)\big).
\]
Since $\xi(P)_{*}\lambda_{[0,1]}=P$, for any $P\in\mathcal{P}_{d}(\mathcal{X})$
and $y\in\mathcal{X}$,
\begin{align*}
\Vert\zeta(P)\Vert_{1} & =\int_{0}^{1}\left|\psi\big(\xi(P)(\Upsilon_{1}(t))\big)\big(\Upsilon_{2}(t)\big)\right|\mathrm{d}t\\
 & =\int_{0}^{1}\int_{0}^{1}\left|\psi\big(\xi(P)(\tau)\big)(t)\right|\mathrm{d}t\mathrm{d}\tau\\
 & =\int_{0}^{1}\Vert\psi(\xi(P)(\tau))\Vert_{1}\mathrm{d}\tau\\
 & =\int\Vert\psi(x)\Vert_{1}P(\mathrm{d}x)\\
 & \le\int\left(\Vert\psi(x)-\psi(y)\Vert_{1}+\Vert\psi(y)\Vert_{1}\right)P(\mathrm{d}x)\\
 & \le\int\left(\theta d(x,y)+\Vert\psi(y)\Vert_{1}\right)P(\mathrm{d}x)\\
 & =\theta\int d(x,y)P(\mathrm{d}x)+\Vert\psi(y)\Vert_{1}.
\end{align*}
 If $\Vert\zeta(P)\Vert_{1}=\infty$, then $\int d(x,y)P(\mathrm{d}x)=\infty$
for all $y\in\mathcal{X}$, and thus $\int\int d(x,y)P(\mathrm{d}x)P(\mathrm{d}y)=\infty$,
contradicting the definition of $\mathcal{P}_{d}(\mathcal{X})$. Therefore
$\Vert\zeta(P)\Vert_{1}<\infty$, and $\zeta(P)\in L_{1}$.

For any $P,Q\in\mathcal{P}_{d}(\mathcal{X})$,
\begin{align*}
\Vert\zeta(P)-\zeta(Q)\Vert_{1} & =\int_{0}^{1}\Vert\psi(\xi(P)(\tau))-\psi(\xi(Q)(\tau))\Vert_{1}\mathrm{d}\tau\\
 & \le\int_{0}^{1}\theta d\left(\xi(P)(\tau),\,\xi(Q)(\tau)\right)\mathrm{d}\tau\\
 & =\theta d_{\lambda}\left(\xi(P),\,\xi(Q)\right)\\
 & \le r\theta C_{d}^{*}(P,Q).
\end{align*}
Also,
\begin{align*}
\Vert\zeta(P)-\zeta(Q)\Vert_{1} & =\int_{0}^{1}\Vert\psi(\xi(P)(\tau))-\psi(\xi(Q)(\tau))\Vert_{1}\mathrm{d}\tau\\
 & \ge\int_{0}^{1}d\left(\xi(P)(\tau),\,\xi(Q)(\tau)\right)\mathrm{d}\tau\\
 & =d_{\lambda}\left(\xi(P),\,\xi(Q)\right)\\
 & \ge C_{d}^{*}(P,Q).
\end{align*}
Therefore $\zeta$ is $(r\theta,1)$-bi-Lipschitz.
\end{proof}
\medskip{}

Since $(\mathbb{R}^{n},\Vert\cdot\Vert_{p})$ is isometrically embeddable
into $L_{1}$ for $1\le p\le2$ \cite{bretagnolle1966lois}, we can
apply Proposition \ref{prop:bilip_L1} to the result in Proposition
\ref{prop:s_rc_ub} to show that when $\mathcal{X}=[0..s]^{n}$, $s\ge1$,
$d(x,y)=\Vert x-y\Vert_{p}$, $1\le p\le2$, $\mathcal{P}([0..s]^{n})$
is bi-Lipschitz embeddable into $L_{1}$ with distortion $28.66n^{1/p}\log(s+1)$.

As a result of Proposition \ref{prop:bilip_L1} and Theorem \ref{thm:rn_rc_ub},
when $\mathcal{X}=\mathbb{R}^{n}$, $d(x,y)=\Vert x-y\Vert_{2}^{q}$,
$0<q<1$, $\mathcal{P}_{d}(\mathcal{X})=\{P\in\mathcal{P}(\mathcal{X}):\,\int\Vert x\Vert_{2}^{q}P(\mathrm{d}x)<\infty\}$
is bi-Lipschitz embeddable into $L_{1}$.
\begin{prop}
\label{prop:ball_bilip}Let $\mathcal{X}=\mathbb{R}^{n}$, $n\ge1$,
$d(x,y)=\Vert x-y\Vert_{2}^{q}$, $0<q<1$. There exists a bi-Lipschitz
embedding function of $(\mathcal{P}_{d}(\mathcal{X}),C_{d}^{*})$
into $L_{1}$ with distortion at most $10.55n^{q/2}/(1-q)$.
\end{prop}

\begin{proof}
[Proof of Proposition \ref{prop:ball_bilip}] To apply Proposition
\ref{prop:bilip}, we design an isometric embedding function of $(\mathbb{R}^{n},d)$
into $L_{1}$. Let $\tilde{\psi}:\mathbb{R}^{n}\times\mathbb{R}^{n}\to\mathbb{R}$
be defined by
\[
\tilde{\psi}(x,y):=\mathbf{1}\{y\neq x\}\Vert x-y\Vert_{2}^{-(n-q)}-\mathbf{1}\{y\neq0\}\Vert y\Vert_{2}^{-(n-q)}.
\]
If $n\ge2$, for $t>0$, we have
\begin{align*}
 & \frac{\mathrm{d}}{\mathrm{d}t}\int_{\mathbb{R}^{n}}\left|\tilde{\psi}(0,y)-\tilde{\psi}(t\mathrm{e}_{1},y)\right|\mathrm{d}y\\
 & =\frac{\mathrm{d}}{\mathrm{d}t}\int_{\mathbb{R}^{n}}\left|\Vert y\Vert_{2}^{-(n-q)}-\Vert t\mathrm{e}_{1}-y\Vert_{2}^{-(n-q)}\right|\mathrm{d}y\\
 & =2\frac{\mathrm{d}}{\mathrm{d}t}\int_{[t/2,\infty)\times\mathbb{R}^{n-1}}\left(\Vert t\mathrm{e}_{1}-y\Vert_{2}^{-(n-q)}-\Vert y\Vert_{2}^{-(n-q)}\right)\mathrm{d}y\\
 & =2\frac{\mathrm{d}}{\mathrm{d}t}\int_{\mathbb{R}^{n}}\left(\mathbf{1}\{y_{1}\ge-t/2\}\Vert y\Vert_{2}^{-(n-q)}-\mathbf{1}\{y_{1}\ge t/2\}\Vert y\Vert_{2}^{-(n-q)}\right)\mathrm{d}y\\
 & =2\frac{\mathrm{d}}{\mathrm{d}t}\int_{[-t/2,t/2]\times\mathbb{R}^{n-1}}\Vert y\Vert_{2}^{-(n-q)}\mathrm{d}y\\
 & =2\int_{\mathbb{R}^{n-1}}\Vert(t/2,y)\Vert_{2}^{-(n-q)}\mathrm{d}y\\
 & =2\int_{\mathbb{R}^{n-1}}\left(t^{2}/4+\Vert y\Vert_{2}^{2}\right)^{-(n-q)/2}\mathrm{d}y\\
 & =2\int_{0}^{\infty}\left(t^{2}/4+\tau^{2}\right)^{-(n-q)/2}(n-1)\mathrm{V}_{n-1,2}\tau^{n-2}\mathrm{d}\tau\\
 & =2(n-1)\mathrm{V}_{n-1,2}\frac{t}{2}\int_{0}^{\infty}\left(t^{2}/4+\tau^{2}t^{2}/4\right)^{-(n-q)/2}(\tau t/2)^{n-2}\mathrm{d}\tau\\
 & =2(n-1)\mathrm{V}_{n-1,2}\frac{t}{2}\left(\frac{t}{2}\right)^{-(n-q)+n-2}\int_{0}^{\infty}\left(1+\tau^{2}\right)^{-(n-q)/2}\tau^{n-2}\mathrm{d}\tau\\
 & =\gamma qt^{-(1-q)},
\end{align*}
where
\[
\gamma:=\begin{cases}
2^{2-q}q^{-1} & \mathrm{if}\;n=1\\
2^{2-q}q^{-1}(n-1)\mathrm{V}_{n-1,2}\int_{0}^{\infty}\left(1+\tau^{2}\right)^{-(n-q)/2}\tau^{n-2}\mathrm{d}\tau & \mathrm{if}\;n\ge2.
\end{cases}
\]
Note that if $n\ge2$,
\begin{align*}
 & \int_{0}^{\infty}\left(1+\tau^{2}\right)^{-(n-q)/2}\tau^{n-2}\mathrm{d}\tau\\
 & \le1+\int_{1}^{\infty}\left(1+\tau^{2}\right)^{-(n-q)/2}\tau^{n-2}\mathrm{d}\tau\\
 & \le1+\int_{1}^{\infty}\left(2\tau^{2}\right)^{-(n-q)/2}\tau^{n-2}\mathrm{d}\tau\\
 & =1+2^{-(n-q)/2}\int_{1}^{\infty}\tau^{-(2-q)}\mathrm{d}\tau\\
 & <\infty.
\end{align*}
If $n=1$, for $t>0$, we have
\begin{align*}
 & \frac{\mathrm{d}}{\mathrm{d}t}\int_{\mathbb{R}}\left|\tilde{\psi}(0,y)-\tilde{\psi}(t,y)\right|\mathrm{d}y\\
 & =2\frac{\mathrm{d}}{\mathrm{d}t}\int_{t/2}^{\infty}\left(\left|t-y\right|^{-(1-q)}-\left|y\right|^{-(1-q)}\right)\mathrm{d}y\\
 & =2\frac{\mathrm{d}}{\mathrm{d}t}\int_{-\infty}^{\infty}\left(\mathbf{1}\{y\ge-t/2\}\left|y\right|^{-(1-q)}-\mathbf{1}\{y\ge t/2\}\left|y\right|^{-(1-q)}\right)\mathrm{d}y\\
 & =2(t/2)^{-(1-q)}\\
 & =\gamma qt^{-(1-q)}.
\end{align*}
Hence, for any $n\ge1$, $t\ge0$,
\[
\int_{\mathbb{R}^{n}}\left|\tilde{\psi}(0,y)-\tilde{\psi}(t\mathrm{e}_{1},y)\right|\mathrm{d}y=\gamma t^{q}.
\]
By symmetry, for any $x,\tilde{x}\in\mathbb{R}^{n}$,
\begin{equation}
\int_{\mathbb{R}^{n}}\left|\tilde{\psi}(x,y)-\tilde{\psi}(\tilde{x},y)\right|\mathrm{d}y=\gamma\Vert x-\tilde{x}\Vert_{2}^{q}.\label{eq:ball_bilip_psit}
\end{equation}
Let $\Upsilon:[0,1]\to[0,1]^{n}$ be a measure-preserving function,
and write $\Upsilon(t)=(\Upsilon_{1}(t),\ldots,\Upsilon_{n}(t))$.
Let $\psi:\mathbb{R}^{n}\to L_{1}$ be defined by
\[
\psi(x)(t):=\gamma^{-1}\left(\prod_{i=1}^{n}\frac{1}{2\Upsilon_{i}(t)(1-\Upsilon_{i}(t))}\right)\tilde{\psi}\left(x,\,\left\{ \tanh^{-1}(2\Upsilon_{i}(t)-1)\right\} _{i\in[1..n]}\right).
\]
For any $x,\tilde{x}\in\mathbb{R}^{n}$,
\begin{align*}
 & \Vert\psi(x_{1})-\psi(x_{2})\Vert_{1}\\
 & =\gamma^{-1}\int_{0}^{1}\left(\prod_{i=1}^{n}\frac{1}{2\Upsilon_{i}(t)(1-\Upsilon_{i}(t))}\right)\left|\tilde{\psi}\left(x,\left\{ \tanh^{-1}(2\Upsilon_{i}(t)-1)\right\} _{i\in[1..n]}\right)-\tilde{\psi}\left(\tilde{x},\left\{ \tanh^{-1}(2\Upsilon_{i}(t)-1)\right\} _{i\in[1..n]}\right)\right|\mathrm{d}t\\
 & =\gamma^{-1}\int_{[0,1]^{n}}\left(\prod_{i=1}^{n}\frac{1}{2y_{i}(1-y_{i})}\right)\left|\tilde{\psi}\left(x,\left\{ \tanh^{-1}(2y_{i}-1)\right\} _{i\in[1..n]}\right)-\tilde{\psi}\left(\tilde{x},\left\{ \tanh^{-1}(2y_{i}-1)\right\} _{i\in[1..n]}\right)\right|\mathrm{d}y\\
 & \overset{(a)}{=}\gamma^{-1}\int_{\mathbb{R}^{n}}\left|\tilde{\psi}(x,y)-\tilde{\psi}(\tilde{x},y)\right|\mathrm{d}y\\
 & \overset{(b)}{=}\Vert x-\tilde{x}\Vert_{2}^{q},
\end{align*}
where (a) is by substituting $y_{i}\leftarrow\tanh^{-1}(2y_{i}-1)$,
and (b) is by \eqref{eq:ball_bilip_psit}. Therefore $\psi$ is isometric.
In particular, for any $x\in\mathbb{R}^{n}$, $\Vert\psi(x)\Vert_{1}=\Vert\psi(x)-\psi(0)\Vert_{1}=\Vert x\Vert_{2}^{q}<\infty$
since $\psi(0)=0$ by definition. Hence $\psi(x)\in L_{1}$. The result
follows from Proposition \ref{prop:bilip_L1} and Theorem \ref{thm:rn_rc_ub}.
\end{proof}
\medskip{}

We now prove Proposition \ref{prop:rn_rc_lb} where $\mathcal{X}=\mathbb{Z}^{n}$,
$n\ge2$, $d(x,y)=\Vert x-y\Vert_{p}^{q}$, $p\in\mathbb{R}_{\ge1}\cup\{\infty\}$,
$0<q\le1$, using a result in \cite{naor2007planar}.
\begin{proof}
[Proof of Proposition \ref{prop:rn_rc_lb}] Assume $n=2$ without
loss of generality. First consider the case $p=2$, $q=1$. Let $s\in\mathbb{N}$
and $\epsilon>0$. We have $r_{\Vert\cdot\Vert_{2}}^{*}(\mathcal{P}([0..s]^{2}))<\infty$
by the ratio bound in Proposition \ref{prop:rc_prop_misc}. Since
$(\mathbb{R}^{2},\Vert\cdot\Vert_{2})$ is isometrically embeddable
into $L_{1}$ \cite{bretagnolle1966lois}, by Proposition \ref{prop:bilip_L1},
there is a bi-Lipschitz embedding function $\xi$ of $(\mathcal{P}([0..s]^{2}),C_{\Vert\cdot\Vert_{2}}^{*})$
into $L_{1}$ with distortion at most $r_{\Vert\cdot\Vert_{2}}^{*}(\mathcal{P}([0..s]^{2}))+\epsilon$.

We invoke a result in \cite{naor2007planar}, which states that 
if $\zeta$ is a bi-Lipschitz embedding function of $(\mathcal{P}([0..s]^{2}),C_{\Vert\cdot\Vert_{2}}^{*})$
into $L_{1}$, then its distortion is at least $\sqrt{\log s}/(64\pi)$.
Since $\xi$ satisfies these conditions, we have
\[
r_{\Vert\cdot\Vert_{2}}^{*}(\mathcal{P}([0..s]^{2}))+\epsilon\ge\frac{\sqrt{\log s}}{64\pi}.
\]
Letting $\epsilon\to0$, we have
\[
r_{\Vert\cdot\Vert_{2}}^{*}(\mathcal{P}([0..s]^{2}))\ge\frac{\sqrt{\log s}}{64\pi}.
\]
For any $p\in\mathbb{R}_{\ge1}\cup\{\infty\}$, $0<q\le1$, by the
ratio bound in Proposition \ref{prop:rc_prop_misc},
\begin{align}
r_{d}^{*}(\mathcal{P}([0..s]^{2})) & \ge\frac{1}{\max\{2^{q/p-1/2},1\}}\cdot\frac{1}{\max\{2^{1/2-q/p}s^{1-q},s^{1-q}\}}r_{\Vert\cdot\Vert_{2}}^{*}(\mathcal{P}([0..s]^{2}))\nonumber \\
 & \ge\frac{s^{q-1}\sqrt{\log s}}{64\sqrt{2}\pi}.\label{eq:rn_rc_lb_order}
\end{align}
If $q=1$, we have $r_{d}^{*}(\mathcal{P}(\mathbb{Z}^{2}))=\infty$
by letting $s\to\infty$. If $0<q<1$, letting $s=\lceil e^{1/(2-2q)}\rceil$,
we have
\begin{align*}
r_{d}^{*}(\mathcal{P}(\mathbb{Z}^{2})) & \ge r_{d}^{*}(\mathcal{P}([0..s]^{2}))\\
 & \ge\frac{s^{q-1}\sqrt{\log s}}{64\sqrt{2}\pi}\\
 & \ge\lceil e^{1/(2-2q)}\rceil^{q-1}\frac{\sqrt{1/(2-2q)}}{64\sqrt{2}\pi}\\
 & \ge\left(\frac{3}{2}e^{1/(2-2q)}\right)^{q-1}\frac{\sqrt{1/(2-2q)}}{64\sqrt{2}\pi}\\
 & \ge\frac{2}{3}\cdot\frac{1}{64\sqrt{2}\pi}\cdot\frac{1}{\sqrt{2e(1-q)}}\\
 & >\frac{1}{1000\sqrt{1-q}}.
\end{align*}
The other cases in Remark \ref{rem:rc_lb_cases} follow from the
same arguments as in Appendix \ref{subsec:bd_ball_cont}.
\end{proof}
\medskip{}

We can also use the result in \cite{khot2006nonembeddability} to
show a lower bound on $r_{c}^{*}$ for the Hamming distance over $\{0,1\}^{n}$.
\begin{prop}
\label{prop:hamming_embed}Let $\mathcal{X}=\{0,1\}^{n}$, $n\ge2$,
$d(x,y)=\Vert x-y\Vert_{1}$, $c(x,y)=(d(x,y))^{q}$, $q>0$. We have
\[
r_{c}^{*}(\mathcal{P}(\{0,1\}^{n}))=\Omega(n^{1-|1-q|})
\]
as $n\to\infty$, where the constant in $\Omega(\cdots)$ does not
depend on $q$.
\end{prop}

\begin{proof}
It is shown in \cite[Corollary 3.7]{khot2006nonembeddability} that
there exists a constant $\gamma>0$ such that if $\zeta$ is a bi-Lipschitz
embedding function of $(\mathcal{P}(\{0,1\}^{n}),C_{d}^{*})$ into
$L_{1}$, then its distortion is at least $\gamma n$. Since $(\{0,1\}^{n},d)$
is clearly isometrically embeddable into $L_{1}$, by Proposition
\ref{prop:bilip_L1}, we have $r_{d}^{*}(\mathcal{P}(\{0,1\}^{n}))\ge\gamma n$.
By the ratio bound in Proposition \ref{prop:rc_prop_misc},
\begin{align*}
r_{c}^{*}(\mathcal{P}(\{0,1\}^{n})) & \ge\frac{1}{\max\{n^{q-1},1\}}\cdot\frac{1}{\max\{n^{1-q},1\}}r_{d}^{*}(\mathcal{P}(\{0,1\}^{n}))\\
 & \ge n^{1-|1-q|}\gamma.
\end{align*}
\end{proof}
Note that $\Vert x-y\Vert_{1}=\Vert x-y\Vert_{p}^{p}$ over $\{0,1\}^{n}$
for $p\ge1$. Therefore, when $\mathcal{X}=\{0,1\}^{n}$, $c(x,y)=\Vert x-y\Vert_{p}^{q}$,
$p\in\mathbb{R}_{\ge1}\cup\{\infty\}$, $q>0$, we have
\[
r_{c}^{*}(\mathcal{P}(\{0,1\}^{n}))=\Omega(n^{1-|1-q/p|}).
\]
\medskip{}

\section{Conjectures\label{sec:unresolved}}

In this section, we list some conjectures and unresolved problems
about $r_{c}^{*}$ that may be of interest.\medskip{}

\begin{enumerate}
\item \emph{For $c(x,y)=|x-y|^{q}$ over $\mathbb{R}$, does there exist
a uniform upper bound for $r_{c}^{*}(\mathcal{P}(\mathbb{R}))$ for
all $q>0$? Is $r_{c}^{*}(\mathcal{P}(\mathbb{R}))=2$ for all $0<q<1$?}\smallskip{}
\\
Theorem \ref{thm:rn_rc_ub} gives an upper bound on $r_{c}^{*}(\mathcal{P}(\mathbb{R}))$
for $0<q<1$. Nevertheless, it is likely not tight, and it tends to
$\infty$ when $q\to1$, which may not be the actual behavior of $r_{c}^{*}$.\medskip{}
\item \emph{Does Theorem \ref{thm:rn_rc_ub} (or a similar bound) hold for
$\mathcal{P}(\mathcal{M})$ where $\mathcal{M}$ is any connected
smooth complete $n$-dimensional Riemannian manifold, and $c(x,y)=(d_{\mathcal{M}}(x,y))^{q}$,
$0<q<1$? Is $r_{c}^{*}(\mathcal{P}(\mathcal{M}))<\infty$ when $\mathcal{M}$
is the hyperbolic space?}\smallskip{}
\\
The upper bound in Theorem \ref{thm:ricci} requires a non-negative
Ricci curvature, and Corollary \ref{cor:dyadic_manifold} requires
finding an embedding into a Euclidean space where the intrinsic distance
can be approximated by the distance in the Euclidean space, which
may not be possible for general Riemannian manifolds. It may be of
interest to find more general bounds.\medskip{}
\item \emph{What is $r_{\mathbf{1}_{\neq}}^{*}(\mathcal{P}([1..4]))$?}\smallskip{}
\\
Theorem \ref{thm:rd_bd} shows that $r_{\mathbf{1}_{\neq}}^{*}(\mathcal{P}([1..4]))\in[3/2,5/3]$.
Nevertheless, its exact value is unknown.
\end{enumerate}
\smallskip{}

\section{Acknowledgements}

The authors acknowledge support from the NSF grants CNS-1527846, CCF-1618145,
the NSF Science \& Technology Center grant CCF-0939370 (Science of
Information), and the William and Flora Hewlett Foundation supported
Center for Long Term Cybersecurity at Berkeley.

\appendix
\[
\]

\section{Deciding Whether $r_{c}^{*}(\{P_{\alpha}\}_{\alpha})=1$ is NP-hard
\label{subsec:hardness}}

We show that the problem of deciding whether $r_{c}^{*}(\{P_{\alpha}\}_{\alpha\in\mathcal{A}})=1$
is NP-hard, where $(\mathcal{X},c)$ is a finite metric space, $c$
only takes values in $\{0,1,2\}$, and $\{P_{\alpha}\}_{\alpha\in\mathcal{A}}$
is a finite collection of probability distributions where each $P_{\alpha}$
is a $b$-type distribution (i.e., $bP_{\alpha}(x)\in\mathbb{Z}_{\ge0}$
for all $x\in\mathcal{X}$), where $b\in\mathbb{N}$ is an input ($b$
is the same for all $P_{\alpha}$). Note that the size of the input
$(|\mathcal{X}|,b,\{c(x,y)\}_{x,y\in\mathcal{X}},\{P_{\alpha}(x)\}_{\alpha\in\mathcal{A},x\in\mathcal{X}})$
is $O(|\mathcal{X}|^{2}+|\mathcal{A}||\mathcal{X}|\log b)$. We will
show this by a polynomial-time reduction from the graph coloring problem
of deciding whether the graph $(V,E)$ admits a proper vertex coloring
with $k$ colors (``proper'' means that every two adjacent vertices
have different colors), which is NP-complete \cite{karp1972reducibility}.

Fix any graph $(V,E)$ with vertex set $V$ and edge set $E$ (without
self-loop or multiple edges) where $(v_{1},v_{2})\in E$ $\Leftrightarrow$
$(v_{2},v_{1})\in E$. Assume $|V|\ge2$. Fix any $k\in[2..|V|]$.
Let $\mathcal{X}:=V\times[1..k]$, and
\[
c\left((v_{1},z_{1}),(v_{2},z_{2})\right):=\mathbf{1}\left\{ (v_{1},z_{1})\neq(v_{2},z_{2})\right\} +\mathbf{1}\left\{ (v_{1},v_{2})\in E\;\mathrm{and}\;z_{1}=z_{2}\right\} .
\]
We can show that $c$ is a metric by the fact that $c$ is symmetric,
$c((v_{1},z_{1}),(v_{2},z_{2}))=0$ $\Leftrightarrow$ $(v_{1},z_{1})=(v_{2},z_{2})$,
and $c$ is $\{0,1,2\}$-valued (if $(v_{1},z_{1}),(v_{2},z_{2}),(v_{3},z_{3})$
are distinct, then $c((v_{1},z_{1}),(v_{2},z_{2}))+c((v_{2},z_{2}),(v_{3},z_{3}))\ge2\ge c((v_{1},z_{1}),(v_{3},z_{3}))$).
Let $\mathcal{A}:=V$, $P_{v}:=\mathrm{Unif}(\{(v,z):\,z\in[1..k]\})$
for $v\in V$. Note that $C_{c}^{*}(P_{v_{1}},P_{v_{2}})=1$ for any
$v_{1}\neq v_{2}$ since $k\ge2$.

We now show that $(V,E)$ is $k$-colorable if and only if $r_{c}^{*}(\{P_{v}\}_{v\in V})=1$.
For the ``only if'' direction, assume $(V,E)$ is $k$-colorable.
Let the proper coloring be $f:V\to[1..k]$. Construct a coupling $\{X_{v}\}_{v}$
by $U\sim\mathrm{Unif}[1..k]$, $X_{v}=(v,U\oplus f(v))$, where $a\oplus b\in[1..k]$
is defined by $a\oplus b\equiv a+b$ ($\mathrm{mod}\;k$). For any
$v_{1}\neq v_{2}$, 
\begin{align*}
c(X_{v_{1}},X_{v_{2}}) & =1+\mathbf{1}\left\{ (v_{1},v_{2})\in E\;\mathrm{and}\;U\oplus f(v_{1})=U\oplus f(v_{2})\right\} \\
 & =1+\mathbf{1}\left\{ (v_{1},v_{2})\in E\;\mathrm{and}\;f(v_{1})=f(v_{2})\right\} \\
 & =1\\
 & =C_{c}^{*}(P_{v_{1}},P_{v_{2}}),
\end{align*}
and hence $r_{c}^{*}(\{P_{v}\}_{v})=r_{c}(\{X_{v}\}_{v})=1$.

For the ``if'' direction, assume $r_{c}^{*}(\{P_{v}\}_{v})=1$.
Let $\{X_{v}\}_{v}$ be a coupling achieving $r_{c}(\{X_{v}\}_{v})\le1+1/(2|E|)$.
Define a coloring by $f_{\{X_{v}\}_{v}}(v):=X_{v,2}$ (where $X_{v,2}$
denotes the second component of the pair $X_{v}\in V\times[1..k]$).
We have
\begin{align*}
 & \mathbf{P}\big(f_{\{X_{v}\}_{v}}\;\text{is not proper}\big)\\
 & =\mathbf{P}\big(\exists(v_{1},v_{2})\in E:\,X_{v_{1},2}=X_{v_{2},2}\big)\\
 & =\mathbf{P}\big(\exists(v_{1},v_{2})\in E:\,c(X_{v_{1}},X_{v_{2}})=2\big)\\
 & \le\sum_{(v_{1},v_{2})\in E}\mathbf{P}\big(c(X_{v_{1}},X_{v_{2}})=2\big)\\
 & =\sum_{(v_{1},v_{2})\in E}\left(\mathbf{E}\big[c(X_{v_{1}},X_{v_{2}})\big]-1\right)\\
 & \le\sum_{(v_{1},v_{2})\in E}\left((1+1/(2|E|))C_{c}^{*}(P_{v_{1}},P_{v_{2}})-1\right)\\
 & =1/2.
\end{align*}
Therefore there exists a proper coloring. The result follows.

We remark that while deciding whether $r_{c}^{*}(\{P_{\alpha}\}_{\alpha})=1$
is NP-hard for general collection of probability distributions $\{P_{\alpha}\}_{\alpha}$,
deciding whether $r_{c}^{*}(\mathcal{P}(\mathcal{X}))=1$ (when $\{P_{\alpha}\}_{\alpha}$
is fixed to the collection of all probability distributions $\mathcal{P}(\mathcal{X})$)
for a finite metric space $(\mathcal{X},c)$ can be performed in $O(|\mathcal{X}|^{2})$
time. This is due to Proposition \ref{prop:rc_prop_misc}, which shows
that if $(\mathcal{X},c)$ is a metric space, then $r_{c}^{*}(\mathcal{P}(\mathcal{X}))=1$
if and only if $(\mathcal{X},c)$ can be isometrically embedded into
$(\mathbb{R},\,(x,y)\mapsto|x-y|)$. We can decide whether such an
embedding function $g:\mathcal{X}\to\mathbb{R}$ exists by fixing
$g(x_{0})=0$ for a point $x_{0}\in\mathcal{X}$, fixing $g(x_{1})=c(x_{0},x_{1})$
for another $x_{1}\in\mathcal{X}$, and then checking whether $g(x):=c(x_{0},x)(1-2\cdot\mathbf{1}\{c(x_{0},x)+c(x_{0},x_{1})=c(x_{1},x)\})$
is an isometric embedding.

\medskip{}

\section{Bounding $r_{\mathbf{1}_{\protect\neq}}^{*}$ for $\{P\in\mathcal{P}(\mathcal{X}):\,\mathrm{supp}(P)\le k\}$
\label{subsec:pf_rd_bd_supp2}}

Here we prove that $r_{\mathbf{1}_{\neq}}^{*}(\{P\in\mathcal{P}(\mathcal{X}):\,\mathrm{supp}(P)\le k\})\le k$
for $k\in\mathbb{N}$. Note that, if $\mathcal{X}$ is uncountable,
there does not exist a $\sigma$-finite measure $\mu$ such that $P\ll\mu$
for all $P\in\{P\in\mathcal{P}(\mathcal{X}):\,\mathrm{supp}(P)=2\}$.
This shows that the existence of a $\sigma$-finite measure $\mu$
such that $P_{\alpha}\ll\mu$ for all non-degenerate $P_{\alpha}\in\{P_{\alpha}\}_{\alpha}$
(non-degenerate means $P_{\alpha}\neq\delta_{x}$ for all $x\in\mathcal{X}$)
is a sufficient condition (by Theorem \ref{thm:rd_bd}), but not a
necessary condition for $r_{\mathbf{1}_{\neq}}^{*}(\{P_{\alpha}\}_{\alpha})\le2$
to hold.
\begin{prop}
\label{prop:rd_bd_supp2}For any Polish space $\mathcal{X}$, $k\in\mathbb{N}$,
and $c(x,y)=\mathbf{1}_{\neq}(x,y)$, we have
\[
r_{\mathbf{1}_{\neq}}^{*}(\{P\in\mathcal{P}(\mathcal{X}):\,\mathrm{supp}(P)\le k\})\le k.
\]
\end{prop}

\begin{proof}
[Proof of Proposition \ref{prop:rd_bd_supp2}]The case where $\mathcal{X}$
is countable follows from Theorem \ref{thm:rd_bd} by letting $\mu$
be the counting measure (the case $k=1$ is trivial). If $\mathcal{X}$
is uncountable, then it is Borel-isomorphic to $[0,1]$. Hence we
assume $\mathcal{X}=[0,1]$ without loss of generality. Let $\{P_{\alpha}\}_{\alpha\in\mathcal{A}}=\{P\in\mathcal{P}(\mathcal{X}):\,\mathrm{supp}(P)\le k\}$
(we can simply take $\mathcal{A}=\{P\in\mathcal{P}(\mathcal{X}):\,\mathrm{supp}(P)\le k\}$,
$P_{\alpha}=\alpha$). We construct a coupling of $\{P_{\alpha}\}_{\alpha}$
by the inverse transform as in Proposition \ref{prop:r_rc}, i.e.,
let $U\sim\mathrm{Unif}[0,1]$, $X_{\alpha}:=F_{P_{\alpha}}^{-1}(U)$.

Fix any two probability distributions $P_{\alpha},P_{\beta}$. We
have
\begin{align*}
 & \mathbf{P}(X_{\alpha}\neq X_{\beta})\\
 & =\sum_{x\in\mathrm{supp}(P_{\alpha})}\mathbf{P}(X_{\alpha}=x\;\mathrm{and}\;X_{\beta}\neq x)\\
 & =\sum_{x\in\mathrm{supp}(P_{\alpha})}\lambda\left([P_{\alpha}([0,x)),\,P_{\alpha}([0,x])]\backslash[P_{\beta}([0,x)),\,P_{\beta}([0,x])]\right)\\
 & \le\sum_{x\in\mathrm{supp}(P_{\alpha})}\left(\max\left\{ P_{\beta}([0,x))-P_{\alpha}([0,x)),\,0\right\} +\max\left\{ P_{\alpha}([0,x])-P_{\beta}([0,x]),\,0\right\} \right)\\
 & =\sum_{x\in\mathrm{supp}(P_{\alpha})}\left(\max\left\{ P_{\beta}([0,x))-P_{\alpha}([0,x)),\,0\right\} +\max\left\{ P_{\beta}((x,1])-P_{\alpha}((x,1]),\,0\right\} \right)\\
 & \le\sum_{x\in\mathrm{supp}(P_{\alpha})}\left(\sum_{y\in\mathrm{supp}(P_{\beta})\cap[0,x)}\max\left\{ P_{\beta}(y)-P_{\alpha}(y),\,0\right\} +\sum_{y\in\mathrm{supp}(P_{\beta})\cap(x,1]}\max\left\{ P_{\beta}(y)-P_{\alpha}(y),\,0\right\} \right)\\
 & \le\sum_{x\in\mathrm{supp}(P_{\alpha})}d_{\mathrm{TV}}(P_{\alpha},P_{\beta})\\
 & \le kd_{\mathrm{TV}}(P_{\alpha},P_{\beta}).
\end{align*}
The result follows.
\end{proof}

\section{Proof of Theorem \ref{thm:rd_bd} for Nonstandard Probability Space
\label{subsec:pf_rd_bd_nonst}}

Here we extend Theorem \ref{thm:rd_bd} to the collection of all distributions
$\mathcal{P}(\mathcal{X})$ (rather than only $\mathcal{P}_{\ll\mu}(\mathcal{X})$)
if the definition of coupling is relaxed to allow a nonstandard probability
space. For a collection of probability distributions $\{P_{\alpha}\}_{\alpha\in\mathcal{A}}$
over a measurable space $(\mathcal{X},\mathcal{F})$, let 
\[
r_{c}^{\mathrm{nst}}(Q):=\inf\left\{ r\ge1:\,\mathbf{E}_{\{X_{\gamma}\}_{\gamma}\sim Q}[c(X_{\alpha},X_{\beta})]\le rC_{c}^{*}(P_{\alpha},P_{\beta})\,\forall\alpha,\beta\in\mathcal{A}\right\} ,
\]
for $Q\in\Gamma(\{P_{\alpha}\}_{\alpha})$, and
\[
r_{c}^{\mathrm{nst}*}(\{P_{\alpha}\}_{\alpha}):=\inf_{Q\in\Gamma(\{P_{\alpha}\}_{\alpha\in\mathcal{A}})}r_{c}(Q).
\]

\begin{prop}
\label{prop:rd_bd_nst}For any Polish space $\mathcal{X}$, and $c(x,y)=\mathbf{1}_{\neq}(x,y)$,
we have
\[
r_{\mathbf{1}_{\neq}}^{\mathrm{nst}*}(\mathcal{P}(\mathcal{X}))\le2.
\]
\end{prop}

\begin{proof}
[Proof of Proposition \ref{prop:rd_bd_nst}]Let $\{P_{\alpha}\}_{\alpha\in\mathcal{A}}=\mathcal{P}(\mathcal{X})$
(we can simply take $\mathcal{A}=\mathcal{P}(\mathcal{X})$, $P_{\alpha}=\alpha$).
Note that if $\mathcal{X}$ is uncountable, there does not exist a
measure $\mu$ such that $P_{\alpha}\ll\mu$ for all $\alpha\in\mathcal{A}$.

For two $\sigma$-finite measures $\mu,\nu$, Lebesgue's decomposition
theorem states that there exist unique $\sigma$-finite measures $\nu_{1},\nu_{2}$
such that $\nu=\nu_{1}+\nu_{2}$, $\nu_{1}\ll\mu$ and $\nu_{2}\perp\mu$.
Write 
\begin{equation}
\nu\cap_{\mathrm{L}}\mu:=\nu_{1},\,\nu\backslash_{\mathrm{L}}\mu:=\nu_{2}.\label{eq:lebesgue_dec}
\end{equation}
By the well-ordering theorem (which requires the axiom of choice),
assume a well-ordering $\le$ of $\mathcal{A}$. We now construct
measures $\{\nu_{\alpha}\}_{\alpha\in\mathcal{A}}$ by transfinite
recursion satisfying $\nu_{\alpha}\le P_{\alpha}$, $\nu_{\alpha}\perp\nu_{\beta}$
for all $\alpha\neq\beta$ and, for any $\alpha\in\mathcal{A}$, there
exists sequence $\{\beta_{\alpha,i}\}_{i\in[1..l_{\alpha}]}$ ($l_{\alpha}\in\mathbb{N}\cup\{\infty\}$)
with distinct elements $\beta_{\alpha,i}\le\alpha$ such that $P_{\alpha}\ll\sum_{i=1}^{l_{\alpha}}2^{-i}\nu_{\beta_{\alpha,i}}$.

Assume such $\nu_{\beta}$'s are constructed for all $\beta<\alpha$.
We now construct $\nu_{\alpha}$. We define $\nu_{\alpha,i}$ recursively.
Let $\nu_{\alpha,1}=P_{\alpha}$. For $i\ge1$, if $t_{i}:=\sup_{\beta<\alpha}((\nu_{\alpha,i}\cap_{\mathrm{L}}P_{\beta})(\mathcal{X}))>0$,
let $\tilde{\beta}_{\alpha,i}<\alpha$ attain at least half of the
supremum, and let $\nu_{\alpha,i+1}=\nu_{\alpha,i}\backslash_{\mathrm{L}}P_{\tilde{\beta}_{\alpha,i}}$.
If $t_{i}=0$ (or if $\alpha$ is the least element), output $\nu_{\alpha}=\nu_{\alpha,i}$,
and the process terminates at time $\tilde{l}_{\alpha}:=i$. If the
process continues indefinitely, then $\tilde{l}_{\alpha}:=\infty$,
and let $\nu_{\alpha}$ be such that
\[
\frac{\mathrm{d}\nu_{\alpha}}{\mathrm{d}P_{\alpha}}(x)=\inf_{i\in\mathbb{N}}\frac{\mathrm{d}\nu_{\alpha,i}}{\mathrm{d}P_{\alpha}}(x).
\]
Note that the pointwise limit of measurable function is measurable.
We have $\nu_{\alpha,i}(\mathcal{X})\le1-(1/2)\sum_{j=1}^{i-1}t_{j}$,
and hence either the process stops at $t_{i}=0$, or $\lim_{i}t_{i}=0$.
Assume the contrary that there exists $\beta<\alpha$ such that $\nu_{\alpha}\cancel{\perp}P_{\beta}$.
Then $(\nu_{\alpha}\cap_{\mathrm{L}}P_{\beta})(\mathcal{X})>0$, and
$(\nu_{\alpha}\cap_{\mathrm{L}}P_{\beta})(\mathcal{X})>t_{i}$ for
some $i$, leading to a contradiction. Hence $\nu_{\alpha}\perp P_{\beta}$
(and $\nu_{\alpha}\perp\nu_{\beta}$) for all $\beta<\alpha$. Also
we have 
\begin{align*}
P_{\alpha} & =\nu_{\alpha}+\sum_{i=1}^{\tilde{l}_{\alpha}}\nu_{\alpha,i}\cap_{\mathrm{L}}P_{\tilde{\beta}_{\alpha,i}}\\
 & \ll\nu_{\alpha}+\sum_{i=1}^{\tilde{l}_{\alpha}}2^{-i-1}P_{\tilde{\beta}_{\alpha,i}}\\
 & \ll\nu_{\alpha}+\sum_{i=1}^{\tilde{l}_{\alpha}}2^{-i-1}\sum_{j=1}^{l_{\tilde{\beta}_{\alpha,i}}}2^{-j}\nu_{\beta_{\tilde{\beta}_{\alpha,i},j}},
\end{align*}
by the induction hypothesis. Therefore we can let $\{\beta_{\alpha,i}\}_{i\in[1..l_{\alpha}]}$
be the set $\{\alpha\}\cup\{\beta_{\tilde{\beta}_{\alpha,i},j}\}_{i,j}$.
Hence the measures $\{\nu_{\alpha}\}_{\alpha\in\mathcal{A}}$ can
be constructed by transfinite recursion.

By the Kolmogorov extension theorem, we can define $\Phi_{\alpha}\sim\mathrm{PP}(\nu_{\alpha}\times\lambda_{\mathbb{R}_{\ge0}})$
independent across $\alpha\in\mathcal{A}$ ($\{\Phi_{\alpha}\}_{\alpha\in\mathcal{A}}$
can be defined on the space $[0,1]^{\mathcal{A}}$ with the product
$\sigma$-algebra).  Let 
\[
X_{\alpha}:=\varrho_{P_{\alpha}\Vert\sum_{i=1}^{l_{\alpha}}2^{-i}\nu_{\beta_{\alpha,i}}}\left(\sum_{i=1}^{l_{\alpha}}\big((x,t)\mapsto(x,2^{i}t)\big)_{*}\Phi_{\beta_{\alpha,i}}\right),
\]
where $((x,t)\mapsto(x,2^{i}t))_{*}\Phi_{\beta_{\alpha,i}}$ denotes
the pushforward measure of $\Phi_{\beta_{\alpha,i}}$ (a random measure)
by the mapping $(x,t)\mapsto(x,2^{i}t)$. Note that $(\sum_{i=1}^{l_{\alpha}}2^{-i}\nu_{\beta_{\alpha,i}})(\mathcal{X})\le1$,
so the measure $\sum_{i=1}^{l_{\alpha}}2^{-i}\nu_{\beta_{\alpha,i}}$
is $\sigma$-finite. By the superposition theorem \cite{last2017lectures},
the point process given by the sum has distribution $\mathrm{PP}(\sum_{i=1}^{l_{\alpha}}2^{-i}\nu_{\beta_{\alpha,i}}\times\lambda_{\mathbb{R}_{\ge0}})$,
and hence $X_{\alpha}\sim P_{\alpha}$. For any $\alpha,\alpha'\in\mathcal{A}$,
let $\{\beta_{i}\}_{i\in[1..l]}=\{\beta_{\alpha,i}\}_{i\in[1..l_{\alpha}]}\cup\{\beta_{\alpha',i}\}_{i\in[1..l_{\alpha'}]}$.
It can be checked that
\[
X_{\alpha}=\varrho_{P_{\alpha}\Vert\sum_{i=1}^{l}2^{-i}\nu_{\beta_{i}}}\left(\sum_{i=1}^{l}\big((x,t)\mapsto(x,2^{i}t)\big)_{*}\Phi_{\beta_{i}}\right),
\]
and similarly for $X_{\alpha'}$. Hence $\mathbf{P}(X_{\alpha}\neq X_{\alpha'})=d_{\mathrm{PC}}(P_{\alpha},P_{\alpha'})\le2d_{\mathrm{TV}}(P_{\alpha},P_{\alpha'})$
by Lemma \ref{lem:pml} and Proposition \ref{prop:dpc_prop}.
\end{proof}
\[
\]

\section{Proof of Lemma \ref{lem:spfr_dist}\label{subsec:pf_hpfr_dist}}

Before we prove Lemma \ref{lem:spfr_dist}, we show that $\bar{\varrho}_{\bar{P}_{I}\Vert\mu_{I}}(\Phi_{I})$
is a random variable. We first show that if $\kappa$ is a probability
kernel from the measurable space $\mathcal{Y}$ to the Polish space
$\mathcal{X}$, $\mu$ is a $\sigma$-finite measure over $\mathcal{X}$,
$Y\sim Q$ independent of $\Phi\sim\mathrm{PP}(\mu\times\lambda_{\mathbb{R}_{\ge0}})$,
then $\varrho_{\kappa(\cdot|Y)\,\Vert\,\mu}(\Phi)$ is a random variable.
Since a Poisson process is a proper point process \cite[Corollary 6.5]{last2017lectures},
there exist random variables $(X_{1},T_{1}),(X_{2},T_{2}),\ldots\in\mathcal{X}\times\mathbb{R}_{\ge0}$
and a random variable $K\in\mathbb{Z}_{\ge0}\cup\{\infty\}$ such
that $\sum_{i=1}^{K}\delta_{X_{i}}=\Phi$ almost surely. Note that
$T_{i}((\mathrm{d}\kappa(\cdot|Y)/\mathrm{d}\mu)(X_{i}))^{-1}$ are
random variables for $i\in\mathbb{N}$ (since $(\mathrm{d}\kappa(\cdot|Y)/\mathrm{d}\mu)(X_{i})$
is a random variable by \cite[Exercise 6.10.72]{bogachev2007measure}),
and the argmin of random variables is a random variable. Therefore
$\varrho_{\kappa(\cdot|Y)\,\Vert\,\mu}(\Phi)$ is a random variable.
As a result, for the case $I=[k..l]$, $\bar{\varrho}_{\bar{P}_{I}\Vert\mu_{I},l}(\Phi_{I})$
is a random variable. For the case $I=(-\infty..l]$, we regard $\bar{\varrho}_{\bar{P}_{I}\Vert\mu_{I},l}(\Phi_{I})=\emptyset$
if the limit in \eqref{eq:spfr_kinfty} does not exist (where $\emptyset$
is regarded as a symbol not in $\mathcal{X}$, and $\bar{\varrho}_{\bar{P}_{I}\Vert\mu_{I},l}(\Phi_{I})$
is in the measurable space $\mathcal{X}\cup\{\emptyset\}$ with $\sigma$-algebra
$\sigma(\mathcal{F}\cup\{\{\emptyset\}\})$, where $\mathcal{F}$
is the $\sigma$-algebra of $\mathcal{X}$). Then $\bar{\varrho}_{\bar{P}_{I}\Vert\mu_{I},l}(\Phi_{I})$
is a random variable since the limit of a sequence of random variables
is a random variable, and the event that the limit does not exist
is measurable. The measurability for the case $\sup I=\infty$ follows
from the previous two cases by definition.

We consider each case separately:

\medskip{}

\noindent \textbf{Case} $I=[k..l]$\textbf{.} If $k=l$, then $\bar{\varrho}_{\{\bar{P}_{i|I_{<i}}\}_{i\in[l..l]}\Vert\mu_{[l..l]}}(\Phi_{[l..l]})=\varrho_{\bar{P}_{l}\Vert\mu_{l}}(\Phi_{l})\sim\bar{P}_{l}$
by the property of Poisson functional representation. Assume $\bar{\varrho}_{\{\bar{P}_{i|I_{<i}}\}_{i\in[k..l-1]}\Vert\mu_{[k..l-1]}}(\Phi_{[k..l-1]})\sim\bar{P}_{[k..l-1]}$.
By the construction in \eqref{eq:spfr}, 
\[
\bar{\varrho}_{\{\bar{P}_{i|I_{<i}}\}_{i\in[k..l]}\Vert\mu_{[k..l]},l}(\Phi_{[k..l]})|\bar{\varrho}_{\{\bar{P}_{i|I_{<i}}\}_{i\in[k..l-1]}\Vert\mu_{[k..l-1]}}(\Phi_{[k..l-1]})\sim\bar{P}_{l|[k..l-1]}.
\]
Hence $\bar{\varrho}_{\{\bar{P}_{i|I_{<i}}\}_{i\in[k..l]}\Vert\mu_{[k..l]},l}(\Phi_{[k..l]})\sim\bar{P}_{[k..l]}$.
Claim \ref{enu:spfr_dist_dist} follows from induction. Claim \ref{enu:spfr_dist_rcd_welldef}
can be proved similarly using induction. Claim \ref{enu:spfr_dist_subint}
follows directly from the definition.

\medskip{}

\noindent \textbf{Case} $I=(-\infty..l]$\textbf{.} We regard $\bar{\varrho}_{\{\bar{P}_{i'|I_{<i'}}\}_{i'\in I}\Vert\mu_{I},l}(\phi_{I})=\emptyset$
if the limit in \eqref{eq:spfr_kinfty} does not exist (where $\emptyset$
is regarded as a symbol not in $\mathcal{X}$). Let $\epsilon>0$.
By \eqref{eq:spfr_kinfty_lim} and \eqref{eq:spfr_kinfty_lim2}, let
$i_{\epsilon}\in I$ such that 
\begin{equation}
\sum_{i\in I_{<i_{\epsilon}}}d_{\mathrm{TV}}\left(\bigg(\prod_{j\in I_{<i}}\nu_{j}\bigg)\bar{P}_{i|I_{<i}},\,\prod_{j\in I_{\le i}}\nu_{j}\right)<\epsilon,\label{eq:pf_spfr_dist_dtv}
\end{equation}
and
\begin{equation}
d_{\mathrm{TV}}\Bigg(\bar{P}_{I_{<i_{\epsilon}}},\,\prod_{j\in I_{<i_{\epsilon}}}\nu_{j}\Bigg)<\epsilon.\label{eq:pf_spfr_dist_dtv2}
\end{equation}
Let $\tilde{Z}_{\epsilon}\in\mathcal{X}^{I}$ be defined by
\begin{equation}
\tilde{Z}_{\epsilon,i}:=\begin{cases}
\varrho_{\nu_{i}\Vert\mu_{i}}(\Phi_{i}) & \,\mathrm{if}\;i<i_{\epsilon}\\
\bar{\varrho}_{\{\bar{P}_{i'|I_{<i'}}(\cdot|(\{\varrho_{\nu_{j}\Vert\mu_{j}}(\Phi_{j})\}_{j<i_{\epsilon}},\cdot))\}_{i'\in[i_{\epsilon}..i]}\Vert\mu_{[i_{\epsilon}..i]},i}(\Phi_{[i_{\epsilon}..i]}) & \,\mathrm{if}\;i\ge i_{\epsilon}.
\end{cases}\label{eq:pf_spfr_zei}
\end{equation}
By \eqref{eq:spfr}, if
\[
\varrho_{\bar{P}_{i|I_{<i}}\big(\cdot\,|\,\{\varrho_{\nu_{j}\Vert\mu_{j}}(\Phi_{j})\}_{j<i}\big)\Vert\mu_{i}}(\Phi_{i})=\varrho_{\nu_{i}\Vert\mu_{i}}(\Phi_{i})\;\forall i<i_{\epsilon},
\]
then 
\[
\bar{\varrho}_{\{\bar{P}_{i'|I_{<i'}}(\cdot|(\{\varrho_{\nu_{j}\Vert\mu_{j}}(\Phi_{j})\}_{j<k},\cdot))\}_{i'\in[k..i]}\Vert\mu_{[k..i]},i}(\Phi_{[k..i]})=\tilde{Z}_{\epsilon,i}\;\forall i\in I,\,k<\min\{i_{\epsilon},i\},
\]
and hence by \eqref{eq:spfr_kinfty},
\begin{align*}
 & \bar{\varrho}_{\{\bar{P}_{i'|I_{<i'}}\}_{i'\in I_{\le i}}\Vert\mu_{I_{\le i}},i}(\Phi_{I_{\le i}})\\
 & =\lim_{k\to-\infty}\bar{\varrho}_{\{\bar{P}_{i'|I_{<i'}}(\cdot|(\{\varrho_{\nu_{j}\Vert\mu_{j}}(\Phi_{j})\}_{j<k},\cdot))\}_{i'\in[k..i]}\Vert\mu_{[k..i]},i}(\Phi_{[k..i]})\\
 & =\tilde{Z}_{\epsilon,i}\;\forall i\in I,
\end{align*}
and hence $\bar{\varrho}_{\{\bar{P}_{i'|I_{<i'}}\}_{i'\in I}\Vert\mu_{I}}(\Phi_{I})=\tilde{Z}_{\epsilon}$.
Therefore,
\begin{align}
 & \mathbf{P}\Big(\bar{\varrho}_{\{\bar{P}_{i'|I_{<i'}}\}_{i'\in I}\Vert\mu_{I}}(\Phi_{I})\neq\tilde{Z}_{\epsilon}\Big)\nonumber \\
 & \le\mathbf{P}\left(\exists i<i_{\epsilon}:\,\varrho_{\bar{P}_{i|I_{<i}}(\cdot\,|\,\{\varrho_{\nu_{j}\Vert\mu_{j}}(\Phi_{j})\}_{j<i})\Vert\mu_{i}}(\Phi_{i})\neq\varrho_{\nu_{i}\Vert\mu_{i}}(\Phi_{i})\right)\nonumber \\
 & \le\sum_{i<i_{\epsilon}}\mathbf{P}\left(\varrho_{\bar{P}_{i|I_{<i}}(\cdot\,|\,\{\varrho_{\nu_{j}\Vert\mu_{j}}(\Phi_{j})\}_{j<i})\Vert\mu_{i}}(\Phi_{i})\neq\varrho_{\nu_{i}\Vert\mu_{i}}(\Phi_{i})\right)\nonumber \\
 & \stackrel{(a)}{=}\sum_{i<i_{\epsilon}}\mathbf{E}\left[d_{\mathrm{PC}}\left(\bar{P}_{i|I_{<i}}\big(\cdot\,|\,\{\varrho_{\nu_{j}\Vert\mu_{j}}(\Phi_{j})\}_{j<i}\big),\,\nu_{i}\right)\right]\nonumber \\
 & \stackrel{(b)}{\le}2\sum_{i<i_{\epsilon}}\mathbf{E}\left[d_{\mathrm{TV}}\left(\bar{P}_{i|I_{<i}}\big(\cdot\,|\,\{\varrho_{\nu_{j}\Vert\mu_{j}}(\Phi_{j})\}_{j<i}\big),\,\nu_{i}\right)\right]\nonumber \\
 & \stackrel{(c)}{=}2\sum_{i<i_{\epsilon}}d_{\mathrm{TV}}\left(\bigg(\prod_{j<i}\nu_{j}\bigg)\bar{P}_{i|I_{<i}},\,\prod_{j\le i}\nu_{j}\right)\nonumber \\
 & \stackrel{(d)}{<}2\epsilon,\label{eq:pf_spfr_dist_peq}
\end{align}
where (a) is by \eqref{eq:dpc_pneq}, (b) is by Proposition \ref{prop:dpc_prop},
(c) is because $\{\varrho_{\nu_{j}\Vert\mu_{j}}(\Phi_{j})\}_{j<i}\sim\prod_{j<i}\nu_{j}$,
and (d) is by \eqref{eq:pf_spfr_dist_dtv}.

Also note that 
\[
\tilde{Z}_{\epsilon}\sim\bigg(\prod_{j<i_{\epsilon}}\nu_{j}\bigg)\bar{P}_{i_{\epsilon}|I_{<i_{\epsilon}}}\bar{P}_{i_{\epsilon}+1|I_{<i_{\epsilon}+1}}\cdots\bar{P}_{l|I_{<l}}.
\]
Hence, by \eqref{eq:pf_spfr_dist_dtv2},
\begin{align*}
 & d_{\mathrm{TV}}\left(\bigg(\prod_{j<i_{\epsilon}}\nu_{j}\bigg)\bar{P}_{i_{\epsilon}|I_{<i_{\epsilon}}}\bar{P}_{i_{\epsilon}+1|I_{<i_{\epsilon}+1}}\cdots\bar{P}_{l|I_{<l}},\,\bar{P}\right)\\
 & =d_{\mathrm{TV}}\left(\bigg(\prod_{j<i_{\epsilon}}\nu_{j}\bigg),\,\bar{P}_{I_{<i_{\epsilon}}}\right)\\
 & <\epsilon.
\end{align*}
Combining this with \eqref{eq:pf_spfr_dist_peq}, and letting $\epsilon\to0$,
we have proved Claim \ref{enu:spfr_dist_dist} and Claim \ref{enu:spfr_dist_kinfty_match},
i.e., $\bar{\varrho}_{\{\bar{P}_{i'|I_{<i'}}\}_{i'\in I}\Vert\mu_{I}}(\Phi_{I})\sim\bar{P}$,
is defined almost surely, and satisfies \eqref{eq:spfr_dist_agree}
(since $\mathbf{P}(\bar{\varrho}_{\{\bar{P}_{i'|I_{<i'}}\}_{i'\in I}\Vert\mu_{I},i}(\Phi_{i})=\varrho_{\nu_{i}\Vert\mu_{i}}(\Phi_{i})\,\forall i<i_{\epsilon})\ge\mathbf{P}(\bar{\varrho}_{\{\bar{P}_{i'|I_{<i'}}\}_{i'\in I}\Vert\mu_{I}}(\Phi_{I})=\tilde{Z}_{\epsilon})\ge1-2\epsilon$).

We now prove Claim \ref{enu:spfr_dist_rcd_welldef} that $\bar{\varrho}_{\{\bar{P}_{i'|I_{<i'}}\}_{i'\in I}\Vert\mu_{I}}(\Phi_{I})$
does not depend on the choice of $\{\bar{P}_{i'|I_{<i'}}\}_{i'\in I}$.
Let $\{\bar{P}_{i|I_{<i}}\}_{i\in I}$ and $\{\bar{P}'_{i|I_{<i}}\}_{i\in I}$
be regular conditional distributions of $\bar{P}$ satisfying \eqref{eq:spfr_abscont_kernel}
and \eqref{eq:spfr_kinfty_lim}. Fix any $\epsilon>0$. Let $i_{\epsilon}\in I$
such that \eqref{eq:pf_spfr_dist_dtv2} is satisfied. Let $\{Z_{i}\}_{i\in I}\sim\bar{P}$.
By \eqref{eq:spfr_kinfty},
\begin{align*}
 & \mathbf{P}\Big(\bar{\varrho}_{\{\bar{P}_{i'|I_{<i'}}\}_{i'\in I}\Vert\mu_{I},l}(\Phi_{I})\neq\bar{\varrho}_{\{\bar{P}'_{i'|I_{<i'}}\}_{i'\in I}\Vert\mu_{I},l}(\Phi_{I})\Big)\\
 & \le\mathbf{P}\bigg(\exists k<i_{\epsilon}:\,\bar{\varrho}_{\{\bar{P}_{i|I_{<i}}(\cdot|(\{\varrho_{\nu_{j}\Vert\mu_{j}}(\Phi_{j})\}_{j<k},\cdot))\}_{i\in[k..l]}\Vert\mu_{[k..l]},l}(\Phi_{[k..l]})\\
 & \;\;\;\;\;\;\;\;\neq\bar{\varrho}_{\{\bar{P}'_{i|I_{<i}}(\cdot|(\{\varrho_{\nu_{j}\Vert\mu_{j}}(\Phi_{j})\}_{j<k},\cdot))\}_{i\in[k..l]}\Vert\mu_{[k..l]},l}(\Phi_{[k..l]})\bigg)\\
 & \stackrel{(a)}{\le}\mathbf{P}\left(\exists k<i_{\epsilon}:\,\prod_{i=k}^{l}\bar{P}_{i|I_{<i}}(\cdot|(\{\varrho_{\nu_{j}\Vert\mu_{j}}(\Phi_{j})\}_{j<k},\cdot))\neq\prod_{i=k}^{l}\bar{P}'_{i|I_{<i}}(\cdot|(\{\varrho_{\nu_{j}\Vert\mu_{j}}(\Phi_{j})\}_{j<k},\cdot))\right)\\
 & \stackrel{(b)}{\le}\mathbf{P}\left(\exists k<i_{\epsilon}:\,\prod_{i=k}^{l}\bar{P}_{i|I_{<i}}(\cdot|(\{Z_{j}\}_{j<k},\cdot))\neq\prod_{i=k}^{l}\bar{P}'_{i|I_{<i}}(\cdot|(\{Z_{j}\}_{j<k},\cdot))\right)+\epsilon\\
 & \stackrel{(c)}{=}\epsilon,
\end{align*}
where the product on the left hand side in (a) is the semidirect product
between a probability distribution $\bar{P}_{k|I_{<k}}(\cdot|\{\varrho_{\nu_{j}\Vert\mu_{j}}(\Phi_{j})\}_{j<k})$
and a sequence of probability kernels. Note that (a) is because if
the semidirect products coincide, then $\{\bar{P}_{i|I_{<i}}(\cdot|(\{\varrho_{\nu_{j}\Vert\mu_{j}}(\Phi_{j})\}_{j<k},\cdot))\}_{i\in[k..l]}$
and $\{\bar{P}'_{i|I_{<i}}(\cdot|(\{\varrho_{\nu_{j}\Vert\mu_{j}}(\Phi_{j})\}_{j<k},\cdot))\}_{i\in[k..l]}$
are two sequences of regular conditional distributions of that semidirect
product distribution, and hence the two $\bar{\varrho}$'s in the
previous line coincide almost surely since we have already proved
that the choice of regular conditional distribuions does not matter
for the case $I=[k..l]$. For (b), it is due to \eqref{eq:pf_spfr_dist_dtv2},
and thus replacing $\{\varrho_{\nu_{j}\Vert\mu_{j}}(\Phi_{j})\}_{j<i_{\epsilon}}$
(which has distribution $\prod_{j<i_{\epsilon}}\nu_{j}$) by $\{Z_{j}\}_{j<i_{\epsilon}}$
affects the probability by at most $\epsilon$. For (c), it is due
to the uniqueness of regular conditional distributions in the sense
that $\bar{P}_{i|I_{<i}}(\cdot|\{z_{j}\}_{j<i})=\bar{P}'_{i|I_{<i}}(\cdot|\{z_{j}\}_{j<i})$
for $\bar{P}$-almost all $\{z_{i}\}_{i\in I}$. The result follows
from letting $\epsilon\to0$.

For Claim \ref{enu:spfr_dist_recur}, we check \eqref{eq:spfr} by
\begin{align*}
 & \mathbf{P}\Big(\bar{\varrho}_{\bar{P}_{I}\Vert\mu_{I},l}(\Phi_{I})\neq\varrho_{\bar{P}_{l|I_{<l}}(\cdot\,|\,\bar{\varrho}_{\bar{P}_{I_{<l}}\Vert\mu_{I_{<l}}}(\Phi_{I_{<l}}))\,\Vert\,\mu_{l}}(\Phi_{l})\Big)\\
 & \stackrel{(a)}{\le}\mathbf{P}\Big(\tilde{Z}_{\epsilon,l}\neq\varrho_{\bar{P}_{l|I_{<l}}(\cdot\,|\,\{\tilde{Z}_{\epsilon,i}\}_{i<l})\,\Vert\,\mu_{l}}(\Phi_{l})\Big)+2\epsilon\\
 & \stackrel{(b)}{=}\mathbf{P}\Big(\bar{\varrho}_{\{\bar{P}_{i'|I_{<i'}}(\cdot|(\{\varrho_{\nu_{j}\Vert\mu_{j}}(\Phi_{j})\}_{j<i_{\epsilon}},\cdot))\}_{i'\in[i_{\epsilon}..l]}\Vert\mu_{[i_{\epsilon}..l]},l}(\Phi_{[i_{\epsilon}..l]})\neq\varrho_{\bar{P}_{l|I_{<l}}(\cdot\,|\,\{\tilde{Z}_{\epsilon,i}\}_{i<l})\,\Vert\,\mu_{l}}(\Phi_{l})\Big)+2\epsilon\\
 & \stackrel{(c)}{=}\mathbf{P}\Big(\varrho_{\bar{P}_{l|I_{<l}}(\cdot\,|\,\{\tilde{Z}_{\epsilon,i}\}_{i<l})\,\Vert\,\mu_{l}}(\Phi_{l})\neq\varrho_{\bar{P}_{l|I_{<l}}(\cdot\,|\,\{\tilde{Z}_{\epsilon,i}\}_{i<l})\,\Vert\,\mu_{l}}(\Phi_{l})\Big)+2\epsilon\\
 & =2\epsilon,
\end{align*}
where (a) is by \eqref{eq:pf_spfr_dist_peq}, (b) is by the definition
of $\tilde{Z}_{\epsilon}$, and (c) is by \eqref{eq:spfr} and the
definition of $\tilde{Z}_{\epsilon}$. Letting $\epsilon\to0$, we
can deduce that \eqref{eq:spfr} is satisfied almost surely.

Claim \ref{enu:spfr_dist_subint} follows directly from the definition.

\medskip{}

\noindent \textbf{Case} $\sup I=\infty$\textbf{.} Since we have proved
that $\bar{\varrho}_{\bar{P}_{I_{\le i}}\Vert\mu_{I_{\le i}}}(\phi_{I_{\le i}})\sim\bar{P}_{I_{\le i}}$
for any $i\in I$, and the distribution of a random process is characterized
by its finite dimensional marginals, we have $\bar{\varrho}_{\bar{P}_{I}\Vert\mu_{I}}(\phi_{I})\sim\bar{P}$.
To check \eqref{eq:spfr_dist_agree} for the case $I=(-\infty..\infty)$,
it suffices to check that $\bar{\varrho}_{\bar{P}_{I_{\le0}}\Vert\mu_{I_{\le0}}}(\phi_{I_{\le0}})$
satisfies \eqref{eq:spfr_dist_agree}, which is proved in the previous
case. All the claims continue to hold.

\[
\]

\section{Proof of the SPFR Condition for Theorem \ref{thm:metric_pow}\label{subsec:pf_metric_pow_spfr}}

We check that the SPFR condition is satisfied for $\bar{P}:=(\prod_{i\in\mathbb{Z}}\mathrm{U}\mathcal{B}_{e^{-\eta(i+\theta)}})\circ P$
for any distribution $P$ over $\mathcal{X}$, where $\theta\in[0,1]$,
$\eta>0$. Let the Borel $\sigma$-algebra of $\mathcal{X}$ be $\mathcal{F}$.
Let $X\sim P$, and $Z_{i}|X\sim\mathrm{U}\mathcal{B}_{e^{-\eta(i+\theta)}}$
be conditionally independent across $i$ given $X$. For $J\subseteq\mathbb{Z}$,
write $\bar{P}_{X|Z_{J}}$ for a regular conditional distribution
of $X$ given $\{Z_{j}\}_{j\in J}$. We choose the regular conditional
distribution of $Z_{i}$ given $\{Z_{j}\}_{j<i}$ to be
\[
\bar{P}_{i|I_{<i}}(\cdot|\{z_{j}\}_{j})=\mathrm{U}\mathcal{B}_{e^{-\eta(i+\theta)}}\circ\bar{P}_{X|Z_{(-\infty..i)}}(\cdot|\{z_{j}\}_{j})
\]
for any $\{z_{j}\}_{j}\in\mathcal{X}^{(-\infty,i)}$, where the precise
choice of the regular conditional distribution $\bar{P}_{X|Z_{(-\infty..i)}}$
is given at the end of Remark \ref{rem:metric_pow_pf_bayesjust}.
By definition, \eqref{eq:spfr_abscont_kernel} is satisfied.

By \eqref{eq:metric_pow_limsup}, we can pick $i_{0}\in\mathbb{Z}$
small enough such that
\begin{equation}
\xi:=\sup_{w\ge e^{-\eta(i_{0}+1)},\,x,y\in\mathcal{X}:\,d(x,y)\le w}\frac{\mu(\mathcal{B}_{w}(x))}{\mu(\mathcal{B}_{w}(y))}<\infty.\label{eq:metric_pow_pf_xi}
\end{equation}
Let
\[
\tilde{\xi}:=2\xi^{2\lceil\eta^{-1}\log(4\Psi/(1-e^{-\eta}))\rceil}.
\]

For \eqref{eq:spfr_kinfty_lim},
\begin{align}
 & \sum_{i\in(-\infty..i_{0}]}d_{\mathrm{TV}}\left(\bigg(\prod_{j\in I_{<i}}\nu_{j}\bigg)\bar{P}_{i|I_{<i}},\,\prod_{j\in I_{\le i}}\nu_{j}\right)\nonumber \\
 & =\sum_{i\in(-\infty..i_{0}]}\int d_{\mathrm{TV}}\left(\bar{P}_{i|I_{<i}}(\cdot|\{z_{j}\}_{j}),\,\nu_{i}\right)\Big(\prod_{j<i}\nu_{j}\Big)(\mathrm{d}\{z_{j}\}_{j})\nonumber \\
 & =\sum_{i\in(-\infty..i_{0}]}\int d_{\mathrm{TV}}\left(\mathrm{U}\mathcal{B}_{e^{-\eta(i+\theta)}}\circ\bar{P}_{X|Z_{(-\infty..i)}}(\cdot|\{z_{j}\}_{j}),\,\mathrm{U}\mathcal{B}_{e^{-\eta(i+\theta)}}(x_{0},\cdot)\right)\Big(\prod_{j<i}\mathrm{U}\mathcal{B}_{e^{-\eta(j+\theta)}}(x_{0},\cdot)\Big)(\mathrm{d}\{z_{j}\}_{j})\nonumber \\
 & \le\sum_{i\in(-\infty..i_{0}]}\int\Bigg(\mathbf{1}\bigg\{ P\Big(\bigcap_{j<i}\mathcal{B}_{e^{-\eta(j+\theta)}}(z_{j})\Big)\ge\frac{1}{2}\bigg\} d_{\mathrm{TV}}\Big(\mathrm{U}\mathcal{B}_{e^{-\eta(i+\theta)}}\circ\bar{P}_{X|Z_{(-\infty..i)}}(\cdot|\{z_{j}\}_{j}),\,\mathrm{U}\mathcal{B}_{e^{-\eta(i+\theta)}}(x_{0},\cdot)\Big)\nonumber \\
 & \;\;\;\;\;\;\;\;+\mathbf{1}\bigg\{ P\Big(\bigcap_{j<i}\mathcal{B}_{e^{-\eta(j+\theta)}}(z_{j})\Big)<\frac{1}{2}\bigg\}\Bigg)\Big(\prod_{j<i}\mathrm{U}\mathcal{B}_{e^{-\eta(j+\theta)}}(x_{0},\cdot)\Big)(\mathrm{d}\{z_{j}\}_{j}).\label{eq:metric_pow_pf_terms}
\end{align}
We consider the two terms separately. First,
\begin{align}
 & \sum_{i\in(-\infty..i_{0}]}\int\mathbf{1}\bigg\{ P\Big(\bigcap_{j<i}\mathcal{B}_{e^{-\eta(j+\theta)}}(z_{j})\Big)\ge\frac{1}{2}\bigg\} d_{\mathrm{TV}}\Big(\mathrm{U}\mathcal{B}_{e^{-\eta(i+\theta)}}\circ\bar{P}_{X|Z_{(-\infty..i)}}(\cdot|\{z_{j}\}_{j}),\,\mathrm{U}\mathcal{B}_{e^{-\eta(i+\theta)}}(x_{0},\cdot)\Big)\nonumber \\
 & \;\;\;\;\;\;\;\;\cdot\Big(\prod_{j<i}\mathrm{U}\mathcal{B}_{e^{-\eta(j+\theta)}}(x_{0},\cdot)\Big)(\mathrm{d}\{z_{j}\}_{j})\nonumber \\
 & \stackrel{(a)}{\le}\sum_{i\in(-\infty..i_{0}]}\int\mathbf{1}\bigg\{ P\Big(\bigcap_{j<i}\mathcal{B}_{e^{-\eta(j+\theta)}}(z_{j})\Big)\ge\frac{1}{2}\bigg\}\int d_{\mathrm{TV}}\left(\mathrm{U}\mathcal{B}_{e^{-\eta(i+\theta)}}(x,\cdot),\,\mathrm{U}\mathcal{B}_{e^{-\eta(i+\theta)}}(x_{0},\cdot)\right)\nonumber \\
 & \;\;\;\;\;\;\;\;\cdot\bar{P}_{X|Z_{(-\infty..i)}}(\mathrm{d}x|\{z_{j}\}_{j})\Big(\prod_{j<i}\mathrm{U}\mathcal{B}_{e^{-\eta(j+\theta)}}(x_{0},\cdot)\Big)(\mathrm{d}\{z_{j}\}_{j})\nonumber \\
 & \stackrel{(b)}{\le}\sum_{i\in(-\infty..i_{0}]}\int\mathbf{1}\bigg\{ P\Big(\bigcap_{j<i}\mathcal{B}_{e^{-\eta(j+\theta)}}(z_{j})\Big)\ge\frac{1}{2}\bigg\}\nonumber \\
 & \;\;\;\;\;\;\;\;\cdot\int\Psi e^{\eta(i+\theta)}d(x,x_{0})\bar{P}_{X|Z_{(-\infty..i)}}(\mathrm{d}x|\{z_{j}\}_{j})\Big(\prod_{j<i}\mathrm{U}\mathcal{B}_{e^{-\eta(j+\theta)}}(x_{0},\cdot)\Big)(\mathrm{d}\{z_{j}\}_{j})\nonumber \\
 & \stackrel{(c)}{\le}\tilde{\xi}\sum_{i\in(-\infty..i_{0}]}\int\mathbf{1}\bigg\{ P\Big(\bigcap_{j<i}\mathcal{B}_{e^{-\eta(j+\theta)}}(z_{j})\Big)\ge\frac{1}{2}\bigg\}\nonumber \\
 & \;\;\;\;\;\;\;\;\cdot\int\Psi e^{\eta(i+\theta)}d(x,x_{0})P\left(\mathrm{d}x\left|\bigcap_{j<i}\mathcal{B}_{e^{-\eta(j+\theta)}}(z_{j})\right.\right)\Big(\prod_{j<i}\mathrm{U}\mathcal{B}_{e^{-\eta(j+\theta)}}(x_{0},\cdot)\Big)(\mathrm{d}\{z_{j}\}_{j})\nonumber \\
 & \le\tilde{\xi}\sum_{i\in(-\infty..i_{0}]}2\int\int\Psi e^{\eta(i+\theta)}\mathbf{1}\bigg\{ x\in\bigcap_{j<i}\mathcal{B}_{e^{-\eta(j+\theta)}}(z_{j})\bigg\} d(x,x_{0})P(\mathrm{d}x)\Big(\prod_{j<i}\mathrm{U}\mathcal{B}_{e^{-\eta(j+\theta)}}(x_{0},\cdot)\Big)(\mathrm{d}\{z_{j}\}_{j})\nonumber \\
 & \stackrel{(d)}{\le}\tilde{\xi}\sum_{i\in(-\infty..i_{0}]}2\int\int\Psi e^{\eta(i+\theta)}\mathbf{1}\left\{ x\in\mathcal{B}_{2e^{-\eta(i-1+\theta)}}(x_{0})\right\} d(x,x_{0})P(\mathrm{d}x)\Big(\prod_{j<i}\mathrm{U}\mathcal{B}_{e^{-\eta(j+\theta)}}(x_{0},\cdot)\Big)(\mathrm{d}\{z_{j}\}_{j})\nonumber \\
 & =2\tilde{\xi}\Psi\int\left(\sum_{i\in(-\infty..i_{0}]}e^{\eta(i+\theta)}\mathbf{1}\left\{ x\in\mathcal{B}_{2e^{-\eta(i-1+\theta)}}(x_{0})\right\} \right)d(x,x_{0})P(\mathrm{d}x)\nonumber \\
 & \le2\tilde{\xi}\Psi\int\int_{-\infty}^{\infty}e^{\eta(t+1+\theta)}\mathbf{1}\left\{ x\in\mathcal{B}_{2e^{-\eta(t-1+\theta)}}(x_{0})\right\} \mathrm{d}t\cdot d(x,x_{0})P(\mathrm{d}x)\nonumber \\
 & =2\tilde{\xi}\Psi\int2\eta^{-1}e^{2\eta}(d(x,x_{0}))^{-1}d(x,x_{0})P(\mathrm{d}x)\nonumber \\
 & =4\tilde{\xi}\Psi\eta^{-1}e^{3\eta},\label{eq:metric_pow_pf_term1}
\end{align}
where (a) is by the convexity of $d_{\mathrm{TV}}$, (b) is by \eqref{eq:metric_pow_delta_dist}
and \eqref{eq:metric_pow_pf_dtv}, and (d) is because $z_{i-1}\in\mathcal{B}_{e^{-\eta(i-1+\theta)}}(x_{0})$,
and hence $\mathcal{B}_{e^{-\eta(i-1+\theta)}}(z_{i-1})\subseteq\mathcal{B}_{2e^{-\eta(i-1+\theta)}}(x_{0})$.
For (c), it is because for any $x,y\in\bigcap_{j<i}\mathcal{B}_{e^{-\eta(j+\theta)}}(z_{j})$,
their likelihood ratio is bounded by

\begin{align}
 & \prod_{j<i}\frac{\left(\mathrm{d}\mathrm{U}\mathcal{B}_{e^{-\eta(j+\theta)}}(x,\cdot)/\mathrm{d}\mu\right)(z_{j})}{\left(\mathrm{d}\mathrm{U}\mathcal{B}_{e^{-\eta(j+\theta)}}(y,\cdot)/\mathrm{d}\mu\right)(z_{j})}\nonumber \\
 & =\prod_{j<i}\frac{\left(\mathrm{d}(\mu_{\mathcal{B}_{e^{-\eta(j+\theta)}}(x)}/\mu(\mathcal{B}_{e^{-\eta(j+\theta)}}(x))/\mathrm{d}\mu\right)(z_{j})}{\left(\mathrm{d}(\mu_{\mathcal{B}_{e^{-\eta(j+\theta)}}(y)}/\mu(\mathcal{B}_{e^{-\eta(j+\theta)}}(y))/\mathrm{d}\mu\right)(z_{j})}\nonumber \\
 & =\prod_{j<i}\frac{(\mu(\mathcal{B}_{e^{-\eta(j+\theta)}}(x)))^{-1}}{(\mu(\mathcal{B}_{e^{-\eta(j+\theta)}}(y)))^{-1}}\nonumber \\
 & \stackrel{(e)}{\le}\xi^{2\lceil\eta^{-1}\log(4\Psi/(1-e^{-\eta}))\rceil}\prod_{j<i-\lceil\eta^{-1}\log(4\Psi/(1-e^{-\eta}))\rceil}\frac{\mu(\mathcal{B}_{e^{-\eta(j+\theta)}}(y))}{\mu(\mathcal{B}_{e^{-\eta(j+\theta)}}(x))}\nonumber \\
 & \stackrel{(f)}{\le}\xi^{2\lceil\eta^{-1}\log(4\Psi/(1-e^{-\eta}))\rceil}\prod_{j<i-\lceil\eta^{-1}\log(4\Psi/(1-e^{-\eta}))\rceil}\left(1-2\Psi e^{\eta(j-i+1)}\right)^{-1}\nonumber \\
 & \le\xi^{2\lceil\eta^{-1}\log(4\Psi/(1-e^{-\eta}))\rceil}\left(1-\sum_{j<i-\lceil\eta^{-1}\log(4\Psi/(1-e^{-\eta}))\rceil}2\Psi e^{\eta(j-i+1)}\right)^{-1}\nonumber \\
 & =\xi^{2\lceil\eta^{-1}\log(4\Psi/(1-e^{-\eta}))\rceil}\left(1-\frac{2\Psi e^{-\eta\lceil\eta^{-1}\log(4\Psi/(1-e^{-\eta}))\rceil}}{1-e^{-\eta}}\right)^{-1}\nonumber \\
 & \le2\xi^{2\lceil\eta^{-1}\log(4\Psi/(1-e^{-\eta}))\rceil}\nonumber \\
 & =\tilde{\xi},\label{eq:metric_pow_f_likelihoodbd}
\end{align}
where (e) is by \eqref{eq:metric_pow_pf_xi} and $\mu(\mathcal{B}_{e^{-\eta(j+\theta)}}(y))\le\xi\mu(\mathcal{B}_{e^{-\eta(j+\theta)}}(z_{j}))\le\xi^{2}\mu(\mathcal{B}_{e^{-\eta(j+\theta)}}(x))$,
and (f) is because
\begin{align*}
\mu(\mathcal{B}_{e^{-\eta(j+\theta)}}(y)) & \le\mu(\mathcal{B}_{e^{-\eta(j+\theta)}}(x))+\mu(\mathcal{B}_{e^{-\eta(j+\theta)}}(y)\backslash\mathcal{B}_{e^{-\eta(j+\theta)}}(x))\\
 & \le\mu(\mathcal{B}_{e^{-\eta(j+\theta)}}(x))+\Psi e^{\eta(j+\theta)}d(x,y)\mu(\mathcal{B}_{e^{-\eta(j+\theta)}}(y))\\
 & \le\mu(\mathcal{B}_{e^{-\eta(j+\theta)}}(x))+\Psi e^{\eta(j+\theta)}2e^{-\eta(i-1+\theta)}\mu(\mathcal{B}_{e^{-\eta(j+\theta)}}(y))\\
 & =\mu(\mathcal{B}_{e^{-\eta(j+\theta)}}(x))+2\Psi e^{\eta(j-i+1)}\mu(\mathcal{B}_{e^{-\eta(j+\theta)}}(y))
\end{align*}
by \eqref{eq:metric_pow_delta_dist} and $x,y\in\mathcal{B}_{e^{-\eta(i-1+\theta)}}(z_{i-1})$.
Since the likelihood ratio is bounded by $\tilde{\xi}$, by Bayes'
rule, 
\begin{equation}
\bar{P}_{X|Z_{(-\infty..i)}}(\cdot|\{z_{j}\}_{j})\le\tilde{\xi}P\big(\cdot\big|\bigcap_{j<i}\mathcal{B}_{e^{-\eta(j+\theta)}}(z_{j})\big)\label{eq:metric_pow_pf_bayescomp}
\end{equation}
whenever $P(\bigcap_{j<i}\mathcal{B}_{e^{-\eta(j+\theta)}}(z_{j}))>0$,
where the comparison is between measures (see Remark \ref{rem:metric_pow_pf_bayesjust}
for a formal justification), and thus (c) holds.

For the second term in \eqref{eq:metric_pow_pf_terms}, let $\tau>0$
be such that $P(\mathcal{B}_{\tau}(x_{0}))>3/4$ (such $\tau$ exists
since $\sum_{j=0}^{\infty}P(\mathcal{B}_{j}(x_{0})\backslash\mathcal{B}_{j-1}(x_{0}))=P(\bigcup_{j=0}^{\infty}(\mathcal{B}_{j}(x_{0})\backslash\mathcal{B}_{j-1}(x_{0})))=1$).
We have
\begin{align*}
 & \sum_{i\in(-\infty..i_{0}]}\int\mathbf{1}\bigg\{ P\Big(\bigcap_{j<i}\mathcal{B}_{e^{-\eta(j+\theta)}}(z_{j})\Big)<\frac{1}{2}\bigg\}\Big(\prod_{j<i}\mathrm{U}\mathcal{B}_{e^{-\eta(j+\theta)}}(x_{0},\cdot)\Big)(\mathrm{d}\{z_{j}\}_{j})\\
 & \le\sum_{i\in(-\infty..i_{0}]}\int\mathbf{1}\bigg\{ P\Big(\mathcal{B}_{\tau}(x_{0})\backslash\bigcap_{j<i}\mathcal{B}_{e^{-\eta(j+\theta)}}(z_{j})\Big)>\frac{1}{4}\bigg\}\Big(\prod_{j<i}\mathrm{U}\mathcal{B}_{e^{-\eta(j+\theta)}}(x_{0},\cdot)\Big)(\mathrm{d}\{z_{j}\}_{j})\\
 & \le4\sum_{i\in(-\infty..i_{0}]}\int P\Big(\mathcal{B}_{\tau}(x_{0})\backslash\bigcap_{j<i}\mathcal{B}_{e^{-\eta(j+\theta)}}(z_{j})\Big)\Big(\prod_{j<i}\mathrm{U}\mathcal{B}_{e^{-\eta(j+\theta)}}(x_{0},\cdot)\Big)(\mathrm{d}\{z_{j}\}_{j})\\
 & =4\sum_{i\in(-\infty..i_{0}]}\int\int_{\mathcal{B}_{\tau}(x_{0})}\mathbf{1}\Big\{ x\notin\bigcap_{j<i}\mathcal{B}_{e^{-\eta(j+\theta)}}(z_{j})\Big\} P(\mathrm{d}x)\Big(\prod_{j<i}\mathrm{U}\mathcal{B}_{e^{-\eta(j+\theta)}}(x_{0},\cdot)\Big)(\mathrm{d}\{z_{j}\}_{j})\\
 & =4\int_{\mathcal{B}_{\tau}(x_{0})}\left(\sum_{i\in(-\infty..i_{0}]}\int\mathbf{1}\Big\{ x\notin\bigcap_{j<i}\mathcal{B}_{e^{-\eta(j+\theta)}}(z_{j})\Big\}\Big(\prod_{j<i}\mathrm{U}\mathcal{B}_{e^{-\eta(j+\theta)}}(x_{0},\cdot)\Big)(\mathrm{d}\{z_{j}\}_{j})\right)P(\mathrm{d}x)\\
 & \le4\int_{\mathcal{B}_{\tau}(x_{0})}\left(\sum_{i\in(-\infty..i_{0}]}\sum_{j<i}\int\mathbf{1}\Big\{ x\notin\mathcal{B}_{e^{-\eta(j+\theta)}}(z_{j})\Big\}\mathrm{U}\mathcal{B}_{e^{-\eta(j+\theta)}}(x_{0},\mathrm{d}z_{j})\right)P(\mathrm{d}x)\\
 & \stackrel{(a)}{\le}4\int_{\mathcal{B}_{\tau}(x_{0})}\left(\sum_{i\in(-\infty..i_{0}]}\sum_{j<i}\Psi e^{\eta(j+\theta)}d(x,x_{0})\right)P(\mathrm{d}x)\\
 & =4\int_{\mathcal{B}_{\tau}(x_{0})}\Psi\frac{e^{\eta(i_{0}-1+\theta)}}{(1-e^{-\eta})^{2}}d(x,x_{0})P(\mathrm{d}x)\\
 & \le4\tau\Psi\frac{e^{\eta(i_{0}-1+\theta)}}{(1-e^{-\eta})^{2}},
\end{align*}
where (a) is by \eqref{eq:metric_pow_delta_dist}. Combining this
with \eqref{eq:metric_pow_pf_term1},
\[
\sum_{i\in(-\infty..i_{0}]}d_{\mathrm{TV}}\left(\bigg(\prod_{j\in I_{<i}}\nu_{j}\bigg)\bar{P}_{i|I_{<i}},\,\prod_{j\in I_{\le i}}\nu_{j}\right)<\infty.
\]
We then check \eqref{eq:spfr_kinfty_lim2}:
\begin{align*}
 & d_{\mathrm{TV}}\Bigg(\bar{P}_{I_{\le i}},\,\prod_{j\le i}\nu_{j}\Bigg)\\
 & \stackrel{(a)}{\le}\int d_{\mathrm{TV}}\Bigg(\prod_{j\le i}\mathrm{U}\mathcal{B}_{e^{-\eta(j+\theta)}}(x,\cdot),\,\prod_{j\le i}\nu_{j}\Bigg)P(\mathrm{d}x)\\
 & \le\int\min\left\{ \sum_{j\le i}d_{\mathrm{TV}}\left(\mathrm{U}\mathcal{B}_{e^{-\eta(j+\theta)}}(x,\cdot),\,\nu_{j}\right),\,1\right\} P(\mathrm{d}x)\\
 & =\int\min\left\{ \sum_{j\le i}d_{\mathrm{TV}}\left(\mathrm{U}\mathcal{B}_{e^{-\eta(j+\theta)}}(x,\cdot),\,\mathrm{U}\mathcal{B}_{e^{-\eta(j+\theta)}}(x_{0},\cdot)\right),\,1\right\} P(\mathrm{d}x)\\
 & \stackrel{(b)}{\le}\int\min\left\{ \sum_{j\le i}\Psi e^{\eta(j+\theta)}d(x,x_{0}),\,1\right\} P(\mathrm{d}x)\\
 & =\int\min\left\{ \frac{\Psi e^{\eta(i+\theta)}}{1-e^{\eta}}d(x,x_{0}),\,1\right\} P(\mathrm{d}x)\\
 & \to0
\end{align*}
as $i\to-\infty$ by Lebesgue's dominated convergence theorem, where
(a) is by the definition of $\bar{P}$ and the convexity of $d_{\mathrm{TV}}$,
and (b) is by \eqref{eq:metric_pow_delta_dist} and \eqref{eq:metric_pow_pf_dtv}.

\medskip{}

\begin{rem}
\label{rem:metric_pow_pf_bayesjust}We now show \eqref{eq:metric_pow_pf_bayescomp}
formally. Let $X\sim P$, and $Z_{i}|X\sim\mathrm{U}\mathcal{B}_{e^{-\eta(i+\theta)}}$
be conditionally independent across $i$ given $X$. We first show
that ``$P(\bigcap_{j<i}\mathcal{B}_{e^{-\eta(j+\theta)}}(Z_{j}))>0$
for all $i\in\mathbb{Z}$'' holds almost surely. Note that ``$P(\mathcal{B}_{w}(X))>0$
for all $w>0$'' holds almost surely (since $P(\exists w>0:\,P(\mathcal{B}_{w}(X))=0)\le P(\bigcup_{\mathrm{open}\,S\subseteq\mathcal{X}:\,P(S)=0}S)=0$
because $\mathcal{X}$ is strongly Lindel\"of). If $P(\mathcal{B}_{w}(X))>0$
for all $w>0$, for any $i\in\mathbb{Z}$, we have
\begin{align*}
 & \mathbf{P}\left(P\big(\bigcap_{j<i}\mathcal{B}_{e^{-\eta(j+\theta)}}(Z_{j})\big)=0\,\Big|\,X\right)\\
 & \le\inf_{w>0}\mathbf{P}\left(P\big(\bigcap_{j<i}\mathcal{B}_{e^{-\eta(j+\theta)}}(Z_{j})\,\big|\,\mathcal{B}_{w}(X)\big)=0\,\Big|\,X\right)\\
 & \le\inf_{w>0}\mathbf{E}\left[1-P\big(\bigcap_{j<i}\mathcal{B}_{e^{-\eta(j+\theta)}}(Z_{j})\,\big|\,\mathcal{B}_{w}(X)\big)\,\Big|\,X\right]\\
 & =\inf_{w>0}\left(1-\mathbf{E}\left[\frac{1}{P(\mathcal{B}_{w}(X))}\int_{\mathcal{B}_{w}(X)}\mathbf{1}\big\{\tilde{x}\in\bigcap_{j<i}\mathcal{B}_{e^{-\eta(j+\theta)}}(Z_{j})\big\} P(\mathrm{d}\tilde{x})\,\Big|\,X\right]\right)\\
 & =\inf_{w>0}\left(1-\frac{1}{P(\mathcal{B}_{w}(X))}\int_{\mathcal{B}_{w}(X)}\mathbf{E}\left[\mathbf{1}\big\{\tilde{x}\in\bigcap_{j<i}\mathcal{B}_{e^{-\eta(j+\theta)}}(Z_{j})\big\}\,\Big|\,X\right]P(\mathrm{d}\tilde{x})\right)\\
 & \le\inf_{w>0}\left(1-\frac{1}{P(\mathcal{B}_{w}(X))}\int_{\mathcal{B}_{w}(X)}\left(1-\sum_{j<i}\mathbf{P}\left(\tilde{x}\notin\mathcal{B}_{e^{-\eta(j+\theta)}}(Z_{j})\,\Big|\,X\right)\right)P(\mathrm{d}\tilde{x})\right)\\
 & \stackrel{(a)}{\le}\inf_{w>0}\left(1-\frac{1}{P(\mathcal{B}_{w}(X))}\int_{\mathcal{B}_{w}(X)}\left(1-\sum_{j<i}\Psi e^{\eta(j+\theta)}w\right)P(\mathrm{d}\tilde{x})\right)\\
 & =\inf_{w>0}\left(1-\frac{1}{P(\mathcal{B}_{w}(X))}\int_{\mathcal{B}_{w}(X)}\left(1-\frac{\Psi e^{\eta(i-1+\theta)}}{1-e^{\eta}}w\right)P(\mathrm{d}\tilde{x})\right)\\
 & =\inf_{w>0}\frac{\Psi e^{\eta(i-1+\theta)}}{1-e^{\eta}}w\\
 & =0,
\end{align*}
where (a) is due to \eqref{eq:metric_pow_delta_dist} and $d(X,\tilde{x})\le w$.
Therefore ``$P(\bigcap_{j<i}\mathcal{B}_{e^{-\eta(j+\theta)}}(Z_{j}))>0$
for all $i\in\mathbb{Z}$'' holds almost surely.

For any $E\in\mathcal{F}$ and $k<i$, for $\bar{P}_{[k..i)}$-almost
all $\{z_{j}\}_{j\in[k..i)}$, it is straightforward to check by Bayes'
rule that

\begin{align*}
 & \bar{P}_{X|Z_{[k..i)}}(E|\{z_{j}\}_{j})\\
 & =\frac{\int_{E\cap\bigcap_{j=k}^{i-1}\mathcal{B}_{e^{-\eta(j+\theta)}}(z_{j})}\Big(\prod_{j=k}^{i-1}\mu(\mathcal{B}_{e^{-\eta(j+\theta)}}(x))\Big)^{-1}P(\mathrm{d}x)}{\int_{\bigcap_{j=k}^{i-1}\mathcal{B}_{e^{-\eta(j+\theta)}}(z_{j})}\Big(\prod_{j=k}^{i-1}\mu(\mathcal{B}_{e^{-\eta(j+\theta)}}(x))\Big)^{-1}P(\mathrm{d}x)}\\
 & \le\frac{\int_{E\cap\bigcap_{j=k}^{i-1}\mathcal{B}_{e^{-\eta(j+\theta)}}(z_{j})}\Big(\inf_{y\in\bigcap_{j=k}^{i-1}\mathcal{B}_{e^{-\eta(j+\theta)}}(z_{j})}\prod_{j=k}^{i-1}\mu(\mathcal{B}_{e^{-\eta(j+\theta)}}(y))\Big)^{-1}P(\mathrm{d}x)}{\int_{\bigcap_{j=k}^{i-1}\mathcal{B}_{e^{-\eta(j+\theta)}}(z_{j})}\Big(\sup_{y\in\bigcap_{j=k}^{i-1}\mathcal{B}_{e^{-\eta(j+\theta)}}(z_{j})}\prod_{j=k}^{i-1}\mu(\mathcal{B}_{e^{-\eta(j+\theta)}}(y))\Big)^{-1}P(\mathrm{d}x)}\\
 & =\frac{\sup_{y\in\bigcap_{j=k}^{i-1}\mathcal{B}_{e^{-\eta(j+\theta)}}(z_{j})}\prod_{j=k}^{i-1}\mu(\mathcal{B}_{e^{-\eta(j+\theta)}}(y))}{\inf_{y\in\bigcap_{j=k}^{i-1}\mathcal{B}_{e^{-\eta(j+\theta)}}(z_{j})}\prod_{j=k}^{i-1}\mu(\mathcal{B}_{e^{-\eta(j+\theta)}}(y))}\cdot\frac{\int_{E\cap\bigcap_{j=k}^{i-1}\mathcal{B}_{e^{-\eta(j+\theta)}}(z_{j})}P(\mathrm{d}x)}{\int_{\bigcap_{j=k}^{i-1}\mathcal{B}_{e^{-\eta(j+\theta)}}(z_{j})}P(\mathrm{d}x)}\\
 & \stackrel{(a)}{\le}\tilde{\xi}\frac{\int_{E\cap\bigcap_{j=k}^{i-1}\mathcal{B}_{e^{-\eta(j+\theta)}}(z_{j})}P(\mathrm{d}x)}{\int_{\bigcap_{j=k}^{i-1}\mathcal{B}_{e^{-\eta(j+\theta)}}(z_{j})}P(\mathrm{d}x)}\\
 & =\tilde{\xi}P\left(E\left|\bigcap_{j=k}^{i-1}\mathcal{B}_{e^{-\eta(j+\theta)}}(z_{j})\right.\right),
\end{align*}
where (a) is by $\prod_{j=k}^{i-1}\mu(\mathcal{B}_{e^{-\eta(j+\theta)}}(y))/\mu(\mathcal{B}_{e^{-\eta(j+\theta)}}(x))\le\tilde{\xi}$
for any $x,y\in\bigcap_{j=k}^{i-1}\mathcal{B}_{e^{-\eta(j+\theta)}}(z_{j})$,
which is proved using the same arguments as in \eqref{eq:metric_pow_f_likelihoodbd}.
By the martingale convergence theorem, $\bar{P}_{X|Z_{[k..i)}}(E|\{Z_{j}\}_{j})\to\bar{P}_{X|Z_{(-\infty..i)}}(E|\{Z_{j}\}_{j})$
almost surely as $k\to-\infty$. Hence, almost surely,
\begin{align}
\bar{P}_{X|Z_{(-\infty..i)}}(E|\{Z_{j}\}_{j}) & =\lim_{k\to-\infty}\bar{P}_{X|Z_{[k..i)}}(E|\{Z_{j}\}_{j})\nonumber \\
 & \le\underset{k\to-\infty}{\lim\inf}\tilde{\xi}P\left(E\left|\bigcap_{j=k}^{i-1}\mathcal{B}_{e^{-\eta(j+\theta)}}(Z_{j})\right.\right)\nonumber \\
 & \le\underset{k\to-\infty}{\lim\inf}\tilde{\xi}\frac{P(E\cap\bigcap_{j=k}^{i-1}\mathcal{B}_{e^{-\eta(j+\theta)}}(Z_{j}))}{P(\bigcap_{j<i}\mathcal{B}_{e^{-\eta(j+\theta)}}(Z_{j}))}\nonumber \\
 & =\tilde{\xi}\frac{P(E\cap\bigcap_{j<i}\mathcal{B}_{e^{-\eta(j+\theta)}}(Z_{j}))}{P(\bigcap_{j<i}\mathcal{B}_{e^{-\eta(j+\theta)}}(Z_{j}))}\nonumber \\
 & =\tilde{\xi}P\left(E\left|\bigcap_{j<i}\mathcal{B}_{e^{-\eta(j+\theta)}}(Z_{j})\right.\right).\label{eq:metric_pow_pf_pcompe}
\end{align}

We now show that $\bar{P}_{X|Z_{(-\infty..i)}}(\cdot|\{Z_{j}\}_{j})\le\tilde{\xi}P(\cdot|\bigcap_{j<i}\mathcal{B}_{e^{-\eta(j+\theta)}}(Z_{j}))$
almost surely. Since $\mathcal{X}$ is a second-countable topological
space, let its base be $\{H_{i}\}_{i\in\mathbb{N}}$ (i.e., $H_{i}\subseteq\mathcal{X}$
is open, and each open set in $\mathcal{X}$ is a union of sets in
$\{H_{i}\}_{i\in\mathbb{N}}$). Let $\{\tilde{H}_{i}\}_{i\in\mathbb{N}}$
be the collection of sets in $\{H_{i}\backslash\bigcup_{j\in J}H_{j}:\,i\in\mathbb{N},\,J\subseteq[1..i-1]\}$
(which is countable). Then each open set in $\mathcal{X}$ is a disjoint
union of sets in $\{\tilde{H}_{i}\}_{i\in\mathbb{N}}$ (since $\bigcup_{j=1}^{\infty}H_{i_{j}}=\bigcup_{j=1}^{\infty}(H_{i_{j}}\backslash\bigcup_{k=1}^{j-1}H_{i_{k}})$
for an increasing sequence $i_{1},i_{2},\ldots$). By \eqref{eq:metric_pow_pf_pcompe},
``$\bar{P}_{X|Z_{(-\infty..i)}}(\tilde{H}_{k}|\{Z_{j}\}_{j})\le\tilde{\xi}P(\tilde{H}_{k}|\bigcap_{j<i}\mathcal{B}_{e^{-\eta(j+\theta)}}(Z_{j}))$
for all $k\in\mathbb{N}$'' holds almost surely. Fix any $\{z_{j}\}_{j}$
such that ``$\bar{P}_{X|Z_{(-\infty..i)}}(\tilde{H}_{k}|\{z_{j}\}_{j})\le\tilde{\xi}P(\tilde{H}_{k}|\bigcap_{j<i}\mathcal{B}_{e^{-\eta(j+\theta)}}(z_{j}))$
for all $k\in\mathbb{N}$'' holds. For any $E\in\mathcal{F}$, since
any Borel probability measure on a Polish space is regular,
\begin{align*}
 & \bar{P}_{X|Z_{(-\infty..i)}}(E|\{Z_{j}\}_{j})\\
 & =\inf\left\{ \bar{P}_{X|Z_{(-\infty..i)}}(\tilde{E}|\{Z_{j}\}_{j}):\,\tilde{E}\supseteq E,\,\tilde{E}\,\text{is open}\right\} \\
 & =\inf\left\{ \bar{P}_{X|Z_{(-\infty..i)}}(\tilde{E}|\{Z_{j}\}_{j}):\,\tilde{E}\supseteq E,\,\tilde{E}\,\text{is a disjoint union of sets in }\{\tilde{H}_{i}\}_{i\in\mathbb{N}}\right\} \\
 & \le\inf\left\{ \tilde{\xi}P\big(\tilde{E}\big|\bigcap_{j<i}\mathcal{B}_{e^{-\eta(j+\theta)}}(Z_{j})\big):\,\tilde{E}\supseteq E,\,\tilde{E}\,\text{is a disjoint union of sets in }\{\tilde{H}_{i}\}_{i\in\mathbb{N}}\right\} \\
 & =\tilde{\xi}P\big(E\big|\bigcap_{j<i}\mathcal{B}_{e^{-\eta(j+\theta)}}(Z_{j})\big),
\end{align*}
and thus $\bar{P}_{X|Z_{(-\infty..i)}}(\cdot|\{z_{j}\}_{j})\le\tilde{\xi}P(\cdot|\bigcap_{j<i}\mathcal{B}_{e^{-\eta(j+\theta)}}(z_{j}))$.
Hence, $\bar{P}_{X|Z_{(-\infty..i)}}(\cdot|\{Z_{j}\}_{j})\le\tilde{\xi}P(\cdot|\bigcap_{j<i}\mathcal{B}_{e^{-\eta(j+\theta)}}(Z_{j}))$
almost surely.

Since we are free to modify $\bar{P}_{X|Z_{(-\infty..i)}}(\cdot|\{z_{j}\}_{j})$
over a set of $\{z_{j}\}_{j}$ with $\bar{P}_{(-\infty..i)}$-measure
zero, we can choose $\bar{P}_{X|Z_{(-\infty..i)}}$ such that ``$P(\bigcap_{j<i}\mathcal{B}_{e^{-\eta(j+\theta)}}(z_{j}))=0$
or $\bar{P}_{X|Z_{(-\infty..i)}}(\cdot|\{z_{j}\}_{j})\le\tilde{\xi}P(\cdot|\bigcap_{j<i}\mathcal{B}_{e^{-\eta(j+\theta)}}(z_{j}))$''
holds for all $\{z_{j}\}_{j}\in\mathcal{X}^{(-\infty..i)}$. This
can be achieved by taking
\[
\bar{P}'_{X|Z_{(-\infty..i)}}(\cdot|\{z_{j}\}_{j}):=\begin{cases}
P & \mathrm{if}\;P(\bigcap_{j<i}\mathcal{B}_{e^{-\eta(j+\theta)}}(z_{j}))=0\\
\bar{P}_{X|Z_{(-\infty..i)}}(\cdot|\{z_{j}\}_{j}) & \mathrm{else}\,\mathrm{if}\;\bar{P}_{X|Z_{(-\infty..i)}}(\cdot|\{z_{j}\}_{j})\le\tilde{\xi}P(\cdot|\bigcap_{j<i}\mathcal{B}_{e^{-\eta(j+\theta)}}(z_{j}))\\
P(\cdot|\bigcap_{j<i}\mathcal{B}_{e^{-\eta(j+\theta)}}(z_{j})) & \mathrm{otherwise}.
\end{cases}
\]
Note that $\{z_{j}\}_{j}\mapsto\bar{P}'_{X|Z_{(-\infty..i)}}(E|\{z_{j}\}_{j})$
is a measurable function for any $E\in\mathcal{F}$, since $P(\bigcap_{j<i}\mathcal{B}_{e^{-\eta(j+\theta)}}(z_{j}))=\inf_{k<i}P(\bigcap_{j=k}^{i-1}\mathcal{B}_{e^{-\eta(j+\theta)}}(z_{j}))$
is a measurable function of $\{z_{j}\}_{j}$, $P(E|\bigcap_{j<i}\mathcal{B}_{e^{-\eta(j+\theta)}}(z_{j}))=\inf_{k<i}P(E\cap\bigcap_{j=k}^{i-1}\mathcal{B}_{e^{-\eta(j+\theta)}}(z_{j}))/P(\bigcap_{j<i}\mathcal{B}_{e^{-\eta(j+\theta)}}(z_{j}))$
is a measurable function of $\{z_{j}\}_{j}$, and 
\begin{align*}
 & \left\{ \{z_{j}\}_{j}:\,\bar{P}_{X|Z_{(-\infty..i)}}(\cdot|\{z_{j}\}_{j})\le\tilde{\xi}P(\cdot|\bigcap_{j<i}\mathcal{B}_{e^{-\eta(j+\theta)}}(z_{j}))\right\} \\
 & =\bigcap_{k\in\mathbb{N}}\left\{ \{z_{j}\}_{j}:\,\bar{P}_{X|Z_{(-\infty..i)}}(\tilde{H}_{k}|\{z_{j}\}_{j})\le\tilde{\xi}P(\tilde{H}_{k}|\bigcap_{j<i}\mathcal{B}_{e^{-\eta(j+\theta)}}(z_{j}))\right\} 
\end{align*}
is measurable. Hence $\bar{P}'_{X|Z_{(-\infty..i)}}$ is a valid regular
conditional distribution, and we can use it in place of $\bar{P}_{X|Z_{(-\infty..i)}}$.

\end{rem}

\medskip{}

\section{Proof that $t\breve{\mathrm{V}}_{n}'(t)/\breve{\mathrm{V}}_{n}(t)$
is non-decreasing\label{sec:hyperbolic_inc}}

Equivalently, we will prove that
\[
\log\frac{t(\sinh t)^{n-1}}{\int_{0}^{t}(\sinh x)^{n-1}\mathrm{d}x}
\]
is non-decreasing. We have
\begin{align*}
 & \frac{\mathrm{d}}{\mathrm{d}t}\log\frac{t(\sinh t)^{n-1}}{\int_{0}^{t}(\sinh x)^{n-1}\mathrm{d}x}\\
 & =t^{-1}+\frac{(n-1)(\sinh t)^{n-2}\cosh t}{(\sinh t)^{n-1}}-\frac{(\sinh t)^{n-1}}{\int_{0}^{t}(\sinh x)^{n-1}\mathrm{d}x}\\
 & =t^{-1}+(n-1)\frac{\cosh t}{\sinh t}-\frac{(\sinh t)^{n-1}}{\int_{0}^{t}(\sinh x)^{n-1}\mathrm{d}x}.
\end{align*}
To prove that this derivative is non-negative, it suffices to prove
that
\[
\int_{0}^{t}(\sinh x)^{n-1}\mathrm{d}x\ge\frac{(\sinh t)^{n-1}}{t^{-1}+(n-1)\frac{\cosh t}{\sinh t}}.
\]
Note that
\[
\lim_{t\to0}\left(\int_{0}^{t}(\sinh x)^{n-1}\mathrm{d}x-\frac{(\sinh t)^{n-1}}{t^{-1}+(n-1)\frac{\cosh t}{\sinh t}}\right)=0,
\]
and

\begin{align*}
 & \frac{\mathrm{d}}{\mathrm{d}t}\left(\int_{0}^{t}(\sinh x)^{n-1}\mathrm{d}x-\frac{(\sinh t)^{n-1}}{t^{-1}+(n-1)\frac{\cosh t}{\sinh t}}\right)\\
 & =(\sinh t)^{n-1}-\frac{\mathrm{d}}{\mathrm{d}t}\frac{(\sinh t)^{n-1}}{t^{-1}+(n-1)\frac{\cosh t}{\sinh t}}\\
 & =(\sinh t)^{n-1}-\frac{(n-1)\cosh t(\sinh t)^{n-2}}{t^{-1}+(n-1)\frac{\cosh t}{\sinh t}}-\frac{(\sinh t)^{n-1}\left(t^{-2}+(n-1)(\sinh t)^{-2}\right)}{\left(t^{-1}+(n-1)\frac{\cosh t}{\sinh t}\right)^{2}}\\
 & =(\sinh t)^{n-1}\left(1-\frac{(n-1)\frac{\cosh t}{\sinh t}}{t^{-1}+(n-1)\frac{\cosh t}{\sinh t}}-\frac{t^{-2}+(n-1)(\sinh t)^{-2}}{\left(t^{-1}+(n-1)\frac{\cosh t}{\sinh t}\right)^{2}}\right)\\
 & =\frac{(\sinh t)^{n-1}}{\left(t^{-1}+(n-1)\frac{\cosh t}{\sinh t}\right)^{2}}\left(t^{-1}\left(t^{-1}+(n-1)\frac{\cosh t}{\sinh t}\right)-t^{-2}-(n-1)(\sinh t)^{-2}\right)\\
 & =\frac{(n-1)(\sinh t)^{n-1}}{\left(t^{-1}+(n-1)\frac{\cosh t}{\sinh t}\right)^{2}(\sinh t)^{2}}\left(t^{-1}\sinh t\cosh t-1\right)\\
 & =\frac{(n-1)(\sinh t)^{n-1}}{\left(t^{-1}+(n-1)\frac{\cosh t}{\sinh t}\right)^{2}(\sinh t)^{2}}\left(\frac{1}{2}t^{-1}\sinh(2t)-1\right)\\
 & \ge0.
\end{align*}
The result follows.

\section{Proof of Theorem \ref{thm:rn_rc_lb_ball}\label{subsec:pf_rn_rc_lb_ball}}

Let $0<q\le1$. \footnote{Theorem \ref{thm:rn_rc_lb_ball} requires that $0<q<1$. Nevertheless,
we also allow the case $q=1$ in the proof.} Let $k\ge3$ be an odd integer. Let $\mathcal{A}:=[-(k-1)/2..\,(k-1)/2]^{n}$
with $|\mathcal{A}|=k^{n}$. For $x,y\in\mathbb{Z}^{n}$, define $x\oplus y\in\mathcal{A}$
such that $(x\oplus y)_{i}\equiv x_{i}+y_{i}\;\mathrm{mod}\;k$ for
$i=1,\ldots,n$, and define $x\ominus y:=x\oplus-y$.

For $\alpha\in\mathcal{A}$, let $P_{\alpha}:=(|\mathcal{A}|-1)^{-1}\sum_{\beta\in\mathcal{A}\backslash\{\alpha\}}\delta_{\beta}$.
Fix any coupling $\{X_{\alpha}\}$ of $\{P_{\alpha}\}$. Write $X_{\alpha}=(X_{\alpha,1},\ldots,X_{\alpha,n})$.
Let $r:=r_{c}(\{X_{\alpha}\}_{\alpha\in\mathcal{A}})$. Let $Z\in\mathbb{Z}^{n}$
be a random vector where $Z_{i}$ is the median of the multiset $\{X_{\alpha,i}\}_{\alpha\in\mathcal{A}}$
(there is a unique integer median since $|\mathcal{A}|$ is odd).
Define
\[
G_{n,p,q}(\gamma):=\sum_{x\in\mathbb{Z}^{n}}\max\left\{ \gamma-\Vert x\Vert_{p}^{q},\,0\right\} .
\]

The proof is divided into 4 parts. First, we bound $r_{c}^{*}(\mathcal{P}(\mathbb{Z}^{n}))$
in terms of $G_{n,p,q}$. Next, we bound $G_{n,p,q}$. Then we find
the rate of growth of $r_{c}^{*}(\mathcal{P}(\mathbb{Z}^{n}))$ as
$n$ increases. Finally, we extend this result to $\mathcal{P}_{\ll\lambda_{S}}(\mathbb{R}^{n})$
for any $S\subseteq\mathbb{R}^{n}$ with $\lambda(S)>0$, and to $\mathcal{P}(\mathcal{M})$
for a Riemannian manifold $\mathcal{M}$.

\medskip{}

\subsection{Bound on $r_{c}^{*}(\mathcal{P}(\mathbb{Z}^{n}))$ in terms of $G_{n,p,q}$}

For any $0\le\gamma\le(k/2)^{q}$, we have
\begin{align*}
 & \mathbf{E}\left[\min\left\{ \sum_{\alpha\in\mathcal{A}}\Vert X_{\alpha}-Z\Vert_{p}^{q},\,G_{n,p,q}(\gamma)\right\} \right]\\
 & \ge\mathbf{E}\left[\min\left\{ \sum_{\alpha\in\mathcal{A}}\Vert X_{\alpha}\ominus Z\Vert_{p}^{q},\,G_{n,p,q}(\gamma)\right\} \right]\\
 & =G_{n,p,q}(\gamma)-\mathbf{E}\left[\max\left\{ G_{n,p,q}(\gamma)-\sum_{\alpha\in\mathcal{A}}\Vert X_{\alpha}\ominus Z\Vert_{p}^{q},\,0\right\} \right]\\
 & \stackrel{(a)}{=}G_{n,p,q}(\gamma)-\mathbf{E}\left[\max\left\{ \sum_{\alpha\in\mathcal{A}}\left(\max\left\{ \gamma-\Vert Z\ominus\alpha\Vert_{p}^{q},\,0\right\} -\Vert X_{\alpha}\ominus Z\Vert_{p}^{q}\right),\,0\right\} \right]\\
 & \stackrel{(b)}{\ge}G_{n,p,q}(\gamma)-\mathbf{E}\left[\sum_{\alpha\in\mathcal{A}}\max\left\{ \gamma-\Vert X_{\alpha}\ominus\alpha\Vert_{p}^{q},\,0\right\} \right]\\
 & =G_{n,p,q}(\gamma)-\frac{|\mathcal{A}|}{|\mathcal{A}|-1}\left(G_{n,p,q}(\gamma)-\gamma\right)\\
 & =\frac{|\mathcal{A}|\gamma-G_{n,p,q}(\gamma)}{|\mathcal{A}|-1}\\
 & \ge\gamma-|\mathcal{A}|^{-1}G_{n,p,q}(\gamma).
\end{align*}
where (a) is because $\{Z\ominus\alpha:\,\alpha\in\mathcal{A}\}=\mathcal{A}$,
and $\gamma-\Vert x\Vert_{p}^{q}\le0$ for $x\in\mathbb{Z}^{n}\backslash\mathcal{A}$
since $\gamma\le(k/2)^{q}$, (b) is because $|\Vert Z\ominus\alpha\Vert_{p}^{q}-\Vert X_{\alpha}\ominus\alpha\Vert_{p}^{q}|\le\Vert X_{\alpha}\ominus Z\Vert_{p}^{q}$
since $\Vert x\ominus y\Vert_{p}^{q}$ is a metric. Let $U\sim\mathrm{Unif}(\mathcal{A})$
independent of $\{X_{\alpha}\}$. Then for any $0\le\gamma\le(k/2)^{q}$,
\begin{align}
 & \mathbf{E}\left[\min\left\{ \mathbf{E}[\Vert X_{U}-Z\Vert_{p}^{q}\,|\,\{X_{\alpha}\}],\,|\mathcal{A}|^{-1}G_{n,p,q}(\gamma)\right\} \right]\nonumber \\
 & \ge|\mathcal{A}|^{-1}\gamma-|\mathcal{A}|^{-2}G_{n,p,q}(\gamma).\label{eq:Epq_G_bd}
\end{align}
For any $j\in[1..n]$, we have
\begin{align*}
 & \sum_{\alpha,\beta\in\mathcal{A}:\,\Vert\alpha\ominus\beta\Vert_{1}=1}|X_{\alpha,j}-X_{\beta,j}|^{q}\\
 & \stackrel{(a)}{\ge}2^{-(1-q)}\sum_{\alpha,\beta\in\mathcal{A}:\,\Vert\alpha\ominus\beta\Vert_{1}=1}\left|\mathrm{sgn}(X_{\alpha,j}-Z_{j})|X_{\alpha,j}-Z_{j}|^{q}-\mathrm{sgn}(X_{\beta,j}-Z_{j})|X_{\beta,j}-Z_{j}|^{q}\right|\\
 & =2^{-(1-q)}\int_{-\infty}^{\infty}\sum_{\alpha,\beta\in\mathcal{A}:\,\Vert\alpha\ominus\beta\Vert_{1}=1}\left|\mathbf{1}\left\{ \mathrm{sgn}(X_{\alpha,j}-Z_{j})|X_{\alpha,j}-Z_{j}|^{q}\le t\right\} -\mathbf{1}\left\{ \mathrm{sgn}(X_{\beta,j}-Z_{j})|X_{\beta,j}-Z_{j}|^{q}\le t\right\} \right|\mathrm{d}t\\
 & \stackrel{(b)}{\ge}2^{-(1-q)}\int_{-\infty}^{\infty}4|\mathcal{A}|^{1-1/n}F_{n}\left(\frac{|\{\alpha\in\mathcal{A}:\,\mathrm{sgn}(X_{\alpha,j}-Z_{j})|X_{\alpha,j}-Z_{j}|^{q}\le t\}|}{|\mathcal{A}|}\right)\mathrm{d}t\\
 & =2^{1+q}|\mathcal{A}|^{1-1/n}\int_{-\infty}^{\infty}F_{n}\left(\mathbf{P}(\mathrm{sgn}(X_{U,j}-Z_{j})|X_{U,j}-Z_{j}|^{q}\le t\,|\,\{X_{\alpha}\})\right)\mathrm{d}t\\
 & \stackrel{(c)}{\ge}2^{1+q}|\mathcal{A}|^{1-1/n}\frac{F_{n}\left(\min\left\{ \mathbf{E}[|X_{U,j}-Z_{j}|^{q}\,|\,\{X_{\alpha}\}],\,1/2\right\} \right)}{\min\left\{ \mathbf{E}[|X_{U,j}-Z_{j}|^{q}\,|\,\{X_{\alpha}\}],\,1/2\right\} }\\
 & \;\;\;\;\;\cdot\int_{-\infty}^{\infty}\min\left\{ \mathbf{P}(\mathrm{sgn}(X_{U,j}-Z_{j})|X_{U,j}-Z_{j}|^{q}\le t\,|\,\{X_{\alpha}\}),\,1-\mathbf{P}(\mathrm{sgn}(X_{U,j}-Z_{j})|X_{U,j}-Z_{j}|^{q}\le t\,|\,\{X_{\alpha}\})\right\} \mathrm{d}t\\
 & \stackrel{(d)}{=}2^{1+q}|\mathcal{A}|^{1-1/n}\frac{F_{n}\left(\min\left\{ \mathbf{E}[|X_{U,j}-Z_{j}|^{q}\,|\,\{X_{\alpha}\}],\,1/2\right\} \right)}{\min\left\{ \mathbf{E}[|X_{U,j}-Z_{j}|^{q}\,|\,\{X_{\alpha}\}],\,1/2\right\} }\\
 & \;\;\;\;\;\cdot\left(\int_{-\infty}^{0}\mathbf{P}(\mathrm{sgn}(X_{U,j}-Z_{j})|X_{U,j}-Z_{j}|^{q}\le t\,|\,\{X_{\alpha}\})\mathrm{d}t+\int_{0}^{\infty}\mathbf{P}(\mathrm{sgn}(X_{U,j}-Z_{j})|X_{U,j}-Z_{j}|^{q}>t\,|\,\{X_{\alpha}\})\mathrm{d}t\right)\\
 & =2^{1+q}|\mathcal{A}|^{1-1/n}\frac{F_{n}\left(\min\left\{ \mathbf{E}[|X_{U,j}-Z_{j}|^{q}\,|\,\{X_{\alpha}\}],\,1/2\right\} \right)}{\min\left\{ \mathbf{E}[|X_{U,j}-Z_{j}|^{q}\,|\,\{X_{\alpha}\}],\,1/2\right\} }\mathbf{E}[|X_{U,j}-Z_{j}|^{q}\,|\,\{X_{\alpha}\}]\\
 & \ge2^{1+q}|\mathcal{A}|^{1-1/n}F_{n}\left(\min\left\{ \mathbf{E}[|X_{U,j}-Z_{j}|^{q}\,|\,\{X_{\alpha}\}],\,1/2\right\} \right),
\end{align*}
where (a) is due to $|x-y|^{q}\ge|x^{q}-y^{q}|$ and $((x+y)/2)^{q}\ge(x^{q}+y^{q})/2$
for $x,y\ge0$, (b) is by the edge-isoperimetric inequality on the
discrete torus \cite[Theorem 8]{bollobas1991edge} with \footnote{The edge-isoperimetric inequality on the discrete torus \cite[Theorem 8]{bollobas1991edge}
states that $\sum_{\alpha,\beta\in\mathcal{A}:\,\Vert\alpha\ominus\beta\Vert_{1}=1}|\mathbf{1}\{\alpha\in S\}-\mathbf{1}\{\beta\in S\}|\ge4|\mathcal{A}|^{1-1/n}F_{n}(|S|/|\mathcal{A}|)$
for any $S\subseteq\mathcal{A}$.}
\begin{align*}
F_{n}(\tau) & :=\min_{b\in[1..n]}b\left(\min\{\tau,\,1-\tau\}\right)^{1-1/b}.
\end{align*}
Note that each edge is counted twice here, and hence the factor 4.
For (c), when $t<0$, since $Z_{j}$ is the median of $X_{U,j}$,
\begin{align*}
 & \mathbf{P}(\mathrm{sgn}(X_{U,j}-Z_{j})|X_{U,j}-Z_{j}|^{q}\le t\,|\,\{X_{\alpha}\})\\
 & \le\mathbf{P}(\mathrm{sgn}(X_{U,j}-Z_{j})|X_{U,j}-Z_{j}|^{q}\le-1\,|\,\{X_{\alpha}\})\\
 & \le\mathbf{E}[|X_{U,j}-Z_{j}|^{q}\,|\,\{X_{\alpha}\}],
\end{align*}
and when $t\ge0$,
\begin{align*}
 & 1-\mathbf{P}(\mathrm{sgn}(X_{U,j}-Z_{j})|X_{U,j}-Z_{j}|^{q}\le t\,|\,\{X_{\alpha}\})\\
 & \le\mathbf{P}(\mathrm{sgn}(X_{U,j}-Z_{j})|X_{U,j}-Z_{j}|^{q}\ge1\,|\,\{X_{\alpha}\})\\
 & \le\mathbf{E}[|X_{U,j}-Z_{j}|^{q}\,|\,\{X_{\alpha}\}],
\end{align*}
and hence $\min\{\mathbf{P}(\mathrm{sgn}(X_{U,j}-Z_{j})|X_{U,j}-Z_{j}|^{q}\le t\,|\,\{X_{\alpha}\}),\,1-\mathbf{P}(\mathrm{sgn}(X_{U,j}-Z_{j})|X_{U,j}-Z_{j}|^{q}\le t\,|\,\{X_{\alpha}\})\}\le\mathbf{E}[|X_{U,j}-Z_{j}|^{q}\,|\,\{X_{\alpha}\}]$,
and (c) follows from the concavity of $F_{n}(t)$. For (d), we have
$\mathbf{P}(\mathrm{sgn}(X_{U,j}-Z_{j})|X_{U,j}-Z_{j}|^{q}\le t\,|\,\{X_{\alpha}\})\le1/2$
when $t<0$ since $Z_{j}$ is the median of $X_{U,j}$ (and similar
for the case $t\ge0$). Hence,
\begin{align*}
 & \sum_{\alpha,\beta\in\mathcal{A}:\,\Vert\alpha\ominus\beta\Vert_{1}=1}\Vert X_{\alpha}-X_{\beta}\Vert_{p}^{q}\\
 & \ge n^{q(1/p-1)}\sum_{\alpha,\beta\in\mathcal{A}:\,\Vert\alpha\ominus\beta\Vert_{1}=1}\Vert X_{\alpha}-X_{\beta}\Vert_{1}^{q}\\
 & \ge2^{1+q}n^{q(1/p-1)}|\mathcal{A}|^{1-1/n}\sum_{j=1}^{n}F_{n}\left(\min\left\{ \mathbf{E}[|X_{U,j}-Z_{j}|^{q}\,|\,\{X_{\alpha}\}],\,1/2\right\} \right)\\
 & \stackrel{(a)}{\ge}2^{1+q}n^{q(1/p-1)}|\mathcal{A}|^{1-1/n}F_{n}\left(\min\left\{ \mathbf{E}[\Vert X_{U}-Z\Vert_{q}^{q}\,|\,\{X_{\alpha}\}],\,1/2\right\} \right)\\
 & \ge2^{1+q}n^{q(1/p-1)}|\mathcal{A}|^{1-1/n}F_{n}\left(\min\left\{ n^{-\max\{1/p-1/q,0\}}\mathbf{E}[\Vert X_{U}-Z\Vert_{p}^{q}\,|\,\{X_{\alpha}\}],\,1/2\right\} \right),
\end{align*}
where (a) is by the concavity of $t\mapsto F_{n}(\min\{t,1/2\})$.
Therefore, when $\mathbf{E}[\Vert X_{U}-Z\Vert_{p}^{q}\,|\,\{X_{\alpha}\}]\ge(1-1/n)^{n(n-1)}$,
\begin{align}
 & \sum_{\alpha,\beta\in\mathcal{A}:\,\Vert\alpha\ominus\beta\Vert_{1}=1}\Vert X_{\alpha}-X_{\beta}\Vert_{p}^{q}\nonumber \\
 & \ge2^{1+q}n^{q(1/p-1)}|\mathcal{A}|^{1-1/n}F_{n}\left(\min\left\{ n^{-\max\{1/p-1/q,0\}}(1-1/n)^{n(n-1)},\,1/2\right\} \right)\nonumber \\
 & =2^{1+q}n^{q(1/p-1)}|\mathcal{A}|^{1-1/n}F_{n}\left(n^{-\max\{1/p-1/q,0\}}(1-1/n)^{n(n-1)}\right)\nonumber \\
 & =\xi_{n,p,q}|\mathcal{A}|^{1-1/n},\label{eq:pq_sumbd_1}
\end{align}
where 
\[
\xi_{n,p,q}:=2^{1+q}n^{q(1/p-1)}F_{n}\left(n^{-\max\{1/p-1/q,0\}}(1-1/n)^{n(n-1)}\right).
\]
When $\mathbf{E}[\Vert X_{U}-Z\Vert_{p}^{q}\,|\,\{X_{\alpha}\}]\le(1-1/n)^{n(n-1)}$,
\begin{align*}
 & \sum_{\alpha,\beta\in\mathcal{A}:\,\Vert\alpha\ominus\beta\Vert_{1}=1}\Vert X_{\alpha}-X_{\beta}\Vert_{p}^{q}\\
 & \stackrel{(a)}{\ge}\sum_{\alpha,\beta\in\mathcal{A}:\,\Vert\alpha\ominus\beta\Vert_{1}=1}\left|\Vert X_{\alpha}-Z\Vert_{p}^{q}-\Vert X_{\beta}-Z\Vert_{p}^{q}\right|\\
 & =\int_{0}^{\infty}\sum_{\alpha,\beta\in\mathcal{A}:\,\Vert\alpha\ominus\beta\Vert_{1}=1}\left|\mathbf{1}\left\{ \Vert X_{\alpha}-Z\Vert_{p}^{q}\ge t\right\} -\mathbf{1}\left\{ \Vert X_{\beta}-Z\Vert_{p}^{q}\ge t\right\} \right|\mathrm{d}t\\
 & \stackrel{(b)}{\ge}\int_{0}^{\infty}4|\mathcal{A}|^{1-1/n}F_{n}\left(\frac{|\{\alpha\in\mathcal{A}:\,\Vert X_{\alpha}-Z\Vert_{p}^{q}\ge t\}|}{|\mathcal{A}|}\right)\mathrm{d}t\\
 & =4|\mathcal{A}|^{1-1/n}\int_{0}^{\infty}F_{n}\left(\mathbf{P}(\Vert X_{U}-Z\Vert_{p}^{q}\ge t\,|\,\{X_{\alpha}\})\right)\mathrm{d}t\\
 & \stackrel{(c)}{\ge}4|\mathcal{A}|^{1-1/n}\int_{0}^{\infty}F_{n}\left(\mathbf{E}[\Vert X_{U}-Z\Vert_{p}^{q}\,|\,\{X_{\alpha}\}]\right)\frac{\mathbf{P}(\Vert X_{U}-Z\Vert_{p}^{q}\ge t\,|\,\{X_{\alpha}\})}{\mathbf{E}[\Vert X_{U}-Z\Vert_{p}^{q}\,|\,\{X_{\alpha}\}]}\mathrm{d}t\\
 & =4|\mathcal{A}|^{1-1/n}F_{n}\left(\mathbf{E}[\Vert X_{U}-Z\Vert_{p}^{q}\,|\,\{X_{\alpha}\}]\right)\\
 & \stackrel{(d)}{=}4n|\mathcal{A}|^{1-1/n}\left(\mathbf{E}[\Vert X_{U}-Z\Vert_{p}^{q}\,|\,\{X_{\alpha}\}]\right)^{1-1/n},
\end{align*}
where (a) is because $\Vert x-y\Vert_{p}^{q}$ is a metric, (b) is
again by \cite[Theorem 8]{bollobas1991edge}, and (c) is by the concavity
of $F_{n}(t)$, and for any $t>0$, since $\Vert X_{U}-Z\Vert_{p}^{q}\ge1$
if $\Vert X_{U}-Z\Vert_{p}^{q}\ge t>0$,
\begin{align*}
 & \mathbf{P}(\Vert X_{U}-Z\Vert_{p}^{q}\ge t\,|\,\{X_{\alpha}\})\\
 & \le\mathbf{P}(\Vert X_{U}-Z\Vert_{p}^{q}\ge1\,|\,\{X_{\alpha}\})\\
 & \le\mathbf{E}[\Vert X_{U}-Z\Vert_{p}^{q}\,|\,\{X_{\alpha}\}],
\end{align*}
and (d) is because $F_{n}(t)=nt^{1-1/n}$ when $t\le(1-1/n)^{n(n-1)}$.
Combining this with \eqref{eq:pq_sumbd_1}, and considering both cases
whether $\mathbf{E}[\Vert X_{U}-Z\Vert_{p}^{q}\,|\,\{X_{\alpha}\}]\le(1-1/n)^{n(n-1)}$,
we have
\begin{align}
 & \sum_{\alpha,\beta\in\mathcal{A}:\,\Vert\alpha\ominus\beta\Vert_{1}=1}\Vert X_{\alpha}-X_{\beta}\Vert_{p}^{q}\nonumber \\
 & \ge4n|\mathcal{A}|^{1-1/n}\left(\min\left\{ \mathbf{E}[\Vert X_{U}-Z\Vert_{p}^{q}\,|\,\{X_{\alpha}\}],\,\tilde{\xi}_{n,p,q}\right\} \right)^{1-1/n},\label{eq:pq_sumbd_2}
\end{align}
where 
\[
\tilde{\xi}_{n,p,q}:=\min\left\{ (1-1/n)^{n(n-1)},\,(\xi_{n,p,q}/(4n))^{n/(n-1)}\right\} .
\]
Note that for any $0<\eta<1$, and random variable $T\ge0$,
\begin{align}
\mathbf{E}\left[T^{\eta}\right] & =\int_{0}^{\infty}\mathbf{P}\left(T^{\eta}\le t\right)\mathrm{d}t\nonumber \\
 & =\int_{0}^{\infty}\mathbf{P}\left(T\le t\right)\eta t^{\eta-1}\mathrm{d}t\nonumber \\
 & =\int_{0}^{\infty}\mathbf{P}\left(T\le t\right)\int_{t}^{\infty}\eta(1-\eta)s^{\eta-2}\mathrm{d}s\mathrm{d}t\nonumber \\
 & =\int_{0}^{\infty}\left(\int_{0}^{s}\mathbf{P}\left(T\le t\right)\mathrm{d}t\right)\eta(1-\eta)s^{\eta-2}\mathrm{d}s\nonumber \\
 & =\int_{0}^{\infty}\mathbf{E}\left[\min\{T,s\}\right]\eta(1-\eta)s^{\eta-2}\mathrm{d}s.\label{eq:ezpow}
\end{align}
Hence, we have
\begin{align*}
 & \mathbf{E}\left[\left(\min\left\{ \mathbf{E}[\Vert X_{U}-Z\Vert_{p}^{q}\,|\,\{X_{\alpha}\}],\,\tilde{\xi}_{n,p,q}\right\} \right)^{1-1/n}\right]\\
 & =\mathbf{E}\left[\int_{0}^{\infty}\min\left\{ \mathbf{E}[\Vert X_{U}-Z\Vert_{p}^{q}\,|\,\{X_{\alpha}\}],\,t,\,\tilde{\xi}_{n,p,q}\right\} \frac{1}{n}\left(1-\frac{1}{n}\right)t^{-(1+1/n)}\mathrm{d}t\right]\\
 & \ge\mathbf{E}\left[\int_{0}^{\tilde{\xi}_{n,p,q}}\min\left\{ \mathbf{E}[\Vert X_{U}-Z\Vert_{p}^{q}\,|\,\{X_{\alpha}\}],\,t\right\} \frac{1}{n}\left(1-\frac{1}{n}\right)t^{-(1+1/n)}\mathrm{d}t\right]\\
 & =\frac{n-1}{n^{2}}\int_{0}^{\tilde{\xi}_{n,p,q}}\mathbf{E}\left[\min\left\{ \mathbf{E}[\Vert X_{U}-Z\Vert_{p}^{q}\,|\,\{X_{\alpha}\}],\,t\right\} \right]t^{-(1+1/n)}\mathrm{d}t\\
 & =\frac{n-1}{n^{2}}\!\int_{0}^{G_{n,p,q}^{-1}(|\mathcal{A}|\tilde{\xi}_{n,p,q})}\!\!\mathbf{E}\left[\min\left\{ \mathbf{E}[\Vert X_{U}-Z\Vert_{p}^{q}\,|\,\{X_{\alpha}\}],\,|\mathcal{A}|^{-1}G_{n,p,q}(\gamma)\right\} \right](|\mathcal{A}|^{-1}G_{n,p,q}(\gamma))^{-(1+1/n)}\mathrm{d}|\mathcal{A}|^{-1}G_{n,p,q}(\gamma)\\
 & \ge\!\frac{n-1}{n^{2}}\!\int_{0}^{\min\{G_{n,p,q}^{-1}(|\mathcal{A}|\tilde{\xi}_{n,p,q}),\,(k/2)^{q}\}}\!\mathbf{E}\left[\min\left\{ \mathbf{E}[\Vert X_{U}-Z\Vert_{p}^{q}\,|\,\{X_{\alpha}\}],|\mathcal{A}|^{-1}G_{n,p,q}(\gamma)\right\} \right]\\
 & \;\;\;\;\;\;\cdot(|\mathcal{A}|^{-1}G_{n,p,q}(\gamma))^{-(1+1/n)}\mathrm{d}|\mathcal{A}|^{-1}G_{n,p,q}(\gamma)\\
 & \stackrel{(a)}{\ge}\frac{n-1}{n^{2}}\!\int_{0}^{\min\{G_{n,p,q}^{-1}(|\mathcal{A}|\tilde{\xi}_{n,p,q}),\,(k/2)^{q}\}}\left(|\mathcal{A}|^{-1}\gamma-|\mathcal{A}|^{-2}G_{n,p,q}(\gamma)\right)(|\mathcal{A}|^{-1}G_{n,p,q}(\gamma))^{-(1+1/n)}\mathrm{d}|\mathcal{A}|^{-1}G_{n,p,q}(\gamma)\\
 & =\frac{n-1}{n^{2}}|\mathcal{A}|^{-(1-1/n)}\int_{0}^{\min\{G_{n,p,q}^{-1}(|\mathcal{A}|\tilde{\xi}_{n,p,q}),\,(k/2)^{q}\}}\left(\gamma-|\mathcal{A}|^{-1}G_{n,p,q}(\gamma)\right)(G_{n,p,q}(\gamma))^{-(1+1/n)}\mathrm{d}G_{n,p,q}(\gamma)\\
 & =\!\frac{n\!-\!1}{n^{2}}|\mathcal{A}|^{-(1-1/n)}\Big(\int_{0}^{\min\{G_{n,p,q}^{-1}(|\mathcal{A}|\tilde{\xi}_{n,p,q}),(k/2)^{q}\}}\!\!\gamma(G_{n,p,q}(\gamma))^{-(1+1/n)}\mathrm{d}G_{n,p,q}(\gamma)\!-\!|\mathcal{A}|^{-1}\Big(1-\frac{1}{n}\Big)^{1-1/n}\!\big(|\mathcal{A}|\tilde{\xi}_{n,p,q}\big)^{1-1/n}\Big)\\
 & =\frac{n-1}{n^{2}}|\mathcal{A}|^{-(1-1/n)}\Big(\int_{0}^{\min\{G_{n,p,q}^{-1}(|\mathcal{A}|\tilde{\xi}_{n,p,q}),\,(k/2)^{q}\}}\gamma(G_{n,p,q}(\gamma))^{-(1+1/n)}\mathrm{d}G_{n,p,q}(\gamma)-\frac{n}{n-1}|\mathcal{A}|^{-1/n}\tilde{\xi}_{n,p,q}^{1-1/n}\Big)
\end{align*}
where $G_{n,p,q}^{-1}$ denotes the inverse function of $G_{n,p,q}$
(which is strictly increasing from $0$ to $\infty$), and (a) is
by \eqref{eq:Epq_G_bd}. Combining this with \eqref{eq:pq_sumbd_2},
\begin{align}
 & \mathbf{E}\left[\sum_{\alpha,\beta\in\mathcal{A}:\,\Vert\alpha\ominus\beta\Vert_{1}=1}\Vert X_{\alpha}-X_{\beta}\Vert_{p}^{q}\right]\nonumber \\
 & \ge4\left(1-\frac{1}{n}\right)\left(\int_{0}^{\min\{G_{n,p,q}^{-1}(|\mathcal{A}|\tilde{\xi}_{n,p,q}),\,(k/2)^{q}\}}\gamma(G_{n,p,q}(\gamma))^{-(1+1/n)}\mathrm{d}G_{n,p,q}(\gamma)-\frac{n}{n-1}|\mathcal{A}|^{-1/n}\tilde{\xi}_{n,p,q}^{1-1/n}\right).\label{eq:Epq_bd}
\end{align}
On the other hand,
\begin{align}
 & \mathbf{E}\bigg[\sum_{\alpha,\beta\in\mathcal{A}:\,\Vert\alpha\ominus\beta\Vert_{1}=1}\Vert X_{\alpha}-X_{\beta}\Vert_{p}^{q}\bigg]\nonumber \\
 & =\sum_{\alpha,\beta\in\mathcal{A}:\,\Vert\alpha\ominus\beta\Vert_{1}=1}\mathbf{E}[c(X_{\alpha},X_{\beta})]\nonumber \\
 & \le r\sum_{\alpha,\beta\in\mathcal{A}:\,\Vert\alpha\ominus\beta\Vert_{1}=1}C^{*}(P_{\alpha},P_{\beta})\nonumber \\
 & \le\frac{r}{|\mathcal{A}|-1}\sum_{\alpha,\beta\in\mathcal{A}:\,\Vert\alpha\ominus\beta\Vert_{1}=1}\Vert\alpha-\beta\Vert_{p}^{q}\nonumber \\
 & =\frac{r}{|\mathcal{A}|-1}\left(2n|\mathcal{A}|+2n|\mathcal{A}|^{1-1/n}((k-1)^{q}-1)\right)\nonumber \\
 & =\frac{2nr|\mathcal{A}|}{|\mathcal{A}|-1}\left(1+|\mathcal{A}|^{-1/n}((k-1)^{q}-1)\right)\nonumber \\
 & \le\frac{2nr|\mathcal{A}|}{|\mathcal{A}|-1}\left(1+|\mathcal{A}|^{-1/n}k^{q}\right)\nonumber \\
 & =\frac{2nr|\mathcal{A}|}{|\mathcal{A}|-1}\left(1+|\mathcal{A}|^{-(1-q)/n}\right).\label{eq:Eqp_ub_pow}
\end{align}
Combining this with \eqref{eq:Epq_bd},
\begin{align*}
r & \ge\left(\frac{2n|\mathcal{A}|}{|\mathcal{A}|-1}\left(1+|\mathcal{A}|^{-(1-q)/n}\right)\right)^{-1}\\
 & \;\;\;\cdot4\left(1-\frac{1}{n}\right)\left(\int_{0}^{\min\{G_{n,p,q}^{-1}(|\mathcal{A}|\tilde{\xi}_{n,p,q}),\,(k/2)^{q}\}}\gamma(G_{n,p,q}(\gamma))^{-(1+1/n)}\mathrm{d}G_{n,p,q}(\gamma)-\frac{n}{n-1}|\mathcal{A}|^{-1/n}\tilde{\xi}_{n,p,q}^{1-1/n}\right).
\end{align*}
Letting $k=|\mathcal{A}|^{1/n}\to\infty$, we have
\begin{align}
 & r_{c}^{*}(\mathcal{P}(\mathbb{Z}^{n}))\nonumber \\
 & \ge\left(1+\mathbf{1}\{q<1\}\right)\frac{1}{n}\left(1-\frac{1}{n}\right)\int_{0}^{\infty}\gamma(G_{n,p,q}(\gamma))^{-(1+1/n)}\mathrm{d}G_{n,p,q}(\gamma)\nonumber \\
 & =\left(1+\mathbf{1}\{q<1\}\right)\frac{1}{n}\left(1-\frac{1}{n}\right)\int_{0}^{\infty}G_{n,p,q}^{-1}(t)t^{-(1+1/n)}\mathrm{d}t.\label{eq:rc_lb_Gnpq}
\end{align}

\medskip{}

\subsection{Bound on $G_{n,p,q}^{-1}(t)$}

We give two methods to bound $G_{n,p,q}^{-1}(t)$ from below. For
the first method, for $v\in\mathbb{Z}^{n}$, we have
\begin{align*}
 & \int_{[-1/2,1/2]^{n}}\left(\Vert v+x\Vert_{p}^{q}-\Vert v\Vert_{p}^{q}\right)\mathrm{d}x\\
 & \le\int_{[-1/2,1/2]^{n}}\Vert x\Vert_{p}^{q}\mathrm{d}x\\
 & \le n^{q/p}\int_{[-1/2,1/2]^{n}}\Vert x\Vert_{\infty}^{q}\mathrm{d}x\\
 & =\frac{2^{-q}n^{q/p+1}}{n+q}.
\end{align*}
Hence, for $\gamma\ge0$,
\begin{align*}
 & \gamma-\Vert v\Vert_{p}^{q}\\
 & \le\int_{[-1/2,1/2]^{n}}\left(\gamma+\frac{2^{-q}n^{q/p+1}}{n+q}-\Vert v+x\Vert_{p}^{q}\right)\mathrm{d}x\\
 & \le\int_{[-1/2,1/2]^{n}}\max\left\{ \gamma+\frac{2^{-q}n^{q/p+1}}{n+q}-\Vert v+x\Vert_{p}^{q},\,0\right\} \mathrm{d}x.
\end{align*}
Hence,

\begin{align*}
 & \max\left\{ \gamma-\Vert v\Vert_{p}^{q},0\right\} \\
 & \le\int_{[-1/2,1/2]^{n}}\max\left\{ \gamma+\frac{2^{-q}n^{q/p+1}}{n+q}-\Vert v+x\Vert_{p}^{q},\,0\right\} \mathrm{d}x.
\end{align*}
We have
\begin{align*}
G_{n,p,q}(\gamma) & =\sum_{v\in\mathbb{Z}^{n}}\max\left\{ \gamma-\Vert v\Vert_{p}^{q},\,0\right\} \\
 & \le\sum_{v\in\mathbb{Z}^{n}}\int_{[-1/2,1/2]^{n}}\max\left\{ \gamma+\frac{2^{-q}n^{q/p+1}}{n+q}-\Vert v+x\Vert_{p}^{q},\,0\right\} \mathrm{d}x\\
 & =\int_{\mathbb{R}^{n}}\max\left\{ \gamma+\frac{2^{-q}n^{q/p+1}}{n+q}-\Vert x\Vert_{p}^{q},\,0\right\} \mathrm{d}x\\
 & =\int_{0}^{\gamma+\frac{2^{-q}n^{q/p+1}}{n+q}}\int_{\mathbb{R}^{n}}\mathbf{1}\left\{ \Vert x\Vert_{p}^{q}\le t\right\} \mathrm{d}x\mathrm{d}t\\
 & =\int_{0}^{\gamma+\frac{2^{-q}n^{q/p+1}}{n+q}}\mathrm{V}_{n,p}t^{n/q}\mathrm{d}t\\
 & =\frac{\mathrm{V}_{n,p}}{n/q+1}\left(\gamma+\frac{2^{-q}n^{q/p+1}}{n+q}\right)^{n/q+1},
\end{align*}
where $\mathrm{V}_{n,p}$ is the volume of the unit $\ell_{p}$ ball
given in \eqref{eq:lp_ball_vol}. Therefore, 
\begin{align}
G_{n,p,q}^{-1}(t) & \ge\left(\frac{n/q+1}{\mathrm{V}_{n,p}}t\right)^{1/(n/q+1)}-\frac{2^{-q}n^{q/p+1}}{n+q}.\label{eq:Ginv_bd_1}
\end{align}

For the second method, since the real-valued function on $\mathbb{R}^{n}$
given by $x\mapsto\max\{\Vert x\Vert_{p}^{q},1\}$ is $q$-Lipschitz
with respect to $\Vert\cdot\Vert_{p}$ (it is the composition of the
$1$-Lipschitz function $x\mapsto\max\{\Vert x\Vert_{p},1\}$ from
$\mathbb{R}^{n}$ to $\mathbb{R}_{\ge1}$ and the $q$-Lipschitz function
$t\mapsto t^{q}$ from $\mathbb{R}_{\ge1}$ to $\mathbb{R}$), for
$v\in\mathbb{Z}^{n}$, we have
\begin{align*}
 & \int_{[-1/2,1/2]^{n}}\left(\max\{\Vert v+x\Vert_{p}^{q},1\}-\max\{\Vert v\Vert_{p}^{q},1\}\right)\mathrm{d}x\\
 & \le q\int_{[-1/2,1/2]^{n}}\Vert x\Vert_{p}\mathrm{d}x\\
 & \le qn^{1/p}\int_{[-1/2,1/2]^{n}}\Vert x\Vert_{\infty}\mathrm{d}x\\
 & =\frac{qn^{1/p+1}}{2(n+1)}.
\end{align*}
Hence, for $\gamma\ge0$,
\begin{align*}
 & \gamma-\max\{\Vert v\Vert_{p}^{q},1\}\\
 & \le\int_{[-1/2,1/2]^{n}}\left(\gamma+\frac{qn^{1/p+1}}{2(n+1)}-\max\{\Vert v+x\Vert_{p}^{q},1\}\right)\mathrm{d}x\\
 & \le\int_{[-1/2,1/2]^{n}}\max\left\{ \gamma+\frac{qn^{1/p+1}}{2(n+1)}-\max\{\Vert v+x\Vert_{p}^{q},1\},\,0\right\} \mathrm{d}x.
\end{align*}
Therefore,

\begin{align*}
 & \max\left\{ \gamma-\max\{\Vert v\Vert_{p}^{q},1\},0\right\} \\
 & \le\int_{[-1/2,1/2]^{n}}\max\left\{ \gamma+\frac{qn^{1/p+1}}{2(n+1)}-\max\{\Vert v+x\Vert_{p}^{q},1\},\,0\right\} \mathrm{d}x.
\end{align*}
We have
\begin{align*}
G_{n,p,q}(\gamma) & =\sum_{v\in\mathbb{Z}^{n}}\max\left\{ \gamma-\Vert v\Vert_{p}^{q},\,0\right\} \\
 & \stackrel{(a)}{\le}\sum_{v\in\mathbb{Z}^{n}}\max\left\{ \gamma-\max\{\Vert v\Vert_{p}^{q},1\},\,0\right\} +1\\
 & \le\sum_{v\in\mathbb{Z}^{n}}\int_{[-1/2,1/2]^{n}}\max\left\{ \gamma+\frac{qn^{1/p+1}}{2(n+1)}-\max\{\Vert v+x\Vert_{p}^{q},1\},\,0\right\} \mathrm{d}x+1\\
 & =\int_{\mathbb{R}^{n}}\max\left\{ \gamma+\frac{qn^{1/p+1}}{2(n+1)}-\max\{\Vert x\Vert_{p}^{q},1\},\,0\right\} \mathrm{d}x+1\\
 & =\int_{1}^{\max\left\{ \gamma+\frac{qn^{1/p+1}}{2(n+1)},\,1\right\} }\int_{\mathbb{R}^{n}}\mathbf{1}\left\{ \Vert x\Vert_{p}^{q}\le t\right\} \mathrm{d}x\mathrm{d}t+1\\
 & =\int_{1}^{\max\left\{ \gamma+\frac{qn^{1/p+1}}{2(n+1)},\,1\right\} }\mathrm{V}_{n,p}t^{n/q}\mathrm{d}t+1\\
 & =\frac{\mathrm{V}_{n,p}}{n/q+1}\left(\left(\max\left\{ \gamma+\frac{qn^{1/p+1}}{2(n+1)},\,1\right\} \right)^{n/q+1}-1\right)+1,
\end{align*}
where (a) is because $\max\{\Vert v\Vert_{p}^{q},1\}\neq\Vert v\Vert_{p}^{q}$
only if $v=0$. Therefore, if $t>1$,
\begin{align*}
G_{n,p,q}^{-1}(t) & \ge\left(\frac{n/q+1}{\mathrm{V}_{n,p}}(t-1)+1\right)^{1/(n/q+1)}-\frac{qn^{1/p+1}}{2(n+1)}.
\end{align*}
If $t>1$ and $(n/q+1)/\mathrm{V}_{n,p}\le1$, then
\begin{align*}
G_{n,p,q}^{-1}(t) & \ge\left(\frac{n/q+1}{\mathrm{V}_{n,p}}t\right)^{1/(n/q+1)}-\frac{qn^{1/p+1}}{2(n+1)}.
\end{align*}
If $t>1$ and $(n/q+1)/\mathrm{V}_{n,p}>1$, then
\begin{align*}
 & G_{n,p,q}^{-1}(t)\\
 & \ge\left(\frac{n/q+1}{\mathrm{V}_{n,p}}t-\left(\frac{n/q+1}{\mathrm{V}_{n,p}}-1\right)\right)^{1/(n/q+1)}-\frac{qn^{1/p+1}}{2(n+1)}\\
 & =\left(\frac{n/q+1}{\mathrm{V}_{n,p}}t\right)^{1/(n/q+1)}-\left(\left(\frac{n/q+1}{\mathrm{V}_{n,p}}t\right)^{1/(n/q+1)}-\left(\frac{n/q+1}{\mathrm{V}_{n,p}}t-\left(\frac{n/q+1}{\mathrm{V}_{n,p}}-1\right)\right)^{1/(n/q+1)}\right)-\frac{qn^{1/p+1}}{2(n+1)}\\
 & \stackrel{(a)}{\ge}\left(\frac{n/q+1}{\mathrm{V}_{n,p}}t\right)^{1/(n/q+1)}-\left(\left(\frac{n/q+1}{\mathrm{V}_{n,p}}\right)^{1/(n/q+1)}-\left(\frac{n/q+1}{\mathrm{V}_{n,p}}-\left(\frac{n/q+1}{\mathrm{V}_{n,p}}-1\right)\right)^{1/(n/q+1)}\right)-\frac{qn^{1/p+1}}{2(n+1)}\\
 & =\left(\frac{n/q+1}{\mathrm{V}_{n,p}}t\right)^{1/(n/q+1)}-\left(\left(\frac{n/q+1}{\mathrm{V}_{n,p}}\right)^{1/(n/q+1)}-1\right)-\frac{qn^{1/p+1}}{2(n+1)},
\end{align*}
where (a) is because $z\mapsto z^{1/(n/q+1)}$ is concave, and hence
$z\mapsto z^{1/(n/q+1)}-(z-a)^{1/(n/q+1)}$ is non-increasing. Therefore,
if $t>1$,
\begin{align*}
G_{n,p,q}^{-1}(t) & \ge\left(\frac{n/q+1}{\mathrm{V}_{n,p}}t\right)^{1/(n/q+1)}-\max\left\{ \left(\frac{n/q+1}{\mathrm{V}_{n,p}}\right)^{1/(n/q+1)}-1,\,0\right\} -\frac{qn^{1/p+1}}{2(n+1)}.
\end{align*}
Combining this with \eqref{eq:Ginv_bd_1}, for $t>1$,
\begin{align*}
G_{n,p,q}^{-1}(t) & \ge\left(\frac{n/q+1}{\mathrm{V}_{n,p}}t\right)^{1/(n/q+1)}-\min\left\{ \frac{2^{-q}n^{q/p+1}}{n+q},\,\max\left\{ \left(\frac{n/q+1}{\mathrm{V}_{n,p}}\right)^{1/(n/q+1)}-1,\,0\right\} +\frac{qn^{1/p+1}}{2(n+1)}\right\} .
\end{align*}
Also note that $G_{n,p,q}^{-1}(t)=t$ when $t\le1$. Hence,
\begin{align*}
 & \int_{0}^{\infty}G_{n,p,q}^{-1}(t)t^{-(1+1/n)}\mathrm{d}t\\
 & =\int_{0}^{1}t\cdot t^{-(1+1/n)}\mathrm{d}t+\int_{1}^{\infty}G_{n,p,q}^{-1}(t)t^{-(1+1/n)}\mathrm{d}t\\
 & \ge\left(1-\frac{1}{n}\right)^{-1}+\int_{1}^{\infty}\Bigg(\left(\frac{n/q+1}{\mathrm{V}_{n,p}}t\right)^{1/(n/q+1)}\\
 & \;\;\;\;-\min\left\{ \frac{2^{-q}n^{q/p+1}}{n+q},\max\left\{ \left(\frac{n/q+1}{\mathrm{V}_{n,p}}\right)^{1/(n/q+1)}\!-1,0\right\} +\frac{qn^{1/p+1}}{2(n+1)}\right\} \Bigg)t^{-(1+1/n)}\mathrm{d}t\\
 & =\frac{1}{n^{-1}-(n/q+1)^{-1}}\left(\frac{n/q+1}{\mathrm{V}_{n,p}}\right)^{1/(n/q+1)}-n\psi_{n,p,q}.
\end{align*}
where
\[
\psi_{n,p,q}:=\min\left\{ \frac{2^{-q}n^{q/p+1}}{n+q},\,\max\left\{ \left(\frac{n/q+1}{\mathrm{V}_{n,p}}\right)^{1/(n/q+1)}-1,0\right\} +\frac{qn^{1/p+1}}{2(n+1)}\right\} -\frac{1}{n-1}
\]
Combining this with \eqref{eq:rc_lb_Gnpq},
\begin{align*}
 & r_{c}^{*}(\mathcal{P}(\mathbb{Z}^{n}))\\
 & \ge\left(1+\mathbf{1}\{q<1\}\right)\frac{1}{n}\left(1-\frac{1}{n}\right)\left(\frac{1}{n^{-1}-(n/q+1)^{-1}}\left(\frac{n/q+1}{\mathrm{V}_{n,p}}\right)^{1/(n/q+1)}-n\psi_{n,p,q}\right)\\
 & =\left(1+\mathbf{1}\{q<1\}\right)\left(1-\frac{1}{n}\right)\left(\frac{n+q}{n+q-nq}\left(\frac{n/q+1}{\mathrm{V}_{n,p}}\right)^{1/(n/q+1)}-\psi_{n,p,q}\right).
\end{align*}

\medskip{}

\subsection{Rate of growth of $r_{c}^{*}(\mathcal{P}(\mathbb{Z}^{n}))$}

We now find the rate of growth of $r_{c}^{*}(\mathcal{P}(\mathbb{Z}^{n}))$
as $n$ increases. If $p<\infty$,
\begin{align*}
 & r_{c}^{*}(\mathcal{P}(\mathbb{Z}^{n}))\\
 & \ge\left(1+\mathbf{1}\{q<1\}\right)\left(1-\frac{1}{n}\right)\left(\frac{n+q}{n+q-nq}\left(\frac{(n/q+1)\mathit{\Gamma}(1+n/p)}{2^{n}(\mathit{\Gamma}(1+1/p))^{n}}\right)^{1/(n/q+1)}+\frac{1}{n-1}-\frac{2^{-q}n^{q/p+1}}{n+q}\right)\\
 & \stackrel{(a)}{\ge}\!\left(1\!+\!\mathbf{1}\{q\!<\!1\}\right)\!\left(1\!-\!\frac{1}{n}\right)\!\left(\frac{n+q}{n+q-nq}\left(\frac{(n/q+1)\sqrt{2\pi}(1+n/p)^{(1+n/p)-1/2}e^{-(1+n/p)}}{2^{n}(\mathit{\Gamma}(1+1/p))^{n}}\right)^{\!1/(n/q+1)}\!\!\!\!+\!\frac{1}{n-1}\!-\!\frac{2^{-q}n^{q/p+1}}{n+q}\right)\\
 & \ge\left(1+\mathbf{1}\{q<1\}\right)\left(1-\frac{1}{n}\right)\left(\frac{n+q}{n+q-nq}\left(\frac{(n/q+1)\sqrt{2\pi}(n/p)^{n/p}e^{-(1+n/p)}}{2^{n}(\mathit{\Gamma}(1+1/p))^{n}}\right)^{1/(n/q+1)}-\frac{2^{-q}n^{q/p+1}}{n+q}\right)\\
 & =\left(1+\mathbf{1}\{q<1\}\right)\left(1-\frac{1}{n}\right)\left(\frac{n+q}{n+q-nq}\left(\sqrt{2\pi}e^{-1}(n/q+1)\right)^{1/(n/q+1)}\left(\frac{(n/p)^{1/p}e^{-1/p}}{2\mathit{\Gamma}(1+1/p)}\right)^{1/(1/q+1/n)}-\frac{2^{-q}n^{q/p+1}}{n+q}\right)\\
 & \stackrel{(b)}{\ge}\left(1+\mathbf{1}\{q<1\}\right)\left(1-\frac{1}{n}\right)\left(\frac{n+q}{n+q-nq}\left(\frac{(n/p)^{1/p}e^{-1/p}}{2}\right)^{1/(1/q+1/n)}-2^{-q}n^{q/p}\right)\\
 & \ge2^{-q}\left(1+\mathbf{1}\{q<1\}\right)\left(1-\frac{1}{n}\right)\left(\frac{n+q}{n+q-nq}\left(\frac{n}{ep}\right)^{1/(p/q+p/n)}-n^{q/p}\right)\\
 & =2^{\mathbf{1}\{q<1\}-q}\left(1-\frac{1}{n}\right)\left(\frac{n+q}{n+q-nq}\left(\frac{n}{ep}\right)^{1/(p/q+p/n)}-n^{q/p}\right),
\end{align*}
where (a) is by Stirling's approximation, and (b) is because $\mathit{\Gamma}(1+1/p)\le1$
and $\sqrt{2\pi}e^{-1}(n/q+1)\ge1$. Hence $r_{c}^{*}(\mathcal{P}(\mathbb{Z}^{n}))=\Omega(n^{1+q/p})$
as $n\to\infty$ if $q=1$. If $q<1$, then
\begin{align*}
 & r_{c}^{*}(\mathcal{P}(\mathbb{Z}^{n}))\\
 & \ge2^{\mathbf{1}\{q<1\}-q}\left(1-\frac{1}{n}\right)\left(\frac{n+q}{n+q-nq}\left(\frac{n}{ep}\right)^{1/(p/q+p/n)}-n^{q/p}\right)\\
 & =2^{\mathbf{1}\{q<1\}-q}\left(1-\frac{1}{n}\right)\left((1-o(1))\frac{1}{1-q}\left(\frac{n}{ep}\right)^{q/p}-n^{q/p}\right)\\
 & =2^{\mathbf{1}\{q<1\}-q}\left(1-\frac{1}{n}\right)\left((1-o(1))\frac{(ep)^{-q/p}}{1-q}-1\right)n^{q/p}\\
 & \stackrel{(a)}{\ge}2^{\mathbf{1}\{q<1\}-q}\left(1-\frac{1}{n}\right)\left((1-o(1))\frac{e^{-q}}{1-q}-1\right)n^{q/p}\\
 & \stackrel{(b)}{=}2^{\mathbf{1}\{q<1\}-q}\left(\frac{e^{-q}}{1-q}-1\right)(1-o(1))n^{q/p}\\
 & =\Omega(n^{q/p})
\end{align*}
where (a) is because $(ep)^{-1/p}$ is increasing in $p\in[1,\infty)$,
and hence $(ep)^{-1/p}\ge e^{-1}$, and (b) is because $e^{-q}/(1-q)$
is strictly increasing in $q\in(0,1)$, and hence $e^{-q}/(1-q)>1$.

If $p=\infty$,
\begin{align*}
 & r_{c}^{*}(\mathcal{P}(\mathbb{Z}^{n}))\\
 & \ge\left(1+\mathbf{1}\{q<1\}\right)\left(1-\frac{1}{n}\right)\left(\frac{1}{1-(1/q+1/n)^{-1}}\left(\frac{(n/q+1)}{2^{n}}\right)^{1/(n/q+1)}+\frac{1}{n-1}-2^{-q}\frac{n}{n+q}\right)\\
 & =\left(1+\mathbf{1}\{q<1\}\right)\left(1-\frac{1}{n}\right)\left(\frac{n/q+1}{n(1/q-1)+1}(n/q+1)^{1/(n/q+1)}2^{-1/(1/q+1/n)}+\frac{1}{n-1}-2^{-q}\frac{n}{n+q}\right),
\end{align*}
and hence $r_{c}^{*}(\mathcal{P}(\mathbb{Z}^{n}))=\Omega(n^{\mathbf{1}\{q=1\}})$.

\medskip{}

\subsection{Extension to $\mathcal{P}_{\ll\lambda_{S}}(\mathbb{R}^{n})$ and
$\mathcal{P}(\mathcal{M})$\label{subsec:bd_ball_cont}}

We extend this result to $\mathcal{P}_{\ll\lambda_{S}}(\mathbb{R}^{n})$
for any $S\subseteq\mathbb{R}^{n}$ with $\lambda(S)>0$. By the definition
of Lebesgue outer measure,
\[
\lambda(S)=\inf\left\{ \sum_{i=1}^{\infty}\lambda(R_{i}):\,R_{i}\in\mathcal{C}_{n},\,S\subseteq\bigcup_{i=1}^{\infty}R_{i}\right\} ,
\]
where $\mathcal{C}_{n}$ is the set of hypercubes in the form $[z_{1}-w/2,z_{1}+w/2]\times\cdots\times[z_{n}-w/2,z_{n}+w/2]$
for some $a\in\mathbb{R}^{n}$, $w\in\mathbb{R}_{>0}$ (while the
original definition is for $R_{i}$ being rectangle sets, any rectangle
set can be covered by a union of hypercubes with total volume arbitrarily
close to the volume of the rectangle). Hence for any $\epsilon>0$,
we can find $R_{1},R_{2},\ldots\in\mathcal{C}_{n}$ such that $S\subseteq\bigcup_{i}R_{i}$
and $\sum_{i}\lambda(R_{i})<(1+\epsilon)\lambda(S)$. Since $\lambda(S)\le\sum_{i}\lambda(R_{i}\cap S)$,
there exists $i$ with $\lambda(R_{i})\le(1+\epsilon)\lambda(R_{i}\cap S)$,
and hence $\lambda(R_{i}\backslash S)\le(\epsilon/(1+\epsilon))\lambda(R_{i})<\epsilon\lambda(R_{i})$.

For brevity, write $R:=R_{i}$, and let $R=[z_{1}-w/2,z_{1}+w/2]\times\cdots\times[z_{n}-w/2,z_{n}+w/2]$.
We now construct probability distributions $\{\tilde{P}_{\alpha}\}_{\alpha\in\mathcal{A}}$,
over $R\cap S$, where $\mathcal{A}:=[-(k-1)/2..\,(k-1)/2]^{n}$ with
$k:=\lfloor n^{-1}\epsilon^{-1/(2n)}\rfloor$ (assume $\epsilon$
is small enough that $k\ge3$). Let 
\[
\gamma:=nw\epsilon^{1/n}/2,
\]
\[
\mathcal{B}_{\gamma}(x):=\{y\in\mathbb{R}^{n}:\,\Vert y-x\Vert_{1}\le\gamma\},
\]
\[
Q_{\alpha}:=\frac{\lambda_{R\cap S\cap\mathcal{B}_{\gamma}((w/k)\alpha+z)}}{\lambda(R\cap S\cap\mathcal{B}_{\gamma}((w/k)\alpha+z))},
\]
\[
\tilde{P}_{\alpha}:=\sum_{\beta\in\mathcal{A}}P_{\alpha}(\beta)Q_{\beta},
\]
where $P_{\alpha}:=(|\mathcal{A}|-1)^{-1}\sum_{\beta\in\mathcal{A}\backslash\{\alpha\}}\delta_{\beta}$
is defined before. Since $k\le n^{-1}\epsilon^{-1/(2n)}\le n^{-1}\epsilon^{-1/n}=w/(2\gamma)$,
it can be checked that $\mathcal{B}_{\gamma}((w/k)\alpha+z)\subseteq R$,
and hence
\begin{align*}
 & \lambda(R\cap S\cap\mathcal{B}_{\gamma}((w/k)\alpha+z))\\
 & \ge\lambda(\mathcal{B}_{\gamma}((w/k)\alpha+z))-\lambda(R\backslash S)\\
 & >(2\gamma/n)^{n}-\epsilon w^{n}\\
 & =0,
\end{align*}
and $Q_{\alpha}$ are valid probability measures. For $\alpha\neq\beta\in\mathcal{A}$,
$x\in\mathcal{B}_{\gamma}((w/k)\alpha+z)$, $y\in\mathcal{B}_{\gamma}((w/k)\beta+z)$,
\begin{align*}
\Vert x-y\Vert_{p} & \ge\Vert((w/k)\alpha+z)-((w/k)\beta+z)\Vert_{p}-2\gamma\\
 & =\frac{w}{k}\Vert\alpha-\beta\Vert_{p}-2\gamma\\
 & \ge\left(\frac{w}{k}-2\gamma\right)\Vert\alpha-\beta\Vert_{p}\\
 & =w\left(\frac{1}{k}-n\epsilon^{1/n}\right)\Vert\alpha-\beta\Vert_{p}\\
 & \ge wn\left(\epsilon^{1/(2n)}-\epsilon^{1/n}\right)\Vert\alpha-\beta\Vert_{p}.
\end{align*}
This is also obviously true for $\alpha=\beta$. Fix any coupling
$\{\tilde{X}_{\alpha}\}_{\alpha\in\mathcal{A}}$ of $\{\tilde{P}_{\alpha}\}_{\alpha\in\mathcal{A}}$.
Let $X_{\alpha}:=\mathrm{round}((k/w)(\tilde{X}_{\alpha}-z))$, where
$\mathrm{round}(x):=(\lfloor x_{1}+1/2\rfloor,\ldots,\lfloor x_{n}+1/2\rfloor)$,
then $X_{\alpha}\sim P_{\alpha}$, and
\begin{align*}
 & \mathbf{E}[c(\tilde{X}_{\alpha},\tilde{X}_{\beta})]\\
 & \ge\left(wn(\epsilon^{1/(2n)}-\epsilon^{1/n})\right)^{q}\mathbf{E}[c(X_{\alpha},X_{\beta})].
\end{align*}
Also, for $\alpha\neq\beta\in\mathcal{A}$, $x\in\mathcal{B}_{\gamma}((w/k)\alpha+z)$,
$y\in\mathcal{B}_{\gamma}((w/k)\beta+z)$,
\begin{align*}
\Vert x-y\Vert_{p} & \le\Vert((w/k)\alpha+z)-((w/k)\beta+z)\Vert_{p}+2\gamma\\
 & =\frac{w}{k}\Vert\alpha-\beta\Vert_{p}+2\gamma\\
 & \le\left(\frac{w}{k}+2\gamma\right)\Vert\alpha-\beta\Vert_{p}\\
 & \le\left(\frac{w}{n^{-1}\epsilon^{-1/(2n)}-1}+nw\epsilon^{1/n}\right)\Vert\alpha-\beta\Vert_{p}\\
 & =wn\left((\epsilon^{-1/(2n)}-n)^{-1}+\epsilon^{1/n}\right)\Vert\alpha-\beta\Vert_{p}.
\end{align*}
Given any coupling $(X,Y)$ of $(P_{\alpha},P_{\beta})$, we can construct
a coupling $(\tilde{X},\tilde{Y})$ of $(\tilde{P}_{\alpha},\tilde{P}_{\beta})$
where $X=\mathrm{round}((k/w)(\tilde{X}-z))$, $Y=\mathrm{round}((k/w)(\tilde{Y}-z))$,
and $\tilde{X}=\tilde{Y}$ when $X=Y$. This can be achieved by letting
$B\sim\mathrm{Unif}[0,1]$ and $\tilde{X}=f_{X}(B)$, $\tilde{Y}=f_{Y}(B)$
where $f_{\alpha}:[0,1]\to\mathbb{R}^{n}$, $f_{\alpha*}\lambda_{[0,1]}=Q_{\alpha}$
for $\alpha\in\mathcal{A}$. Hence,
\[
C_{c}^{*}(\tilde{P}_{\alpha},\tilde{P}_{\beta})\le\left(wn((\epsilon^{-1/(2n)}-n)^{-1}+\epsilon^{1/n})\right)^{q}C_{c}^{*}(P_{\alpha},P_{\beta}).
\]
Therefore,
\[
r_{c}(\{\tilde{X}_{\alpha}\}_{\alpha})\ge\frac{\left(wn(\epsilon^{1/(2n)}-\epsilon^{1/n})\right)^{q}}{\left(wn((\epsilon^{-1/(2n)}-n)^{-1}+\epsilon^{1/n})\right)^{q}}r_{c}(\{X_{\alpha}\}_{\alpha}).
\]
Let $\epsilon\to0$ (and hence $k\to\infty$). Since
\[
\frac{\left(wn(\epsilon^{1/(2n)}-\epsilon^{1/n})\right)^{q}}{\left(wn((\epsilon^{-1/(2n)}-n)^{-1}+\epsilon^{1/n})\right)^{q}}\to1,
\]
we obtain the same result as in \eqref{eq:rc_lb_Gnpq}:
\begin{align*}
 & r_{c}^{*}(\mathcal{P}_{\ll\lambda_{S}}(\mathbb{R}^{n}))\\
 & \ge\left(1+\mathbf{1}\{q<1\}\right)\frac{1}{n}\left(1-\frac{1}{n}\right)\int_{0}^{\infty}G_{n,p,q}^{-1}(t)t^{-(1+1/n)}\mathrm{d}t.
\end{align*}

\medskip{}

We finally consider the case $\mathcal{P}(\mathcal{M})$ where $\mathcal{M}$
is a connected smooth complete $n$-dimensional Riemannian manifold.
We take $p=2$ and $c(x,y)=(d_{\mathcal{M}}(x,y))^{q}$, where $d_{\mathcal{M}}$
denotes the intrinsic distance on the manifold $\mathcal{M}$. Let
$x_{0}\in\mathcal{M}$, and let $\psi:E\to F$ be a coordinates chart
at $x_{0}$ corresponding to the geodesically normal coordinates at
$x_{0}$, where $E\subseteq\mathcal{M}$ contains a neighborhood of
$x_{0}$, $F\subseteq\mathbb{R}^{n}$, and $\psi(x_{0})=0$. Let $\epsilon>0$,
and let $E'\subseteq E$ be an open neighborhood of $x_{0}$ small
enough that 
\begin{equation}
(1-\epsilon)\Vert\psi(x)-\psi(y)\Vert_{2}\le d_{\mathcal{M}}(x,y)\le(1+\epsilon)\Vert\psi(x)-\psi(y)\Vert_{2}\label{eq:manifold_dist_bd}
\end{equation}
for all $x,y\in E'$. This is possible due to the smoothness of the
Riemannian metric tensor (see \cite[Definition 1.4.1, Theorem 1.4.4]{jost2008riemannian}).
Let $w>0$ be small enough that $\{y:\Vert y\Vert_{\infty}\le w\}\subseteq\psi(E')$,
and let $\bar{P}_{\alpha}:=\sum_{\beta\in\mathcal{A}}P_{\alpha}(\beta)\delta_{\psi^{-1}((w/k)\beta)}$,
where $\mathcal{A}$ and $P_{\alpha}$ are defined before. By \eqref{eq:manifold_dist_bd}
and the ratio bound in Proposition \ref{prop:rc_prop_misc},
\[
r_{c}^{*}(\mathcal{P}(\mathcal{M}))\ge r_{c}^{*}(\{\bar{P}_{\alpha}\}_{\alpha})\ge\left(\frac{1-\epsilon}{1+\epsilon}\right)^{q}r_{c}^{*}(\{P_{\alpha}\}_{\alpha}).
\]
Letting $k\to\infty$, and then $\epsilon\to0$, we obtain the same
result as in \eqref{eq:rc_lb_Gnpq}:
\begin{align*}
 & r_{c}^{*}(\mathcal{P}(\mathcal{M}))\\
 & \ge\left(1+\mathbf{1}\{q<1\}\right)\frac{1}{n}\left(1-\frac{1}{n}\right)\int_{0}^{\infty}G_{n,2,q}^{-1}(t)t^{-(1+1/n)}\mathrm{d}t.
\end{align*}
\[
\]

\medskip{}

\section{Proof of Corollary \ref{cor:rn_rc_lb_infdim}\label{subsec:pf_rn_rc_lb_infdim}}

The strategy is to find a sequence of embedding functions $g_{n}:S\to\mathcal{X}$,
where either $S=\mathbb{Z}^{n}$ or $S\subseteq\mathbb{R}^{n}$ with
$\lambda(S)>0$, such that $c(g_{n}(u),g_{n}(v))=\theta_{n}\Vert u-v\Vert_{p}^{q}$
for some constant $\theta_{n}>0$ that only depends on $n$ (i.e.,
$c(g_{n}(u),g_{n}(v))$ is proportional to $\Vert u-v\Vert_{p}^{q}$).
We will then have
\begin{align*}
r_{c}^{*}(\mathcal{P}(\mathcal{X})) & \ge r_{\Vert u-v\Vert_{p}^{q}}^{*}(\mathcal{P}(S)).
\end{align*}
If $0<q<1$, we invoke Theorem \ref{thm:rn_rc_lb_ball} to show that
$r_{\Vert u-v\Vert_{p}^{q}}^{*}(\mathcal{P}(S))=\Omega(n^{q/p})$,
which tends to $\infty$ as $n\to\infty$ if $p<\infty$. If $q\ge1$,
we use Proposition \ref{prop:rn_rc_lb_2} instead to show $r_{\Vert u-v\Vert_{p}^{q}}^{*}(\mathcal{P}(S))=\infty$.

For the case $\mathcal{X}=\{x\in\mathbb{R}^{\mathbb{N}}:\Vert x\Vert_{p}<\infty\}$,
$c(x,y)=\Vert x-y\Vert_{p}^{q}$, $p\in\mathbb{R}_{\ge1}$, $q>0$,
the result follows directly by considering the embedding $g_{n}:\mathbb{R}^{n}\to\mathcal{X}$
defined by $g_{n}(u):=(u_{1},\ldots,u_{n},0,0,\ldots)$. This result
also holds for $r_{c}^{*}(\mathcal{P}(\mathcal{X}'))$, where $\mathcal{X}'\subseteq\mathcal{X}$
is the set of non-negative, non-increasing sequences that sum to 1.
In this case, consider the embedding $g_{n}:[-1,1]^{n}\to\mathcal{X}'$
defined by
\[
(g_{n}(u))_{i}=\mathbf{1}\{i\le2n\}\frac{2n+1-i+(-1)^{i}u_{\lceil i/2\rceil}/2}{n(2n+1)}.
\]
It is straightforward to check that $\Vert g_{n}(u)-g_{n}(v)\Vert_{p}^{q}$
is proportional to $\Vert u-v\Vert_{p}^{q}$.

Next we consider the case where $\mathcal{X}$ is the space of all
continuous functions $f:[0,1]\to\mathbb{R}$ (with the topology and
$\sigma$-algebra generated by the $L_{\infty}$ metric), $c(f,g)=\Vert f-g\Vert_{p}^{q}$
is the $L_{p}$ metric to the power $q$, where $p\in\mathbb{R}_{\ge1}\cup\{\infty\}$,
$q>0$ satisfy $p<\infty$ or $q\ge1$. Refer to Remark \ref{rem:contfcn_meas}
for the proof that $c$ is measurable. Consider the embedding $g_{n}:\mathbb{Z}^{n}\to\mathcal{X}$
defined by
\[
(g_{n}(u))(t)=\left(1-\left|2\left(nt-\lfloor nt\rfloor\right)-1\right|\right)u_{\min\{\lfloor nt\rfloor+1,n\}}.
\]
It is straightforward to check that $\Vert g_{n}(u)-g_{n}(v)\Vert_{p}^{q}$
is proportional to $\Vert u-v\Vert_{p}^{q}$. The result follows.

Note that this result also holds for $r_{c}^{*}(\mathcal{P}(\mathcal{X}'))$,
where $\mathcal{X}'\subseteq\mathcal{X}=\mathrm{C}([0,1],\mathbb{R})$
is the set of non-negative, infinitely differentiable, $1$-Lipschitz
functions $f$ with $\int_{0}^{1}f=1$. In this case, we consider
the embedding $g_{n}:[-1,1]^{n}\to\mathcal{X}'$ defined by
\[
(g_{n}(u))(t)=1+\frac{1}{4n}\sum_{i=1}^{n}\left(\psi(4nt-4i+3)-\psi(4nt-4i+1)\right)u_{i},
\]
where $\psi(t):=\mathbf{1}\{|t|<1\}e^{-\frac{1}{1-t^{2}}}$ is infinitely
differentiable and $1$-Lipschitz. It is straightforward to check
that $\Vert g_{n}(u)-g_{n}(v)\Vert_{p}^{q}$ is proportional to $\Vert u-v\Vert_{p}^{q}$.

Next we consider the case $\mathcal{X}=\{x\in\{0,1\}^{\mathbb{N}}:\,\sum_{i}x_{i}<\infty\}$,
$c(x,y)=\Vert x-y\Vert_{1}^{q}$, $q>0$. Let $f:\mathbb{N}\to\mathbb{N}\times\mathbb{Z}$
be a bijection (write $f(i)=(f_{1}(i),f_{2}(i))$). Let $g_{n}:\mathbb{Z}^{n}\to\mathcal{X}$
be defined by
\[
(g_{n}(u))_{i}=\mathbf{1}\left\{ f_{1}(i)\le n\;\mathrm{and}\;f_{2}(i)\in[u_{f_{1}(i)}..-1]\cup[0..u_{f_{1}(i)}-1]\right\} .
\]
Note that $[u_{f_{1}(i)}..-1]=\emptyset$ if $u_{f_{1}(i)}\ge0$,
and $[0..u_{f_{1}(i)}-1]=\emptyset$ if $u_{f_{1}(i)}\le0$. We have
\begin{align*}
 & \Vert g_{n}(u)-g_{n}(v)\Vert_{1}\\
 & =\Big|\Big\{ i:\,\mathbf{1}\{f_{1}(i)\le n\;\mathrm{and}\;f_{2}(i)\in[u_{f_{1}(i)}..-1]\cup[0..u_{f_{1}(i)}-1]\}\\
 & \;\;\;\;\;\;\;\neq\mathbf{1}\{f_{1}(i)\le n\;\mathrm{and}\;f_{2}(i)\in[v_{f_{1}(i)}..-1]\cup[0..v_{f_{1}(i)}-1]\}\Big\}\Big|\\
 & =\Big|\Big\{(i,j)\in\mathbb{N}\times\mathbb{Z}:\,\mathbf{1}\{i\le n\;\mathrm{and}\;j\in[u_{i}..-1]\cup[0..u_{i}-1]\}\\
 & \;\;\;\;\;\;\;\neq\mathbf{1}\{i\le n\;\mathrm{and}\;j\in[v_{i}..-1]\cup[0..v_{i}-1]\}\Big\}\Big|\\
 & =\sum_{i=1}^{n}\bigg(\left|\left([u_{i}..-1]\cup[0..u_{i}-1]\right)\backslash\left([v_{i}..-1]\cup[0..v_{i}-1]\right)\right|\\
 & \;\;\;\;\;\;\;+\left|\left([v_{i}..-1]\cup[0..v_{i}-1]\right)\backslash\left([u_{i}..-1]\cup[0..u_{i}-1]\right)\right|\bigg)\\
 & =\Vert u-v\Vert_{1},
\end{align*}
and hence $c(g_{n}(u),g_{n}(v))=\Vert u-v\Vert_{1}^{q}$.
\begin{rem}
\label{rem:contfcn_meas}Suppose $\mathcal{X}$ is the space of all
continuous functions $f:[0,1]\to\mathbb{R}$ (with the topology and
$\sigma$-algebra generated by the $L_{\infty}$ metric), $c(f,g)=\Vert f-g\Vert_{p}^{q}$,
$p\in\mathbb{R}_{\ge1}\cup\{\infty\}$, $q>0$. Here we show that
$c$ is measurable for the sake of completeness. Let $\Xi\subseteq\mathcal{X}$
be the set of all polynomials with rational coefficients. Note that
$\Xi$ is countable. By the Weierstrass approximation theorem, any
$f\in\mathcal{X}$ can be lower-bounded (or upper-bounded) by a function
in $\Xi$ that is arbitrarily close to $f$ in the $L_{\infty}$ (and
hence $L_{p}$) metric. As a result, for any $\gamma>0$,
\begin{align*}
 & \left\{ (f,g):\,\Vert f-g\Vert_{p}<\gamma\right\} \\
 & =\bigcup_{\overset{\alpha_{1},\alpha_{2},\beta_{1},\beta_{2}\in\Xi:}{\overset{\Vert\alpha_{1}-\alpha_{2}\Vert_{p}+\Vert\alpha_{1}-\beta_{1}\Vert_{p}+\Vert\beta_{1}-\beta_{2}\Vert_{p}<\gamma}{}}}\left\{ (f,g):\,\alpha_{1}\le f\le\alpha_{2},\,\beta_{1}\le g\le\beta_{2}\right\} 
\end{align*}
is in the product $\sigma$-algebra since $\{f:f\ge\alpha_{1}\}=\bigcup_{i\in\mathbb{N}}\{f:\Vert f-(\alpha_{1}+i)\Vert_{\infty}\le i\}$.
Hence $c$ is measurable.
\end{rem}

\medskip{}

\section{Proof of Proposition \ref{prop:circle} and Proposition \ref{prop:circle_trunc}\label{subsec:pf_circle}}

We represent the circle as $\mathcal{X}:=[0,1)$ with the metric $d(x,y):=\min\{|x-y|,\,1-|x-y|\}$,
and $c(x,y)=(d(x,y))^{q}$, $q>0$. While the length of $\mathcal{X}=\mathcal{M}$
in Proposition \ref{prop:circle} is $2\pi$ instead of $1$, scaling
the whole space would not affect $r_{c}^{*}(\mathcal{P}(\mathcal{X}))$.

We first prove that $r_{c}^{*}(\mathcal{P}(\mathcal{X}))\ge2$ for
any $q>0$. Applying Proposition \ref{prop:rn_rc_lb_1} on $x_{i}=(i-1)/k$
for $i=1,\ldots,k$, we have
\begin{align*}
r_{c}^{*}(\mathcal{P}(\mathcal{X})) & \ge\frac{2(k-1)\min_{1\le i<j\le k}c(x_{i},x_{j})}{\sum_{i=1}^{k}c(x_{i},x_{i+1})}\\
 & =\frac{2(k-1)k^{-q}}{k\cdot k^{-q}}\\
 & =\frac{2(k-1)}{k}.
\end{align*}
The result follows from letting $k\to\infty$.

We then prove that $r_{c}^{*}(\mathcal{P}(\mathcal{X}))=\infty$ for
$q>1$. Applying Proposition \ref{prop:rn_rc_lb_2} on $x_{i}=(i-1)/(2k)$
for $i=1,\ldots,2k$, we have
\begin{align}
r_{c}^{*}(\mathcal{P}(\mathcal{X})) & \ge r_{c}^{*}\left(\mathcal{P}(\{x_{i}:i\in[1..2k]\})\right)\nonumber \\
 & \ge\left(\sum_{i=1}^{k}\frac{c(x_{i},x_{i+1})+c(x_{i+k},x_{i+k+1})}{c(x_{i},x_{i+k+1})+c(x_{i+k},x_{i+1})-c(x_{i},x_{i+1})-c(x_{i+k},x_{i+k+1})}\right)^{-1}+1\nonumber \\
 & =\left(k\cdot\frac{2(2k)^{-q}}{2((k-1)/(2k))^{q}-2(2k)^{-q}}\right)^{-1}+1\nonumber \\
 & =\left(k\cdot\frac{1}{(k-1)^{q}-1}\right)^{-1}+1\nonumber \\
 & =\frac{(k-1)^{q}-1}{k}+1\nonumber \\
 & =\Omega(k^{q-1}).\label{eq:circle_order}
\end{align}
The result follows from letting $k\to\infty$.

The upper bound in Proposition \ref{prop:circle} for $0<q<1$ follows
from Corollary \ref{cor:dyadic_manifold}.

We then prove that $r_{c}^{*}(\mathcal{P}(\mathcal{X}))\le2$ for
$q=1$ using a similar idea as in \cite{karp1989competitive}. Consider
the collection of probability distribution $\{P_{\alpha}\}_{\alpha\in\mathcal{A}}$
over $\mathcal{X}$. For $z\in\mathbb{R}$, let $\psi_{z}:[0,1)\to[0,1)$
defined by $\psi_{z}(x)\equiv x+z\;\mathrm{mod}\;1$. We construct
a coupling by $X_{\alpha,z}:=\psi_{-z}(F_{\psi_{z*}P_{\alpha}}^{-1}(U))$
(where $F_{\psi_{z*}P_{\alpha}}^{-1}$ is the inverse of the cdf of
$\psi_{z}(X)$ when $X\sim P_{\alpha}$), $X_{\alpha}:=X_{\alpha,Z}$,
where $U\sim\mathrm{Unif}[0,1]$ independent of $Z\sim\mathrm{Unif}[0,1]$.
It is straightforward to check that $X_{\alpha,z}\sim P_{\alpha}$
for any $z$, and hence $X_{\alpha}\sim P_{\alpha}$.

Consider two probability distributions $P_{\alpha},P_{\beta}$. Fix
any coupling $(\tilde{X}_{\alpha},\tilde{X}_{\beta})\in\Gamma_{\lambda}(P_{\alpha},P_{\beta})$
(assume $(\tilde{X}_{\alpha},\tilde{X}_{\beta})$ is independent of
$(U,Z)$). We have
\begin{align*}
 & \mathbf{E}\left[c(X_{\alpha},X_{\beta})\right]\\
 & =\mathbf{E}\left[c\big(\psi_{Z}(X_{\alpha,Z}),\psi_{Z}(X_{\beta,Z})\big)\right]\\
 & \le\mathbf{E}\left[\left|\psi_{Z}(X_{\alpha,Z})-\psi_{Z}(X_{\beta,Z})\right|\right]\\
 & =\mathbf{E}\left[\mathbf{E}\left[\left|F_{\psi_{Z*}P_{\alpha}}^{-1}(U)-F_{\psi_{Z*}P_{\beta}}^{-1}(U)\right|\,\big|\,Z\right]\right]\\
 & \stackrel{(a)}{\le}\mathbf{E}\left[\mathbf{E}\left[\left|\psi_{Z}(\tilde{X}_{\alpha})-\psi_{Z}(\tilde{X}_{\beta})\right|\,\big|\,Z\right]\right]\\
 & =\mathbf{E}\left[\mathbf{E}\left[\left|\psi_{Z}(\tilde{X}_{\alpha})-\psi_{Z}(\tilde{X}_{\beta})\right|\,\big|\,\tilde{X}_{\alpha},\tilde{X}_{\beta}\right]\right]\\
 & \stackrel{(b)}{=}\mathbf{E}\left[2d(\tilde{X}_{\alpha},\tilde{X}_{\beta})\left(1-d(\tilde{X}_{\alpha},\tilde{X}_{\beta})\right)\right]\\
 & \le2\mathbf{E}\left[d(\tilde{X}_{\alpha},\tilde{X}_{\beta})\right]\left(1-\mathbf{E}\left[d(\tilde{X}_{\alpha},\tilde{X}_{\beta})\right]\right)\\
 & \le2\mathbf{E}\left[d(\tilde{X}_{\alpha},\tilde{X}_{\beta})\right],
\end{align*}
where (a) is by the optimality of the quantile coupling, and (b) is
by $\mathbf{E}[|\psi_{Z}(x)-\psi_{Z}(y)|]=2d(x,y)(1-d(x,y))$. Hence,
\begin{align*}
 & \mathbf{E}\left[c(X_{\alpha},X_{\beta})\right]\\
 & \le\inf_{(\tilde{X}_{\alpha},\tilde{X}_{\beta})\in\Gamma_{\lambda}(P_{\alpha},P_{\beta})}2\mathbf{E}\left[d(\tilde{X}_{\alpha},\tilde{X}_{\beta})\right]\\
 & =2C_{c}^{*}(P_{\alpha},P_{\beta}).
\end{align*}

The same construction can also be used to prove Proposition \ref{prop:circle_trunc}
for $q\ge1$. For $\gamma\ge0$, let $h_{\gamma}:\mathbb{R}_{\ge0}\to\mathbb{R}_{\ge0}$
be defined by
\[
h_{\gamma}(x):=\begin{cases}
x^{q} & \mathrm{if}\;x\le\gamma^{1/q}\\
q\gamma^{1-1/q}(x-\gamma^{1/q})+\gamma & \mathrm{if}\;x>\gamma^{1/q}.
\end{cases}
\]
It can be checked that $h_{\gamma}$ is convex and $h_{\gamma}(x)\le\min\{x^{q},q\gamma^{1-1/q}x\}$.
Consider two probability distributions $P_{\alpha},P_{\beta}$. Fix
any coupling $(\tilde{X}_{\alpha},\tilde{X}_{\beta})\in\Gamma_{\lambda}(P_{\alpha},P_{\beta})$
(assume $(\tilde{X}_{\alpha},\tilde{X}_{\beta})$ is independent of
$(U,Z)$). For any $\gamma\ge0$, we have
\begin{align*}
 & \mathbf{E}\left[\min\left\{ c(X_{\alpha},X_{\beta}),\,\gamma\right\} \right]\\
 & =\mathbf{E}\left[\min\left\{ c\big(\psi_{Z}(X_{\alpha,Z}),\psi_{Z}(X_{\beta,Z})\big),\,\gamma\right\} \right]\\
 & \le\mathbf{E}\left[\min\left\{ \left|\psi_{Z}(X_{\alpha,Z})-\psi_{Z}(X_{\beta,Z})\right|^{q},\,\gamma\right\} \right]\\
 & \le\mathbf{E}\left[h_{\gamma}\left(|\psi_{Z}(X_{\alpha,Z})-\psi_{Z}(X_{\beta,Z})|\right)\right]\\
 & =\mathbf{E}\left[\mathbf{E}\left[h_{\gamma}\left(\left|F_{\psi_{Z*}P_{\alpha}}^{-1}(U)-F_{\psi_{Z*}P_{\beta}}^{-1}(U)\right|\right)\,\bigg|\,Z\right]\right]\\
 & \stackrel{(a)}{\le}\mathbf{E}\left[\mathbf{E}\left[h_{\gamma}\left(\left|\psi_{Z}(\tilde{X}_{\alpha})-\psi_{Z}(\tilde{X}_{\beta})\right|\right)\,\big|\,Z\right]\right]\\
 & =\mathbf{E}\left[\mathbf{E}\left[h_{\gamma}\left(\left|\psi_{Z}(\tilde{X}_{\alpha})-\psi_{Z}(\tilde{X}_{\beta})\right|\right)\,\big|\,\tilde{X}_{\alpha},\tilde{X}_{\beta}\right]\right]\\
 & =\mathbf{E}\left[d(\tilde{X}_{\alpha},\tilde{X}_{\beta})h_{\gamma}\left(1-d(\tilde{X}_{\alpha},\tilde{X}_{\beta})\right)+(1-d(\tilde{X}_{\alpha},\tilde{X}_{\beta}))h_{\gamma}\left(d(\tilde{X}_{\alpha},\tilde{X}_{\beta})\right)\right]\\
 & \stackrel{(b)}{\le}\mathbf{E}\left[d(\tilde{X}_{\alpha},\tilde{X}_{\beta})q\gamma^{1-1/q}\left(1-d(\tilde{X}_{\alpha},\tilde{X}_{\beta})\right)\right]+\mathbf{E}\left[(d(\tilde{X}_{\alpha},\tilde{X}_{\beta}))^{q}\right]\\
 & \le q\gamma^{1-1/q}\mathbf{E}\left[d(\tilde{X}_{\alpha},\tilde{X}_{\beta})\right]+\mathbf{E}\left[(d(\tilde{X}_{\alpha},\tilde{X}_{\beta}))^{q}\right],
\end{align*}
where (a) is by the optimality of the quantile coupling for convex
costs, and (b) is by $h_{\gamma}(x)\le\min\{x^{q},q\gamma^{1-1/q}x\}$.
Substituting $\gamma=\eta C_{c}^{*}(P_{\alpha},P_{\beta})$ and taking
the infimum over $(\tilde{X}_{\alpha},\tilde{X}_{\beta})$, we have
\begin{align*}
 & \mathbf{E}\left[\min\left\{ c(X_{\alpha},X_{\beta}),\,\eta C_{c}^{*}(P_{\alpha},P_{\beta})\right\} \right]\\
 & \le\inf_{(\tilde{X}_{\alpha},\tilde{X}_{\beta})\in\Gamma_{\lambda}(P_{\alpha},P_{\beta})}\left(q\left(\eta C_{c}^{*}(P_{\alpha},P_{\beta})\right)^{1-1/q}\mathbf{E}\left[d(\tilde{X}_{\alpha},\tilde{X}_{\beta})\right]+\mathbf{E}\left[(d(\tilde{X}_{\alpha},\tilde{X}_{\beta}))^{q}\right]\right)\\
 & =\inf_{(\tilde{X}_{\alpha},\tilde{X}_{\beta})\in\Gamma_{\lambda}(P_{\alpha},P_{\beta})}C_{c}^{*}(P_{\alpha},P_{\beta})\left(q\eta^{1-1/q}\frac{\mathbf{E}[d(\tilde{X}_{\alpha},\tilde{X}_{\beta})]}{(C_{c}^{*}(P_{\alpha},P_{\beta}))^{1/q}}+\frac{\mathbf{E}[(d(\tilde{X}_{\alpha},\tilde{X}_{\beta}))^{q}]}{C_{c}^{*}(P_{\alpha},P_{\beta})}\right)\\
 & \le\inf_{(\tilde{X}_{\alpha},\tilde{X}_{\beta})\in\Gamma_{\lambda}(P_{\alpha},P_{\beta})}C_{c}^{*}(P_{\alpha},P_{\beta})\left(q\eta^{1-1/q}\frac{(\mathbf{E}[(d(\tilde{X}_{\alpha},\tilde{X}_{\beta}))^{q}])^{1/q}}{(C_{c}^{*}(P_{\alpha},P_{\beta}))^{1/q}}+\frac{\mathbf{E}[(d(\tilde{X}_{\alpha},\tilde{X}_{\beta}))^{q}]}{C_{c}^{*}(P_{\alpha},P_{\beta})}\right)\\
 & =(q\eta^{1-1/q}+1)C_{c}^{*}(P_{\alpha},P_{\beta}).
\end{align*}

\medskip{}

\section{Proof of Proposition \ref{prop:finite_col} \label{subsec:pf_finite_col}}

We use the result in \cite{fakcharoenphol2004tight} to construct
a random tree over $\mathcal{A}$, and use the tree to construct the
coupling. We remark that the same strategy is used in \cite{archer2004approximate}.
We include the proof here for the sake of completeness.

We invoke the result in \cite{fakcharoenphol2004tight}, which states
that for any finite metric space $(\mathcal{Y},d_{\mathcal{Y}})$,
there exists a random laminar family $T_{U}\subseteq2^{\mathcal{Y}}$
(i.e., for any $v_{1},v_{2}\in T_{U}$, either $v_{1}\subseteq v_{2}$,
$v_{2}\subseteq v_{1}$, or $v_{1}\cap v_{2}=\emptyset$) indexed
by a random variable $U\sim\varpi_{U}$ (where $\varpi_{U}$ is a
probability distribution over $\mathbb{N}$ with finite support\footnote{We can assume finiteness since the number of choices of $T_{U}\subseteq2^{\mathcal{Y}}$
is finite.}), satisfying that $\mathcal{Y}\in T_{U}$ (the root node), $\{y\}\in T_{U}$
for all $y\in\mathcal{Y}$ (the leaf nodes), and the property in \eqref{eq:pf_finite_col_treebd}
below. We can regard $T_{U}$ as a rooted tree where $v_{1}\in T_{U}$
is the parent of $v_{2}\in T_{U}$ if $v_{2}\subsetneqq v_{1}$ is
a maximal subset (i.e., there does not exist $v_{3}\in T_{U}$ such
that $v_{2}\subsetneqq v_{3}\subsetneqq v_{1}$), and the edge $(v_{1},v_{2})$
has length $\mathrm{diam}(v_{1})/2$ (the diameter is with respect
to $d_{\mathcal{Y}}$). Let $d_{T_{U}}(y_{1},y_{2})$ be the length
of the path connecting the leaf nodes $\{y_{1}\}$ and $\{y_{2}\}$
in the tree $T_{U}$. Note that $d_{T_{U}}(y_{1},y_{2})\ge d_{\mathcal{Y}}(y_{1},y_{2})$
since the lowest common ancestor of $\{y_{1}\}$ and $\{y_{2}\}$
has diameter at least $d_{\mathcal{Y}}(y_{1},y_{2})$. The random
tree $T_{U}$ in \cite{fakcharoenphol2004tight} satisfies that, for
any $y_{1},y_{2}\in\mathcal{Y}$,
\begin{equation}
\mathbf{E}_{U}[d_{T_{U}}(y_{1},y_{2})]\le\frac{8\log|\mathcal{Y}|}{\log2}d_{\mathcal{Y}}(y_{1},y_{2}).\label{eq:pf_finite_col_treebd}
\end{equation}

We now use this result to prove Proposition \ref{prop:finite_col}.
Consider the metric space $(\mathcal{A},\,(\alpha,\beta)\mapsto C_{c}^{*}(P_{\alpha},P_{\beta}))$.
Let $T_{U}\subseteq2^{\mathcal{A}}$ satisfy the aforementioned conditions.
For any $u\in\mathrm{supp}(\varpi_{U})$ and $\alpha,\beta\in\mathcal{A}$,
define $v_{u,\alpha,\beta}$ to be the lowest common ancestor of $\{\alpha\}$
and $\{\beta\}$ in $T_{u}$, and define $w_{u,\alpha,\beta,0},\ldots,w_{u,\alpha,\beta,l_{u,\alpha,\beta}}$
to be the path connecting $v_{u,\alpha,\beta}$ and $\{\alpha\}$
in $T_{u}$ (where $w_{u,\alpha,\beta,0}=v_{u,\alpha,\beta}$ and
$w_{u,\alpha,\beta,l_{u,\alpha,\beta}}=\{\alpha\}$). Note that the
path connecting $\{\alpha\}$ and $\{\beta\}$ in $T_{u}$ is $\{\alpha\}=w_{u,\alpha,\beta,l_{u,\alpha,\beta}},w_{u,\alpha,\beta,l_{u,\alpha,\beta}-1},\ldots,w_{u,\alpha,\beta,1},\,w_{u,\alpha,\beta,0}=w_{u,\beta,\alpha,0}=v_{u,\alpha,\beta},\,w_{u,\beta,\alpha,1},\ldots,w_{u,\beta,\alpha,l_{u,\beta,\alpha}}=\{\beta\}$.

Fix any $\epsilon>0$. For any $\alpha,\beta\in\mathcal{A}$, let
$Q_{\alpha,\beta}\in\Gamma(P_{\alpha},P_{\beta})$ such that $\mathbf{E}_{(X,Y)\sim Q_{\alpha,\beta}}[c(X,Y)]\le(1+\epsilon)C_{c}^{*}(P_{\alpha},P_{\beta})$,
and write $Q_{\alpha|\beta}$ for the conditional distribution of
$X$ given $Y$ when $(X,Y)\sim Q_{\alpha,\beta}$. Define random
variables $X_{u,v}\in\mathcal{X}$ for $u\in\mathrm{supp}(\varpi_{U})$,
$v\in T_{u}$ recursively as $X_{u,\mathcal{A}}\sim P_{\mathrm{cen}(\mathcal{A})}$
(where $\mathrm{cen}(v):=\arg\min_{y\in v}\max_{y'\in v}d_{\mathcal{Y}}(y,y')$
with arbitrary tie-breaking), and for any node $v_{2}$ with parent
$v_{1}$ in $T_{u}$, let $X_{u,v_{2}}|X_{u,v_{1}}\sim Q_{\mathrm{cen}(v_{2})|\mathrm{cen}(v_{1})}$.
The random vectors $\{X_{u,v}\}_{v\in T_{u}}$ are coupled arbitrarily
across different $u\in\mathrm{supp}(\varpi_{U})$. We have
\begin{align*}
 & \mathbf{E}[c(X_{u,\{\alpha\}},X_{u,\{\beta\}})]\\
 & \stackrel{(a)}{\le}\sum_{i=1}^{l_{u,\alpha,\beta}}\mathbf{E}[c(X_{u,w_{u,\alpha,\beta,i-1}},X_{u,w_{u,\alpha,\beta,i}})]+\sum_{i=1}^{l_{u,\beta,\alpha}}\mathbf{E}[c(X_{u,w_{u,\beta,\alpha,i-1}},X_{u,w_{u,\beta,\alpha,i}})]\\
 & =\sum_{i=1}^{l_{u,\alpha,\beta}}\mathbf{E}_{(X,Y)\sim Q_{\mathrm{cen}(w_{u,\alpha,\beta,i-1}),\,\mathrm{cen}(w_{u,\alpha,\beta,i})}}[c(X,Y)]\\
 & \;\;\;\;\;+\sum_{i=1}^{l_{u,\beta,\alpha}}\mathbf{E}_{(X,Y)\sim Q_{\mathrm{cen}(w_{u,\beta,\alpha,i-1}),\,\mathrm{cen}(w_{u,\beta,\alpha,i})}}[c(X,Y)]\\
 & \le\sum_{i=1}^{l_{u,\alpha,\beta}}(1+\epsilon)C_{c}^{*}(P_{\mathrm{cen}(w_{u,\alpha,\beta,i-1})},P_{\mathrm{cen}(w_{u,\alpha,\beta,i})})\\
 & \;\;\;\;\;+\sum_{i=1}^{l_{u,\beta,\alpha}}(1+\epsilon)C_{c}^{*}(P_{\mathrm{cen}(w_{u,\beta,\alpha,i-1})},P_{\mathrm{cen}(w_{u,\beta,\alpha,i})})\\
 & \stackrel{(b)}{\le}\sum_{i=1}^{l_{u,\alpha,\beta}}(1+\epsilon)\mathrm{diam}(w_{u,\alpha,\beta,i-1})+\sum_{i=1}^{l_{u,\beta,\alpha}}(1+\epsilon)\mathrm{diam}(w_{u,\beta,\alpha,i-1})\\
 & =2(1+\epsilon)d_{T_{u}}(\alpha,\beta),
\end{align*}
where (a) is by the triangle inequality, and (b) is because $\mathrm{cen}(w_{u,\alpha,\beta,i-1}),\mathrm{cen}(w_{u,\alpha,\beta,i})\in w_{u,\alpha,\beta,i-1}$.
We then define $X_{\alpha}:=X_{U,\{\alpha\}}$, where $U\sim\varpi_{U}$.
We have
\begin{align*}
\mathbf{E}[c(X_{\alpha},X_{\beta})] & =\mathbf{E}\left[\mathbf{E}[c(X_{U,\{\alpha\}},X_{U,\{\beta\}})\,|\,U]\right]\\
 & \le\mathbf{E}[2(1+\epsilon)d_{T_{U}}(\alpha,\beta)]\\
 & \stackrel{(a)}{\le}2(1+\epsilon)\frac{8\log|\mathcal{A}|}{\log2}C_{c}^{*}(P_{\alpha},P_{\beta})\\
 & \le23.09(1+\epsilon)(\log|\mathcal{A}|)C_{c}^{*}(P_{\alpha},P_{\beta}),
\end{align*}
where (a) is by \eqref{eq:pf_finite_col_treebd}. The result follows
from letting $\epsilon\to0$.

\[
\]

\medskip{}

\[
\]

\bibliographystyle{IEEEtran}
\bibliography{ref}

\end{document}